\documentclass[reqno]{memo-l}
\usepackage{amssymb,amsmath,amsthm,stmaryrd,enumerate,geometry}
\usepackage[all]{xy} 
\usepackage{hyperref}
\newdimen\proofrulebreadth \proofrulebreadth=.05em
\newdimen\proofdotseparation \proofdotseparation=1.25ex
\newdimen\proofrulebaseline \proofrulebaseline=2ex
\newcount\proofdotnumber \proofdotnumber=3
\let\then\relax
\def\hfi{\hskip0pt plus.0001fil}
\mathchardef\squigto="3A3B
\newif\ifinsideprooftree\insideprooftreefalse
\newif\ifonleftofproofrule\onleftofproofrulefalse
\newif\ifproofdots\proofdotsfalse
\newif\ifdoubleproof\doubleprooffalse
\let\wereinproofbit\relax
\newdimen\shortenproofleft
\newdimen\shortenproofright
\newdimen\proofbelowshift
\newbox\proofabove
\newbox\proofbelow
\newbox\proofrulename
\def\shiftproofbelow{\let\next\relax\afterassignment\setshiftproofbelow\dimen0 }
\def\shiftproofbelowneg{\def\next{\multiply\dimen0 by-1 }%
\afterassignment\setshiftproofbelow\dimen0 }
\def\setshiftproofbelow{\next\proofbelowshift=\dimen0 }
\def\setproofrulebreadth{\proofrulebreadth}

\def\prooftree{
\ifnum  \lastpenalty=1
\then   \unpenalty
\else   \onleftofproofrulefalse
\fi
\ifonleftofproofrule
\else   \ifinsideprooftree
        \then   \hskip.5em plus1fil
        \fi
\fi
\bgroup
\setbox\proofbelow=\hbox{}\setbox\proofrulename=\hbox{}%
\let\justifies\proofover\let\leadsto\proofoverdots\let\Justifies\proofoverdbl
\let\using\proofusing\let\[\prooftree
\ifinsideprooftree\let\]\endprooftree\fi
\proofdotsfalse\doubleprooffalse
\let\thickness\setproofrulebreadth
\let\shiftright\shiftproofbelow \let\shift\shiftproofbelow
\let\shiftleft\shiftproofbelowneg
\let\ifwasinsideprooftree\ifinsideprooftree
\insideprooftreetrue
\setbox\proofabove=\hbox\bgroup$\displaystyle 
\let\wereinproofbit\prooftree
\shortenproofleft=0pt \shortenproofright=0pt \proofbelowshift=0pt
\onleftofproofruletrue\penalty1
}

\def\eproofbit{
\ifx    \wereinproofbit\prooftree
\then   \ifcase \lastpenalty
        \then   \shortenproofright=0pt  
        \or     \unpenalty\hfil         
        \or     \unpenalty\unskip       
        \else   \shortenproofright=0pt  
        \fi
\fi
\global\dimen0=\shortenproofleft
\global\dimen1=\shortenproofright
\global\dimen2=\proofrulebreadth
\global\dimen3=\proofbelowshift
\global\dimen4=\proofdotseparation
\global\count255=\proofdotnumber
$\egroup  
\shortenproofleft=\dimen0
\shortenproofright=\dimen1
\proofrulebreadth=\dimen2
\proofbelowshift=\dimen3
\proofdotseparation=\dimen4
\proofdotnumber=\count255
}

\def\proofover{
\eproofbit 
\setbox\proofbelow=\hbox\bgroup 
\let\wereinproofbit\proofover
$\displaystyle
}%
\def\proofoverdbl{
\eproofbit 
\doubleprooftrue
\setbox\proofbelow=\hbox\bgroup 
\let\wereinproofbit\proofoverdbl
$\displaystyle
}%
\def\proofoverdots{
\eproofbit 
\proofdotstrue
\setbox\proofbelow=\hbox\bgroup 
\let\wereinproofbit\proofoverdots
$\displaystyle
}%
\def\proofusing{
\eproofbit 
\setbox\proofrulename=\hbox\bgroup 
\let\wereinproofbit\proofusing
\kern0.3em$
}

\def\endprooftree{
\eproofbit 
  \dimen5 =0pt
\dimen0=\wd\proofabove \advance\dimen0-\shortenproofleft
\advance\dimen0-\shortenproofright
\dimen1=.5\dimen0 \advance\dimen1-.5\wd\proofbelow
\dimen4=\dimen1
\advance\dimen1\proofbelowshift \advance\dimen4-\proofbelowshift
\ifdim  \dimen1<0pt
\then   \advance\shortenproofleft\dimen1
        \advance\dimen0-\dimen1
        \dimen1=0pt
        \ifdim  \shortenproofleft<0pt
        \then   \setbox\proofabove=\hbox{%
                        \kern-\shortenproofleft\unhbox\proofabove}%
                \shortenproofleft=0pt
        \fi
\fi
\ifdim  \dimen4<0pt
\then   \advance\shortenproofright\dimen4
        \advance\dimen0-\dimen4
        \dimen4=0pt
\fi
\ifdim  \shortenproofright<\wd\proofrulename
\then   \shortenproofright=\wd\proofrulename
\fi
\dimen2=\shortenproofleft \advance\dimen2 by\dimen1
\dimen3=\shortenproofright\advance\dimen3 by\dimen4
\ifproofdots
\then
        \dimen6=\shortenproofleft \advance\dimen6 .5\dimen0
        \setbox1=\vbox to\proofdotseparation{\vss\hbox{$\cdot$}\vss}%
        \setbox0=\hbox{%
                \advance\dimen6-.5\wd1
                \kern\dimen6
                $\vcenter to\proofdotnumber\proofdotseparation
                        {\leaders\box1\vfill}$%
                \unhbox\proofrulename}%
\else   \dimen6=\fontdimen22\the\textfont2 
        \dimen7=\dimen6
        \advance\dimen6by.5\proofrulebreadth
        \advance\dimen7by-.5\proofrulebreadth
        \setbox0=\hbox{%
                \kern\shortenproofleft
                \ifdoubleproof
                \then   \hbox to\dimen0{%
                        $\mathsurround0pt\mathord=\mkern-6mu%
                        \cleaders\hbox{$\mkern-2mu=\mkern-2mu$}\hfill
                        \mkern-6mu\mathord=$}%
                \else   \vrule height\dimen6 depth-\dimen7 width\dimen0
                \fi
                \unhbox\proofrulename}%
        \ht0=\dimen6 \dp0=-\dimen7
\fi
\let\doll\relax
\ifwasinsideprooftree
\then   \let\VBOX\vbox
\else   \ifmmode\else$\let\doll=$\fi
        \let\VBOX\vcenter
\fi
\VBOX   {\baselineskip\proofrulebaseline \lineskip.2ex
        \expandafter\lineskiplimit\ifproofdots0ex\else-0.6ex\fi
        \hbox   spread\dimen5   {\hfi\unhbox\proofabove\hfi}%
        \hbox{\box0}%
        \hbox   {\kern\dimen2 \box\proofbelow}}\doll%
\global\dimen2=\dimen2
\global\dimen3=\dimen3
\egroup 
\ifonleftofproofrule
\then   \shortenproofleft=\dimen2
\fi
\shortenproofright=\dimen3
\onleftofproofrulefalse
\ifinsideprooftree
\then   \hskip.5em plus 1fil \penalty2
\fi
}

\SelectTips{cm}{}
\newdir{ >}{{}*!/-7pt/@{>}}
\newdir{m}{->}
\newcommand{\xycenter}[1]{\vcenter{\hbox{\xymatrix{#1}}}}

\newcommand{\ie}{\text{i.e.\ }}
\newcommand{\eg}{\text{e.g.}}
\newcommand{\resp}{\text{resp.\ }}
\newcommand{\myemph}{\textit} 

\numberwithin{equation}{section}

\numberwithin{section}{chapter}

\newtheorem{theorem}{Theorem}[section]
\newtheorem*{theorem*}{Theorem}
\newtheorem{lemma}[theorem]{Lemma} 
\newtheorem{proposition}[theorem]{Proposition} 
\newtheorem{corollary}[theorem]{Corollary}

\theoremstyle{definition}
\newtheorem{definition}[theorem]{Definition}	
\newtheorem*{definition*}{Definition}	

\theoremstyle{remark}

\newtheorem{remark}[theorem]{Remark} 
\newtheorem*{remark*}{Remark} 
\newtheorem{example}[theorem]{Example}

\newtheorem*{example*}{Example}
\newtheorem*{examples*}{Examples}

\newcommand{\defeq}{=_{\mathrm{def}}}
\newcommand{\co}{\colon}
\newcommand{\iso}{\cong} 

\newcommand{\op}{\mathrm{op}}

\newcommand{\coend}{\int}  
\newcommand{\mat}{\to}
\newcommand{\sym}{\to}

\newcommand{\HOM}{\mathrm{HOM}}
\newcommand{\REG}{\mathrm{REG}}

\newcommand{\Nat}{\mathbb{N}}
\newcommand{\Ob}{\mathrm{Obj}}
\newcommand{\id}{\mathrm{Id}}

\newcommand{\yon}{\mathrm{y}}

\newcommand{\Mon}{\mathrm{Mon}}
\newcommand{\CMon}{\mathrm{CMon}}

\newcommand{\ev}{\mathrm{ev}}
\newcommand{\Alg}{\mathrm{Alg}}

\newcommand{\colim}{\operatorname{colim}}

\newcommand{\overbar}[1]{\mkern 1.5mu\overline{\mkern-1.5mu#1\mkern-1.5mu}\mkern 1.5mu}
\renewcommand{\vec}[1]{\overbar{#1}}

\newcommand{\tensorvcat}{\otimes}
\newcommand{\tensorloc}{\mathbin{\widehat{\otimes}}}

\newcommand{\conv}{\ast}

\newcommand{\obj}[1]{#1}
\newcommand{\objX}{\obj{X}}

\newcommand{\objY}{\obj{Y}}

\newcommand{\objU}{\obj{U}}
\newcommand{\objV}{\obj{V}}
\newcommand{\objW}{\obj{W}}
\newcommand{\objZ}{\obj{Z}}
\newcommand{\objC}{\obj{C}}
\newcommand{\objK}{\obj{K}}

\newcommand{\Obj}{\mathrm{Obj}}

\newcommand{\cat}[1]{\mathbb{#1}}
\newcommand{\catA}{\cat{A}}
\newcommand{\catB}{\cat{B}}
\newcommand{\catC}{\cat{C}}
\newcommand{\catD}{\cat{D}}
\newcommand{\catK}{\cat{K}}
\newcommand{\catM}{\cat{M}}
\newcommand{\catN}{\cat{N}}
\newcommand{\catP}{\cat{P}}

\newcommand{\catX}{\cat{X}}
\newcommand{\catY}{\cat{Y}}
\newcommand{\catZ}{\cat{Z}}

\newcommand{\catW}{\cat{W}}

\newcommand{\moncatV}{\mathcal{V}}
\newcommand{\catV}{\mathcal{V}}

\newcommand{\psh}{P}
\newcommand{\pshX}{\psh({\catX})}
\newcommand{\pshY}{\psh({\catY})}

\newcommand{\pshM}{\psh({\catM})}
\newcommand{\pshN}{\psh({\catN})}
\newcommand{\pshA}{\psh({\catA})}
\newcommand{\pshB}{\psh({\catB})}

\newcommand{\bcat}[1]{\mathcal{#1}}
\newcommand{\bcatE}{\bcat{E}}
\newcommand{\bcatF}{\bcat{F}}
\newcommand{\bcatG}{\bcat{G}}

\newcommand{\pcat}[1]{\mathcal{#1}}
\newcommand{\pcatE}{\pcat{E}}
\newcommand{\pcatF}{\pcat{F}}

\newcommand{\rigR}{\mathcal{R}}
\newcommand{\rigS}{\mathcal{S}}

\newcommand{\concretecatfont}{\mathrm}
\newcommand{\zerocat}{\concretecatfont{0}}
\newcommand{\onecat}{\concretecatfont{1}}
\newcommand{\Scat}{\concretecatfont{S}}
\newcommand{\Set}{\concretecatfont{Set}}
\newcommand{\Ab}{\concretecatfont{Ab}}
\newcommand{\UnitCat}{\concretecatfont{I}}

\newcommand{\bicatfont}{\mathrm}

\newcommand{\Cat}{\bicatfont{Cat}}
\newcommand{\CAT}{\bicatfont{CAT}}
\newcommand{\CCAT}{\bicatfont{CCAT}}
\newcommand{\SMCat}{\bicatfont{SMonCat}}

\newcommand{\CatV}{\Cat_{\catV}}
\newcommand{\CATV}{\CAT_{\catV}}
\newcommand{\CCATV}{\CCAT_{\catV}}
\newcommand{\PCATV}{\bicatfont{PCAT}_{\catV}}

\newcommand{\lax}{\mathrm{lax}}

\newcommand{\LMCatV}{\bicatfont{MonCat}_\catV^\lax}
\newcommand{\LMCATV}{\bicatfont{MONCAT}_\catV^\lax}
\newcommand{\LSMCatV}{\bicatfont{SMonCat}_\catV^\lax}
\newcommand{\LSMCATV}{\bicatfont{SMONCAT}_\catV^\lax}

\newcommand{\LSCatV}{\bicatfont{SMonCat}_\catV^\lax}

\newcommand{\MCatV}{\bicatfont{MonCat}_\catV}

\newcommand{\SMCatV}{\bicatfont{SMonCat}_\catV}

\newcommand{\RigV}{\bicatfont{Rig}_\catV}
\newcommand{\SRigV}{\bicatfont{SRig}_\catV}
\newcommand{\LRigV}{\bicatfont{Rig}_\catV^\lax}
\newcommand{\LSRigV}{\bicatfont{SRig}_\catV^\lax}

\newcommand{\DistV}{\bicatfont{Dist}_\catV}
\newcommand{\MDistV}{\bicatfont{MonDist}_\catV}
\newcommand{\SMDistV}{\bicatfont{SMonDist}_\catV}
\newcommand{\LMDistV}{\bicatfont{MonDist}_\catV^\lax}
\newcommand{\LSMDistV}{\bicatfont{SMonDist}_\catV^\lax}

\newcommand{\SDistV}{S\text{-}\bicatfont{Dist}_\catV}
\newcommand{\MatV}{\bicatfont{Mat}_\catV}
\newcommand{\SMatV}{S\text{-}\bicatfont{Mat}_\catV}
\newcommand{\SymV}{\bicatfont{Sym}_\catV}
\newcommand{\CatSymV}{\bicatfont{CatSym}_\catV}
\newcommand{\OpdV}{\bicatfont{OpdBim}_\catV}
\newcommand{\OpdR}{\bicatfont{OpdBim}_\rigR}

\newcommand{\Mnd}{\mathrm{Mnd}}
\newcommand{\MndE}{\Mnd(\bcatE)}

\newcommand{\EM}{\mathrm{EM}}
\newcommand{\emE}{\EM(\bcatE)}
\newcommand{\Res}{\mathrm{Res}}

\newcommand{\Bim}{\mathrm{Bim}}
\newcommand{\bimE}{\Bim(\bcatE)}
\newcommand{\bimF}{\Bim(\bcatF)}

\newcommand{\Id}{\mathrm{Id}}

\newcommand{\myJ}{\mathrm{J}}
\newcommand{\JE}{\myJ_\bcatE}
\newcommand{\JF}{\myJ_\bcatF}

\author[]{Nicola Gambino}
\address{School of Mathematics, University of Leeds, Leeds LS2 9JT, United Kingdom}

\email{n.gambino@leeds.ac.uk}

\author[]{Andr\'e Joyal}
\address{Department de Math\`ematiques, Universit\`e du Qu\`ebec \`a Montr\'eal,
 Case Postale 8888, Succursale Centre-Ville, Mont\'real (Qu\`ebec) H3C 3P8,
 Canada}

\email{joyal.andre@uqam.ca}

\title[]{On operads, bimodules and analytic functors}



\date{June 4th, 2014; Revised February 9th, 2015; Revised May 24th, 2015}

\begin{document}

\frontmatter

\begin{abstract}
We develop further the theory of operads and analytic functors. In particular, we introduce the bicategory $\OpdV$ of operad bimodules, that has operads as $0$-cells, operad bimodules as $1$-cells and operad bimodule maps as 2-cells, and prove that it is cartesian closed. In order to obtain this result, we extend the theory of distributors and the formal theory of monads. 
\end{abstract}

\maketitle

\tableofcontents

\mainmatter

\chapter*{Introduction}

\subsection*{Operads}
Operads originated in algebraic topology~\cite{BoardmanJ:homias,MayJP:geoils} and, while they continue to be very important in that area 
(see, \eg,~\cite{BergerC:axihto,ElmendorfAD:rinma,ElmendorfAD:percma,RezkC:spaasc}), they have found applications in  several other branches of 
mathematics, including  geometry~\cite{GinzburgV:kosdo,KapranovM:modmto},  algebra~\cite{LodayJL:algo,MillesJ:andqca},
combinatorics~\cite{AguiarM:monfsh,LivernetM:frolma} and category 
theory~\cite{BaezJ:higdan,BataninM:mongcn,LeinsterT:higohc}. Recently, operads have started to be considered also in
theoretical computer science~\cite{CurienPL:opecdl}. See~\cite{FresseB:modoof,GetzlerE:oper,MarklM:opeatp,SmirnovV:simomat} for recent accounts of the theory of operads.

There are several variants of operads in the literature (symmetric or non-symmetric, many-sorted or single-sorted, enriched or non-enriched). Here, by an operad we mean a many-sorted (sometimes called coloured) symmetric operad,  enriched over a fixed symmetric monoidal closed presentable category $\catV$. Thus, we call operad what other authors call a $\catV$-enriched symmetric multicategory. 
An operad $A$ has a set of objects
$|A|$ and, for every tuple~$(x_1, \ldots, x_n, x) \in |A|^{n+1}$, an object $A[x_1, \ldots, x_n; x] \in \catV$ of operations with inputs of 
sort~$x_1, \ldots, x_n$ and output of sort $x$. 
An operation $f \in A[x_1, x_2, x_3; x]$, \ie an arrow $f \co I \to A[x_1, x_2, x_3; x]$ in $\catV$, 
can be represented as the tree 
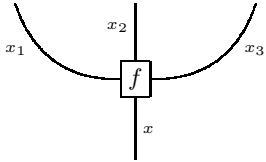
\begin{small}
\[
\vcenter{\hbox{\xymatrix@C=0.3cm@R=0.8cm{
            &  \ar@/_1pc/@{-}[drrr]_(0.3){x_1}       &  & &   \ar@{-}[d]_(0.4){x_2}   &  & &   \ar@/^1pc/@{-}[dlll]^(0.3){x_3} \\
            &   &  & & *+[F] {f} \ar@{-}[d]^(0.6){x}   &  & & \\ 
            &   &  &  & \quad & &   }  }} 
\] 
\end{small}
Even if single-sorted operads suffice for many purposes, many-sorted operads are essential for some applications, such as those 
developed in~\cite{ElmendorfAD:rinma,ElmendorfAD:percma}. However, the theory of operads is less developed than that of
single-sorted operads. For example, recent monographs on operads~\cite{FresseB:modoof,MarklM:opeatp} deal almost 
exclusively with single-sorted operads. 
 One of the original motivations for this work was to develop further
the theory of operads, so as to facilitate  applications. As we will see, dropping the restriction of working with 
single-sorted operads not only does not create problems, but in fact brings to light new mathematical structures. 

Operads are typically studied for their categories of algebras. For an operad $A$, an $A$-algebra
 consists of a family $T \in \catV^{|A|}$ equipped with a left~$A$-action, which can be represented 
 as a composition law for trees of the form
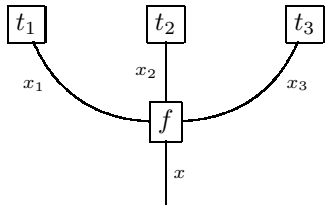
\[
\vcenter{\hbox{\xymatrix@C=0.3cm@R=0.8cm{
            & *+[F] {t_1}  \ar@/_1pc/@{-}[drrr]_(0.3){x_1}       &  & &   *+[F] {t_2} \ar@{-}[d]_(0.5){x_2}   &  & &   *+[F] {t_3} \ar@/^1pc/@{-}[dlll]^(0.3){x_3} \\
            &   &  & & *+[F] {f} \ar@{-}[d]^(0.6){x}   &  & & \\ 
            &   &  &  & \quad & &   }  }} 
\] 
where $f \in A[x_1, x_2, x_3; x]$ and $t_i \in T(x_i)$ for $i = 1, 2, 3$. One can consider 
algebras not only in~$\catV$, but also in any symmetric $\catV$-rig, \ie symmetric monoidal closed
presentable $\catV$-category $\rigR = (\rigR, \diamond, e)$. For an operad $A$, we write $\Alg_\rigR(A)$ for the category of 
$A$-algebras in~$\rigR$. The forgeful functor $U \co \Alg_\rigR(A) \to \rigR^{|A|}$ has a left adjoint $F \co \rigR^{|A|} \to \Alg_\rigR(A)$, the free $A$-algebra functor. 
Informally, we may view the category $\rigR^{|A|}$ as an affine variety, the operad~$A$ as a system of equations, and 
the category $\Alg_\rigR(A)$ as the sub-variety of~$\rigR^{|A|}$ defined by the system of equations $A$.
The left action on an $A$-algebra 
The notions of an operad bimodule and operad bimodule map subsume those of
algebra and algebra map, as we explain below, and play an
important role in the theory of operads~\cite{DwyerW:boavtp,FresseB:modoof,KapranovM:modmto,MarklM:modo,RezkC:spaasc}. 
   Explicitly, for operads $A$ and $B$, a $(B,A)$-bimodule in a $\catV$-rig $\rigR$ is a family of objects~$M[x_1, \ldots, x_n; y] \in \rigR$, indexed by 
sequences~$(x_1, \ldots, x_n) \in |A|^n$  and~$y \in |B|$, subject to suitable functoriality conditions, equipped with a left~$B$-action and a right $A$-action that commute with each other. Informally, an element~$m \in M[x_1, x_2; y]$ may be represented as a tree of the form
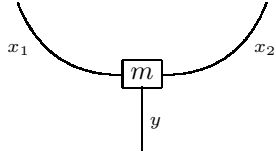
\[
\vcenter{\hbox{\xymatrix@C=0.3cm@R=0.8cm{
            &  \ar@/_1pc/@{-}[drrr]_(0.3){x_1}       &  & &      &  & &   \ar@/^1pc/@{-}[dlll]^(0.3){x_2} \\
            &   &  & & *+[F] {m} \ar@{-}[d]^(0.6){y}   &  & & \\
            &   &  &  & \quad & &   }  }}
\]
The right $A$-action can be viewed as composition operation for trees of the form
 \[
         \smallskip
\vcenter
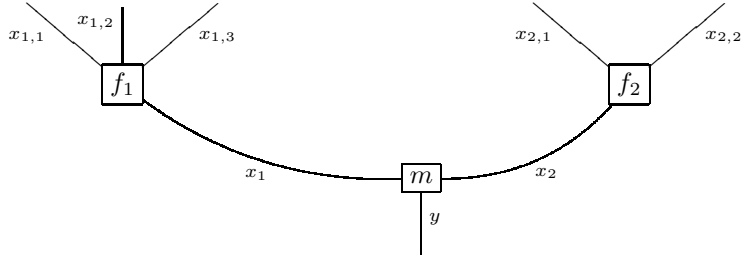
{\hbox{\xymatrix@C=1cm@R=0.8cm{
\ar@{-}[dr]_(0.3){x_{1,1}} &  \ar@{-}[d]_(0.3){x_{1,2}} & \ar@{-}[dl]^(0.3){x_{1,3}}  &     & & \ar@{-}[dr]_(0.3){x_{2, 1}} & &  \ar@{-}[dl]^(0.3){x_{2,2}}     \\
            & *+[F]{f_1}    \ar@/_1pc/@{-}[drrr]_{x_1} &  & &  & & *+[F]{f_2}  \ar@/^1pc/@{-}[dll]^{x_2}  &   \\
            &   &  & & *+[F]{m} \ar@{-}[d]^y   &  & & \\
            &   &  & & \quad &  & &   } }}
            \]
where $f_1 \in A[x_{1,1}, x_{1,2}, x_{1,3}; x_1]$ and $f_2 \in A[x_{2,1}, x_{2,2}; x_2]$, while
 the left $B$-action as a composition operation for trees of the form 
 \[
         \smallskip
\vcenter
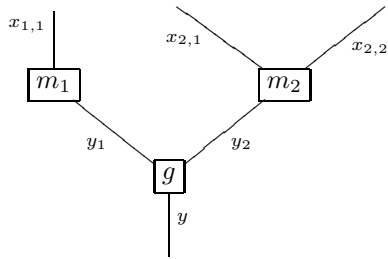
{\hbox{\xymatrix@C=1cm@R=0.8cm{
      \ar@{-}[d]_(0.3){x_{1,1}}    &  \ar@{-}[dr]_(0.3){x_{2,1}} & &     \ar@{-}[dl]^(0.4){x_{2,2}}  \\
                  *+[F]{m_1}    \ar@{-}[dr]_{y_1}  & & *+[F]{m_2}  \ar@{-}[dl]^{y_2}  &    \\
                  & *+[F]{g} \ar@{-}[d]^y     & & \\
                  \quad &  & &   } }}
            \]
            where $m_1 \in M[m_{1,1}; y_1]$, $m_2 \in M[x_{2,1}, x_{2,2}; y_2]$ and $g \in B[y_1, y_2; y]$.
            We write $\OpdR[A,B]$ for the category of $(B,A)$-bimodules and bimodule maps in $\rigR$. As
            we will see, a useful way of understanding operad bimodules is to observe that a $(B,A)$-bimodule $M$ determines a functor 
\[
\Alg_\rigR(M) \co \Alg_\rigR(A) \to \Alg_\rigR(B) \, . 
\]
which fits into the commutative diagram of the form
\[
\xymatrix@C=1.5cm{
\Alg_\rigR(A) \ar[r]^{\Alg_\rigR(M)} & \Alg_\rigR(B) \ar[d]^U\\
\rigR^{|A|} \ar[r]_{M} \ar[u]^F & \rigR^{|B|} \, . }
\]
Here, the vertical arrows are the evident free and forgetful functors and, for $T \in \rigR^{|A|}$ and $y \in |B|$, we have
\[
M(T)(y) \defeq  \int^{\bar{x} \in S(|A|)} M[\bar{x}; y] \otimes T^{\bar{x}}  \, ,
\]
where $S(|A|)$ is the free symmetric monoidal category on $|A|$, which has tuples of elements of $|A|$ as
objects and, for such a tuple $\bar{x} = (x_1, \ldots, x_n)$, we let $T^{\bar{x}} \defeq T(x_1) \diamond \ldots \diamond T(x_n)$.
We refer to such functors  as \emph{analytic functors}, since they are a natural extension the analytic functors originally
introduced in~\cite{JoyalA:fonaes} and their generalization defined in~\cite{FioreM:carcbg}. Examples of  
analytic functors between categories of operad algebras include generalizations of the extension and restriction functors 
for single-sorted operads studied in~\cite{FresseB:modoof}. 

\medskip

One of the upshots of this paper is that
for two operads~$A$ and~$B$, there are new operads $A \sqcap B$ and  $B^A$ such 
that we have natural equivalences 
\[
\Alg_\rigR(A  \sqcap B)  \simeq \Alg_\rigR(A) \times \Alg_\rigR(B) \, , \quad  \Alg_\rigR(B^A)  \simeq \OpdR[A, B] \, .
\]
Thus, $A \sqcap B$ is the operad whose algebras are pairs consisting of an $A$-algebra and a $B$-algebra, while~$B^A$ is the operad whose algebras are $(B,A)$-bimodules. We have~$| A \sqcap B| = |A| \sqcup |B|$, an instance of a duality phenomenon which pervades this paper. Furthermore, there is an operad $\top$ (the operads whose set of sorts is empty) such that, for any operad $A$, there are equivalences 
\[
\OpdR[A, \top] \simeq 1 \, , \quad \OpdR[ \top, A] \simeq \Alg_\rigR(A) \, ,
\] 
where $1$ is the terminal category. Here, note that the second equivalence shows that operad bimodules subsume operad algebras.

\medskip

A fundamental notion in our development is that of a bicategory, introduced in~\cite{BenabouJ:intb}, which generalizes the notion of a monoidal category in the same
way in which the notion of a category generalizes that of a monoid. Indeed, it will be 
very useful for us to regard the notion of an operad as a special case of the general notion 
of a monad in a bicategory~\cite{BenabouJ:intb}, which generalizes the  notion of a monoid in a
monoidal category~\cite{MacLaneS:catwm}. Taking further ideas in~\cite{BaezJ:higdan,FioreM:carcbg}, we introduce a bicategory, 
called the bicategory of symmetric sequences and 
denoted by~$\SymV$, and show that monads therein are exactly operads. This  
extends to
the many-sorted case of the well-known fact that single-sorted operads can be viewed as monoids in the category of 
single-sorted symmetric sequences equipped with the substitution monoidal 
structure~\cite{KellyGM:opejpm,SmirnovV:coccts}. Furthermore, the notions of an
operad bimodule and an operad bimodule map  can be seen as instances of the the general
notions of a bimodule and a bimodule map in a  bicategory~\cite{StreetR:enrcc}.

The characterization of operads, operad bimodules and operad bimodule maps as monads, monad bimodules  and monad 
bimodule maps in the bicategory~$\SymV$, makes it natural to assemble them  in a bicategory, called the bicategory of operad bimodules
and denoted by~$\OpdV$, using the so-called bimodule construction~\cite{StreetR:enrcc}.
This construction takes what we call a tame bicategory  $\bcatE$ (\ie a bicategory whose hom-categories have reflexive coequalizers and whose composition functors preserve coequalizers in each variable) and returns  a new 
 bicategory, called the bicategory of bimodules  in~$\bcatE$ and denoted by~$\bimE$,
which has monads, monad bimodules  and monad bimodule maps in~$\bcatE$ as 
 0-cells, 1-cells and 2-cells, respectively. The composition operation in $\bimE$ is defined using the assumption that $\bcatE$ is 
 tame, generalizing the definition of the  tensor product of ring bimodules.  
We will  prove that the bicategory $\SymV$ of symmetric sequences is a tame bicategory (Corollary~\ref{thm:smattame}),
a result that allows us to define the bicategory of operad bimodules by letting
\[
\OpdV \defeq \Bim(\SymV) \, . 
\]
Remarkably, the  composition of bimodules in~$\OpdV$ obtained by specializing to this case the general definition of composition
in bicategories of bimodules is a generalization to many-sorted operads of the circle-over construction  for single-sorted operads defined in~\cite{RezkC:spaasc}. Furthermore, we will show that composition in~$\OpdV$
corresponds to composition of  analytic functors between categories of operad algebras. 
The category $\SymV$ can be naturally viewed as a sub-bicategory of~$\OpdV$ and in between the two there is a further bicategory,
called the bicategory of categorical symmetric sequences and denoted by~$\CatSymV$, so that we have inclusions:
\[
\SymV \subseteq \CatSymV \subseteq \OpdV \, .
\]
When $\catV = \Set$, the bicategory $\CatSymV$ is exactly the bicategory of generalized species of structures defined in~\cite{FioreM:carcbg}.  

\subsection*{Main results}
 Our main results show that~$\CatSymV$ and~$\OpdV$ are 
cartesian closed bicategories. These facts may be of interest to researchers in mathematical
logic and theoretical computer science, since they provide new models of versions of the simply-typed
$\lambda$-calculus, and contribute to the general program of generalizing domain 
theory~\cite{CattaniGL:proomb,HylandM:somrgd} and to the study of models of the
differential $\lambda$-calculus~\cite{EhrhardT:diflc} inspired by the theory of analytic functors~\cite{FioreM:matmcc}.
One of our original motivation for showing that~$\OpdV$ is cartesian closed 
was to show that for any two operads $A$ and $B$, the category of~$(B,A)$-bimodules can be 
regarded as the category of algebras for an operad. This allows, in particular, to apply to categories
of bimodules the known results concerning the existence of Quillen model structures on categories
of operad algebras~\cite{BergerC:axihto,BergerC:rescorh}, although we do not pursue this idea here. As we will explain below, proving 
that~$\CatSymV$ is cartesian closed can  be considered also as a useful step towards show that~$\OpdV$ is cartesian closed. 
It should be noted that, in order for~$\OpdV$ to be 
cartesian closed, it is essential to consider operads and not just single-sorted operads. Indeed, given 
two single-sorted operads, their exponential in $\OpdV$ is, in general, not single-sorted (see below for more information). We would also like to
mention that our results are intended to contribute to an ongoing research programme, which we like to
refer to as `2-algebraic geometry', that is concerned with the development of a variant of commutative
algebra and algebraic geometry where commutative rings are replaced by symmetric monoidal categories~\cite{BaezJ:higdan,ChirvasituA:funpga}.

 Our first main result (Theorem~\ref{thm:smondistiscc}) is that the bicategory $\CatSymV$ is cartesian closed.
For small $\catV$-categories $\catX$ and $\catY$, their product $\catX \sqcap \catY$ and the exponential $\catY^\catX$ 
in $\CatSymV$ are given
by the formulas:
\[
\catX \sqcap \catY \defeq  \catX \sqcup \catY \, , \quad 
\catY^\catX \defeq S(\catX)^\op \otimes \catY \, ,
\] 
where $\catX \sqcup \catY$ denotes the coproduct of $\catX$ and $\catY$ in the 2-category of small $\catV$-categories. Thus,
when $\catV = \Set$, we obtain the cartesian closed structure of~\cite{FioreM:carcbg}.  An interesting aspect of this result is that 
one obtains a cartesian closed bicategory even when $\catV$ is not a cartesian closed category, but has an arbitrary monoidal closed structure. 
Our approach to the definition of the bicategory~$\CatSymV$ and to the proof that it is cartesian closed differs significantly from the one adopted
in~\cite{FioreM:carcbg}  in the non-enriched case. In particular, its construction  and the proof that it has finite products 
follow  immediately from the results  that we obtain in the first sections of the paper. 
There, we develop further the theory of 
distributors (also known as bimodules or profunctors)~\cite{BenabouJ:dis,LawvereFW:metsgl} and introduce and study the notions of a (lax) monoidal distributor and of a symmetric (lax) monoidal distributor, and show how they can be seen as morphisms of appropriate bicategories.  These bicategories and the bicategory~$\CatSymV$ are defined using the notion of a Gabriel factorization of a homomorphism (which we introduce in Chapter~\ref{cha:bac}), rather than via
the theory of Kleisli bicategories and pseudo-distributive laws~\cite{ChengE:psedl,MarmolejoF:dislpm,MarmolejoF:cohpdl}, 
as done in~\cite{FioreM:carcbg}. We prefer this approach since it allows us to avoid the verification of several coherence conditions. 
Furthermore, our proof that the bicategory~$\CatSymV$ is cartesian closed is broken down and organized into several observations on symmetric monoidal distributors that admit relatively short proofs and does not involve lengthy coend calculations like the proof of the corresponding fact in~\cite{FioreM:carcbg}. 

The second main result in this paper (Theorem~\ref{thm:opdcc}) is that the bicategory $\OpdV$ is cartesian closed. Indeed,
we have already described above the universal properties that characterize products, written $A \sqcap B$,  
exponentials, written $B^A$, and the terminal object, written $\top$, in $\OpdV$.
Our proof that $\OpdV$ is cartesian closed relies on two results. The first is that for a tame bicategory~$\bcatE$, if $\bcatE$ is cartesian closed, then $\bimE$ is cartesian closed (Theorem~\ref{thm:ccbimod}). Applying this fact to~$\CatSymV$ (which we 
show to be tame as well), we obtain that $\Bim(\CatSymV)$ is cartesian closed. The second is that the inclusion 
\[
\OpdV  = \Bim(\SymV) \subseteq \Bim(\CatSymV)
\]
induced by the inclusion $\SymV \subseteq \CatSymV$ is an equivalence of bicategories (Theorem~\ref{thm:keybiequiv}). 
In order to prove these two auxiliary facts, we develop some aspects of the formal
theory of monads (in the sense of~\cite{StreetR:fortm}) in tame bicategories.
In particular, we establish a universal property of~$\bimE$, namely that of being the Eilenberg-Moore completion of~$\bcatE$ as a tame bicategory (a notion
that we will define precisely in Section~\ref{sec:montbb}), which is a special case of a result obtained independently by Richard Garner and Michael Shulman
 in the context of the theory of categories enriched in a bicategory~\cite{GarnerR:enrcfc}.

\medskip

\noindent
\textbf{Organization of the paper.}
The paper is organized as follows. Chapter~\ref{cha:bac} provides an overview of the background material used in the paper. We also introduce the notion of a Gabriel factorization of a homomorphism, which we use several times to construct the bicategories of interest.  

The rest of the paper is organized in two parts, each leading up to one of our two main results. 
The first part includes Chapter~\ref{cha:mond} and Chapter~\ref{cha:syms}. Chapter~\ref{cha:mond} develops some 
auxiliary material, needed for Chapter~\ref{cha:syms}. In particular, it introduces the auxiliary notions of a monoidal distributor and of a symmetric monoidal distributor, shows how they can be seen as the morphisms of appropriate bicategories, and establishes some useful facts about these bicategories. Chapter~\ref{cha:syms} introduces the notion of an $S$-distributor and the corresponding bicategory $\SDistV$. We then define $\CatSymV$ as the opposite of $\SDistV$. We then establish our first main result (Theorem~\ref{thm:smondistiscc}), asserting that $\CatSymV$ is cartesian closed. 
 
The second part of the paper comprises Chapter~\ref{cha:bicob} and Chapter~\ref{cha:carcob}. Chapter~\ref{cha:bicob} recalls the notions of a monad, bimodule and bimodule map in a bicategory~$\bcatE$ and the definition of the bicategory~$\bimE$, under the assumption that  $\bcatE$ is tame. We then show that $\CatSymV$ and its subcategory $\SymV$ are 
tame (Corollary~\ref{thm:smattame}), thus allowing us to define the bicategory of operad bimodules using the bimodule construction. Chapter~\ref{cha:carcob} is devoted to the proof of our second main result. First, 
we show that, for a tame bicategory $\bcatE$, if $\bcatE$ is cartesian closed, then so is $\bimE$ (Theorem~\ref{thm:ccbimod}).
Secondly, we investigate some aspects of the theory of monads in a tame bicategory, leading up to the fact that $\bimE$ can be
viewed as the Eilenberg-Moore completion of $\bcatE$ as a tame bicategory (Theorem~\ref{thm:completiontheorem}). A corollary of this fact leads to prove that $\OpdV$ and $\Bim(\CatSymV)$ are equivalent (Theorem~\ref{thm:keybiequiv}). Combining this fact with earlier results, as described above, we obtain that~$\OpdV$ is cartesian closed (Theorem~\ref{thm:opdcc}).

 \smallskip

\noindent
\emph{Note for expert readers.} Since we tried to make the paper as self-contained as possible, some material will be familar to
expert readers. In particular, readers who are already acquainted with~\cite{FioreM:carcbg} and are willing to assume that the results therein
transfer to the enriched setting may skip the first paper of the paper and start reading from Chapter~\ref{cha:bicob}. Similarly, 
readers who are already familiar with the universal property of the bimodule construction from~\cite{GarnerR:enrcfc} may
skip the final sections of the paper (Section~\ref{sec:montrb}, \ref{sec:montbb} and \ref{sec:bicbem}), apart from 
Theorem~\ref{thm:keybiequiv}.

\medskip

\noindent
\textbf{Acknowledgements}
We had the opportunity of working on the material presented in this paper at the Fields Institute, the Centre de Recerca Matem\`atica, the Institute for Advanced Study, 
the Mathematisches Forschungsinstitut Oberwolfach and the Institut Henri Poincar\'e. We are grateful to these institutions for their support and the excellent working condiitons. 

\smallskip

\noindent
 \emph{Acknowledgements for the first-named author.}  (i)  This material is based on research sponsored by the Air Force Research Laboratory, under agreement 
number FA8655-13-1-3038. The U.S.\ Government is authorized to reproduce and distribute reprints for Governmental purposes notwithstanding any copyright notation thereon. The views and conclusions contained herein are those of the authors and should not be interpreted as necessarily representing the official policies or endorsements, either expressed or implied, of the Air Force Research Laboratory or the U.S.\ Government.
(ii) This research was supported by a grant from the John Templeton Foundation. (iii) This research was supported by the Engineering and 
Physical Sciences Research Council,  Grant EP/K023128/1. 

\smallskip

\noindent
 \emph{Acknowledgements for the second-named author.} The support of NSERC is gratefully acknowledged.

\chapter{Background}
\label{cha:bac}

This chapter reviews some notions and results that will be used in the reminder of the paper. In particular, we review the basics of the theory of bicategories (Section~\ref{sec:revbt}), elements of enriched category theory (Section~\ref{sec:vcat}) and the definition of the bicategory of distributors (Section~\ref{sec:dist}). All the material in this chapter is essentially well-known, with the possible 
exception of the definition of the Gabriel factorization of a homomorphism,
which is the bicategorical counterpart of the factorization of a functor as an essentially surjective functor followed by a fully faithful one. 
Since Gabriel factorizations are for the definition of several bicategories in the remainder of the paper, we illustrate in some detail how this construction works in the discussion of the bicategory of distributors.

\section{Review of bicategory theory}
\label{sec:revbt}

For the convenience of the reader, the definitions of bicategory, homomorphism, pseudo-natural transformation and modification
are recalled in Appendix~\ref{sec:combd}.

For a bicategory $\bcatE$, 
we write $\bcatE[\objX, \objY]$ for the hom-category 
between two objects $\objX, \objY \in \bcatE$.
A bicategory~$\bcatE$ is said to be \emph{locally small} when
$\bcatE[\objX, \objY]$ is a small category for every $\objX, \objY \in \bcatE$.
A \emph{morphism},  or a \emph{$1$-cell}  ~$F \co  \objX \to \objY$
is an object of the category $\bcatE[\objX, \objY]$, and a \emph{$2$-cell} $\alpha\co  F\to F'$
is a morphism of the category $\bcatE[\objX, \objY]$. We write $1_\objX \co  \objX \rightarrow \objX$
for  the identity morphism of an object~$\objX \in \bcatE$.  
The composition operation of 2-cells, \ie the the composition operation of the hom-categories of $\bcatE$, 
is usually referred to as the \emph{vertical composition} in~$\bcatE$ and its effect 
on~ $\alpha \co  F \rightarrow F'$, $\beta \co  F' \rightarrow F''$  is written~$\beta \cdot \alpha \co  
F \rightarrow F''$. The identity arrow of an object $F\in \bcatE[\objX,\objY]$ is called an \emph{identity $2$-cell} of~$\bcatE$ 
and written~$1_F \co  F \to F$.  We 
refer to the composition operation
\[
(- ) \circ (-) \co  \bcatE[\objY, \objZ]\times \bcatE[\objX, \objY]\to \bcatE[\objX, \objZ]
\]
as the \emph{horizontal composition} of $\bcatE$.
The horizontal composite of~$F \co  \objX \rightarrow \objY$ and~$G \co   \objY \rightarrow~  \objZ$ 
is denoted by~$G \circ F \co  \objX \rightarrow \objZ$. 
The horizontal composition of $\alpha \co  F \rightarrow F'$ with
$\beta \co  G \rightarrow G'$
is written $\beta \circ \alpha \co  G \circ F \rightarrow~G' \circ F'$.
This 2-cell is the common value of the composites in the following naturality square 
\[
\xymatrix@C=2cm{
G \circ F \ar[r]^{G \circ \alpha} \ar[d]_{\beta \circ F} & G \circ F' \ar[d]^{\beta \circ F'} \\
G' \circ F \ar[r]_{G' \circ \alpha} & G' \circ F' \,. }
\]
We say that a bicategory is \emph{strict} if its composition operation 
is strictly associative and if the units~$1_\objX$  are strict.
A strict bicategory is the same thing as a 2-category, \ie a category enriched over
the category of locally small categories and functors~$\CAT$. Both $\CAT$ and 
the category of small categories~$\Cat$ have also the structure of a 2-category.

\begin{example} A monoidal category $\catC = (\mathbb{C},\otimes, I)$ can be identified with a bicategory,
here denoted by~$\Sigma(\catC)$, which has a single object and $\catC$ as its hom-category. The horizontal composition 
 of $\Sigma(\catC)$
is then given by the tensor product of $\catC$. Every bicategory with one object is of the
form~$\Sigma(\catC)$ for some monoidal category $\mathbb{C}$.
\end{example}

Given two bicategories~$\bcatE$ and~$\bcatF$, their \emph{cartesian product}~$\bcatE\times \bcatF$  is the bicategory 
with objects
\[
\Ob(\bcatE\times \bcatF) \defeq \Ob(\bcatE)\times \Ob(\bcatF)
\]  
and hom-categories
 \[
 (\bcatE\times \bcatF)[(\objX,\objY),(\objX',\objY')] \defeq \bcatE[\objX,\objX']\times  \bcatF[\objY,\objY'] \, , 
 \]
  for  $\objX,\objX'\in \bcatE$ and $\objY,\objY'\in \bcatF$. Composition is defined in the obvious way. 
 For a bicategory $\bcatE$, we write 
 $\bcatE^{\op}$ for the \emph{opposite bicategory} of $\bcatE$, which is obtained by formally reversing the direction of the morphisms of~$\bcatE$, but not that of the 2-cells. 
For a morphism $F \co \objX \to \objY$ in $\bcatE$, we write~$F^\op \co \objY \to \objX$ for the corresponding morphism in~$\bcatE^\op$. 

\subsection*{Equivalences and adjunctions in a bicategory}
A morphism $F \co  \objX \rightarrow \objY $ in  a bicategory~$\bcatE$  is said to be
 an \emph{equivalence}  if there exists a morphism $U \co  \objY  \rightarrow \objX $
together with invertible 2-cells $\alpha \co   G \circ F \to 1_\objX$
and $\beta \co   F \circ U \to 1_\objY$. We  write $\objX \simeq \objY$ to indicate that~$\objX$ and~$\objY$ are equivalent.
 An \emph{adjunction}  $(F, U, \eta, \varepsilon) \co \objX \to \objY$ in $\bcatE$ consists of morphisms $F \co \objX \to \objY$
 and $U \co \objY \to \objX$ and 2-cells $\eta\co  1_\objX\to U\circ F$ and $\varepsilon\co  F\circ U\to 1_\objX$ satisfying
the triangular laws, expressed the commutative diagrams
\begin{equation}
\label{equ:triangleidentities}
\xycenter{
F \  \ar[r]^-{ F \circ \eta} \ar@/_1pc/[dr]_{1_F} &  F\circ U\circ F  \ar[d]^{\varepsilon\circ F}  \\
& F \, ,
 }\quad \quad 
 \xycenter{
  U\circ F \circ U  \ar[d]_{U\circ \varepsilon}  & U \  \ar[l]_-{\eta\circ U} \ar@/^1pc/[dl]^{1_U} &  \\
\, U .&   \ \
 }
\end{equation} 
The morphism $F$ is the~\emph{left adjoint} and
the morphism $U$ is the~\emph{right adjoint};
the 2-cell $\eta$ 
is the~\emph{unit} of the adjunction and the 2-cell $\varepsilon$ is the~\emph{counit}.
Recall that the unit~$\eta$ of an adjunction~$(F, U, \eta,\varepsilon)$ determines the counit~$\varepsilon$
and conversely. More precisely, $\varepsilon\co   F\circ U\to 1_{\objY}$   is the unique 2-cell 
such that~$(U\circ \varepsilon )\cdot (\eta\circ  U)=1_U$,
and $\eta\co  1_{\objX }\to U\circ F$  is the  unique 2-cell
such that~$(\varepsilon \circ F)\cdot (F\circ \eta)=1_F$.
We  often write~$F\dashv U$ to indicate an 
adjunction~$(F, U, \eta,\varepsilon)$.
An adjunction is called a \emph{reflection} (\resp \emph{coreflection}) if its counit  (\resp unit) is invertible. 
If~$(F, U, \eta,\varepsilon) \co \objX \to \objY$ is an adjunction
in~$\bcatE$, then $(U^\op, F^\op, \eta,\varepsilon) \co  \objX \rightarrow  \objY$ is an adjunction
in the opposite bicategory~$\bcatE^\op$.

\subsection*{Homomorphisms} For bicategories $\bcatE$ and $\bcatF$, we write $\Phi \co  \bcatE \to \bcatF$ to indicate that $\Phi$ is a homomorphism
from $\bcatE$ to $\bcatF$. 
The homomorphisms from $\bcatE$ to $\bcatF$ are the objects of a 
bicategory~$\HOM\bigl[ \bcatE, \bcatF\bigr]$ whose morphisms are pseudo-natural transformations
 and  2-cells are modifications. 
A \emph{contravariant homomorphism} $\Phi\co  \bcatE\to \bcatF$ is defined to be a homomorphism $\Phi\co   \bcatE^\op \to \bcatF$.  
The canonical homomorphism 
\[
\bcatE[-,-]\co  \bcatE^\op\times \bcatE \to \CAT 
\]
takes a pair of objects $(\objX,\objY)$ to the category $\bcatE[\objX,\objY]$.
In particular, there is a
a covariant homomorphism $\bcatE[ \objK , - ]\co \bcatE \to \CAT$ and
a  contravariant homomorphism
$\bcatE[ -,\objK  ]\co \bcatE \to \CAT$
for each object~$\objK \in \bcatE$.

\medskip

We  say that a homomorphism $\Phi\co \bcatE\to \bcatF$
is \emph{full and faithful} if  for every $\objX,\objY\in \bcatE$ the  functor 
\[
\Phi_{\objX, \objY} \co  \bcatE[\objX, \objY] \to \bcatF[\Phi \objX, \Phi \objY]
\]
is an equivalence of categories.
We say that $\Phi\co \bcatE\to \bcatF$ is \emph{essentially surjective} 
if for every object~$\objY\in \bcatF$ there exists an object $\objX\in \bcatE$
together with an equivalence $\Phi\objX \simeq \objY$.
We say that $\Phi\co \bcatE\to \bcatF$ is an \emph{equivalence} if it is
full and faithful and essentially surjective.

The coherence theorem for bicategories, asserts that 
every bicategory is equivalent to a 2-category~\cite{MacLaneS:cohbic}. 
Thanks to this result, it is possible to treat the horizontal composition in a bicategory as if it were strictly associative and unital, which we will often do in the following.
We say that a homomorphism $\Phi\co \bcatE\to \bcatF$ is an \emph{inclusion}
if it is injective on objects and full and faithful. In this case, we will often write
$\bcatE \subseteq \bcatF$ and treat the action of $\Phi$ on objects as if it were the identity.

A homomorphism  $\Phi\co \bcatE \to \bcatF$
takes an adjunction $(F, U, \eta,\varepsilon) \co \objX  \rightarrow \objY$ in~$\bcatE$
to an adjunction $(\Phi F, \Phi U, \Phi\eta,\Phi \varepsilon)\co  \Phi \objX  \rightarrow \Phi \objY$
in $ \bcatF$. Dually, a contravariant homomorphism  $\Phi\co \bcatE \to \bcatF$
takes an adjunction $(F, U, \eta,\varepsilon) \co \objX  \rightarrow \objY$ in~$\bcatE$
to an adjunction $(\Phi U, \Phi F, \Phi\eta,\Phi \varepsilon)\co \Phi \objY  \rightarrow \Phi \objX$
in~$\bcatF$.  For example,  if~$\objK $ is an object of~$\bcatE$, then
the homomorphism $\bcatE[ \objK , - ]\co \bcatE \to \CAT$
takes  an adjunction $(F, U)\co \objX  \rightarrow \objY$  in~$\bcatE$
to an adjunction $(\bcatE[\objK, F], \bcatE[\objK, U]) \co \bcatE[ \objK ,  \objX ]  \to \bcatE[\objK , \objY]$
in $\CAT$. Dually, the contravariant homomorphism  $\bcatE[ -,\objK  ]\co \bcatE \to \CAT$
takes  an adjunction $(F, U) \co \objX  \rightarrow \objY$  in~$\bcatE$
to an adjunction $(\bcatE[U, \objK], \bcatE[F, \objK]) \co  \bcatE[ \objY ,  \objK] \to \bcatE[ \objX , \objK ]$.

\subsection*{Prestacks}  By a  \myemph{prestack} on a (locally small) bicategory $\bcatE$ we mean a
contravariant homomorphism~$\Phi \co  \bcatE \to \Cat$.
The bicategory of prestacks
 \[
\mathbf{P}(\bcatE) \defeq   \HOM\bigl[\bcatE^{\op}, \Cat\bigr]\, ,
\]
 is a 2-category, since $\Cat$ is a 2-category.  
The Yoneda homomorphism
\[
\yon_\bcatE \co  \bcatE \rightarrow \mathbf{P}(\bcatE)
\]
takes an object $ \objX  \in \bcatE$
to the prestack~$\yon_\bcatE(\objX) \defeq \bcatE[-, \objX]$.
For $\Phi\in  \mathbf{P}(\bcatE)$ and $\objX  \in \bcatE$, 
there is a functor
\[
\mathbf{P}(\bcatE)[\yon(\objX ), \Phi]\to  \Phi( \objX )
\]
which takes a pseudo-natural transformation $\alpha \co  \yon_\bcatE( \objX )\to  \Phi$
to the object $\alpha_\objX(1_ \objX )\in  \Phi( \objX )$, and the 
bicategorical Yoneda lemma asserts that this functor 
is an equivalence of categories.
It follows that the Yoneda homomorphism 
is full and faithful.  

We  say that a prestack  $\Phi : \bcatE \to \Cat$
on a  bicategory~$\bcatE $ is  \emph{representable} if
there exists an object~$\objX \in \bcatE$ together with a pseudo-natural equivalence $\alpha \co  \yon( \objX )\to  \Phi$.
It follows from Yoneda lemma that  $\alpha$ is determined by
the object~$\alpha_0 \defeq \alpha(1_\objX)\in \Phi(\objX)$.
In this case, we  say that~$\Phi$ is \emph{represented} by the pair~$(\objX, \alpha_0)$. 
Such a pair~$(\objX,\alpha_0)$ is unique up to equivalence.

\subsection*{Gabriel factorization} In analogy with the way a functor can be factored
as a fully faithful functor followed by an essentially surjective one, 
every homomorphism $\Phi\co \bcatE\to \bcatF$ admits a factorization of the form
\begin{equation}
\label{equ:gabrielfactorization}
{\vcenter{\hbox{\xymatrix{
\bcatE\ar[rr]^{\Phi} \ar[dr]_{\Gamma}&& \bcatF\\
&\bcatG \ar[ur]_{\Delta} \, , &
}}}}
\end{equation}
where $\Gamma$ is essentially surjective and $\Delta$ is full and faithful. In fact, we may suppose that  $\Gamma$
is the identity on objects, in which case we have a~\emph{Gabriel factorization} of $\Phi$. In order
to obtain a Gabriel factorization, the bicategory $\bcatG$ is defined as having the same objects as $\bcatE$ and letting, for $\objX, \objY \in\bcatE$, 
\[
\bcatG[\objX,\objY]\defeq\bcatF[\Phi\objX,\Phi\objY] \, .
\]
The composition law of $\bcatG$ is defined via the composition law of $\bcatF$
in the evident way. The homomorphism $\Gamma \co \bcatE \to \bcatG$ is  the identity on objects,
while~$\Gamma_{\objX, \objY} \co  \bcatE[\objX,\objY]\to \bcatG[  \objX,  \objY]$ 
is defined to be~$\Phi_{\objX, \objY} \co  \bcatE[\objX,\objY]\to \bcatF[\Phi\objX,\Phi\objY]$. 
The homomorphism $\Delta \co  \bcatG \to \bcatF$ is defined on objects by 
letting~$\Delta(\objX) \defeq \Phi(\objX)$, for $\objX \in \bcatE$, 
while~$\Delta_{\objX, \objY} \co  \bcatG[\objX,\objY]\to \bcatF[\Phi\objX,  \Phi\objY]$ 
is the identity functor. These definitions are illustrated in Section~\ref{sec:dist}, where we
show how the bicategory of distributors arises from a Gabriel factorization.

\begin{example*} Let us consider the Gabriel factorization of the  Yoneda homomorphism  $\yon \co  \bcatE \rightarrow \mathbf{P}(\bcatE)$,
  \[
\xymatrix{
\bcatE\ar[rr]^{\yon_\bcatE} \ar[dr]_{\Gamma}&& \mathbf{P}(\bcatE)\\
&\bcatG \ar[ur]_{\Delta}. &
}
\]
The bicategory $\bcatG$
is a 2-category, since the bicategory $ \mathbf{P}(\bcatE)$ is a 2-category,
Moreover, the homomorphism $\Gamma$ is an equivalence of bicategories,
since the Yoneda homomorphism is full and faithful.
Hence, the bicategory $ \bcatE$ is equivalent to a 2-category, giving 
a proof of  the coherence theorem for bicategories ~\cite{MacLaneS:cohbic}.
\end{example*}

There is a slight variation of the definitions given above  which arises when we are given
not only the homomorphism $\Phi \co \bcatE \to \bcatF$ but also, for each $\objX, \objY \in \bcatE$, 
a category $\bcatG[\objX, \objY]$ and an equivalence
\begin{equation}
\label{equ:gabrielnonstandard}
   \Delta_{\objX, \objY} \co \bcatG[\objX,\objY]\to \bcatF[\Phi\objX,\Phi\objY] \, .
\end{equation}
 In this case, we obtain again a Gabriel factorization of $\Phi$.
The bicategory~$\bcatG$ has again the same objects as~$\bcatE$ and its hom-categories
are given by the given categories~$\bcatG[\objX, \objY]$. The composition functors of $\bcatG$
are determined (up to unique isomorphism) by requiring that the following diagram commutes up to natural isomorphism:
\[
\xymatrix@C=1.8cm{
\bcatG[\objY, \objZ] \times \bcatG[\objX, \objY] \ar[r]^-{(-) \circ (-)} \ar[d]_-{\Delta_{\objY, \objZ} \times \Delta_{\objX, \objY}} & 
\bcatG[\objX, \objZ] \ar[d]^-{\Delta_{\objX, \objZ}} \\
\bcatF[ \Phi \objX, \Phi \objZ] \times \bcatF[\Phi \objX, \Phi \objY] \ar[r]_-{(-) \circ (-)} & 
\bcatF[\Phi \objX, \Phi \objZ] \, .}
\]
Similarly, the identity morphism $1_\objX \co \objX \to \objX$ on an object $\objX \in \bcatG$, is determined (up to unique isomorphism)
by requiring that there is an isomorphism $\Delta(1_\objX) \iso 1_{\Phi X}$. These associativity and unit isomorphisms can be defined
in a similar way. The definition of the required homomorphism~$\Delta \co \bcatG \to \bcatF$ now follows easily. Its action on objects
is given by mapping~$\objX \in \bcatG$ to~$\Phi \objX \in \bcatF$ and its action on hom-categories is given by the equivalences in~\eqref{equ:gabrielnonstandard},
so that $\Delta$ is full and faithful by construction. The homomorphism $\Delta \co \bcatE \to \bcatG$ can then be defined as the identity on objects,
while its action on hom-categories is essentially determined by requiring that the diagram in~\eqref{equ:gabrielfactorization} commutes up to
pseudo-natural equivalence. We will illustrate this method of constructing bicategories in Section~\ref{sec:dist}
and apply it again in Section~\ref{sec:mondist} and Section~\ref{sec:smondist}.

\subsection*{Adjunctions between bicategories}
If $ \bcatE$  and  $\bcatF$ are bicategories, 
then an \emph{adjunction}  (sometimes also referred to as a \emph{biadjunction}) $\theta\co \Phi\dashv \Psi$ between two homomorphisms $\Phi\co  \bcatE\to \bcatF$
and $\Psi\co  \bcatF\to \bcatE$  is defined to be a pseudo-natural equivalence 
\[
\theta\co  \bcatE[\objX,\Psi\objY] \simeq  \bcatF[\Phi\objX,\objY]  \, .
\]
The homomorphism $\Phi$ is said to be the \emph{left adjoint} and
the homomorphism $\Psi$ to be the \emph{right adjoint}.
A homomorphism $\Phi\co  \bcatE\to \bcatF$ has a right adjoint if and only
if the prestack 
\[
\bcatE[\Phi(-),\objY]\co  \bcatE \to \Cat 
\]
is representable for every object $\objY\in  \bcatF$.
The \emph{counit} of the adjunction is a pseudo-natural transformation $\varepsilon \co  \Phi\circ \Psi \to \Id_\bcatF$
defined by letting $\varepsilon_{\objY} \defeq \theta(1_{\Psi\objY} )$
for~$\objY\in  \bcatF$.
The  \emph{unit} of the adjunction  is a pseudo-natural transformation $\eta\co  \Id_\bcatE \to \Psi\circ \Phi$
defined by letting~$\eta_{\objX} \defeq \theta^{-1}(1_{\Phi\objX} )$
for~$\objX\in  \bcatE$, where $\theta^{-1}$ is a quasi-inverse of $\theta$.
Either of the pseudo-natural transformations~$\eta$ and~$\varepsilon$ determine
the adjunction~$\theta$.

\subsection*{Cartesian, cocartesian and cartesian closed bicategories}
 We recall the notion of cartesian bicategory.
We say that an object~$\top$ in a bicategory~$\bcatE$ is \emph{terminal}
if the category~$\bcatE[\objC,\top ]$ is equivalent to the terminal category
for every object~$\objC\in \bcatE$.
A terminal object~$\top \in \bcatE$
is unique up to equivalence when it exists. 
Given two objects $\objX_1, \objX_2 \in \bcatE$, 
we say that an object~$\objX \in \bcatE$ equipped
with  two morphisms~$\pi_1\co  \objX \to \objX_1$ and~$\pi_2\co  \objX \to \objX_2$
is the \emph{cartesian product} of~$\objX_1$ and~$\objX_2$
if the  functor 
\[
\pi\co \bcatE[\objC,\objX] \to  \bcatE[\objC,\objX_1]\times \bcatE[\objC,\objX_2]  \, , 
\]
defined by letting~$\pi(F) \defeq (\pi_1\circ F,\pi_2\circ F)$
is an equivalence of categories for every object~$\objC \in \bcatE$.
The cartesian product of the objects~$\objX_1$ and~$\objX_2$
is unique up to equivalence when it exists. In this case, 
we will denote it by~$\objX_1\sqcap \objX_2$ and
refer to the morphisms~$\pi_k \co  \objX_1\sqcap \objX_2 \to \objX_k$ ($k = 1, 2$)
as the \emph{projections}.
When every pair of objects in ~$\bcatE$ has a cartesian product,
then the diagonal homomorphism~$\Delta_\bcatE \co \bcatE \to \bcatE\times \bcatE$
has a right adjoint, 
\[
(-) \sqcap (-) \co \bcatE\times \bcatE\to \bcatE \, ,
\] 
which associates
to $(\objX_1,\objX_2)$ the cartesian product $ \objX_1\sqcap \objX_2$.
We say that a bicategory ~$\bcatE$ with a terminal object is \emph{cartesian} 
if every pair of objects in ~$\bcatE$ has a cartesian product. 
Dually, we say that a bicategory~$\bcatE$ is \emph{cocartesian} 
if the opposite bicategory $\bcatE^\op$ is cartesian.
We write $\bot$ for the initial object and $\objX_1\sqcup \objX_2$ for the coproduct of two objects~$\objX_1$ and~$\objX_2$,
and refer to the  morphisms~$\iota_k \co \objX_k \to \objX_1\sqcup \objX_2$ ($k = 1, 2$)
as the \emph{inclusions}.

\medskip

 We recall the notion of cartesian closed bicategory.
 Given objects $\objX \, , \objY$ of a cartesian bicategory $\bcatE$,
we will say that an object $\obj{E}\in \bcatE$
equipped with a morphism $\ev\co  \obj{E} \sqcap \objX \to \objY$ 
is the \emph{exponential} of $\objY$ by $\objX$
if the functor
\[
\xymatrix{
\bcatE[\objK,\obj{E}] \ar[rr]^(0.4){(-)\sqcap \objX }&& \bcatE[\objK \sqcap \objX ,\obj{E} \sqcap \objX]  
\ar[rr]^(0.55){\bcatE[\objK \sqcap \objX , \ev ] }  & & 
\bcatE[\objK \sqcap \objX,  \objY]
}
\]
is an equivalence of categories for every object $\objK \in \bcatE$.
 This condition means that the prestack
\[
\bcatE[(-)\sqcap\objX, \objY]\co  \bcatE^\op \to \Cat
\]
is represented by the pair $(\obj{E},\ev)$.
The  exponential of~$\objY$ by~$\objX$
is unique up to equivalence when it exists and 
we denote it by~$\objY^{\objX}$ or $[\objX, \objY]$ and refer to the morphism
\[
\ev \co  \objY^{\objX} \sqcap \objX \to \objY
\]
as the \emph{evaluation}. 
We say that 
a cartesian  bicategory~$\bcatE$ is \emph{closed} if 
 the exponential $\objY^\objX$ exists for every $\objX,\objY\in \bcatE$.
A cartesian bicategory $\bcatE$ is closed if and only if, for every object $\objX\in \bcatE$, the homomorphism
$(-)\sqcap \objX\co \bcatE\to \bcatE$
has a right adjoint $(-)^\objX\co \bcatE\to \bcatE$.
The resulting homomorphism mapping~$(\objX,\objY)$ to~$\objY^\objX$
 is contravariant in the first variable  and covariant in the second.

\subsection*{Monoidal bicategories}
A cartesian bicategory is an example of a symmetric monoidal bicategory, a notion that we limit ourselves to review in outline.
First, recall that by definition, a monoidal bicategory is a tricategory with one object~\cite{GordonR:coht,GurskiN:cohtdc}.
We will not describe this notion here because of its complexity (see~\cite{ChengE:pertnc,GordonR:coht,GurskiN:cohtdc,StayM:comcb} for details).
It will suffice to say that a monoidal structure on bicategory  $\bcatE$
is a 9-tuple 
\[
(\otimes,  \mathbb{I}, \alpha^1, \alpha^2, \lambda^1, \lambda^2, \rho^1, \rho^2, \mu) \, ,
\]
where the tensor product 
\[
(-) \otimes (-) \co \bimE\times \bimE \to  \bimE
\]
is a homomorphism, $ \alpha^1$, $\lambda^1$ and $\rho^1$ are pseudo natural (adjoint) equivalences
and $ \alpha^2, \lambda^2$, $ \rho^2$ and $ \mu$ are invertible modifications.
More precisely, 
\[
\alpha^1(\objX,\objY,\objZ)\co (\objX\otimes \objY) \otimes \objZ \simeq \objX\otimes (\objY \otimes \objZ)
\]
is the 1-associativity constraint and the 2-associativity constraint $\alpha^2(\objX,\objY,\objZ, \objW)$ 
is  a 2-cell fitting in the pentagon
\[
\begin{xy} 
(0,20)*+{((\objX \otimes \objY) \otimes \objZ) \otimes \objW} ="1"; 
(-40,0)*+{(\objX \otimes (\objY \otimes \objZ) ) \otimes \objW}="2"; 
(-25,-20)*+{\objX \otimes ( ( \objY \otimes \objZ) \otimes \objW) }="3"; 
(25,-20)*+{\objX \otimes (\objY  \otimes (\objZ \otimes \objW)) \, . }="4" ; 
(40,0)*+{ (\objX \otimes \objY) \otimes (\objZ \otimes \objW)}="5";
{\ar_{\alpha^1 \otimes \objW} "1";"2"};
{\ar^{\alpha^1} "1";"5"};
{\ar^{\alpha^1} "5";"4"};
{\ar_{\alpha^1} "2";"3"};
{\ar_{\objX \otimes \alpha^1} "3";"4"};
\end{xy}
\] 
The associativity constraints satisfy coherence conditions that we omit.
We also omit the coherence conditions for the unit object $\mathbb{I}$
and its constraints $(\lambda^1, \rho^1,  \lambda^2, \rho^2, \mu)$.
A \emph{symmetry structure} on a monoidal bicategory as above is a pseudo-natural (adjoint) equivalence
\[
\sigma^1_{\objX,\objY}\co  \objX \otimes \objY \simeq \objY\otimes \objX
\]
together with certain higher dimensional constraints~\cite{GurskiN:inflsc}.

\medskip

For further information on the theory of bicategories, we invite the reader to refer 
to~\cite[Volume~I, Chapter~7]{BorceuxF:hanca} and~\cite{BenabouJ:intb,LackS:a2cc,StreetR:fibb}.

\section{$\catV$-categories and presentable $\catV$-categories}
\label{sec:vcat}

Since we will focus on enriched categories and enriched operads, it is convenient to  recall some aspects of enriched category theory from~\cite{KellyGM:bascec}.  Let $\moncatV = (\catV, \otimes, I, [-,-])$ be a locally presentable symmetric monoidal closed category, which we
shall consider fixed throughout this paper. Examples of such a category include the category of sets, the category of pointed simplicial sets (with the smash product as tensor product), categories of chain complexes of vector spaces over a field, and the category of spectra.

 If $\catX$ is a small $\catV$-category, we write $\catX[x,y]$ or simply $[x,y]$ for
the hom-object between two objects $x,y \in \catX$.
We write~$\CatV$ (\resp $\CATV$) for the 2-category of  small (\resp locally small) $\catV$-categories,
 $\moncatV$-functors and  $\moncatV$-natural transformations.
 The category~$\CatV$ is complete and cocomplete. In particular, its terminal object is the $\catV$-category~$\UnitCat$ defined 
 by
 letting~$\Obj(\UnitCat) \defeq \{ \ast \}$ and~$\UnitCat[\ast, \ast] \defeq \top$, where~$\top$ is the terminal object of~$\catV$. 
The terminal object of  $\CatV$ is the $\catV$-category~$\onecat$ defined by letting~$\Obj(\onecat) \defeq \{ \ast \}$ 
and~$\onecat[\ast,\ast] \defeq \top$, where~$\top$ is the terminal object of~$\catV$.
The category~$\CatV$ has also a symmetric monoidal closed structure.
We write $\catX\tensorvcat \catY$ for the tensor product  two small  $\moncatV$-categories~$\catX$ and~$\catY$.
This is defined by letting~$\Ob(\catX\tensorvcat \catY) \defeq \Ob(\catX)\times \Ob(\catY)$
and 
\[
(\catX\tensorvcat \catY)\bigl[ (x,y),(x',y') \bigr] \defeq \catX[x,x'] \otimes \catY[y,y'] \, .
\]
Sometimes we write~$x \otimes y \in \catX \otimes \catY$ instead of $(x,y) \in \catX \otimes \catY$.
The unit object for this monoidal structure is 
the~$\moncatV$-category $\UnitCat$ defined by letting~$\Ob(  \UnitCat) \defeq \{ \ast \} $ and 
$\UnitCat[\ast,\ast] \defeq I$. The hom-object~$[\catX, \catY]$
is the~$\moncatV$-category of $\moncatV$-functors from $\catX$ to $\catY$ and $\catV$-natural
transformations.  For $\catV$-categories $\catX, \catY, \catZ$, a \emph{$\catV$-functor of two variables} $F\co \catX\times \catY \to \catZ$
is a~$\catV$-functor~$F\co \catX \tensorvcat \catY \to \catZ$.

\medskip

The next definition recalls the notion of a (locally) presentable $\catV$-category, with which we will work throughout the paper.
The reader is invited to refer to~\cite{KellyGM:strdle} for further information about it.

\begin{definition} 
We say that a $\catV$-category $\pcatE$
is (locally)~\emph{presentable} if it is $\catV$-cocomplete
and its underlying ordinary category is (locally) presentable in the usual sense.
\end{definition}

We write~$\PCATV$ for the
2-category of presentable $\catV$-categories, 
 cocontinuous $\catV$-functors and 
$\catV$-natural transformations. 
For example, the $\catV$-category $\pshX \defeq [\catX^\op, \catV]$ of presheaves on
a small $\catV$-category $\catX$ is 
presentable.
In particular,  the terminal $\catV$-category $\onecat \simeq \psh(\zerocat)$ is 
presentable, where $\zerocat$ is the $\catV$-category with no objects.
For a small $\catV$-category $\catX$, we write $\yon_\catX \co \catX\to \pshX$ for the Yoneda $\catV$-functor, which is
defined by mapping $x\in \catX$ to $\yon_\catX(x) \defeq \catX[-,x]$. By the enriched version of the Yoneda lemma, there is
an isomorphism
\[
\pshX[ \yon_\catX(x), A] \iso A(x) \, ,
\]
for every $A \in \pshX$ and $x \in \catX$.
It follows that  $\yon_\catX$
is full and faithful; we will often regard it as an inclusion by writing $x$ instead of $\yon_\catX(x)$.
If $\catX$ is a small $\catV$-category, then the $\catV$-category
$\pshX$ is  cocomplete and freely generated by~$\catX$.
More precisely,  the Yoneda functor $\yon_\catX  \co  \catX\to \pshX$
exhibits $\pshX$ as the free cocompletion of $\catX$. This means that
if $\pcatE$ is a cocomplete $\catV$-category,
and $\CCATV[\pshX,\pcatE]$ denotes the (large, locally small)
$\catV$-category of cocontinuous $\catV$-functors from $\pshX$ to $\pcatE$,
 then the restriction functor
 \begin{equation}
\label{equ:restrictionlocpres}
\xycenter{
\yon_\catX^\ast \co \CCATV[\pshX,\pcatE]\to [\catX, \pcatE] \, ,}
\end{equation}
defined by letting $\yon_\catX^\ast(F) \defeq F\circ \yon_\catX$, is an equivalence of $\catV$-categories.
In particular,  every $\catV$-functor $F \co  \catX \to \pcatE$
admits a cocontinuous extension $F_{c} \co  \pshX\to \pcatE$
which is unique up to a unique $\catV$-natural isomorphism,
\[
\xymatrix{
 \catX \ar@/_1pc/[drr]_{F}  \ar[rr]^-{\yon_\catX} \ar@{}[drr]|{\quad \iso} & &\pshX \ar[d]^{F_{c}} \\
 && \pcatE.}
 \]
The $\catV$-functor $F_c$ is the left Kan extension of $F$ along $\yon_\catX$ and
  its action on $A \in \pshX$ is given by the coend formula
\begin{equation}
 \label{equ:cocextension}
 F_{c}(A) \defeq \int^{x\in \catX} F(x) \otimes A(x)  \, .
\end{equation}
  The $\catV$-functor $F_{c}$ is left adjoint  to the ``singular $\catV$-functor'' $F^s\co  \pcatE\to \pshX$, 
given by letting 
 \[
 F^s(y)(x) \defeq \pcatE[F(x),Y] \, ,
 \] 
 for $Y\in \pcatE$ and $x\in \catX$. We write
\begin{equation}
 \label{equ:homomP}
P\co \Cat_\catV\to\PCATV
\end{equation}
for the homomorphism which takes
a small $\catV$-category $\catX$ to $\pshX$.
 If $u\co \catX \to  \catY$ is a $\catV$-functor between small $\catV$-categories,
we define $P(u)\defeq u_!\co \pshX\to \pshY$, \ie as the cocontinuous 
 extension of~$\yon_\catY\circ u\co \catX\to \pshY$.
 Hence, the diagram
\[
\xymatrix{
\catX \ar[d]_{u} \ar[r]^-{\yon_\catX}   & \pshX \ar[d]^{u_!} \\
\catY \ar[r]_-{\yon_\catY} & \pshY
}
\]
 commutes up to a canonical isomorphism and, for every $A\in \psh(\catX)$ and $y\in \catY$, we have 
\begin{equation}
 \label{equ:lowershrieck}
u_!(A)(y) \defeq \int^{x\in \catX} \catX[y,u(x)] \otimes A(x) \, .
\end{equation}
 The functor $u_{!} \co \pshX \to \pshY$ has a right adjoint $u^\ast \co \pshY \to \pshX$ defined
 by letting
 \begin{equation}
 \label{equ:upperstar}
 u^*(B)(x) \defeq B(u(x)) \, , 
 \end{equation}
 for~$B \in \pshY$ and $x \in \catX$.

\medskip

If $\pcatE$ is a presentable $\catV$-category and $ \mathcal{C}$
is a cocomplete $\catV$-category, then any cocontinuous $\catV$-functor from $\pcatE$ to $\mathcal{C}$
has a right adjoint. Because of this, products and coproducts in $\PCATV$ are intimately related, as we now recall. 
First of all, the cartesian  product $\pcatE=\bigsqcap_{k \in K} \pcatE_k$
of a family of presentable $\catV$-categories $(\pcatE_k \ | \  k \in K)$ is presentable.
Each projection $\pi_k \co \pcatE_k \to \pcatE$ has a left adjoint $\iota_k \co \pcatE \to \pcatE_k$
and the family $(\iota_k \ | \ k \in  K)$ is a coproduct diagram in $\PCATV$.
In particular, the terminal $\catV$-category $1$
is both initial and terminal in the bicategory $\PCATV$.

\begin{lemma}  \label{bulletcomp} The homomorphism  $P\co \CatV \to\PCATV$ preserves 
coproducts.
\end{lemma}

\begin{proof} This follows from the universal property of $\pshX$.  Indeed, let us consider 
a  family of small $\catV$-categories $(\catX_k \ | \ k \in K)$ and 
let $\catX = \bigsqcup_{k \in K} \catX_k$ and let $\iota_k \co \catX_k \to \catX$ be the inclusion for $k \in K$.
We  prove that the family of maps $(\iota_k)_!\co P(\catX_k)\to \pshX$ $(k\in K)$ is a coproduct diagram in~$\PCATV$.
For this, it suffices to show for every presentable $\catV$-category $\pcatE$, the
 family of functors
\[ 
\PCATV[(\iota_k)_!, \pcatE] \co  \PCATV[\pshX, \pcatE] \to  \PCATV[P(\catX_k), \pcatE] \quad \quad (k \in K)
\]
is a product diagram in the 2-category $\CAT$.
 But this family is equivalent to the family of functors
 \[
 [\iota_k, \pcatE] \co [\catX, \pcatE] \to [\catX_k, \pcatE]\quad \quad (k\in K)
 \]
by the equivalence in~\eqref{equ:restrictionlocpres}. This proves the result, since the family $\iota_k \co \catX_k \to \catX$, for $k \in K$, 
is a coproduct diagram  also  in the 2-category of locally small $\catV$-categories.
\end{proof}

\begin{remark} \label{thm:tensorpcat} 
If $\pcatE$ and $\pcatF $ are presentable $\catV$-categories, then so is the 
$\catV$-category $\CCATV[\pcatE,\pcatF]$
of cocontinuous $\catV$-functors from $\pcatE$ to $\pcatF$.
This defines the hom-object of a symmetric monoidal closed structure 
on the 2-category $\mathrm{PCAT}_\catV$.
By definition, the \emph{completed tensor
product} $\pcatE \tensorloc  \pcatF $ of two presentable 
$\catV$-categories $\pcatE$ and $\pcatF $ 
is a presentable $\catV$-category
equipped with a $\catV$-functor in two variables from
$\pcatE \times \pcatF$ to $\pcatE \tensorloc  \pcatF$
that is $\catV$-cocontinuous in each variable and universal with respect to that property.
The unit object for the completed tensor product 
is~$\catV$.
If we consider the 2-categories $\CatV$ and $\PCATV$ as 
equipped with the symmetric monoidal structures, the homomorphism
$P\co \Cat_\catV\to\PCATV$ is symmetric monoidal. Indeed, for $\catX \, , \catY \in \CatV$, we have a
$\catV$-functor of two variables 
\[
\phi_{\catX, \catY} \co \pshX  \times \pshY \to \psh(\catX\otimes \catY)
\]
 defined by letting $\phi_{\catX, \catY}(F, G)(x \tensorvcat y) \defeq F(x)\otimes G(y)$. This 
 $\catV$-functor  exhibits $ \psh(\catX \tensorvcat \catY)$ as the completed tensor product 
 of~$\pshX $ and~$\pshY$ and so we have an equivalence
 \[
 \pshX \tensorloc  \pshY \simeq \psh(\catX \tensorvcat \catY) \, .
 \]  
 \end{remark}
 
 \medskip

We conclude this section with a straightforward observation, which we state explicitly for future reference. Recall
that if $\catX$ is a small $\catV$-category and $\pcatE$ is a locally small $\catV$-category, then the $\catV$-category $[\catX,\pcatE]$ of $\catV$-functors from $\catX$ to $\pcatE$  is locally small.

 \begin{proposition}  \label{thm:lambda} Let $\catX, \catY$ be a small $\catV$-categories and $\pcatE$ be a locally small $\catV$-category.
 The  $\catV$-functors
\[
\lambda^{\catY}\co [\catX  \tensorvcat \catY, \pcatE] \to [\catX, [\catY,\pcatE]] \, , \quad 
\lambda^{\catX}\co [\catX  \tensorvcat \catY, \pcatE] \to [\catY, [\catX,\pcatE]] 
\]
defined by letting 
\[
(\lambda^{\catY}F)(x)(y) \defeq F(x,y) \, , \quad (\lambda^{\catX}F)(y)(x) \defeq F(x,y) \, , 
\]
for  $F \co  \catX  \tensorvcat \catY \to \pcatE$, $x\in  \catX$ and $y\in \catY$,  are 
equivalences of $\catV$-categories. \qed
\end{proposition}

\section{Distributors}
\label{sec:dist}

Let us recall  the notion of a distributor (sometimes called bimodule or profunctor in the 
literature)~\cite{BenabouJ:dis,LawvereFW:metsgl} and some basic facts about it. In
particular, we show how the bicategory of distributors fits into a Gabriel factorization.

\begin{definition} \label{thm:dis}
Let $\catX, \catY \in \CatV$. A \emph{distributor}  $F \co  \catX \mat \catY$ 
is a $\catV$-functor $F \co  \catY^\op  \tensorvcat \catX \to \catV$. 
\end{definition}

For a distributor $F \co  \catX \mat \catY$, we write $F[y,x]$ for the result of applying
$F$ to $(y, x) \in \catY^\op \otimes \catX$. Small $\catV$-categories, distributors and $\catV$-transformations form a bicategory, called the bicategory of distributors and
denoted by~$\DistV$, in which the hom-category of
morphisms between two small $\catV$-categories $\catX$ and $\catY$ is defined by letting
\[
 \DistV[\catX, \catY] \defeq [\catY^\op \otimes \catX, \catV ] \, .
 \]
The bicategory $\DistV$ fits into a Gabriel factorization of the form 
\begin{equation}
\label{equ:gabrieldistributors}
\xycenter{
\CatV \ar[rr]^{\psh} \ar[dr]_{(-)_\bullet}&&\PCATV \\
&\DistV \ar[ur]_{(-)^\dag} .&
}
\end{equation}

The Gabriel factorization essentially determines  the composition operation and the unit morphisms 
of the bicategory $\DistV$ and provides a proof that they satisfy the appropriate coherence conditions.
We illustrate this fact since we will use the same method  to define other bicategories 
in Chapter~\ref{cha:mond} and Chapter~\ref{cha:syms}.
First of all, for small $\catV$-categories $\catX, \catY$, we define
\begin{equation}
\label{equ:dagger}
(-)^\dag\co \DistV[\catX,\catY] \to \PCATV[\pshX, \pshY] 
\end{equation}
to be the composite of first  equivalence in Proposition~\ref{thm:lambda} and the quasi-inverse of the equivalence in~\eqref{equ:restrictionlocpres}:
\[
\xymatrix{
\DistV[\catX,\catY] \ar[r]^-{\lambda} &  \CATV[\catX, \pshY] \ar[r]^-{(-)_c} & \PCATV[\pshX,\pshY] \, . }
\]
It is convenient to express the effect of this functor with the following notation:
 \[
 \smallskip
\begin{prooftree}
\[
F\co \catX \mat \catY
\justifies
\lambda F\co \catX\to  \pshY
\]
\justifies
F^\dag \co  \psh(\catX) \to  \psh(\catY) \, . 
\end{prooftree}
\]
Here and in the rest of the paper, a horizontal line means that the data above it allows us to define the data 
data below it.
Explicitly, for  $A\in \pshX$, we have 
\begin{align*} 
F^\dag(A)(y) & =(\lambda F)_c(A)(y) \\
 & = \int^{x\in \catX} (\lambda F)(x)(y) \otimes A(x) \\
 & = \int^{x\in \catX} F[y,x] \otimes A(x)  \, . 
 \end{align*}
The functor in~\eqref{equ:dagger} is an equivalence of categories, since it is the composite of  equivalences. 
Because of this, the composite $G \circ F$ of two distributors $F \co \catX
\mat \catY$ and $G \co \catY \mat \catZ$ is determined up to unique isomorphism by the requirement that there
is an isomorphism
\begin{equation}
\label{equ:isofordagger1}
\varphi_{F. G} \co (G \circ F)^\dag \to G^\dag \circ F^\dag \, . 
\end{equation}
Thus,  $G\circ F\co   \catX\mat \catZ$ can be defined by letting
\begin{equation}
\label{equ:compdist}
(G \circ F)[z,x]  \defeq  \int^{y \in \catY}  G[z,y] \otimes  F[y,x]    \, .
\end{equation}
Similarly, the identity distributor $\id_\catX \co \catX \mat \catX$ is determined up to unique isomorphism 
by the requirement that there is an isomorphism
\begin{equation}
\label{equ:isofordagger2}
\varphi_\catX \co (\id_\catX)^\dagger \to 1_{\pshX} \, .
\end{equation}
Thus, $\id_\catX\co  \catX\mat \catX$ can be defined by letting
 \[
 \id_\catX[x,y] \defeq \catX[x,y]  \, .
 \]
Using the same reasoning, it is possible to show that  horizontal composition of distributors is functorial
and associative up to coherent isomorphism, and that the identity morphisms provide two-sided units for this operation 
up to coherent isomorphism. For example, for distributors $F \co \catX \mat \catY$, $G \co \catY \mat \catZ$,
$H \co \catZ \mat \catW$, the associativity isomorphism 
\[
\alpha_{F, G, H} \co (H \circ G) \circ F \to H \circ (G \circ F)
\]
can be defined as the unique 2-cell such that the following diagram commutes (where we omit subscripts of the 2-cells
to improve readability):
\begin{equation}
\label{equ:cohcomp}
\xycenter{
\big( (H \circ G) \circ F \big)^\dagger \ar[r]^{\alpha^\dagger}  \ar[d]_{\phi} & \big( H \circ (G \circ F) \big)^\dagger \ar[d]^{\varphi} \\
(H \circ G)^\dagger \circ F^\dagger \ar[d]_{\varphi \circ F^\dagger}      &   H^\dagger \circ (G \circ F)^\dagger \ar[d]^{H^\dagger \circ \varphi}  \\
(H^\dagger \circ G^\dagger) \circ F^\dagger \ar@{=}[r] & H^\dagger \circ (G^\dagger \circ F^\dagger) \, .} 
\end{equation}
It follows that we can define a homomorphism $(-)^\dagger \co \DistV \to \PCATV$ by letting 
\[
\catX^\dagger  \defeq P(\catX)
\] 
and taking its action on morphisms and 2-cells be defined by the functor in~\eqref{equ:dagger}.
The required isomorphisms expressing pseudo-functoriality are then given by the 2-cells in~\eqref{equ:isofordagger1}
and~\eqref{equ:isofordagger2}, which satisfy the required coherence conditions by the definition of the associativity
and unit isomorphisms in~$\DistV$, as done above. For example, the diagram in~\eqref{equ:cohcomp} 
states exactly one the coherence conditions. Furthermore, by construction, the homomorphism $(-)^\dagger \co \DistV \to \PCATV$ is
full and faithful, as required from the second part of a Gabriel factorization.

\medskip

We now define the homomorphism $(-)_\bullet \co \Cat_\catV\to \DistV$ which provides the first part of the Gabriel factorization in~\eqref{equ:gabrieldistributors}. Its action on objects is the identity. Furthermore, the requirement that the
diagram in~\eqref{equ:gabrieldistributors} commutes up to pseudo-natural isomorphism forces us to send a 
$\catV$-functor  $u\co \catX \to  \catY$ to a distributor  $u_\bullet\co \catX \mat  \catY$
such that $u_{!} \iso (u_\bullet)^\dag$.  Such a distributor, which is unique up to unique isomorphism, is defined by letting, for $x \in \catX$, $y \in \catY$, 
\[
u_\bullet[y,x] \defeq \catY[y,u(x)] \, . 
\]
Indeed,  for~$A\in \pshX$ and~$y\in \catY$, we have 
\begin{align*} 
(u_\bullet)^\dag (A)(y)
 & =\int^{x\in \catX} \catY[y,u(x)] \otimes A(x)   \\
& = u_!(A)(y) \, . \phantom{\int^{y \in \catY}}  
\end{align*} 
The distributor  $u_\bullet\co \catX \mat  \catY$ has a right adjoint $u^\bullet\co \catY \mat  \catX$, which is 
defined by letting for $x \in \catX$, $y \in \catY$, 
\[
u^\bullet[x,y] \defeq \catY[u(x),y] \, .
\]
Indeed, we have that $u^\ast \iso (u^\bullet)^\dag$, since for $B \in \pshY$ and $x \in \catX$, we have 
\begin{align*} 
 (u^\bullet)^\dag (B)(x)  & =  \int^{y\in  \catY}  \catY[u(x),y] \otimes B(y)  \\ 
 & \iso  B(u(x)) \phantom{\int^{y \in \catY}} \\ 
 &  = u^\ast(B)(x) \, . 
 \end{align*} 
Since there is an adjunction $u_!\dashv u^\ast$, we also have an
adjunction $u_\bullet \dashv u^\bullet$.
The components of its unit $\eta\co  \Id_\catX\to  u^\bullet \circ u_\bullet $ are the maps
$u_{x,x'} \co \catX[x,x'] \to \catY[u(x),u(x')] $ given by $u$.
The components of the counit~$\varepsilon\co  u_\bullet \circ u^\bullet \to  \Id_\catY$ are the canonical maps
\[
\varepsilon_{y, y'} \co \int^{x\in  \catX}  \catY[y,u(x)]  \otimes  \catY[u(x),y']  \to \catY[y,y']  \, .
\]
For $u\co \catX\to \catY \, , v\co \catX \to  \catY$ we have $ (v\circ u)_! \iso v_!\circ u_!$, and for $\catX \in \CatV$,
we have $(1_\catX)_{!} \iso 1_{\pshX}$. Therefore, there are canonical isomorphisms
\[
(v\circ u)_\bullet \iso v_\bullet \circ u_\bullet \,  , \quad  (1_\catX)_\bullet \iso \id_\catX \, , 
\] 
which necessarily satisfy the coherence conditions for a homomorphism  $(-)_\bullet \co \Cat_\catV\to \DistV$.

\medskip

Part (i) of the next lemma will be used to prove Theorem~\ref{freesymmtricdist}, while part (ii) will be used in the proof
of Proposition~\ref{lem:multiplicative2}

\begin{lemma}  \label{bulletlambda} Let $F\co \catX\mat \catY$ be a distributor.
\begin{enumerate}[(i)]
\item For all $\catV$-functors $u:\catX'\to \catX$,  $\lambda(F\circ u_\bullet) \iso \lambda(F) \circ u$.
\item For all $\catV$-functors $u:\catX'\to \catX$ and $v:\catY'\to \catY$,  $(v^\bullet\circ  F\circ u_\bullet )[y',x'] \iso F[v(y'),u(x')]$.
 \end{enumerate}
\end{lemma}

\begin{proof} For (i), let $x' \in \catX$. Then
\begin{align*} 
 \lambda(F\circ u_\bullet)(x')(y) & = (F \circ u_\bullet)[y,x'] \\
  & = \int^{x \in \catX} F[y,x] \otimes u_\bullet[x,x'] \\
  & = \int^{x \in \catX} F[y,x] \otimes \catX'[x, u(x')]  \\
  & \iso F[y, u(x')]  \phantom{\int^{x \in \catX}} \\
  & = (\lambda(F) \circ u) (x') (y) \, . \phantom{\int^{x \in \catX}}  \qedhere
  \end{align*}
For (ii), let $x' \in \catX'$ and $y' \in \catY'$. Then
\begin{align*}
(v^\bullet\circ  F\circ u_\bullet )[y',x'] & \iso \int^{x \in \catX} \int^{y \in \catY} v^\bullet[y',y] \otimes F[y,x] \otimes u_{\bullet}[x,x']  \\
 & = \int^{x \in \catX} \int^{y \in \catY} \catY[v(y'),y] \otimes F(y,x) \otimes \catX[x, u(x')] \\
 & \iso F[v(y'), u(x')] \, . \phantom{\int^{x \in \catX}} 
 \end{align*}
 
  \end{proof}

\begin{proposition}  \label{bulletcomp2} The bicategory $\DistV$ has coproducts and
the homomorphisms 
\[
(-)_\bullet:\CatV \to \DistV \, , \quad (-)^\dag:\DistV\to  \PCATV
\]
preserve coproducts.
\end{proposition}

\begin{proof} This follows from Lemma~\ref{bulletcomp} and the fact that $\DistV$ fits into a Gabriel factorization.
\end{proof}

\begin{remark} Although we will not need it in the following, let us recall that the symmetric monoidal
structure on $\CatV$ extends to a symmetric monoidal structure on $\DistV$, defined in the same
way on objects. The tensor product $F_1 \tensorvcat F_2 \co  \catX_1 \tensorvcat \catX_2 \mat \catY_1   \tensorvcat \catY_2$ of 
two distributors~$F_1\co  \catX_1\mat \catY_1$ and $F_2\co  \catX_2\mat  \catY_2$
is defined by letting 
 \[
 (F_1 \tensorvcat F_2)[(y_1,y_2) , (x_1,x_2)] \defeq  F_1[y_1,x_1] \otimes F_2[ y_2,x_2]  \, ,
 \]
 for $x_1\in \catX_1$, $x_2\in \catX_2$, $y_1\in \catY_1$ and $y_2\in \catY_2$.
 This defines a symmetric monoidal structure on the bicategory $\DistV$.
The homomorphisms involved in the Gabriel factorization in~\eqref{equ:gabrieldistributors}
are symmetric monoidal.  Let us also remark that the symmetric monoidal bicategory $\DistV$
is compact~\cite{StayM:comcb} (also called rigid): the dual of a small $\catV$-category $\catX$
is the opposite $\catV$-category $\catX^\op$. The counit $\varepsilon\co \catX^{\op} \tensorvcat \catX\mat \UnitCat$
is given by the hom-functor $ \UnitCat^{\op} \tensorvcat \catX^\op  \tensorvcat \catX=\catX^\op  \tensorvcat \catX \to \catV$
and similarly for the unit $\eta\co \UnitCat\mat \catX  \tensorvcat \catX^{\op}$. Here, $\UnitCat$ is the $\catV$-category
giving the unit for the tensor product on $\CatV$, as defined in Section~\ref{sec:vcat}.
\end{remark}

\begin{remark}[The bicategory of matrices]  \label{thm:matrices}
For our goals in Chapter~\ref{cha:syms}, 
it will be useful for  to give an explicit description of the sub-bicategory of $\DistV$ spanned by sets, viewed as discrete $\catV$-categories.
In order to do so, we need to recall the definition and some basic properties of the functor mapping an ordinary category to the free $\catV$-category on it.  
If $I$ is the unit object of the monoidal category~$\catV$,
then the functor~$\catV(I,-)\co \catV\to\Set$ has a left adjoint~$(-)\cdot I\co \Set\to \catV$ which associates to a set $S$
the coproduct $S\cdot I=\bigsqcup_S I$ of $S$ copies of $I$.
This left adjoint functor 
is  symmetric (strong) monoidal. Hence, for any pair of sets~$S$ and~$T$, we have an isomorphism
\[
(S\times T)\cdot I \iso (S\cdot I) \tensorvcat (T\cdot I) \, .
\]
A similar situation occurs for the functor $(-)\cdot  \UnitCat\co \Set \to \CatV$
which takes a set $S$ to the $\catV$-category $S\cdot \UnitCat=\bigsqcup_S  \UnitCat$.
The functor $\mathrm{Und} \co \CatV \to \Cat$ mapping a $\catV$-category to its underlying category
has also a left adjoint $(-)\cdot \UnitCat\co \Cat \to \CatV$ which associates
to a category $\catC$ the $\catV $-category $\catC \cdot \UnitCat$
defined by letting $\Obj(\catC \cdot \UnitCat) \defeq
\Obj(\catC)$ and
$(\catC \cdot \UnitCat)[x,y] \defeq \catC[x,y] \cdot I$.
This left adjoint is  symmetric (strong) monoidal. Hence, for every $\catC, \catD \in \Cat$, we have an isomorphism
\[
(\catC \times \catD )\cdot  \UnitCat  \iso ( \catC \cdot  \UnitCat) \tensorvcat (\catD \cdot  \UnitCat) \, .
\]
Recall that, for sets $X$ and $Y$, a \emph{matrix} $F \co X \mat Y$ is a functor $F \co Y \times X \to \catV$,
\ie a family of sets $F(y,x)$, for $(y, x) \in Y \times X$.  Sets, matrices and natural transformations form a bicategory, called the bicategory of matrices and denoted by~$\MatV$, which can be identified with the full
sub-bicategory of $\DistV$ spanned by discrete $\catV$-categories. Indeed, for a set $X$, the discrete $\catV$-category 
with set of objects $X$ is the same thing as the free $\catV$-category on the discrete category with set of objects $X$, 
which is denoted here by~$X \cdot \UnitCat$. Furthermore, for every pair of sets $X$ and $Y$, we have isomorphisms of categories
\begin{align*}
\MatV[X, Y] & =  \CAT[ Y \times X, \catV] \\ 
 & \iso \CATV[  (Y \times X) \cdot \UnitCat, \catV] \\
 & \iso \CATV[ (Y \cdot \UnitCat) \otimes (X \cdot \UnitCat), \catV ] \\ 
 & \iso \DistV[  X \cdot \UnitCat \, ,  Y \cdot \UnitCat] \, . 
\end{align*}
The composition and identity morphisms in $\MatV$ can be then defined so that we have a full and faithful homomorphism from $\MatV$ to $\DistV$. Given two matrices $F \co X \mat Y$ and $G \co Y \mat Z$, their composite 
$G\circ F \co X\mat Z$ is defined so that there is an isomorphism $(G\circ F) \cdot \UnitCat \iso (G \cdot \UnitCat )\circ 
(F \cdot \UnitCat)$. It follows that, for $x \in X$ and $z \in Z$, we can define
 \[
(G\circ F)[z,x]  \defeq  \bigsqcup_{y \in Y} G[z,y] \otimes F[y,x]  \, .
\]
For a set $X$, the identity matrix~$\Id_X : X\mat X$ is defined so that~$\Id_X \cdot \UnitCat \iso \Id_{X \cdot \UnitCat}$. Hence, for $x, y \in X$, we can define
\[
\Id_X[x,y]  \defeq  
\left\{
\begin{array}{ll}
I & \text{if } x = y \, , \\
0 & \text{otherwise} \, ,
\end{array}
\right.
\]
where $I$ and $0$ denote the unit object and the initial object of $\catV$, respectively. These definitions determine
an inclusion $\MatV \subseteq \DistV$.
\end{remark}

We conclude this section by defining the operation of composition of a distributor with a $\catV$-functor, which will be useful in the
discussion of composition of analytic functors 
in Section~\ref{sec:symsaf}. If $\pcatE$ is a presentable $\catV$-category, then we define the~\emph{composite} of a $\catV$-distributor $F\co \catX\mat \catY$
with a $\catV$-functor $T\co  \catY\to \pcatE$
as the~$\catV$-functor $T\circ F\co  \catX\to \pcatE$
defined by mapping $x\in \catX$ to
\begin{equation}
\label{equ:compdistfun}
(T\circ F)(x)  \defeq \int^{y\in \catY} T(y)\otimes F(y, x) \, .
\end{equation}

 \begin{lemma}\label{contraactionlemma} 
Let $\catX, \catY$ be small $\catV$-categories and $\pcatE$ a  presentable $\catV$-category.
 Let $F\co \catX\mat \catY$ be a distributor and  $T\co  \catY\to \pcatE$ be a $\catV$-functor. 
 There is an isomorphism
 \[
(T\circ F)_c \iso T_c\circ F^\dag \, ,
\]
where  $T_c\co \pshY\to \pcatE$ is the cocontinuous extension of $T\co  \catY\to \pcatE$.
\end{lemma}
 
  \begin{proof} 
    The functor $T_c\circ F^\dag\co \pshX \to \pcatE$ is cocontinuous and
 for every $x\in \catX$ we have
 \begin{align*} 
 (T_c\circ F^\dag) (\yon_\catX(x)) & \iso T_c(\lambda F)(x) \\
  &  =\int^{y\in \catY} T(y)\otimes \lambda(F)(x)(y) \\
  & =\int^{y\in \catY} T(y)\otimes F(y, x)  \\
  & =(T\circ F)(x) \, .
  \end{align*}
 Thus, $(T\circ F)_c \iso T_c\circ F^\dag$ by the uniqueness up to unique isomorphism of the concontinous extension of a functor.
  \end{proof}

 \begin{proposition}\label{contraactionprop} Let $\catX, \catY, \catZ \in \CatV$ and $\pcatE \in \PCATV$.
\begin{enumerate}[(i)]   
\item For all distributors $F\co \catX\mat \catY$, $G\co \catY\mat \catZ$ and $\catV$-functors $T\co  \catZ\to \pcatE$, $(T\circ G)\circ F \iso T\circ (G\circ F)$.
\item For all $T\co  \catZ\to \pcatE$, $T\circ \id_\catZ \iso T$.
\end{enumerate}
\end{proposition}
 
  \begin{proof} For part~(i), it suffices to show that we have $((T\circ G)\circ F)_c\iso T\circ (G\circ F)_c$.
  By Lemma~\ref{contraactionlemma} and the isomorphism in~\eqref{equ:isofordagger1} we have 
  \begin{align*} 
  ((T\circ G)\circ F)_c & \iso (T\circ G)_c\circ F^\dag \\
   & \iso (T_c\circ G^\dag)\circ F^\dag \\
   & = T_c\circ (G^\dag\circ F^\dag)  \\
   & \iso  T_c \circ  (G \circ F)^\dag \\
   & =  T\circ (G\circ F) \, .
  \end{align*}
 Part~(ii) follows by a similar reasoning.
 \end{proof}

 \medskip
 
See~\cite[Volume~I, Chapter~7]{BorceuxF:hanca} and~\cite{BenabouJ:dis,BenabouJ:distw,LawvereFW:metsgl}
for further information and~\cite{CattaniGL:proomb} for applications of distributors in theoretical computer science.

 \chapter{Monoidal distributors}
 \label{cha:mond}
 
 This chapter introduces several auxiliary bicategories of distributors, with the ultimate goal
 of defining  the bicategory of symmetric monoidal distributors, denoted by~$\SMDistV$. 
This bicategory will be used in 
 Chapter~\ref{cha:syms} to define bicategories of symmetic sequences.   
 The chapter is organized in two parts: the first consists of
Section~\ref{sec:moncat} and
 Section~\ref{sec:mondist}, and the second of Section~\ref{sec:smoncat} and Section~\ref{sec:smondist}.
 The two parts have a parallel development: we begin by defining a suitable homomorphisms and then
 define bicategories of interest by 
 considering their Gabriel factorization.  In particular, Section~\ref{sec:moncat} and
 Section~\ref{sec:mondist} lead up to the definition of the bicategories of lax monoidal distributors 
 (Theorem~\ref{thm:bicatlmdistv}) and of monoidal distributors (Theorem~\ref{thm:bicatlmdistv}),
 while Sections~\ref{sec:smoncat} and Section~\ref{sec:smondist} lead to the definition of the bicategories of symmetric lax monoidal distributors (Theorem~\ref{thm:bicatlsmdistv}) and of symmetric monoidal distributors (Theorem~\ref{thm:bicatsmdistv}).

\section{Monoidal $\catV$-categories and $\catV$-rigs}
\label{sec:moncat}

 We suppose that the reader familiar with
the notions of monoidal $\catV$-category, lax monoidal $\catV$-functor, 
and monoidal $\catV$-natural transformation. We limit ourselves to recalling
 that, for monoidal $\catV$-categories $(\catM,\otimes, e)$ and $(\catN, \otimes, e)$, 
 a \emph{lax monoidal $\catV$-functor} $F\co \catM\to \catN$ is equipped
 with multiplication and a unit
 \[
 \mu(x,y)\co F(x)\otimes F(y)\to F(x\otimes y) \, , \quad \eta\co e\to F(e) \, .
\]
We say that $F$ is a \emph{monoidal $\catV$-functor}
if~$\mu$ and~$\eta$ are invertible. 
Recall also that a $\catV$-natural transformation between
lax monoidal $\catV$-functors is \emph{monoidal} if it respects the multiplication and unit.
We write $\LMCatV$ (\resp $\MCatV$) for the 2-category of small monoidal $\catV$-categories, 
lax monoidal (\resp  monoidal) $\catV$-functors and monoidal $\catV$-natural transformations. 
If $\catM$ and $\catN$ are monoidal $\catV$-categories, 
then so is the $\catV$-category $\catM \tensorvcat \catN$.
This defines a symmetric monoidal structure on the categories  $\LMCatV$
and $\MCatV$.
The unit object is the $\catV$-category $\UnitCat$ that is the unit for the tensor product on
$\CatV$, defined in Section~\ref{sec:vcat}.
It is easy to verify that $\UnitCat$ is initial in the bicategory
$\MCatV$, in the sense that for every $\catM\in \MCatV$ we have an equivalence of categories
$\MCatV[\UnitCat, \catM] \simeq 1$, where $1$ is the terminal category. 

\medskip

Definition~\ref{thm:vrig} below introduces the notion of a $\catV$-rig. The comparison between this notion
and  notions already existing in
in the literature is simpler to describe in the symmetric case, so we postpone it until Definition~\ref{thm:symvrig}, where we introduce the notion of a symmetric $\catV$-rig.

\begin{definition} \label{thm:vrig}
A $\catV$-\emph{rig} is a monoidal closed presentable $\catV$-category.
\end{definition}

A $\catV$-rig can be defined equivalently as a monoid (in an
appropriately weak sense) in the monoidal bicategory $(\PCATV, \tensorloc,\catV)$. If $\rigR$ and $ \rigS$ are $\catV$-rigs,
we say that a cocontinuous $\catV$-functor $F\co \rigR\to \rigS$ 
is a \emph{lax homomorphism} (\resp homomorphism) of $\catV$-rigs if it is a lax monoidal (\resp monoidal) functor.
We write $\LRigV$ (\resp $\RigV$) for the 2-category 
of $\catV$-rigs, lax homomorphisms (\resp homomorphism)
and monoidal natural transformations.

We need to recall some basic facts about Day's convolution tensor product~\cite{DayBJ:conbc,DayBJ:clocf,ImGB:unipcm}. 
For a small monoidal $\catV$-category $\catM=(\catM,\oplus, 0)$ and  a $\catV$-rig $\rigR=(\rigR, \diamond, e)$, 
the $\catV$-category $[\catM,\rigR]$ can be equipped with a monoidal structure, called the \emph{convolution tensor product},
making it into a~$\catV$-rig. By definition,  the convolution product $A_1 \ast A_2$ of two $\catV$-functors $A_1,A_2\co \catM \to \rigR$
is defined by letting, for $x\in \catM$, 
\begin{equation}
\label{equ:convolution}
(A_1\conv A_2)(x) \defeq 
\int^{x_1\in \catM} \int^{x_2\in \catM} A_1(x_1)\diamond  A_2(x_2) \otimes \catM[x_1\oplus x_2,x] \, .
\end{equation}
The unit object for the convolution product is the functor $E \defeq \catM(0,-)\otimes e$. 
An important case  of the convolution tensor product  is given by considering $\catV$-rigs
of the form $\pshM = [\catM^\op,\catV]$, where  $\catM=(\catM,\oplus, 0)$ is a small monoidal 
$\catV$-category. In this case, for~$A_1 ,A_2 \in \pshM$, $x \in \catM$, we have
\[
(A_1 \conv A_2)(x)=\int^{x_1\in \catM} \int^{x_2\in \catM} A_1(x_1)\otimes  A_2(x_2) \otimes  \catM[x,x_1\oplus x_2] \, .
\]
The function mapping a small monoidal $\catV$-category $\catM$ to the $\catV$-rig $P(\catM)$ extends to a
homomorphism
\begin{equation}
\label{equ:plmcat}
P\co \LMCatV\to \LRigV \, . 
\end{equation}
Indeed, for every small monoidal $\catV$-category $\catM$ the Yoneda embedding $\yon_\catM \co  \catM\to \pshM$ becomes a monoidal functor and it exhibits $\pshM$ as the free $\catV$-rig on $\catM$. More precisely, the restriction functor
\begin{equation}
\label{equ:yonmoncat}
\yon_\catM^\ast \co \LRigV[\pshM,\rigR] \to \LMCATV[\catM, \rigR] 
\end{equation}
along  $\yon_\catM\co \catM\to \pshM$ is an equivalence of categories
for any $\catV$-rig $\rigR$. The inverse equivalence takes a lax monoidal $\catV$-functor
$F \co  \catM \to \rigR$ to its cocontinuous extension $F_{c} \co  \pshM\to \rigR$, defined as in~\eqref{equ:lowershrieck},
which can be equipped with a lax monoidal structure. Thus, the homomorphism in~\eqref{equ:plmcat} 
takes a lax monoidal $\catV$-functor $u\co \catM\to \catN$ to the lax homomorphism of rigs 
$P(u)\defeq u_!\co \pshM\to \pshN$. All of the above restricts in an evident way to the 2-category $\MCatV$
so as to give also a homomorphism
\begin{equation}
\label{equ:pmcat}
P\co \MCatV\to\RigV \, .
\end{equation}

\begin{remark} \label{thm:tensorvrig} If $\rigR$ and $\rigS$ are $\catV$-rigs, then so is the presentable~$\catV$-category $\rigR \tensorloc \rigS$ discussed in~Remark~\ref{thm:tensorpcat}.
This defines the tensor product of a symmetric monoidal closed structure on the 2-categories $\LRigV$ and  $\RigV$, with unit the category~$\catV$.
Furthermore, $\catV$~is initial in the 2-category
$\RigV$, in the sense that for every  $\rigR \in \RigV$
we have an equivalence $\RigV[\catV, \rigR] \simeq \onecat$.
 If we consider the 2-categories $\LMCatV$ and $\LRigV$ (\resp $\MCatV$ and $\RigV$) as equipped with their symmetric monoidal structures,
the homomorphism  $P\co \LMCatV\to \LRigV$ (\resp $P\co \MCatV\to \RigV$) is symmetric monoidal. Indeed, if  $\catM$ and $\catN$ are 
small monoidal~$\catV$-categories, then the equivalence of presentable categories
\[
\phi_{\catM, \catN} \co \pshM \tensorloc \pshN \to \psh(\catM\tensorvcat \catN)
\]
of Remark~\ref{thm:tensorpcat} is an equivalence of $\catV$-rigs. 
\end{remark}

The homomorphisms in~\eqref{equ:plmcat} and~\eqref{equ:pmcat}  will be used in Section~\ref{sec:mondist} to
define the bicategories of lax monoidal distributors and of monoidal distributors, respectively, via a Gabriel
factorization. In order to do this, we establish some auxiliary results.

\begin{lemma}\label{monoidalequiv} 
Let $\catM \, , \catN$ be small monoidal $\catV$-categories, and $\rigR$ be a $\catV$-rig.
The  equivalences
\[
\lambda^{\catM}\co [\catM  \tensorvcat \catN, \rigR] \to [\catN, [\catM,\rigR]]  \, , \quad
\lambda^{\catN}\co [\catM  \tensorvcat \catN, \rigR] \to [\catM, [\catN,\rigR]] 
\]
are monoidal.
\end{lemma}

\begin{proof} Let $\lambda = \lambda^\catM$. For $\catV$-functors $A_1 \, , A_2 \co \catM  \tensorvcat \catN\to  \rigR$, we  construct
a natural isomorphism 
\[
\lambda(A_1)\conv \lambda(A_2)\iso  \lambda(A_1\conv A_2) \, .
\]
By definition, for $y\in \catN$, we have
\[
\bigl( \lambda(A_1)\conv \lambda^\catM(A_2)\bigr)(y) 
=
\int^{y_1\in \catN}\int^{y_2\in \catN}  \lambda(A_1)(y_1) \conv  \lambda(A_2)(y_2) \otimes  \catN[y_1\oplus y_2,y]  \, .
\]
Hence,  for $x\in \catM$ and $y\in \catN$, we have
 \begin{equation}
 \label{equ:convolutionlemma}
 \bigl( \lambda(A_1)\conv \lambda^\catM(A_2)\bigr)(y)(x)=  \int^{y_1\in \catN}\int^{y_2\in \catN} \bigl(\lambda(A_1)(y_1) \conv  \lambda(A_2)(y_2)\bigr)(x) \otimes  \catN[y_1\oplus y_2,y] \, .
 \end{equation}
But, for $x\in \catM$ and $y_1,y_2\in \catN$, we have 
\begin{align}
 \bigl(\lambda(A_1)(y_1) \conv  \lambda(A_2)(y_2)\bigr)(x) 
&  = \int^{x_1\in \catM}\int^{x_2\in \catM}  \lambda(A_1)(y_1)(x_1) \diamond  \lambda(A_2)(y_2)(x_2) \otimes  \catM[ x_1\oplus x_2,x] \notag \\
&  \iso  \int^{x_1\in \catM}\int^{x_2\in \catM} A_1(x_1, y_1) \diamond  A_2(x_2, y_2) \otimes  \catM[ x_1\oplus x_2,x]  \label{equ:substitutlhs} \, .
\end{align}
By substituting the right-hand side of~\eqref{equ:substitutlhs} in the right-hand side of~~\eqref{equ:convolutionlemma}, it follows  that 
\begin{align*} 
& \bigl( \lambda(A_1)\conv \lambda^\catM(A_2)\bigr)(y)(x) \\
& \qquad \iso \int^{y_1\in \catN}\int^{y_2\in \catN}
\int^{x_1\in \catM}\int^{x_2\in \catM} A_1(x_1, y_1) \diamond  A_2(x_2, y_2)\otimes  \catM[ x_1\oplus x_2,x]  \otimes  \catN[ y_1\oplus y_2,y ]   \\
& \qquad \iso \int^{(x_1,y_1)\in \catM\otimes \catN}\int^{(x_2,y_2)\in \catM\otimes  \catN}
A_1(x_1, y_1) \diamond  A_2(x_2, y_2) \otimes (\catM\otimes \catN)[ (x_1,y_1)\oplus (x_2, y_2), (x,y)]  \\
& \qquad = (A_1 \conv A_2)(x,y)  \phantom{\int^{x_1}} \\
 & \qquad  = \lambda^\catM (A_1 \conv A_2)(y)(x) \, , 
 \end{align*}
as required.
\end{proof}

Let $\catM=(\catM,\oplus, 0)$ be a small monoidal $\catV$-category and $\rigR=(\rigR, \diamond, e)$ be a $\catV$-rig.
By the definition of the convolution product  of $A_1,A_2\in [\catM,\rigR]$, as given in~\eqref{equ:convolution}, there is a canonical map 
\begin{equation}
\label{equ:starp19}
\mathit{can} \co A_1(x_1)\diamond  A_2(x_2) \to (A_1 \conv A_2)(x_1\oplus x_2) \, .
\end{equation}
If $A=(A,\mu, \eta)$ is a monoid object in $[\catM,\rigR]$, then the composite
\[
\xymatrix{A(x_1)\diamond A(x_2)\ar[r]^(0.45){can} & (A\conv A)(x_1\oplus x_2) \ar[rr]^(0.55){\mu(x_1\oplus x_2)} && A(x_1\oplus x_2)
}
\]
is a lax monoidal structure on $A$ with components $\mu(x_1,x_2)\co A(x_1)\diamond A(x_2)\to A(x_1\oplus x_2)$
on~$A$. This defines a functor  $\rho\co  \Mon[\catM,\rigR]\to \LMCATV[\catM, \rigR]$, where $\Mon[\catM,\rigR]$
denotes  the category of monoids in~$\CATV[\catM,\rigR]$ and $\LMCATV[\catM, \rigR]$ denotes the
category of lax monoidal $\catV$-functors from~$\catM$ to~$\rigR$. The next lemma is essentially as in~\cite[Example~3.2.2]{DayBJ:conbc} 
and~\cite[Proposition~22.1]{MandellM:modcds}.

\begin{lemma}\label{monoidobjectlaxmonoidal} The functor $\rho\co  \Mon[\catM,\rigR]\to \LMCATV[\catM, \rigR]$ 
 is an equivalence of categories.
\end{lemma}

\begin{proof} Let us describe an inverse to~$\rho$.
Let $\mu(x_1,x_2)\co A(x_1)\diamond A(x_2)\to A(x_1\oplus x_2)$
be a lax monoidal structure on a functor $A\co \catM \to \rigR$. 
We have 
\[
(A\conv A)(x)=\int^{x_1\in \catM}\int^{x_2\in \catM} \catM(x_1\oplus x_2,x)\otimes A(x_1)\diamond A(x_2)
\]
and the natural transformation  $\mu(x_1,x_2)\co A(x_1)\diamond A(x_2)\to A(x_1\oplus x_2)$
induces a map 
\[
(A\conv A)(x) \to \int^{x_1\in \catM}\int^{x_2\in \catM} \catM(x_1\oplus x_2,x)\otimes A(x_1\oplus x_2) \to A(x)
\]
which defines the multiplication $\mu\co A\conv A\to A$
of a monoid object $(A,\mu, \eta)$ in $[\catM,\rigR]$.
It is easy to verify that this is an inverse to~$\rho$.
\end{proof}

We can now extend Proposition~\ref{thm:lambda} to categories of monoidal functors.

\begin{proposition}\label{lambdalaxmonoida} The equivalences of categories
\[
\lambda^{\catM} \co  [\catM  \tensorvcat \catN, \rigR] \to [\catN,  [\catM,\rigR]] \, , \quad 
\lambda^{\catN}\co [\catM  \tensorvcat \catN, \rigR] \to \CATV[\catM, [\catN,\rigR]] 
\]
restrict to equivalences of categories
\begin{gather*}
\lambda^\catM \co \LMCATV[\catM  \tensorvcat \catN, \rigR] \to \LMCATV[\catN,  [\catM,\rigR]] \, , \\
\lambda^\catN \co \LMCATV[\catM  \tensorvcat \catN, \rigR] \to \LMCATV[\catM,  [\catN,\rigR]] \, . 
\end{gather*}
\end{proposition}

\begin{proof} We consider only $\lambda^\catM$, since the argument for $\lambda^\catN$ is  analogous. First of all, observe 
that~$\lambda^{\catM}$ is monoidal by Lemma \ref{monoidalequiv}.
It thus induces an equivalence between the category of monoids in 
the monoidal category $[\catM  \tensorvcat \catN, \rigR]$
and the category of monoids in 
the monoidal category $[\catN,  [\catM,\rigR]]$.
The claim then follows by Lemma \ref{monoidobjectlaxmonoidal}.
\end{proof}

\begin{corollary} \label{thm:bijectionmonoidal} Let $\catM \, , \catN$ be small monoidal $\catV$-categories.
For every distributor $F \co \catM \mat \catN$, there is a bijective correspondence between the 
lax monoidal structures on the  $\catV$-functors
\[
F \co  \catN^\op  \tensorvcat \catM \to \catV \, , \quad
\lambda F\co \catM\to \pshN \, , \quad
F^\dag \co \pshM\to \pshN \, .
\]
\end{corollary} 

\begin{proof} By Proposition~\ref{lambdalaxmonoida} and the equivalence in~\eqref{equ:yonmoncat}.
\end{proof} 

\section{Monoidal distributors} 
\label{sec:mondist}

\begin{definition}\label{deflaxmonodist} Let $\catM$, $\catN$ be small monoidal $\catV$-categories. 
A \emph{lax monoidal distributor} $F \co  \catM \mat~\catN$ is a lax monoidal functor 
$F \co  \catN^\op  \tensorvcat \catM \to \catV$.
\end{definition}

The next theorem introduces the bicategory of lax monoidal distributors.

\begin{theorem} \label{thm:bicatlmdistv}
Small monoidal $\catV$-categories, lax monoidal distributors and monoidal $\catV$-trans\-for\-ma\-tions form a bicategory,
called the bicategory of lax monoidal distributors and denoted by~$\LMDistV$, 
which fits into a Gabriel factorization diagram 
\[
\xymatrix@R=1.3cm{
\LMCatV \ar[rr]^{\psh} \ar[dr]_{(-)_\bullet} & &\LRigV \\
&\LMDistV \ar[ur]_{(-)^\dag} \, . &
}
\]
\end{theorem} 

\begin{proof} For small monoidal $\catV$-categories $\catM$ and $\catN$, we define the hom-category of lax
monoidal distributors from~$\catM$ 
to~$\catN$ by letting
\[
 \LMDistV[\catM, \catN] \defeq \LMCATV[\catN^\op \otimes \catM, \catV]  \, .
 \]
Then, we define the functor
 \[
(-)^\dag\co \LMDistV[\catM,\catN] \to \LRigV[\pshM, \pshN] 
\]
as the composite
\[
\xymatrix@C=1.2cm{\LMDistV[\catM,\catN] \ar[r]^-{\lambda} &  \LMCATV[\catM, \pshN] \ar[r]^-{(-)_c} & \LRigV[\pshM,\pshN] \, .}
\]
The functor $\lambda$ is an equivalence by Proposition~\ref{lambdalaxmonoida}. Since~$(-)_c$ is also an equivalence (being a quasi-inverse to composition with $\yon_\catM$), it follows that~$(-)^\dag$ is an equivalence, as required. The rest of the data necessary to have a bicategory is determined by the requirement to have a Gabriel factorization.  In particular, the second part of the Gabriel factorization is then defined
by mapping a small monoidal $\catV$-category to $\catM^\dag \defeq \pshM$, viewed as a $\catV$-rig with the convolution monoidal structure.
For the first part of the factorization, we need to show that if  $u\co \catM\to \catN$ is a lax monoidal $\catV$-functor, then the distributor $u_\bullet\co \catM \mat  \catN$ is lax monoidal. But the functor~$u_!\co \pshM\to \pshN$ is lax monoidal, since the functor $u\co \catM\to \catN$ is  lax monoidal.
Corollary~\ref{thm:bijectionmonoidal} implies that the distributor $u_\bullet$ is lax monoidal, since we have $(u_\bullet)^\dag \iso u_!$.
\end{proof}

The composition operation 
of $\LMDistV$ is obtained as a restriction of the composition operation of $\DistV$.  Indeed, for lax monoidal distributors $F\co \catM \mat \catN$  and $G\co \catN \mat \catP$, the composite distributor $G\circ F\co \catM \mat \catP$ is lax monoidal, since there is an isomorphism 
\[
(G\circ F)^\dag \iso G^\dag \circ F^\dag
\] 
and lax monoidal  $\catV$-functors are closed under composition. Similarly, the identity morphism on a small monoidal $\catV$-category is the 
identity distributor  $\id_\catM\co  \catM\mat\catM$, which is lax monoidal since the hom-functor of $\catM$ is a lax monoidal $\catV$-functor.
All of the above admits a restriction to the case of monoidal, rather than lax monoidal, $\catV$-functors.
In order to make the theory work out smoothly, however, it is appropriate to define the notion of a monoidal distributor
as follows.

\begin{definition}\label{defmonodist} Let $ \catM \, , \catN$ be small monoidal $\catV$-categories.
A \emph{monoidal distributor} $F \co  \catM \mat \catN$ is a lax monoidal distributor such that 
the lax monoidal functor $\lambda F\co \catN \to \pshM$ is monoidal.
\end{definition}

Let us point out that requiring a lax monoidal distributor $F \co  \catM \mat \catN$ to be monoidal is not equivalent to requiring the lax monoidal $\catV$-functor $F \co 
\catN^\op \tensorvcat \catM \to \catV$ to be monoidal. For example, consider the identity distributor $\id_\catM\co  \catM\mat\catM$, which is given by the hom-functor 
$\catM(-,-) \co \catM^\op \otimes \catM \to \catV$. This $\catV$-functor  is lax monoidal, but not monoidal. However, 
$\Id_\catM \co  \catM\mat\catM$ is a monoidal distributor since~$\lambda(\id_\catM) \co \catM \to \pshM$ is the Yoneda embedding~$\yon_\catM \co \catM \to \pshM$, which is 
 monoidal.
Note that, by Corollary~\ref{thm:bijectionmonoidal}, a lax monoidal distributor  $F \co  \catM \mat \catN$ is monoidal if and only if the lax monoidal 
functor $F^\dag \co \pshN \to \pshM$ is monoidal.

\medskip

The next theorem defines the bicategory of monoidal distributors.

\begin{theorem} \label{thm:bicatmdistv}
Small monoidal $\catV$-categories, monoidal distributors and monoidal $\catV$-trans\-for\-ma\-tions form a 
bicategory, called the bicategory of monoidal distributors and denoted by~$\MDistV$, which fits in a Gabriel factorization diagram
\[
\xymatrix@R=1.3cm{
\MCatV \ar[rr]^{\psh} \ar[dr]_{(-)_\bullet} & &\RigV \\
&\MDistV \ar[ur]_{(-)^\dag} \, . &
}
\]
\end{theorem} 

\begin{proof} For small monoidal $\catV$-categories $\catM$ and $\catN$, we define the category $\MDistV[\catM, \catN]$  
as the full sub-category of $\LMDistV[ \catM, \catN]$ spanned by monoidal distributors. 
The rest of the proof follows the pattern of the one of Theorem~\ref{thm:bicatlmdistv}. In particular, 
the functor $\lambda$ used in the proof of Theorem~\ref{thm:bicatlmdistv} restricts to an equivalence
\[
\lambda \co \MDistV[\catM, \catN] \to \RigV[ \catM, \pshN]
\]
by the very definition of the notion of a monoidal distributor.
\end{proof}

\begin{remark} 
The tensor product $F_1 \tensorvcat F_2 \co  \catM_1 \tensorvcat \catM_2 \mat \catN_1   \tensorvcat \catN_2$ of lax
monoidal (\resp  monoidal) distributors $F_1\co  \catM_1\mat \catN_1$ and $F_2\co  \catM_2\mat  \catN_2$
is lax monoidal (\resp monoidal).
The operation defines a
symmetric monoidal structure on the bicategories $\LMDistV$ and $\MDistV$.
Moreover, the homomorphisms in the Gabriel 
factorizations of Theorem~\ref{thm:bicatlmdistv} and Theorem~\ref{thm:bicatmdistv} are symmetric monoidal.
Let us also point out that the symmetric monoidal bicategory $\LMDistV$
is compact: the dual of a monoidal $\catV$-category $\catM$
is the opposite $\catV$-category $\catM^\op$. The counit $\varepsilon\co \catM^{\op} \tensorvcat \catM\mat \UnitCat$
is given by the hom-functor $ \UnitCat^{\op} \tensorvcat \catM^\op  \tensorvcat \catM=\catM^\op  \tensorvcat \catM \to \catV$
and similarly for the unit $\eta\co \UnitCat\mat \catM  \tensorvcat \catM^{\op}$.
In contrast,  the symmetric monoidal bicategory $\MDistV$ is \emph{not} compact.
\end{remark}

\medskip

We conclude this section by restricting to the  monoidal case the operation of composition of a 
distributor with a functor, defined in~\eqref{equ:compdistfun}.

\begin{proposition} \label{actionmono} Let $\rigR$ be a $\catV$-rig. Then the  composite  
of a monoidal distributor $F\co \catM\mat \catN$
with a monoidal functor $T\co  \catN \to \rigR$ is a monoidal functor $T\circ F\co \catM\to \rigR$.
\end{proposition}

\begin{proof} It suffices to show that  $(T\circ F)_c\co \pshM\to \rigR$
is monoidal. The functor $F^\dag\co \pshM \to \pshN$
is monoidal, since the distributor $F\co \catM\mat \catN$ is monoidal, and
the functor $T_c\co \pshN\to \rigR$
is monoidal, since  $T$ is monoidal.
Hence, the functor $T_c\circ F^\dag\co \pshM\to \rigR$
is monoidal. This proves the result, since we have
$(T\circ F)_c \iso T_c\circ F^\dag$ by Lemma~\ref{contraactionlemma}.
\end{proof}

\section{Symmetric monoidal $\catV$-categories and symmetric $\catV$-rigs}
\label{sec:smoncat}

The aim of this section is to develop the counterpart for symmetric monoidal $\catV$-categories of the material in Section~\ref{sec:moncat}. Let us recall
that, for symmetric monoidal $\catV$-categories $\catA$ and $\catB$, a lax monoidal $\catV$-functor $F\co \catA\to \catB$ 
is said to be \emph{symmetric} if, for any pair of objects $x\in \catA$ and~$y\in \catB$,  the following square commutes
\[
\xymatrix{
F(x)\otimes F(y) \ar[d]_{\sigma} \ar[rr]^{\mu(x,y)}&& F(x\otimes y)\ar[d]^{F(\sigma)}\\
F(y)\otimes F(x) \ar[rr]_{\mu(y,x)} &&F(y\otimes x) \, , }
\]
where we use $\sigma$ to denote the symmetry isomorphism of both $\catA$ and $\catB$.
We write $\LSMCatV$ (\resp $\SMCatV$) for the 2-category of small symmetric monoidal $\catV$-categories, 
symmetric lax monoidal (\resp  symmetric monoidal) $\catV$-functors and monoidal $\catV$-natural transformations. 
If $\catA$ and $\catB$ are symmetric monoidal $\catV$-categories, 
then so is the $\catV$-category $\catA\tensorvcat \catB$.
This operation defines a symmetric monoidal structure on the categories  $\LSMCatV$
and $\SMCatV$. The unit object  is the $\catV$-category $\UnitCat$ giving the unit for the tensor product of $\CatV$, as defined in Section~\ref{sec:vcat}.
It is easy to verify that $\UnitCat$ is initial in the 
2-ca\-te\-go\-ry~$\SMCatV$. If $\catA= (\catA, \oplus, 0, \sigma)$ 
is a symmetric monoidal category, then the interchange law
\[
(x_1\oplus x_2)\oplus (y_1\oplus y_2) \iso (x_1\oplus y_1) \oplus (x_2\oplus y_2)
\]
is a natural (symmetric) monoidal structure on the
 tensor product functor of $\catA$.

 \begin{lemma} \label{thm:smoncathascoprod}
  The $2$-category $\SMCatV$ has finite coproducts. In particular,
  the coproduct of two small symmetric monoidal $\catV$-categories $\catA_1$ and $\catA_2$
  is their tensor product $\catA_1 \tensorvcat \catA_2$.
  \end{lemma}

 \begin{proof} Let $\catA_1 = (\catA_1, \oplus, 0) \,  , \catA_2 = (\catA_2, \oplus, 0)$ 
 be two symmetric monoidal categories. We define the functors $\iota_1\co \catA_1 \to \catA_1  \tensorvcat \catA_2$
 and $\iota_2\co  \catA_2 \to \catA_1  \tensorvcat \catA_2$ by letting $\iota_1(x) \defeq (x,0)$ and $\iota_2(x) \defeq (0,x)$.
 We now have to show  that the functor
 \[
 \pi\co \SMCatV[\catA_1  \tensorvcat \catA_2, \catB ]\to  \SMCatV[\catA_1, \catB]\times \SMCatV[\catA_2, \catB] \, , 
 \]
defined by letting $\pi(F) \defeq (F\circ \iota_1,F\circ \iota_2)$, is an equivalence of categories for any 
symmetric monoidal $\catV$-category~$\catB$.
The tensor product functor $\mu\co \catB\tensorvcat \catB\to \catB$
is symmetric monoidal, since the monoidal category  $\catB$ is symmetric.
Thus, if $F_1\co  \catA_1 \to  \catB$ and $F_2\co  \catA_2 \to  \catB$ are symmetric monoidal functors, then 
the functor $\mu\circ (F_1\otimes F_2)\co \catA_1  \tensorvcat \catA_2 \to\catB$
is symmetric monoidal. Hence, we obtain a functor 
 \[
 \rho\co  \SMCatV[\catA, \catB]\times \SMCatV[\catA_2, \catB]\to \SMCatV[\catA_1  \tensorvcat \catA_2, \catB ]
 \]
defined by letting $\rho(F_1,F_2) \defeq \mu\circ (F_1\otimes F_2)$.
The verification that the functors $\pi$ and $\rho$
are mutually pseudo-inverse is left to the reader.
 \end{proof}

The notion of a symmetric $\catV$-rig, introduced in Definition~\ref{thm:symvrig} below, generalizes to the
enriched case the notion of a 2-ring introduced in~\cite{ChirvasituA:funpga} (where it is defined as a 
symmetric closed monoidal presentable category), which is a special case of the
notion of a 2-rig introduced in~\cite{BaezJ:higdan} (where it is defined as a symmetric monoidal cocomplete category in which the tensor product preserves colimits in both variables).

\begin{definition} \label{thm:symvrig}
A \emph{symmetric $\catV$-rig} is a symmetric monoidal closed presentable $\catV$-category.
\end{definition}

If $\rigR$ and $ \rigS$ are symmetric $\catV$-rigs,
we say that a lax homomorphism (\resp homomorphism) of $\catV$-rigs $F\co \rigR\to \rigS$ 
is a \emph{symmetric} if it is symmetric as a  lax monoidal (\resp monoidal) $\catV$-functor.
We write  $\LSRigV$ (\resp $\SRigV$) for the 2-category of symmetric $\catV$-rigs, 
symmetric lax homomorphisms (\resp symmetric  homomorphisms) and monoidal $\catV$-natural transformations. 

\medskip

The convolution tensor product extends in a natural way to the symmetric case~\cite{DayBJ:clocf,ImGB:unipcm}. Indeed, 
if~$\catA=(\catA,\oplus,0,\sigma_\catA)$ is a small symmetric monoidal $\catV$-category
and $\rigR=(\rigR, \diamond, e,\sigma_{\rigR})$ is symmetric $\catV$-rig
then the $\catV$-rig $[\catA,\rigR]$ is symmetric. By definition, for $F,G\co \catA\to \rigR$, 
 the value at $x\in \catA$ of the symmetry isomorphism $\sigma\co F\conv G\to G\conv F$
is the coend over $x_1,x_2\in \catA$ of the maps
\[
\sigma_{\rigR} \otimes \catA[\sigma_\catA,x] \co F(x_1)\diamond G(x_2)\otimes \catA[ x_1\oplus x_2,x] \to 
G(x_2)\diamond F(x_1)\otimes \catA[ x_2\oplus x_1,x]  \, .
\]

If $\catA$ is a small symmetric monoidal $\catV$-category, then
 $\pshA=[\catA^\op,\catV]$ is a symmetric $\catV$-rig.
The Yoneda functor $\yon_\catA\co  \catA\to \pshA$
is symmetric monoidal and exhibits $\pshA$ as the free symmetric $\catV$-rig on $\catA$.
More precisely, this means that the restriction functor  
\begin{equation}
\label{equ:yonsmoncat}
\yon_\catA^\ast \co \LSRigV[\pshA,\rigR] \to \LSMCATV[\catA, \rigR] 
\end{equation}
along the  Yoneda functor $\yon_\catA\co  \catA\to \pshA$
is an equivalence of categories for any symmetric $\catV$-rig~$\rigR$.
The inverse equivalence takes a  symmetric lax monoidal $\catV$-functor
$F \co  \catA \to \rigR$ to the functor $F_{c} \co  \pshA\to \rigR$, which
can be equipped with the structure of a lax homomorphism of $\catV$-rigs.
We write
\[
P\co \LSCatV\to\LSRigV
\]
for the homomorphism of bicategories which takes  a symmetric  lax  monoidal functor
$u\co \catA\to \catB$ to the lax symmetric homomorphism of symmetric $\catV$-rigs $P(u)\defeq u_!\co \pshA\to \pshB$.
All of the above restricts to symmetric monoidal $\catV$-functors and symmetric
homomorphisms of $\catV$-rigs and so we obtain a homomorphism
\[
P\co \SMCatV\to\SRigV \, .
\]

\begin{remark} If $\rigR$ and $\rigS$ are symmetric $\catV$-rigs,
then so is their completed tensor product $\rigR \tensorloc \rigS$ (cf.\ Remark~\ref{thm:tensorpcat})
This defines  a symmetric monoidal structure on the bicategories $\LSRigV$ and  $\SRigV$.
The unit object is the category $\catV$.
If $\catA$ and $\catB$ are symmetric monoidal $\catV$-categories,
then the canonical functor 
 \[
 \pshA  \tensorloc  \pshB \to \psh(\catA\otimes \catB)
 \] 
 is an equivalence of symmetric $\catV$-rigs.
This witnesses the fact that the homomorphisms of bicategories $P\co \SMCatV\to \SRigV$
and $P\co \LSMCatV\to \LSRigV$ are symmetric monoidal.
\end{remark}

We now proceed to extend Proposition~\ref{lambdalaxmonoida} to functor categories of symmetric lax monoidal $\catV$-functors.
The first step is the following lemma, which is a counterpart of Lemma~\ref{monoidalequiv} for symmetric monoidal $\catV$-categories.

\begin{lemma}\label{symmonoidalequiv} 
Let $\catA\, , \catB$ be small symmetric monoidal $\catV$-categories and ~$\rigR$ be a symmetric $\catV$-rig.
Then, the monoidal equivalences
\[
\lambda^{\catA}\co [\catA  \tensorvcat \catB, \rigR] \to [\catB, [\catA,\rigR]]  \, , \quad
\lambda^{\catB}\co [\catA  \tensorvcat \catB, \rigR] \to [\catA, [\catB,\rigR]] 
\]
are symmetric. \end{lemma}

\begin{proof}  Similar to the proof of Lemma~\ref{monoidalequiv}.
\end{proof}

 Let $\catA=(\catA,\oplus,0,\sigma_\catA)$ be a small symmetric monoidal $\catV$-category and $\rigR=(\rigR, \diamond, e,\sigma_{\rigR})$ be a symmetric $\catV$-rig. As we have just seen, the $\catV$-rig $[\catA,\rigR]$ is symmetric. Now, for $F,G\co  \catA\to \rigR$, if
\[
\sigma\co F\conv G\to G\conv F
\] 
is the symmetry isomorphism, then the following diagram commutes for  every $x_1,x_2\in \catA$, 
 \begin{equation}
\label{equ:symmetryconvolsquare}
\xycenter{
F(x_1)\diamond G(x_2)\ar[d]_{\sigma_{\rigR}} \ar[rr]^{can} && (F\conv G)(x_1\oplus x_2) \ar[d]^{\sigma(\sigma_\catA)} \\
G(x_2)\diamond F(x_1)\ar[rr]_{can} && (G\conv F)(x_2\oplus x_1) \, ,
}
\end{equation}
where the maps labelled $\mathit{can}$ are as in~\eqref{equ:starp19}. For a small symmetric monoidal $\catV$-category $\catA$ and a symmetric $\catV$-rig $\rigR$, we write $\CMon[\catA,\rigR]$ 
for the category of commutative monoid objects in the symmetric $\catV$-rig $[\catA,\rigR]$, which is a full subcategory of the category 
$\Mon[\catA,\rigR]$ of monoid objects in $[\catA,\rigR]$. The next lemma is a counterpart of Lemma~\ref{monoidobjectlaxmonoidal} for symmetric monoidal $\catV$-categories.

\begin{lemma}\label{commonoidobjectlaxsymmonoidal} Let $\catA$ be a small symmetric monoidal $\catV$-category
and $\rigR$ be a symmetric $\catV$-rig. Then, the equivalence of categories 
\[
\rho\co \Mon[\catA,\rigR]\to \LMCATV[\catA, \rigR]
\]
restricts to an equivalence of categories 
\[
\CMon[\catA,\rigR]\to \LSMCATV[\catA, \rigR] \, .
\]
\end{lemma}

\begin{proof}
Let us show that the functor $\rho$ takes a commutative monoid object 
$F=(F,\mu, \eta)$ in~$[\catA,\rigR]$ to a symmetric lax monoidal functor.
By definition, the lax monoidal structure 
\[
\mu(x_1,x_2)\co F(x_1)\diamond F(x_2) \to  F(x_1\oplus x_2)
\]
on the functor $F\co \catA\to \catV$ is obtained by composing the maps
\begin{equation}
\label{equ:composingadd}
\xymatrix{
F(x_1)\diamond F(x_2)\ar[r]^-{can} & (F\conv F)(x_1\oplus x_2) \ar[rr]^{\mu(x_1\oplus x_2)} && F(x_1\oplus x_2).
}
\end{equation}
Let us now consider the following diagram:
\[
\xymatrix{
(F\conv F)(x_1\oplus x_2)\ar[d]_{\sigma(x_1\oplus x_2)} \ar[rr]^{\mu(x_1\oplus x_2)} && F(x_1\oplus x_2) \ar@{=}[d] \\
(F\conv F)(x_1\oplus x_2)\ar[d]_{(F\conv F)(\sigma_\catA)} \ar[rr]^{\mu(x_1\oplus x_2)} && F(x_1\oplus x_2)\ar[d]^{F(\sigma_\catA)} \\
(F\conv F)(x_2\oplus x_1)\ar[rr]_{\mu(x_2\oplus x_1)} && F(x_2\oplus x_1) \, .
}
\] 
Its top square commutes since the product $\mu\co F\conv F\to F$ is commutative, while the bottom square commutes by naturality.
It follows by composing that following square commutes,
 \begin{equation}
\label{equ:symmetryconvols2}
\xycenter{
(F\conv F)(x_1\oplus x_2)\ar[d]_{\sigma(\sigma_\catA)} \ar[rr]^{\mu(x_1\oplus x_2)} &&  F(x_1\oplus x_2) \ar[d]^{F(\sigma_\catA) } \\
(F\conv F)(x_2\oplus x_1)\ar[rr]_{\mu(x_2\oplus x_1)} && F(x_2\oplus x_1) \, .
}
\end{equation}
But the following square commutes, being an instance of the diagram in~\eqref{equ:symmetryconvolsquare}
 \begin{equation}
\label{equ:symmetryconvolsquare3}
\xycenter{
F(x_1)\diamond F(x_2)\ar[d]_{\sigma_{\rigR}} \ar[rr]^{can} && (F\conv F)(x_1\oplus x_2) \ar[d]^{\sigma(\sigma_\catA)} \\
F(x_2)\diamond F(x_1)\ar[rr]_{can} && (F\conv F)(x_2\oplus x_1).
}\end{equation}
If we compose horizontally the squares  in~\eqref{equ:symmetryconvols2}
and~\eqref{equ:symmetryconvolsquare3}, we obtain the following commutative square
\[
\xymatrix{
F(x_1)\diamond F(x_2)\ar[d]_{\sigma_\rigR} \ar[rr]^{\mu(x_1,x_2)} && F(x_1\oplus x_2) \ar[d]^{F(\sigma_\catA) } \\
F(x_2)\diamond F(x_1)\ar[rr]_{\mu(x_2,x_1)}  && F(x_2\oplus x_1) \, . }
\]
This shows that the lax monoidal structure in~\eqref{equ:composingadd}
is symmetric.
\end{proof}

It is now possible to extend Proposition~\ref{lambdalaxmonoida} to the symmetric case.

\begin{proposition}\label{lambdalaxmonoidasymm} The equivalences of categories
\[
\lambda^{\catA}\co [\catA  \tensorvcat \catB, \rigR] \to [\catB,  [\catA,\rigR]] \, , \quad
\lambda^{\catB}\co [\catA  \tensorvcat \catB, \rigR] \to [\catA,  [\catB,\rigR]] 
\]
restrict to equivalences of categories
\begin{align*}
\LSMCATV[\catA  \tensorvcat \catB, \rigR] & \simeq \LSMCATV[\catB,  [\catA,\rigR]] \, , \\
\LSMCATV[\catA  \tensorvcat \catB, \rigR] & \simeq \LSMCATV[\catA,  [\catB,\rigR]] \, .
\end{align*} 
\end{proposition}

\begin{proof} This follows from Lemma~\ref{symmonoidalequiv}
and Lemma~\ref{commonoidobjectlaxsymmonoidal}.
\end{proof}

The next corollary is the counterpart  of Corollary~\ref{thm:bijectionmonoidal} in the symmetric monoidal case.

\begin{corollary} \label{thm:symmlaxmonoidalbijection} Let $\catA \, , \catB$ be small symmetric $\catV$-categories.
For every distributor $F \co \catA \mat \catB$,  the  symmetric lax monoidal structures on the $\catV$-functors
\[
F \co \catB^\op \otimes \catA \to~\catV \, ,  \quad \lambda F\co \catA\to \pshB \, ,  \quad  F^\dag \co \pshA\to \pshB 
\]
are in bijective correspondence.
\end{corollary}

\begin{proof} The claim follows by Proposition~\ref{lambdalaxmonoidasymm} and the equivalence in~\eqref{equ:yonsmoncat}.
\end{proof}

\section{Symmetric monoidal distributors} 
\label{sec:smondist}

\begin{definition}\label{defsymlaxmonodist} Let $ \catA$, $\catB$ be small symmetric monoidal $\catV$-categories. 
A \emph{symmetric lax monoidal distributor} $F \co  \catA \mat \catB$ is a symmetric lax monoidal functor 
$F \co  \catB^\op  \tensorvcat \catA \to \catV$.
\end{definition}

The next theorem introduces the bicategory of symmetric lax monoidal distributors.

\begin{theorem} \label{thm:bicatlsmdistv}
Small symmetric monoidal $\catV$-categories, symmetric lax monoidal distributors and monoidal $\catV$-natural transformations
form a bicategory, called the bicategory of symmetric lax monoidal distributors and denoted by~$\LSMDistV$, which fits in a Gabriel factorization
\[
\xymatrix@R=1.3cm{
\LSMCatV \ar[rr]^{\psh} \ar[dr]_{(-)_\bullet} & &\LSRigV \\
&\LSMDistV \ar[ur]_{(-)^\dag} \, . &
}
\]
\end{theorem}

\begin{proof} For small symmetric monoidal $\catV$-categories $\catA$ and $\catB$ we define the
hom-category of symmetric lax monoidal distributors from $\catA$ to $\catB$ by letting 
\[
 \LSMDistV[\catA, \catB] \defeq \LSMCATV[\catB^\op \otimes \catA, \catV]  \, .
 \]
With this definition, we have an equivalence of categories
\[
(-)^\dag\co \LSMDistV[\catA,\catB] \to \LSRigV[\pshA, \pshB] \, ,
\]
which is defined as the composite of the following two equivalences:
\[
\xymatrix@C=1.2cm{ \LSMDistV[\catA,\catB] \ar[r]^-{\lambda}  &  \LSMCATV[\catA, \pshB] \ar[r]^-{(-)_c}  & \LSRigV[\pshA,\pshB] \, .}
\]
The rest of the data is determined by the requirement of having a Gabriel factorization. In particular, the action of~$(-)^\dag\co \LSMDistV \to \LSRigV$ 
on objects by letting $\catA^\dag \defeq \pshA$. If~$u\co \catA\mat \catB$ is a  symmetric lax monoidal $\catV$-functor between 
symmetric monoidal $\catV$-categories, then the distributor $u_\bullet\co \catA \mat  \catB$
is symmetric lax monoidal. By Corollary~\ref{thm:symmlaxmonoidalbijection}
the lax monoidal functor $u_!\co \pshA\to \pshB$ is symmetric, since the lax monoidal functor $u\co \catA\to \catB$ is 
 symmetric. Hence the lax  monoidal distributor $u_\bullet$ is symmetric, since we have $(u_\bullet)^\dag \iso u_!$.
\end{proof}

Note that the composition law and the identity morphisms are defined as in $\DistV$. 
Indeed, for symmetric lax monoidal distributors $F\co \catA \mat \catB$  and $G\co \catB \mat \catC$,
the distributor  $G\circ F\co \catA \mat \catC$ is  symmetric lax monoidal, since $(G\circ F)^\dag \iso G^\dag \circ F^\dag$
 and  symmetric lax monoidal  functors are closed under composition. Furthermore,  the identity distributor $\id_\catA\co  \catA\mat\catA$ is symmetric 
 lax monoidal, since the hom-functor from $\catA^{\op}\otimes \catA$ to $\catV$ is  symmetric lax monoidal.

\begin{definition}\label{defsymmonodist} If $ \catA$ and $\catB$ are small symmetric monoidal $\catV$-categories, 
we say that a monoidal distributor $F \co  \catA \mat \catB$ is \emph{symmetric} if the
monoidal functor $\lambda F\co \catA \to \pshB$ is symmetric.
\end{definition}

The next theorem introduces the bicategory of symmetric monoidal distributors.

\begin{theorem} \label{thm:bicatsmdistv}
Small symmetric monoidal $\catV$-categories, symmetric monoidal distributors and monoidal $\catV$-trans\-for\-ma\-tions form 
a bicategory $\SMDistV$ which fits in a Gabriel factorization
\[
\xymatrix@R=1.3cm{
\SMCatV \ar[rr]^{\psh} \ar[dr]_{(-)_\bullet} & &\SRigV \\
&\SMDistV \ar[ur]_{(-)^\dag} \, . &
}
\]
\end{theorem} 

\begin{proof} For  small symmetric monoidal $\catV$-categories $\catA$ and $\catB$, we define 
$\SMDistV[\catA, \catB]$ as the full sub-category of $\LSMDistV[\catA,\catB]$
spanned by symmetric monoidal distributors. In this way, the functor $\lambda$ in the proof
of Theorem~\ref{thm:bicatlsmdistv} restricts to an equivalence 
\[
\lambda \co \SMDistV[\catA,\catB] \to  \SRigV[\catA, \pshB ] \, .
\] 
The rest of the proof follows the usual pattern.
\end{proof}

Once again, the composition law and  the identity morphisms of $\SMDistV$ are defined as in $\DistV$.

\begin{remark}
The tensor  product $F_1 \tensorvcat F_2 \co  \catA_1 \tensorvcat \catA_2 \mat \catB_1   \tensorvcat \catB_2$ of 
symmetric lax monoidal (\resp  symmetric monoidal) distributors $F_1\co  \catA_1\mat \catB_1$ and $F_2\co  \catA_2\mat  \catB_2$
is symmetric lax monoidal (\resp  symmetric monoidal).
The operation defines a
symmetric monoidal structure on the bicategories $\LSMDistV$ and  $\SMDistV$.
Moreover, the homomorphisms involved in the Gabriel factorizations
of Theorem~\ref{thm:bicatlsmdistv} and Theorem~\ref{thm:bicatsmdistv} are symmetric monoidal.
  The symmetric monoidal bicategory $\LSMDistV$
is compact: the dual of a symmetric monoidal $\catV$-category $\catA$
is the opposite $\catV$-category $\catA^\op$. The counit $\varepsilon\co \catA^{\op} \tensorvcat \catA\mat \UnitCat$
is given by the hom-functor $ \UnitCat^{\op} \tensorvcat \catA^\op  \tensorvcat \catA=\catA^\op  \tensorvcat \catA \to \catV$
and similarly for the unit $\eta\co \UnitCat\mat \catA  \tensorvcat \catA^{\op}$.
In contrast,  the symmetric monoidal bicategory $\SMDistV$ is \emph{not} compact.
\end{remark}

\medskip

We conclude the section by extending Proposition~\ref{actionmono} to the symmetric case.

\begin{proposition} \label{actionsymmetricmono} Let $\rigR$ be a symmetric $\catV$-rig. Then the  composite  
of a symmetric monoidal distributor $F\co \catA\mat \catB$
with a symmetric monoidal functor $T\co  \catB \to \rigR$ is a symmetric monoidal functor $T\circ F\co \catA\to \rigR$.
\end{proposition}

\begin{proof} Similar to the proof of Proposition~\ref{actionmono}.
\end{proof}

\chapter{Symmetric sequences}
\label{cha:syms}

This chapter introduces the bicategory of categorical symmetric sequences, which is denoted by~$\CatSymV$, and presents our
first main result, asserting that $\CatSymV$ is cartesian closed (Theorem~\ref{thm:smondistiscc}). This generalizes the main result 
in~\cite{FioreM:carcbg} to the enriched setting. In order to define the bicategory~$\CatSymV$, we construct an auxiliary 
bicategory, called the bicategory of $S$-distributors and denoted by~$\SDistV$. The definition of this bicategory is based on
that of the bicategory of symmetric monoidal distributors (presented in Section~\ref{sec:smondist}), some basic facts about
the construction of free symmetric monoidal $\catV$-categories, and  the notion of a Gabriel factorization. 
The special case of this bicategory obtained by considering $\catV = \Set$ was introduced in~\cite{FioreM:carcbg} using  the theory of Kleisli bicategories and pseudo-distributive laws. Although that approach is likely to carry over to the enriched case and lead to an equivalent definition to the one given here, we prefer to use the idea of a Gabriel factorization, since it allows us to avoid almost entirely the discussion of coherence conditions. Furthermore, when $\catV = \Set$, the monoidal categories
$\SymV[X,X]$ (with tensor product given by composition in $\SymV$) are exactly those defined in~\cite{BaezJ:higdan} in
order to characterize operads as monoids. 

We begin in Section~\ref{sec:fresmc} by reviewing some material on free symmetric monoidal categories. Section~\ref{sec:sdist}
introduces the bicategory of $S$-distributors and establishes some basic facts about it. In particular, we prove that it has coproducts.
In Section~\ref{sec:symsaf} we then define the bicategory of categorical symmetric sequences by duality, letting
\[
\CatSymV \defeq (\SDistV)^\op \, .
\]
We also define the
analytic functor associated to a categorical symmetric sequence and present the sub-category $\SymV$ of~$\CatSymV$ spanned by sets.
The morphisms in that bicategory are called symmetric sequences and they can be thought of as the many-sorted generalization
of the single-sorted symmetric sequences used in connection with single-sorted operads. Indeed, we will use $\SymV$ in 
Chapter~\ref{cha:bicob} to define the bicategory of operad bimodules.  The proof that $\CatSymV$ is cartesian closed is given in Section~\ref{sec:carcsc}, which concludes the chapter.

\section{Free symmetric monoidal $\catV$-categories}
\label{sec:fresmc}

The forgetful 2-functor $U \co  \SMCatV\to \CatV$ has a left adjoint $S \co \CatV \to \SMCatV$
which associates to a small $\catV$-category $\catX$
the symmetric monoidal $\catV$-category $S(\catX)$ freely generated by~$\catX$~\cite{BlackwellR:twodmt}.
This $\catV$-category is defined by letting
 \begin{equation}
\label{equ:sigmastarset}
S(\catX) \defeq \bigsqcup_{n \in \Nat} S^n (\catX) \, ,
\end{equation}
 where $S^n(\catX)$ is the \emph{symmetric $n$-power} of $\catX$.
 More explicitly, observe that the $n$-th symmetric group~$\Sigma_n$ acts naturally on the $\catV$-category $\catX^n$
with the right action defined by letting
\[
\vec{x} \cdot \sigma \defeq (x_{\sigma 1}, \ldots, x_{\sigma n}) \, ,
\]
for~$\vec{x}=(x_{1}, \ldots, x_{ n})\in \catX^n$ and $\sigma \in \Sigma_n$.
If we apply the Grothendieck construction
to this right action, we obtain the symmetric $n$-power of $\catX$, 
 \[
S^n(\catX)  \defeq \Sigma_n\int \catX^n \, .
 \]
Explicitly, an object of $S^n(\catX)$
 is a sequence $\vec{x}=(x_{1}, \ldots, x_{n})$
of objects of  $\catX$, and the hom-object between $\vec{x}, \vec{y} \in S^n(\catX)$ is defined by letting
 \[
S^n(\catX)\big[\vec{x}, \vec{y}\big] \defeq 
\bigsqcup_{\sigma \in \Sigma_n}  \catX[ x_1, y_{\sigma(1)}]  \otimes \cdots \otimes \catX[ x_n, y_{\sigma(n)}]  \, ,
\]
where the coproduct on the right-hand side is taken in $\catV$. 
 The tensor product of~$\vec{x}\in S^m(\catX) $ and~$\vec{y}\in S^n(\catX) $ is the concatenation
\[
\vec{x} \oplus \vec{y}\defeq (x_1, \ldots, x_m, y_1,\ldots, y_n) \, .
\]
The symmetry $\sigma_{\vec{x}, \vec{y}} \co  \vec{x} \oplus \vec{y} \to \vec{y} \oplus \vec{x}$
is the shuffle permutation swapping the first $m$-elements of the first 
sequence with the last $m$-elements of the second. The unit is the empty sequence $e$.
The inclusion $\catV$-functor
\[
\iota_\catX \co  \catX \rightarrow S(\catX) \, ,
\] 
which 
takes $x \in \catX$ to the one-element sequence~$(x) \in S^1(\catX)$, 
exhibits $S(\catX)$ as the free symmetric monoidal $\catV$-category on $\catX$.
More precisely, for every symmetric monoidal $\catV$-category $\catA=(\catA,\oplus,0,\sigma)$
the restriction functor
\begin{equation*}
\iota_\catX^\ast \co   \SMCatV[S(\catX),\catA] \to \CatV[ \catX, \catA] \, , 
\end{equation*} 
defined by letting $\iota_\catX^\ast(v) \defeq v\circ \iota_\catX$, 
is an equivalence of categories.
It follows  that every $\catV$-functor $T \co  \catX \to \catA$
admits a symmetric monoidal extension $T^e\co  S(\catX ) \to \catA$ fitting in the 
diagram,
\[
\xymatrix{
 \catX \ar@/_1pc/[drr]_{T}  \ar[rr]^-{\iota_\catX} &&S(\catX ) \ar[d]^{T^e} \\
 && \catA, }
 \]
and that $T^e$ is unique up to unique isomorphism of symmetric monoidal $\catV$-functors.
Explicitly,  for  $\vec{x}=(x_{1}, \ldots, x_{n})\in S(\catX)$, we have
 \[
 T^e(\vec{x}) \defeq T(x_1) \oplus \ldots \oplus T(x_n) \, .
  \]

\medskip

Our next theorem shows how the adjunction between $\CatV$ and $\SMCatV$ extends to an
adjunction between~$\DistV$ and~$\SMDistV$.

\begin{theorem}  \label{freesymmtricdist} The forgetful homomorphism $U \co  \SMDistV\to \DistV$
has a left adjoint 
\[
S\co \DistV\to \SMDistV \, .
\] 
\end{theorem}

\begin{proof} Let $\iota = \iota_\catX \co \catX \to S(\catX)$, so that we have a distributor $\iota_\bullet \co \catX \mat S(\catX)$. 
We need to show that the restriction functor
\begin{equation*}
\iota_\bullet^\ast \co   \SMDistV[S(\catX),\catA] \to \DistV[ \catX, \catA]
\end{equation*} 
defined by letting $\iota_\bullet^\ast (F) \defeq F\circ \iota_\bullet$ 
is an equivalence of categories for any
symmetric monoidal $\catV$-category $\catA$.
The following diagram commutes up to isomorphism by 
part~(i) of~Lemma~\ref{bulletlambda}
\[
\xymatrix{
\SMDistV[S(\catX), \catA] 
  \ar[rr]^{(-)\circ \iota_\bullet} \ar[d]_\lambda && \DistV[\catX,\catA] \ar[d]^\lambda \\
\SMCatV[S(\catX), \psh(\catA) \big] \ar[rr]_{(-)\circ \iota} &&  \CATV \big[ \catX, \psh(\catA)\big] \,  . }
\]
The bottom side of this diagram is an equivalence of categories.
Hence also the top side, since the vertical sides are equivalences.
\end{proof}

We give an explicit formula for the symmetric monoidal extension $F^e\co  S(\catX ) \mat \catA$
of a distributor $F\co  \catX \mat \catA$, which fits in the diagram of $S$-distributors
\[
\xymatrix{
 \catX \ar@/_1pc/[drr]_{F}  \ar[rr]^-{(\iota_\catX)_\bullet} &&S(\catX ) \ar[d]^{F^e}, \\
 && \catA \, }
 \]
and is the unique such distributor up to a unique isomorphism of symmetric monoidal $\catV$-distributors. 
By construction, the functor $\lambda (F^e)\co S(\catX)\to \pshA$
 is the symmetric monoidal extension of the $\catV$-functor $\lambda F\co  \catX\to \pshA$.
 Thus, $\lambda (F^e)=(\lambda F)^e$
 and it follows that 
 \begin{equation}
\label{equ:fprime}
F^e[ y; \vec{x}]  \defeq 
\int^{y_1\in  \catA} \ldots \int^{y_n \in  \catA} 
 F[ y_1,x_1]  \otimes \cdots\otimes  F[ y_n, x_n]   \otimes \catA[ y, y_1 \oplus \cdots \oplus y_n ] \, ,
\end{equation} 
 for $\vec{x} = (x_1, \ldots, x_n)\in S^n(\catX)$
 and $y\in \catA$. A special case of these definitions that will be of importance for our development arises
 by considering $\catA = S(\catY)$, where $\catY$ is a small $\catV$-category. In this case,
the symmetric monoidal extension $F^e\co  S(\catX) \mat S(\catY)$ of  a distributor $F\co \catX\mat S(\catY)$
is defined by letting
\begin{equation}
\label{equ:fprime2}
F^e[ \vec{y}; \vec{x}]  \defeq 
\int^{\vec{y}_1\in  S(\catY)} \cdots \int^{\vec{y}_n \in S( \catY)} 
 F[ \vec{y}_1;x_1]  \otimes \cdots\otimes  F[ \vec{y}_n; x_n]  \otimes S(\catY)[ \vec{y}, \vec{y}_1 \oplus \cdots \oplus \vec{y}_n]  \, ,
\end{equation} 
for $\vec{x} = (x_1, \ldots, x_n)\in S^n(\catX)$  and $\vec{y} \in S(\catY)$.

\medskip

It will be useful to describe the action of the homomorphism $S \co \DistV \to \SMDistV$ on morphisms. 
For a distributor $F\co \catX\mat \catY$, the symmetric monoidal distributor
 $S(F) \co S(\catX) \mat S(\catY)$ is defined by letting
 $S(F) \defeq (\iota_\bullet \circ F)^e$
 and therefore makes  following diagram commute up to a canonical isomorphism
\[
\xymatrix@C=1.5cm{
\catX \ar[d]_{F} \ar[r]^-{\iota_\bullet}   & S(\catX) \ar[d]^{S(F)} \\
\catY \ar[r]_-{\iota_\bullet} & S(\catY).
}
\]
Explicitly, for $\vec{x}=(x_1,\ldots,x_n)$ and $\vec{y}=(y_1,\ldots, y_m)$, we have 
\[
S(F)[ \vec{y},\vec{x}]  \defeq 
\left\{
\begin{array}{ll}
\bigsqcup_{\sigma \in \Sigma_n}  F[ y_1, x_{\sigma(1)}]   \otimes \cdots \otimes F[ y_n, x_{\sigma(n)}]  \, ,   & \text{if } m = n \, , \\[1ex]
0 & \text{otherwise} \, .
\end{array}
\right.
\]
We conclude this section by stating an important property of the interaction between tensor products and coproducts of $\catV$-categories.
For $\catX, \catY \in \CatV$, define the $\catV$-functor 
\begin{equation}
\label{equ:cfunctordef}
c_{\catX, \catY} \co S(\catX) \tensorvcat S(\catY) \to S( \catX\sqcup \catY) \, , 
\end{equation}
by letting $c_{\catX, \catY}(\vec{x} \tensorvcat \vec{y}) \defeq \vec{x}\oplus \vec{y}$. The next proposition will be
used in the proofs of Proposition~\ref{lem:multiplicative2} and Theorem~\ref{thm:smondistiscc}.

\begin{proposition} \label{thefunctorc} For every $\catX, \catY \in \CatV$, 
the  $\catV$-functor $c_{\catX, \catY} \co S(\catX) \tensorvcat S(\catY) \to S( \catX\sqcup \catY)$
is a symmetric monoidal equivalence.
\end{proposition} 

\begin{proof} Let us write $c$ for $c_{\catX, \catY}$. 
It is easy to verify that  $c\co S(\catX) \tensorvcat S(\catY) \to 
S(\catX\sqcup \catY)$ is symmetric monoidal.
Let us show that it is an equivalence.
The 2-functor $S \co  \CatV \to \SMCatV$ preserves 
coproducts, since it is a left adjoint.
Thus, if $\iota_1\co  \catX\to  \catX\sqcup \catY$ and  $\iota_2\co  \catY\to \catX\sqcup \catY$
are the inclusions, then the functors  $S(\iota_1)$ and $S(\iota_2)$
exhibit $S(\catX \sqcup \catY)$ as the coproduct of~$S(\catX)$ and~$S(\catY)$.
But the functors  $\mathit{in}_1\co  S(\catX) \to  S(\catX)\otimes S(\catY)$ and  $\mathit{in}_2\co  S(\catY)\to S(\catX) \otimes S(\catY)$
defined by letting $\mathit{in}_1(\vec{x}) \defeq \vec{x}\otimes e$
and  $\mathit{in}_2(\vec{y}) \defeq e \otimes \vec{y}$, where we write $e$ for the empty sequence in both~$S(\catX)$ and~$S(\catY)$, 
exhibit~$S(\catX)\otimes S(\catY)$ as the coproduct of~$S(\catX)$ and~$S(\catY)$ by Lemma~\ref{thm:smoncathascoprod}
It follows that the $\catV$-functor  $c\co S(\catX) \tensorvcat S(\catY) \to 
S(\catX\sqcup \catY)$ is an equivalence,
since the diagram
\[
\xymatrix@C=1.5cm{
S(\catX) \ar[r]^-{\mathit{in}_1}  \ar@/_1pc/[dr]_{S(\iota_1)} & S(\catX) \tensorvcat S(\catY) \ar[d]_{c}  & S(\catY) \ar[l]_-{\mathit{in}_2}  \ar@/^1pc/[dl]^{S(\iota_2)} \\
 & S(\catX \sqcup \catY)  & }
 \]
 commutes.
\end{proof}

\section{$S$-distributors} 
\label{sec:sdist}

\begin{definition} Let $\catX, \catY$ be small $\catV$-categories. 
An \emph{$S$-distributor} $F \co  \catX \sym \catY$ is
a distributor $F\co  \catX\mat S(\catY)$, \ie a $\catV$-functor
$F \co  S(\catY)^\op  \tensorvcat \catX \to \catV$.
\end{definition}

We write $F[\vec{y};x]$ for the result of applying an $S$-distributor $F \co \catX \to \catY$
 to~$(\vec{y}, x) \in S(\catY)^\op \otimes \catX$. The next theorem introduces the bicategory of $S$-distributors. Once again, its
proof uses a Gabriel factorization.

\begin{theorem} Small $\catV$-categories, $S$-distributors and $\catV$-transformations
form a bicategory, called the bicategory of $S$-distributors and denoted by~$\SDistV$, which fits into a Gabriel factorization diagram
\[
\xymatrix@R=1.3cm{
\DistV \ar[rr]^{S} \ar[dr]_{\delta}&& \SMDistV \\
&\; \SDistV \, . \ar[ur]_{(-)^e} & }
\]
\end{theorem}

\begin{proof} Let $\catX$, $\catY$ be small $\catV$-categories. We define the hom-category of
$S$-distributors from~$\catX$ to~$\catY$ by letting $\SDistV[\catX,\catY]\defeq \DistV[\catX,S(\catY)]$.
We then define a functor 
\begin{equation}
\label{equ:efullandfaith}
(-)^e \co \SDistV[\catX,\catY] \to \SMDistV[S(\catX), S(\catY)]
\end{equation}
by mapping an $S$-distributor $F \co \catX \sym \catY$, given by a distributor $F \co \catX \mat S(\catY)$,
 to its symmetric monoidal extension $F^e \co S(\catX) \mat S(\catY)$, defined as 
in~\eqref{equ:fprime2}. We represent the
action of this functor by the derivation 

\[
\begin{prooftree}
F\co \catX \sym \catY
\justifies
 F^e\co S(\catX) \mat S( \catY).
\end{prooftree} \medskip
\]
Given two $S$-distributors $F \co  \catX \sym \catY$ and $G \co  \catY \sym \catZ$, we define their 
composite $G \circ F \co  \catX \sym \catZ$ so as to have an isomorphism $(G\circ F)^e \iso G^e\circ F^e$.
For $\vec{x}\in  S(\catX)$ and $\vec{z}\in S(\catZ)$, we have
\[
(G^e \circ F^e)[ \vec{z}; \vec{x}]   = 
\int^{\vec{y} \in S(\catY)} 
G^e[ \vec{z};  \vec{y}]  \, \otimes  F^e[ \vec{y};\vec{x}]  \, ,
\]
and therefore,  for  $x\in \catX$ and $\vec{z}\in  S(\catZ)$, we let
\begin{equation}
\label{equ:Sdistcomp} 
(G \circ F)[ \vec{z}; x]  \defeq  
\int^{\vec{y} \in S(\catY)} 
G^e[ \vec{z}; \vec{y}]  \, \otimes  F[ \vec{y};x ] \, .
\end{equation}
Here, the distributor $G^e \co S(\catY) \mat S(\catZ)$ is defined via the formula in~\eqref{equ:fprime2}.
The horizontal composition of  $S$-distributors is functorial, coherently associative  since the  horizontal composition 
of the bicategory $\SMDistV$ is so. The identity $S$-distributor $\Id_\catX\co \catX\sym \catX$ 
is determined in an analogous way as the composition operation; explicitly, we define it 
 as the distributor $(\iota_\catX)_\bullet \co  \catX\mat S(\catX)$, where $\iota_\catX\co  \catX\to S(\catX)$
is the inclusion $\catV$-functor. Thus,
\begin{equation}
\label{equ:idSdist}
\Id_\catX[ x_1, \ldots, x_n; x]  \defeq   S(\catX)\big[  (x_1, \ldots, x_n), (x)  \big] 
= \left\{
\begin{array}{ll}
\catX[ x_1, x]  & \text{if }  n = 1 \, , \\
0 & \text{otherwise.} 
\end{array} 
\right.
\end{equation} 
This completes the definition of the bicategory $\SDistV$. The homomorphism 
\[
(-)^e\co \SDistV \to \SMDistV 
\]
is then defined as follows. Its action on objects is defined by letting $\catX^e \defeq S(\catX)$,
while its action on hom-categories is  given by the functors in~\eqref{equ:efullandfaith}.  
We complete the definition of the required Gabriel factorization by defining the homomorphism $\delta\co \DistV \to 
\SDistV$. Its action on object is the identity. Given a distributor
 $F\co \catX\mat \catY$ we define the S-distributor $\delta(F) \co \catX \sym \catY$ by letting $\delta(F) \defeq \iota_\bullet \circ F$, where $\iota\co \catY \to S(\catY)$ is the inclusion
 $\catV$-functor. In this
way, we have
 \begin{equation}
 \label{equ:pscommsdist}
 S(F) \iso (\iota_\bullet \circ F)^e = \delta(F)^e  \,  \, .
 \end{equation}
Explicitly, we have 
\[
\delta(F)[ \vec{y},x]  \defeq
\left\{
\begin{array}{ll}
F[ y,x]   & \text{ if } \vec{y}=(y)\in S^1(\catY) \, , \\
0   & \text{ otherwise.}
\end{array}
\right.
\]
For distributors $F \co  \catX \mat \catY$ and $G \co  \catY \mat \catZ$, the pseudo-functoriality isomorphism 
\[
\delta(G)\circ \delta(F) \iso \delta(G\circ F)
\] 
is obtained combining the fact that the functor in~\eqref{equ:efullandfaith} is an equivalence with the existence of an isomorphism 
\[
 \delta(G\circ F)^e  \iso S(G\circ F)  \iso S(G) \circ S(F)  \iso \delta(G)^e \circ \delta(F)^e \, ,
\]
 where we used  the isomorphism in~\eqref{equ:pscommsdist} twice. 
In a similar way one obtains an isomorphism~$\delta(\id_\catX) \iso  \id_\catX$ for every small $\catV$-category~$\catX$.
 \end{proof}

\begin{remark} For every pair of small $\catV$-categories $\catX$, $\catY$, the following diagram commutes:
\[
\xymatrix@C=1.5cm{
\DistV[\catX, \catY] \ar[r]^{\lambda} \ar[d]_-{(-)^e} & \CatV[\catX, PS(\catY) \ar[d]^-{(-)^e} \\
\SMDistV[S(\catX), S(\catY)] \ar[r]_-{\lambda} & \SMCatV[S(\catX), PS(\catY)] \, .}
\]
Thus, for a distributor $F \co \catX \to \catY$, we write $\lambda F^e \co S(\catX) \to PS(\catY)$ for
the common value of the composite functors. Note that the functor $(\lambda F^e)_c\co \psh{S(\catX)}\to \psh{S(\catY)}$
is a homomorphism of symmetric $\catV$-rigs.
 Symbolically,
 \[
\begin{prooftree} 
\[
\[
F\co \catX \sym \catY \text{ in } \SDistV
 \justifies
F^e \co S(\catX) \mat S(\catY)  \text{ in } \SMDistV
\]
\justifies
\lambda F^e \co  S(\catX)\to  \psh{S(\catY)}  \text{ in } \SMCatV
\]
\justifies
(\lambda F^e)_c \co \psh{S(\catX)}\to \psh{S(\catY)   \text{ in } \SRigV  \, .}
\end{prooftree}
\]
\end{remark}

\begin{lemma} \label{lemmacoprodSdist}  
If $\catX_1$ and $\catX_2$ are small $\catV$-categories
and  $\iota_k\co \catX_k\to \catX_1\sqcup \catX_2$ is the $k$-th inclusion, 
then the symmetric monoidal distributors 
$S(\iota_k)_\bullet \co  S(\catX_k)\mat S(\catX_1\sqcup \catX_2)$
 for $k=1,2$ 
 exhibit $S(\catX_1 \sqcup \catX_2)$ as the 
 coproduct of $S(\catX_1)$ and $S(\catX_2)$ in the bicategory $\SMDistV$.
\end{lemma} 

\begin{proof} The distributors $(\iota_k)_\bullet\co \catX_k\to \catX_1\sqcup \catX_2$ (for $k=1,2$)
exhibit $\catX_1\sqcup \catX_2$ as the coproduct of~$\catX_1$ and~$\catX_2$ in the bicategory $\DistV$ by Proposition~\ref{bulletcomp2}.
Hence the symmetric monoidal distributors $S(\iota_k)_\bullet =S((\iota_k)_\bullet)\co S(\catX_k)\to S(\catX_1\sqcup \catX_2)$ 
exhibits the coproduct of $S(\catX_1)$ and $S(\catX_2)$ in the bicategory $\SMDistV$,
 since the homomorphism $S\co  \DistV \to \SMDistV$
is a left adjoint and hence preserves finite coproducts.
\end{proof}

\begin{proposition} \label{coprodSdist} 
The bicategory $\SDistV$ has finite coproducts and the homomorphisms 
\[
\delta \co  \DistV\to \SDistV \, ,   \quad (-)^e\co \SDistV \to \SMDistV 
\]
 preserve finite coproducts. 
\end{proposition}

\begin{proof} This follows from Lemma~\ref{lemmacoprodSdist}, since the homomorphism $S\co  \DistV \to \SMDistV$
preserves finite coproducts.
\end{proof}

We establish an explicit formula for  the coproduct homomorphism on $S$-distributors, which will be used in the proof of Theorem~\ref{thm:smondistiscc}.
By definition, the coproduct of two $S$-distributors 
$F_1\co  \catX_1 \sym \catY_1$ and $F_2\co  \catX_2\sym  \catY_2$
is a $S$-distributor $F_1\sqcup_S F_2\co  
 \catX_1 \sqcup  \catX_2 \sym  \catY_1 \sqcup  \catY_2$
 fitting in a commutative diagram of $S$-distributors,
\begin{equation}
\label{equ:sdistcommdiag}
\vcenter{\hbox{\xymatrix@C=2cm{
\catX_1 \ar[r]^{F_1} \ar[d]_{\delta(i^1_\bullet)}  & \catY_1\ar[d]^{\delta(j^1_\bullet)}  \\
  \catX_1\sqcup \catX_2 \ar[r]^{F_1\sqcup_S F_2} & \catY_1\sqcup \catY_2\\
\catX_2   \ar[u]^{\delta(i^2_\bullet)} \ar[r]_{F_2} &\catY_2 \ar[u]_{\delta(j^2_\bullet)} \ , }}}
\end{equation}
where $i^k\co \catX_k\to  \catX_1\sqcup   \catX_2$
and $j^k\co \catY_k\to  \catY_1\sqcup   \catY_2$ are the inclusions.

\begin{proposition} \label{lem:multiplicative2} Given $S$-distributors $F_1\co  \catX_1\sym \catY_1$ and $F_2\co  \catX_2\sym  \catY_2$,
then 
\[
 (F_1\sqcup_S F_2)^e [ \vec{y}_1\oplus \vec{y}_2 ,\vec{x}_1\oplus \vec{x}_2]  \iso
 F_1^e[ \vec{y}_1,\vec{x}_1] \otimes F_2^e[\vec{y}_2,\vec{x}_2] 
 \]
 for $\vec{x}_1\in S(\catX_1)$, $\vec{x}_2\in S(\catX_2)$,
 $\vec{y}_1\in S(\catY_1)$ and $\vec{y}_2\in S(\catY_2)$.
\end{proposition}

\begin{proof} 
The image of the diagram in~\eqref{equ:sdistcommdiag} by the homomorphism  $(-)^e\co \SDistV \to \SMDistV $
is a commutative diagram of symmetric monoidal distributors,
\[
\xymatrix@C=2.5cm{
S(\catX_1) \ar[r]^{F_1^e} \ar[d]_{S(i_1)_\bullet}  & S( \catY_1)\ar[d]^{S(j_1)_\bullet}  \\
S( \catX_1\sqcup \catX_2) \ar[r]^{(F_1\oplus F_2)^e} & S( \catY_1\sqcup \catY_2)\\
S(\catX_2)   \ar[u]^{S(i_2)_\bullet} \ar[r]_{F_2^e}  & S(\catY_2)  \ar[u]_{S(j_2)_\bullet} \, ,}
\]
from which we obtain the following commutative square
of symmetric monoidal distributors,
\[
\xymatrix@C=2,5cm{
S(\catX_1) \tensorvcat S(\catX_2)  \ar[r]^{F_1^e \tensorvcat F_2^e} \ar[d]_{c_\bullet}  &
S( \catY_1)\ar[d]^{c_\bullet} \tensorvcat S(\catY_2)   \\
S( \catX_1\sqcup \catX_2) \ar[r]_{(F_1\oplus F_2)^e}& S(\catY_1\sqcup \catY_2) \, ,}
\]
where $c\co S(\catX_1) \tensorvcat S(\catX_2) \to 
S( \catX_1\sqcup \catX_2)$ is the $\catV$-functor of~\eqref{equ:cfunctordef}, which in this case is defined by 
letting $c(\vec{x}_1 \tensorvcat \vec{x}_2) \defeq \vec{x}_1\oplus \vec{x}_2$.
We have $c^\bullet \circ c_\bullet \iso \Id_{S(\catY_1) \otimes S(Y_2)}$, since $c$ is an equivalence by Proposition~\ref{thefunctorc}.
Hence the following diagram of distributors commutes 
\[
\xymatrix@C=2.5cm{
S(\catX_1) \tensorvcat S(\catX_2)  \ar[r]^{F_1^e  \tensorvcat F_2^e}  \ar[d]_{c_\bullet}  &
S( \catY_1) \tensorvcat S(\catY_2)   \\
S( \catX_1\sqcup \catX_2) \ar[r]_{(F_1\oplus F_2)^e} &S( \catY_1\sqcup \catY_2)
 \ar[u]_{c^\bullet}\, . }
\]
This proves the result, since 
\[
(F_1^e \tensorvcat F_2^e)[ \vec{y}_1  \tensorvcat \vec{y}_2,  \vec{x}_1  \tensorvcat \vec{x}_2] 
\defeq  
F_1^e[ \vec{y}_1,\vec{x}_1]  \otimes F_2^e[ \vec{y}_2,\vec{x}_2] 
\]
and, by part (ii) of Lemma~\ref{bulletlambda}, we have
\[
(c^\bullet \circ (F_1\oplus F_2)^e \circ c_\bullet)[ \vec{y}_1  \tensorvcat \vec{y}_2,  \vec{x}_1  \tensorvcat \vec{x}_2]
 \iso 
 (F_1\oplus F_2)^e[ \vec{y}_1\oplus \vec{y}_2 ,\vec{x}_1\oplus \vec{x}_2]   \, . \qedhere
\]
\end{proof}

\begin{remark}[The bicategory of $S$-matrices] \label{thm:Smatrices}
 We give an explicit description of the full the sub-bicategory of~$\SDistV$ spanned by sets, in analogy with what we did in Remark~\ref{thm:matrices}, where we showed how the bicategory of matrices~$\MatV$ can be seen as the full sub-bicategory of the bicategory of distributors spanned by sets. If~$\mathbb{M}$ is a symmetric monoidal category, then~$\mathbb{M} \cdot \UnitCat$
has the structure of a symmetric monoidal $\catV$-category and the functor mapping $\mathbb{M}$
to  $\mathbb{M} \cdot \UnitCat$ is left adjoint to the functor~$\mathrm{Und} \co \SMCatV \to \SMCat$
mapping  a symmetric monoidal $\catV$-category to its underlying symmetric monoidal category.
We also have a natural isomorphism,
\[
 S (\catC) \cdot \UnitCat  \iso S(  \mathbb{C} \cdot \UnitCat ) 
\]
since the diagram 
\[
\xymatrix{
\SMCatV \ar[r]^{\mathrm{Und}} \ar[d]_U &\SMCat  \ar[d] ^U\\
\CatV \ar[r]_{\mathrm{Und}} & \, \Cat }
\]
commutes and a composite of left adjoints is left adjoint to the composite. As a special case of the definitions in Section~\ref{sec:fresmc}, the free symmetric monoidal
category $S(X)$ on a set $X$ admits the following direct description. For $n \in \Nat$, let $S^n(X)$ be the category whose objects
are sequences $\vec{x} = (x_1, \ldots, x_n)$
of elements of $X$ and whose morphisms $\sigma \co (x_1, \ldots, x_n) \to (x_1', \ldots, x_n')$ are permutations $\sigma \in \Sigma_n$ such that~$x_i' = x_{\sigma(i)}$ for $1 \leq i \leq n$. We then let $S(X)$ be the coproduct of the categories $S_n(X)$, for $
n \in \Nat$,
\[
S(X) \defeq  \bigsqcup_{n \in \Nat} S^n(X) \, . 
\]
For sets $X$ and $Y$, we define an \emph{$S$-matrix} $F \co X \sym Y$ to be a functor $F \co S(Y)^\op \times X \to \catV$. 
Sets, $S$-matrices and natural transformations form a bicategory $\SMatV$ which can be identified with the
full sub-bicategory of the bicategory $\SDistV$ of $S$-distributors spanned by discrete $\catV$-categories. 
Indeed, for sets $X$ and $Y$, we have the following chain of isomorphisms:
\begin{align*} 
\SMatV[X, Y] & = \Cat[ S(Y)^\op \times X \, , \catV ] \\
  & \iso \CatV \big[ \big( S(Y)^\op \times X \big) \cdot \UnitCat, \catV ] \\
  & \iso \CatV \big[  \big( S(Y)^\op \cdot \UnitCat \big) \otimes \big(   X \cdot \UnitCat \big), \catV \big] \\
  & \iso \CatV \big[  \big( S(Y)\cdot \UnitCat \big)^\op \otimes \big(   X \cdot \UnitCat \big), \catV \big] \\
  & \iso \SDistV[ X \cdot \UnitCat \, ,  Y \cdot \UnitCat ] \, .
  \end{align*}
The composition operation and the identity morphisms of~$\SMatV$ are determined by those of~$\SDistV$ analogously to 
the way in which composition operation and the identity morphisms of~$\MatV$ are determined by those of~$\DistV$. We
do not unfold these definitions, since in Section~\ref{sec:regbss} we will describe explicitly its opposite bicategory. Note that we obtain 
the following diagram of inclusions:
\[
\xymatrix{
\MatV \ar[r] \ar[d] & \SMatV \ar[d] \\
\DistV \ar[r] & \, \SDistV \, . }
\]
\end{remark}

We conclude this section by defining the operation of compositon of an $S$-distributor with a functor, in analogy with 
the definition of composition of a distributor with a functor, given in~\eqref{equ:compdistfun}.  Let $\rigR$ be a symmetric $\catV$-rig.
Then the \emph{composite} of an $S$-distributor  $F\co \catX\sym \catY$ with a $\catV$-functor $T\co  \catY\to \rigR$
is the $\catV$-functor $T\circ F\co \catX\to \rigR$ defined by letting
\begin{equation}
\label{equ:extrastartf}
(T\circ F)(x)  \defeq \int^{\vec{y}\in S(\catY)}T^e(\vec{y})\otimes F[\vec{y}; x] \, ,
\end{equation}
for $x\in \catX$.

\begin{lemma}\label{actiononrighomomo} Let $\catX, \catY$ be small $\catV$-categories and $\rigR$ be a symmetric $\catV$-rig.
For every $S$-distributor  $F\co \catX\sym \catY$ and $\catV$-functor $T\co  \catY\to \rigR$, we have 
\[
(T\circ F)^e \iso T^e\circ F^e \, ,
\]
where 
\[
(T^e\circ F^e)(\vec{x}) = \int^{\vec{y} \in S(\catY)}  T^e(\vec{y}) \otimes F^e[\vec{y};\vec{x}]  \, .
\]
\end{lemma}

\begin{proof} 
The functor  $T^e\circ F^e\co S(\catX)\to \rigR$ is symmetric monoidal by Proposition~\ref{actionsymmetricmono},
since the distributor $F^e\co S(\catX)\mat S(\catY)$ is symmetric monoidal
and the $\catV$-functor $T^e\co  S(\catY)\to \rigR$  is symmetric monoidal.
Moreover, if $\vec{x}=(x)$ with $x\in \catX$,
then
\[
(T^e\circ F^e)((x)) = \int^{\vec{y}\in S(\catY)}  T^e(\vec{y}) \otimes F^e[ \vec{y};(x)] =
 \int^{\vec{y}\in S(\catY)}  T^e(\vec{y}) \otimes F[\vec{y}; x] =( T\circ F)(x) \, .
 \]
Hence the symmetric monoidal functor $T^e\circ F^e\co S(\catX)\to \rigR$ is an
extension of~$T\circ F\co \catX\to \rigR$.
This shows that $(T\circ F)^e \iso T^e\circ F^e$ by the uniqueness up to unique isomorphism of the extension.
\end{proof}

\begin{proposition}\label{assoactionSdist} Let $\catX, \catY, \catZ$ be small $\catV$-categories and $\rigR$ 
be a symmetric $\catV$-rig.
\begin{enumerate}[(i)]
\item  For all distributors $F\co \catX \sym \catY$, $G\co \catY\sym \catZ$ and $\catV$-functors $T\co  \catZ\to   \rigR$, 
$(T\circ G)\circ F \iso T\circ (G\circ F)$.
\item For every $\catV$-functor $T\co  \catX \to  \rigR$,  $T\circ \id_\catX \iso T$.
\end{enumerate} 
\end{proposition}

\begin{proof} For part~(i), it suffices to show that $((T\circ G)\circ F)^e \iso (T\circ (G\circ F))^e$. But,
by Proposition~\ref{contraactionprop} and Lemma~\ref{actiononrighomomo}, we have 
 \begin{align*} 
 ((T\circ G)\circ F)^e & \iso (T\circ G)^e\circ F^e \\
  & \iso  (T^e\circ G^e)\circ F^e \\
  & \iso T^e\circ (G^e\circ F^e) \\
  & \iso T^e \circ (G\circ F)^e \, .
 \end{align*} 
 The proof of part (ii) is analogous.
 \end{proof}

\section{Symmetric sequences and analytic functors}
\label{sec:symsaf}

\begin{definition} Let $\catX, \catY$ be small  $\catV$-categories. A \emph{categorical symmetric sequence} $F \co \catX \to \catY$ is
an $S$-distributor from $\catY$ to $\catX$, \ie a $\catV$-functor $F \co S(\catX)^\op \otimes \catY \to \catV$.
\end{definition}

If $F \co \catX \to \catY$ is a categorical symmetric sequence, then for a symmetric $\catV$-rig $\rigR = (\rigR, \diamond, e)$ we  define 
its associated analytic functor $F \co \rigR^\catX\to \rigR^\catY$ by letting, for $T \in \rigR^\catX$ and~$y \in \catY$, 
\[
F(T)(y) \defeq \int^{\vec{x}\in S(\catX)} F[\vec{x};y] \otimes  T^{\vec{x}} \, , 
\]
where, for $\vec{x} = (x_1, \ldots, x_n) \in S(\catX)$, $T^{\vec{x}} \defeq  T^e(\vec{x}) = T(x_1) \diamond \ldots \diamond T(x_n)$. We 
represent the correspondence between categorical symmetric sequences and analytic functors as follows:
\[
\begin{prooftree}
F\co  \catX \sym \catY
\justifies
F \co \rigR^\catX\to \rigR^\catY \, .
\end{prooftree} 
\]

\begin{example} For $\catV = \rigR  = \Set$ and $\catX = \catY = \onecat$, where $\onecat$ is the terminal category, we obtain exactly the notion of an analytic functor introduced 
in~\cite{JoyalA:fonaes}. Indeed, a symmetric sequence~$F \co \onecat \sym \onecat$ is the same thing as a functor $F \co \Scat \to \Set$. Here, $\Scat = S(\onecat)$ is the category of natural numbers and
permutations. The analytic functor  $F \co \Set \to \Set$ associated to such a symmetric sequence has the following form:
\[
F(T) \defeq \int^{n \in \Scat} F[n] \otimes T^{n} \, .
\]
See~\cite{AguiarM:monfsh,BergeronF:comstls,JoyalA:thecsf} for applications of the theory of
analytic functors to combinatorics and~\cite{AdamekJ:anafwp} for recent work on categorical
aspects of the theory.
\end{example}

\begin{example} For $\catV = \rigR  = \Set$,
we obtain the notion of  analytic functor between categories of covariant presheaves\footnote{The analytic functors studied in~\cite{FioreM:carcbg} 
were between categories of presheaves, but we prefer to consider covariant preshaves to match our earlier definitions.} considered in~\cite{FioreM:carcbg}. In that
context, the analytic functor $F \co \Set^\catX \to \Set^\catY$ of a symmatric sequence
$F \co \catX \sym \catY$, \ie a functor $F \co S(\catX)^\op \otimes \catY \to \Set$, is obtained as a left Kan extension, fitting in the diagram
\[
\xymatrix{
S(\catX)^\op \ar@/_1pc/[dr]_{\lambda F} \ar[r]^{\sigma_\catX} & \Set^\catX \ar[d]^{F} \\
 & \Set^\catY \, ,}
 \]
where $\sigma_\catX \co S(\catX)^\op  \to \Set^\catX$ is the functor defined by letting
\[ 
\sigma_\catX(x_1, \ldots, x_n) \defeq  \bigsqcup_{1 \leq k \leq n } \catX[x_k, -] \, . 
\]
Hence, for $T \in \Set^\catX$ and $y \in \catY$, we have
\begin{align*}
F(T)(y) & = \int^{\vec{x}  \in S(\catX)} \lambda F(\vec{x})(y) \otimes \big[ \sigma_\catX(\vec{x}), T \big]  \\
 & = \bigsqcup_{n \in \mathbb{N}} \; 
 \int^{\vec{x} \in S^n(\catX)} F[\vec{x}; y] \otimes  \Big[   \bigsqcup_{1 \leq k \leq n} \catX[x_k, -] , T \Big]  \\
 & = \bigsqcup_{n \in \mathbb{N}} \; 
 \int^{\vec{x} \in S^n(\catX)} F[\vec{x}; y] \otimes  \bigsqcap_{1 \leq k \leq n} \pshX \big[    \catX[x_k, -] , T \big]  \\
 & = \bigsqcup_{n \in \mathbb{N}} \; 
 \int^{\vec{x} \in S^n(\catX)} F[\vec{x}; y] \otimes  \bigsqcap_{1 \leq k \leq n} T(x_k)   \\
 & = \int^{\vec{x} \in S(\catX)} F[\vec{x}; y] \otimes  T^{\vec{x}} \, , 
 \end{align*}
 where for $\vec{x} = (x_1, \ldots, x_n)$, we have $T^{\vec{x}} = T(x_1) \times \ldots \times T(x_n)$. Note how this construction
of analytic functors as a left Kan extension does not carry over to the enriched setting.
\end{example}

\medskip
 
Small $\catV$-categories, categorical symmetric sequences and~$\catV$-natural transformations form a bicategory, 
called the bicategory of categorical symmetric sequences and denoted by~$\CatSymV$. This bicategory is defined as the opposite of the bicategory~$\SDistV$ of
$S$-distributors:
\[
\CatSymV \defeq (\SDistV)^\op \, .
\]
In particular, for small $\catV$-categories $\catX$ and $\catY$, we have
\[
\CatSymV[\catX ,\catY ] =  \SDistV[\catY,\catX] = \DistV[\catY, S(\catX)] = \CATV[S(\catX)^\op \otimes \catY, \catV] \, .
\]

\begin{remark}[The bicategory of symmetric sequences]
We write $\SymV$ for the full sub-bicategory of $\CatSymV$ spanned by sets, viewed as discrete $\catV$-categories. This
bicategory can be defined simply as
\[
\SymV \defeq (\SMatV)^\op \, ,
\]
where $\SMatV$ is  full sub-bicategory of $\SDistV$ spanned by sets, as defined 
in Remark~\ref{thm:Smatrices}. We unfold some definitions since they will be useful in Section~\ref{sec:regbss}. 
The objects of $\SymV$ are sets and the hom-category between two sets~$X$ and~$Y$ is defined by
\[
\SymV[X,Y] \defeq \SMatV[Y,X] = \Cat[S(X)^\op \times Y, \catV] \, .
\]
For sets $X$ and $Y$, we define a \emph{symmetric sequence} $F \co X \to Y$  to be a functor $F \co S(X)^\op \times Y \to~\catV$.
For a symmetric $\catV$-rig $\rigR$, the extension $F \co \rigR^X \to \rigR^Y$ of such a symmetric sequence is given by
\begin{equation}
\label{equ:anasym}
F(T)(y) \defeq  \bigsqcup_{n \in \Nat} \int^{(x_1, \ldots, x_n) \in S^n(X)} 
F[x_1, \ldots, x_n; y] \otimes T(x_1) \otimes \ldots \otimes T(x_n) \, .
\end{equation}
It will be useful to have also an explicit description of the composition operation in $\SymV$. For sets $X, Y, Z$ and symmetric sequences
$F \co X \to Y$, $G \co Y \to Z$, their composite $G \circ F \co X \to Z$ is given by 
\begin{multline}
\label{equ:symcomp}
(G  \circ F)(\vec{x}; z) \defeq   \bigsqcup_{m \in \Nat} \int^{(y_1, \ldots, y_m) \in S^m(Y)} G[y_1, \ldots, y_m; z] \, \otimes  \\
\coend^{\vec{x}_1 \in S(X)} \cdots \coend^{\vec{x}_m \in S(X)}  
[\vec{x}, \vec{x}_1 \oplus \ldots \oplus \vec{x}_m] \otimes   F[\vec{x}_1; y_1] \otimes \cdots \otimes F[\vec{x}_m; y_m] \, .
\end{multline}
For a set $X$, the identity symmetric sequence $\Id_X \co X \to X$ is defined by letting
\begin{equation}
\label{equ:symid}
\Id_X[ \vec{x}; x] = 
\left\{
\begin{array}{ll}
I &  \text{ if } \vec{x} = (x)  \, , \\
0 & \text{ otherwise. }
\end{array}
\right.
\end{equation}
\end{remark}

\begin{example} \label{thm:keyyremark}
Let $\Scat = S(1)$, the category of natural numbers and permutations. The monoidal structure on $\SymV[1,1] \iso [\Scat^\op, \catV]$ given by
the horizontal composition in~$\SymV$ as defined in~\eqref{equ:symcomp} is exactly the substitution monoidal structure discussed in the 
introduction, which characterizes the notion of a single-sorted operad, in the sense that
monoids in $[\Scat^\op, \catV]$ with respect to this monoidal structure are exactly single-sorted operads~\cite{KellyGM:opejpm}. Indeed, as a special case of the formula 
in~\eqref{equ:symcomp}, we get
\[
(G \circ F)[n]  = \int^{m \in \Scat} G[m] \, \otimes  
\coend^{n_1 \in \Scat} \cdots \coend^{n_m \in \Scat}  
[n, n_1 + \ldots n_m] \otimes   F[ n_1]  \otimes \cdots \otimes F[n_m]  \, .
\]
Similarly, the unit $J$ for the substitution tensor product is given by a special case of the formula in~\eqref{equ:symid}:
\[
J[n]  = 
\left\{
\begin{array}{ll}
I &  \text{ if } n = 1  \, , \\
0 & \text{ otherwise. }
\end{array}
\right.
\]
We will see in Section~\ref{sec:monmb} that the horizontal composition of $\SymV$ can be used to characterize the notion of an operad.
\end{example}

We conclude this section by relating the composition of categorical symmetric sequences with the composition of analytic functors. In particular, it shows how the analytic functor associated to the composite of two categorical symmetric sequences is isomorphic
to the composites of the analytic functors associated to the categorical symmetric sequences. We also show
that the analytic functor associated to the identity categorical symmetric sequence is naturally isomorphic to the identity functor. This generalizes
Theorem~3.2 in~\cite{FioreM:carcbg} to the enriched setting.

\begin{theorem} \label{thm:anafuncomp} Let $\catX, \catY, \catZ$ be small $\catV$-categories and $\rigR$ 
be a symmetric $\catV$-rig.
\begin{enumerate}[(i)]
\item For every pair of categorical symmetric sequences $F \co  \catX \sym \catY$ and $G \co  \catY \sym \catZ$, there is a natural isomorphism with components
\[
(G\circ F)(T)  \iso G (F (T)) \, .
\]
\item There is a natural isomorphism with components
\[
 \Id_\catX(T) \iso T \, .
\]
\end{enumerate}
\end{theorem}

\begin{proof}  For this proof, it is  convenient to have some auxiliary notation: for a categorical symmetric sequence $F \co \catX \to \catY$,
we write $F^\op \co \catY \sym \catX$ for the corresponding $S$-distributor.
For part (i), let $T\co \catX \to \rigR$. We begin by showing
that~$F(T) = T \circ F^\op$, where $T \circ F^\op$ denotes the composite of the $S$-distributor $F^\op \co \catY \sym \catX$ with
the $\catV$-functor $T \co \catX \to \rigR$, defined in~\eqref{equ:extrastartf}. For~$y \in \catY$, we have
\begin{align*} 
F (T)(y) & = \int^{\vec{x}\in S(\catX)}  F[\vec{x}; y] \otimes T^{\vec{x}} \\
&  =  \int^{\vec{x}\in S(\catX)}  F[\vec{x}; y] \otimes  T^e(\vec{x}) \\
& = (T \circ F^\op)(y)  \, . \phantom{\int^{\vec{z}}}
\end{align*} 
It then follows by part~(i) of Proposition~\ref{assoactionSdist} that we have
\begin{align*} 
(G\circ F)(T) & =   T \circ (G \circ F)^\op \\
 & =  T  \circ (F^\op \circ G^\op) \\ 
 & \iso (T \circ F^\op) \circ G^\op  \\
  & =  F(T) \circ G^\op \\
  & = G (F (T))\, .
  \end{align*}
  For part (ii), if $T\co \catX\to \rigR$, then by part (ii) of Proposition~\ref{assoactionSdist} we have
  \begin{align*} 
  \id_{\catX} (T) & =  T \circ \Id_\catX^\op \\
  & \iso T  \, . \qedhere
  \end{align*} 
\end{proof}

\section{Cartesian closure of categorical symmetric sequences}
\label{sec:carcsc}

We conclude this chapter with our first main result, asserting that that the bicategory of categorical symmetric sequences is cartesian closed. This  generalizes to the enriched setting the main result of~\cite{FioreM:carcbg}. The existence of products
in~$\CatSymV$ follows easily by our earlier results, but we state the result explicitly for emphasis.

\begin{proposition} \label{thm:anafprod}
The bicategory $ \CatSymV$ has finite products. In particular, the product of 
two small $\catV$-categories $\catX$ and $\catY$ in $\CatSymV$ is given by 
their coproduct $\catX \sqcup \catY$ in $\CatV$.
\end{proposition}

\begin{proof} This follows by Proposition~\ref{coprodSdist} by duality, since the bicategory $\SDistV$ has finite coproducts 
and these are given by coproducts in $\CatV$.
\end{proof}

As done in~\cite{FioreM:carcbg} in the case $\catV = \Set$, the fact that coproducts in  $\CatSymV$ are given by coproducts in $\CatV$ can be seen intuitively by the following chain of equivalences:
\begin{align*}
\CatSymV[ \catZ, \catX] \times \CatSymV[\catZ, \catY]  & =  
\SDistV[\catX, \catZ] \times \SDistV[\catY, \catZ] \\
& = \DistV[ \catX, S(\catZ)] \times \DistV[\catY, S(\catZ)] \\
& \simeq  \DistV[ \catX \sqcup \catY, S(\catZ)]  \\
 & = \SDistV[ \catX \sqcup \catY, \catZ] \\
 & = \CatSymV[ \catZ, \catX \sqcup \catY] \, .
\end{align*}
We now consider the definition of exponentials in~$\CatSymV$.
For small $\catV$-categories~$\catX$ and~$\catY$, 
let us define
\[
[\catX, \catY] \defeq  S(\catX)^{\op} \otimes \catY \, .
\]
Then for every a small $\catV$-category $\catZ$,
 we have
 \begin{eqnarray*}
 \CatSymV[\catZ, [\catX, \catY] ] & = &\SDistV[S(\catX)^{\op} \otimes \catY, \catZ] \\
 & = &\DistV[S(\catX)^{\op} \otimes \catY, S(\catZ)] \\
  & = &\CatV[S(\catZ)^\op\otimes S(\catX)^{\op} \otimes \catY, \catV] \\
    & = &\DistV[\catY, S(\catZ)\otimes S(\catX) ] \\
    & \simeq & \DistV[ \catY, S(\catZ \sqcup \catX) ] \\
    & = & \SDistV[ \catY, \catZ \sqcup \catX] \\
    & = &  \CatSymV[ \catZ \sqcap \catX, \catY]    \, .
     \end{eqnarray*}

Let us consider the effect of this chain of equivalences on a categorical symmetric sequence $F \co  \catZ \to [\catX, \catY]$, which is
is a distributor $F\co  \catY \mat S(\catZ)\otimes S(\catX) $.
Let $c = c_{\catZ, \catX} \co S(\catZ) \tensorvcat S(\catX) \mat S( \catZ\sqcup \catX)$ be the functor in~\eqref{equ:cfunctordef}, which
in this case is given by $c(\vec{z} \tensorvcat \vec{x}) \defeq \vec{z}\oplus \vec{x}$.
The distributor $c_\bullet\co S(\catZ) \tensorvcat S(\catX) \to S( \catZ\sqcup \catX)$
is an equivalence, since the functor $c$ is an equivalence, as stated in Proposition~\ref{thefunctorc}. We then have
 \[
\begin{prooftree} 
\[ 
\[ 
\[ 
F  \co  \catZ   \to  [\catX, \catY]  \text{ in } \CatSymV
\justifies
F\co  \catY \mat S(\catZ)\otimes S(\catX)  \text{ in } \DistV
\] 
 \justifies
 c_\bullet \circ F\co    \objY  \mat  S(\catZ\sqcup \catX)   \text{ in } \DistV
\] 
  \justifies
  c_\bullet \circ F \co  \objY   \sym \catZ\sqcup \catX  \text{ in } \SDistV
\] 
  \justifies
  c_\bullet \circ F  \co  \catZ \sqcap \catX  \to \catY   \text{ in }  \CatSymV \, .
  \bigskip
\end{prooftree} 
\]

\smallskip

\noindent
By considering the particular case of $\catZ =[\catX,\catY]$ we define the categorical symmetric sequence
\[
\ev  \co   [\catX ,\catY ]\sqcap  \catX \to \catY  \, ,
\]
by letting $\ev \defeq c_\bullet \circ \Id$,  where $\Id  \co [\catX,\catY]  \to [\catX,\catY] $ is the identity categorical symmetric sequence on $[\catX,\catY]$. By definition, we have
\[
\ev =  (c_\bullet \circ \Id) \co   \catY \sym \big( S(\catX)^\op \otimes \catY \big) \sqcup  \catX  \, 
\]
where $\Id \co S(\catX)^{\op} \otimes \catY \sym S(\catX)^{\op} \otimes \catY$ is the identity $S$-distributor, as in~\eqref{equ:idSdist}, which in
this case is given by the distributor $\Id\co S(\catX)^{\op} \otimes \catY \mat S(S(\catX)^{\op} \otimes \catY)$ defined by
\[
\Id(\vec{v},\vec{x}^{\op}\otimes  y)=\big[  \vec{v}, \vec{x}^{\op} \otimes y \big] \defeq   
S(S(\catX)^{\op} \otimes \catY)[ \vec{v}, \vec{x}^{\op} \otimes y] \, ,
\]
for $\vec{v}\in S(S(\catX)^{\op} \otimes \catY)$, $\vec{x}\in S(\catX)$ and $y\in \catY$. Hence, 
\begin{equation}
\label{equ:defevalu}
\ev[\vec{w}; y]  =  \int^{\vec{v}\in   S(S(\catX)^{\op} \otimes \catY) }
 \int^{\vec{x}\in  S(\catX)} 
\big[\vec{w}, \vec{v}\oplus \vec{x}]\otimes \big[ \vec{v}, \vec{x}^{\op} \otimes y\big]   =   \int^{\vec{x}\in  S(\catX)}  
  [\vec{w}, (\vec{x}^{\op} \otimes y) \oplus \vec{x}\big]  \, , 
\end{equation}
for $\vec{w}\in S(S(\catX)^\op\otimes \catY \sqcup  \catX)$
and $y\in \catY$. 

\medskip

We can now state and prove our first main result, which generalizes to the enriched case the main result in~\cite{FioreM:carcbg}.

\begin{theorem} \label{thm:smondistiscc}
 The bicategory $ \CatSymV$ is cartesian closed. 
More precisely, the analytic functor $\ev \co [\catX,\catY] \sqcap \catX \to  
\catY$ exhibits  the exponential of $\catY$ by $\catX$.
\end{theorem}

\begin{proof}  
We have to show that the functor 
\[
\varepsilon\co   \CatSymV[\catZ, [\catX,\catY] ] \to   \CatSymV[\catZ \sqcap \catX, \catY]
\]
defined by letting $\varepsilon(F) \defeq \ev \circ (F \sqcap \catX)$
is an equivalence of categories for every small $\catV$-category~$\catZ$.
By duality, this amounts to showing that the functor 
\[
\varepsilon\co   \SDistV[S(\catX)^{\op}\otimes \catY, \catZ] \to   \SDistV[\cat Y, \catZ\sqcup \catX]
\]
defined by letting $\varepsilon(F) \defeq (F\sqcup \catX)\circ \ev$
is an equivalence of categories for every small $\catV$-category~$\catZ$.
Notice that $\varepsilon$ has the form
\[
\varepsilon\co    \DistV[\catY , S(\catZ)\otimes S(\catX)]  \to   \DistV[\catY, S(\catZ\sqcup \catX)]
\]
since 
\[
 \SDistV[S(\catX)^{\op}\otimes \catY, \catZ] =\CatV[S(\catZ)^\op \otimes S(\catX)^{\op}\otimes \catY, \catV]= \DistV[\catY , S(\catZ)\otimes S(\catX)]  \, .
 \]
If $c\co S(\catZ) \tensorvcat S(\catX) \to S( \catZ\sqcup \catX)$
is the functor in~\eqref{equ:cfunctordef}, which in this case is defined by letting $c(\vec{z}\otimes \vec{x}) \defeq \vec{z}\oplus \vec{x}$,
let us show that we have $\varepsilon(F)=c_\bullet \circ F$
\[
\xymatrix{
\catY \ar[rr]^F \ar@/_1pc/[rrd]_{\varepsilon(F)} && S(\catZ)\otimes S(\catX)\ar[d]^{c_\bullet}\\
&& S(\catZ \sqcup \catX) \, . }
\]
Observe that we have $ c_\bullet \circ c^\bullet \iso \id_{S(\catZ) \otimes S(\catX)}$, since the functor $c$
is an equivalence of categories by Proposition~\ref{thefunctorc}.
Hence it suffices to show that we have  $c^\bullet  \circ \varepsilon(F)=F$, as in the following diagram of distibutors:
\[
\xymatrix{
\catY \ar[rr]^F \ar@/_1pc/[rrd]_{\varepsilon(F)} && S(\catZ)\otimes S(\catX) \\
&& S(\catZ \sqcup \catX) \ar[u]_{c^\bullet} \, . }
\]
In other words, it suffice to show that for $\vec{z}\in  S(\catZ)$, $ \vec{x}\in  S(\catX)$ and $y\in \catY$ we have
\[
\varepsilon(F)[ \vec{z}\oplus \vec{x};y] =F[\vec{z} \otimes \vec{x};y]  \, .
\]
By definition, $\varepsilon(F)$ is a composite
of $S$-distributors:
\[
\xymatrix{
\catY \ar[rr]^-{\ev} \ar@/_1pc/[rrd]_-{\varepsilon(F)} && \big( S(\catX)^\op\otimes \catY \big) \sqcup \catX \ar[d]^{F\sqcup \catX} \\
&& \catZ \sqcup \catX \, .  }
\]
 Thus,
 \[
 \varepsilon(F)[ \vec{z}\oplus \vec{x};y] =
\int^{\vec{w} \in S(S(\catX)^{\op}\otimes \catY\sqcup \catX) } 
(F\sqcup \catX)^e[ \vec{z}\oplus \vec{x}, \vec{w}]  \otimes \ev(\vec{w};y) \, .
\]
 It then follows from~\eqref{equ:defevalu} that
  \begin{eqnarray*}
\varepsilon(F)(\vec{z}\oplus \vec{x};y) &=& \int^{\vec{w}\in S (\catZ^\catY \sqcup \catY) }
  \int^{\vec{v}\in  S(\catY)}  
(F\sqcup \catY)^e[ \vec{z}\oplus \vec{x}, \vec{w}] \otimes   [\vec{w}, (\vec{v}^{\op} \otimes y) \oplus \vec{v}\big]\\
&=& 
 \int^{\vec{v}\in  S(\catY)}  
 (F\sqcup \catY)^e[ \vec{z}\oplus \vec{x},  (\vec{v}^{\op} \otimes y) \oplus \vec{v}]  \, .
  \end{eqnarray*}
  But we have 
   $(F\sqcup \catY)^e[ \vec{z}\oplus \vec{x},  (\vec{v}^{\op} \otimes y) \oplus \vec{v}] = F^e[ \vec{z}, \vec{v}^{\op} \otimes y]   \otimes  [\vec{x},\vec{v}]  $
  by Proposition~\ref{lem:multiplicative2}. Thus, 
  \begin{eqnarray*}
 \varepsilon(F)(\vec{z}\oplus \vec{x};y) &=& 
 \int^{\vec{v}\in  S(\catY)}   F^e[ \vec{z}, \vec{v}^{\op} \otimes y]  \otimes  [\vec{x},\vec{v}] \\
 &=&  F^e[ \vec{z}, \vec{x}^{\op} \otimes y]   \\
  &=& F^e[ \vec{z} \otimes \vec{x}, y]  \\
  &=& 
F(\vec{z} \otimes \vec{x}, y) \, .
  \end{eqnarray*}
This proves that $\varepsilon(F)=c_\bullet \circ F$. 
Hence the functor mapping $F$ to $\varepsilon(F)$
is an equivalence, as required, since the distributor $c_\bullet$ is an equivalence.
\end{proof}

\chapter{The bicategory of operad bimodules}
\label{cha:bicob}

The aim of this chapter is to define the bicategory of operad bimodules, which we denote by~$\OpdV$. The first step to do this 
is to identify operads, operad bimodules and operad bimodules maps as 
monads, monad bimodules and bimodule maps in the bicategory of symmetric sequences~$\SymV$
introduced in Chapter~\ref{cha:syms}. 
The second step is to apply to $\SymV$ the so-called bimodule construction~\cite{StreetR:enrcc}, which takes a 
bicategory $\bcatE$ satisying appropriate assumptions and produces a bicategory $\bimE$ with monads, 
monad bimodules and bimodule maps in $\bcatE$ as 0-cells, 1-cells and 2-cells, respectively. The appropriate 
assumptions on $\bcatE$ to perform this construction are that its hom-categories  have reflexive coequalizers and the composition 
functors preserve coequalizers in each variable. Although there are important examples of such bicategories,  they do not 
seem to have been isolated with specific terminology; we will call them \emph{tame}. Thus, we are led to establish that $\SymV$ is tame, a fact that allows
us to define the bicategory of operad bimodules by letting 
\[
\OpdV \defeq \Bim(\SymV) \, .
\]
We establish that $\SymV$ is tame as a consequence of the fact that the bicategory of $\CatSymV$ is tame, a result that will  
be useful also to prove that $\OpdV$ is cartesian closed in Chapter~\ref{cha:carcob}

This chapter is organized as follows. Section~\ref{sec:monmb} recalls the notions of a monad, monad bimodule and 
bimodule map in a bicategory, and shows how operads, operad bimodules and operad bimodule maps are instances of 
these notions. Section~\ref{sec:regbbb} recalls the bimodule construction and establishes some elementary facts about
it. In Section~\ref{sec:monadmorphisms} relates, for a tame bicategory $\bcatE$, the bicategory of bimodules 
$\bimE$ with the bicategory of monads $\MndE$, defined as in~\cite{StreetR:fortm}. The results established
there are used in Section~\ref{sec:anaf} to provide examples of analytic functors and in Chapter~\ref{cha:carcob}
to prove that~$\OpdV$ is cartesian closed. Section~\ref{sec:regbss} is devoted to proving that $\SymV$ is
tame. We end the chapter in Section~\ref{sec:anaf} by introducing the analytic functor associated to an operad
bimodule and providing examples thereof.

\section{Monads, modules and bimodules}
\label{sec:monmb}

For this section, let $\bcatE$ be a fixed bicategory.

\begin{definition} \label{thm:monad} Let $\objX \in \bcatE$. 
\begin{enumerate}[(i)]
\item A \emph{monad} on $\objX$ is a triple $(A, \mu, \eta)$ consisting of a 
morphism $A \co  \objX \to \objX$, a 2-cell $\mu \co  A \circ A \rightarrow A$,
called the {\em multiplication} of the monad, and a 2-cell $\eta \co  1_\objX \rightarrow A$,
called the {\em unit}  of the monad, such that the following diagrams (expressing
associativity and unit axioms) commute:
\[
\xymatrix@C=1.5cm{
A\circ A\circ A \ar[r]^-{A \circ \mu} \ar[d]_{\mu\circ A} & A\circ A  \ar[d]^{\mu} \\
A\circ A  \ar[r]_{\mu} & A ,} \qquad
\xymatrix{
A \ar[r]^-{A \circ \eta} \ar@/_1pc/[dr]_{1_A} & A\circ A  \ar[d]^{\mu} & A \ar[l]_-{\eta \circ A} \ar@/^1pc/[dl]^{1_A} \\
 & \, A  .&  }
\]
\item Let $A = (A, \mu_A, \eta_A)$ and $B = (B, \mu_B, \eta_B)$ be monads on $\objX$. A \emph{monad map} from~$A$ to~$B$ is a 2-cell $\pi \co  A \to B$ 
such that the following diagrams commute:
\[
 \xymatrix@C=1.5cm{
  A \circ A \ar[d]_{\mu_A} \ar[r]^{\pi \circ \pi} & B \circ B \ar[d]^{\mu_B} \\
   A \ar[r]_{\pi} & B } \qquad
\xymatrix{
1_\objX \ar[r]^{\eta_A} \ar@/_1pc/[dr]_{\eta_B} & A \ar[d]^{\pi} \\
  & B }  
\]
\end{enumerate}
\end{definition}

For $\objX \in \bcatE$, we write $\Mon(\objX)$ for the category of monads on $\objX$ and monad maps. Sometimes we will use~$\mu$ and~$\eta$ for the multiplication and the unit of 
different monads, whenever the context does not lead to confusion. Note that the notion of a monad is self-dual, in the sense that a monad in~$\bcatE$ is the same thing as a monad in~$\bcatE^\op$. 
The category $\Mon(\objX)$ can be  defined equivalently as the category of monoids and monoid morphisms in the category~$\bcatE[\objX,\objX]$, 
considered as a monoidal
category with composition as tensor product and the identity morphism $1_\objX \co  \objX  \to \objX$ as unit. Hence, 
a homomorphism~$\Phi\co  \bcatE \to \bcatF$ sends a monad $A \co  \objX \to \objX$  to a monad $\Phi(A) \co  \Phi(\objX) \to \Phi(\objX)$, since 
it induces a monoidal functor $\Phi_{\objX, \objX} \co  \pcatE[\objX,\objX]\to \bcatF[\Phi(\objX),\Phi(\objX)]$.
Clearly, monads on a small $\catV$-category in the 2-category $\CatV$ are $\catV$-monads in the usual sense. We give some further examples below.

\begin{example}[Monoids as monads] For a monoidal category $\catC$, 
monads in the bicategory~$\Sigma(\catC)$ are monoids in $\catC$~\cite[Section~5.4.1]{BenabouJ:intb}. In particular, monads in
$\Sigma(\Ab)$ are rings~\cite[Section~VII.3]{MacLaneS:catwm}. 
\end{example}

\begin{example}[Categories as monads]
We recall from~\cite{BenabouJ:intb} that a monad on a set~$X$ in the bicategory of matrices $\MatV$ of Section~\ref{sec:dist} is the same thing as a $\catV$-category with $X$ as its set of objects. Indeed, if $A \co  X \mat X$ is a monad, we can define a $\catV$-category $\catX$ with $\Obj(\catX) = X$ 
by letting $\catX[x,y] \defeq A[x,y]$, for $x, y \in X$, since the matrix $A \co  X \mat X$ is a function~$A \co  X \times X \to \catV$. The composition operation and the identity morphisms of
$\catX$ are given by the multiplication and unit of the monad, since they have components of the form
\[
\mu_{x,z} \co  \bigsqcup_{y \in X} \catX[y,z] \times \catX[x,y] \to \catX[x,z] \, , \quad
\eta_x \co  I \to \catX[x,x] \, .
\]
The associativity and unit axioms for a monad, as stated in Definition~\ref{thm:monad}, then reduce to the associativity and unit axioms for the composition operation in a~$\catV$-category.
\end{example}

\begin{example}[Operads as monads] \label{thm:opdasmnd} 
A monad on a set $X$ in~$\SymV$ is the same thing as an operad (by which we 
mean a symmetric many-sorted $\catV$-operad, which is the same thing as a symmetric $\catV$-multicategory), 
with $X$ as its set of sorts (or set of objects). For $\catV = \Set$, this was shown in~\cite[Proposition~9]{BaezJ:higdan}.
We give an outline of the proof for a general $\catV$, which follows from an an immediate generalization of a result by Kelly~\cite{KellyGM:opejpm} recalled in Remark~\ref{thm:keyyremark}. Let $A \co  X \sym X$ be a monad in~$\SymV$, \ie a symmetric sequence $A \co  X \sym X$, given by a functor $A \co  S(X)^\op \times X \to \catV$, equipped with a multiplication and  a unit. We define an operad with set of objects $X$ as follows. 
First of all, for~$x_1, \ldots, x_n, x \in X$, the object of  operations with inputs of sorts~$x_1, \ldots, x_n$ and output of sort~$x$ to be $A[x_1,\ldots, x_n; x]$. 
The symmetric group actions required to have an operad follow from the functoriality of $A$, since we have a morphism
\[
 \sigma^* \co  \catX[x_{\sigma(1)}, \ldots, x_{\sigma(n)}; x ] \to \catX[ x_1, \ldots, x_n; x] 
 \]
for each each permutation $\sigma \in \Sigma_n$ and $(x_1, \ldots, x_n, x) \in S_n(X)^\op \times X$.  By the definition of the composition operation in $\SymV$,
as given in~\eqref{equ:symcomp},  the multiplication $\mu \co  A \circ A \to A$ amounts to having a family of morphisms
\[
\theta_{\vec{x}_1,  \ldots, \vec{x}_m, \vec{x}, x}  
\co   A[ \vec{x}; x]  \otimes   A[ \vec{x}_1; x_1]  \otimes \ldots \otimes  A[\vec{x}_m; x_m]   \to 
A[ \vec{x}_1 \oplus \ldots \oplus \vec{x}_m; x]  \, , 
\]
where $\vec{x} = (x_1, \ldots, x_m)$, which is  natural in  $\vec{x}, \vec{x}_1, \ldots, \vec{x}_m \in S(X)^\op $ and
satisfies the equivariance condition expressed by the 
commutativity of the following diagram:

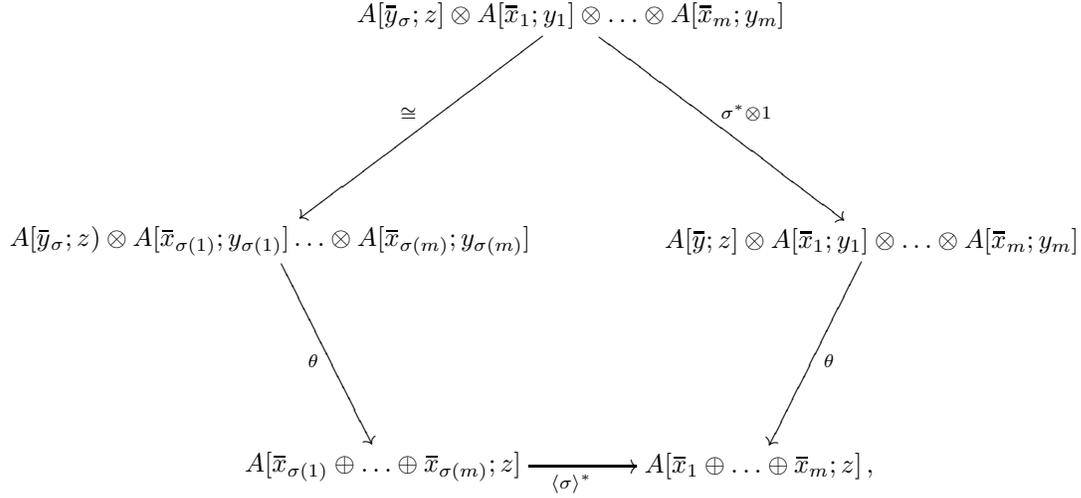
\[
\begin{xy} 
(0,30)*+{A[ \vec{y}_\sigma ; z]  \otimes A[ \vec{x}_1; y_1]  \otimes \ldots \otimes A[ \vec{x}_m; y_m] } ="1"; 
(-40,0)*+{A[ \vec{y}_\sigma; z) \otimes A[ \vec{x}_{\sigma(1)}; y_{\sigma(1)}] \ldots \otimes A[ \vec{x}_{\sigma(m)}; y_{\sigma(m)}] }="2"; 
(-25,-30)*+{ A[\vec{x}_{\sigma(1)} \oplus \ldots \oplus \vec{x}_{\sigma(m)}; z] }="3"; 
(25,-30)*+{ A[\vec{x}_1 \oplus \ldots \oplus \vec{x}_m; z] \, , }="4"; 
(40,0)*+{A[ \vec{y};z] \otimes A[ \vec{x}_1; y_1]  \otimes \ldots \otimes A[ \vec{x}_m; y_m] }="5";
{\ar_{\iso}  "1";"2"};
{\ar^{\sigma^* \otimes 1} "1";"5"};
{\ar^{\theta} "5";"4"};
{\ar_{\theta} "2";"3"};
{\ar_{\langle \sigma \rangle^*} "3";"4"};
\end{xy}
\] 
where $\langle \sigma \rangle \co  \vec{x}_{\sigma(1)} \oplus \ldots \oplus \vec{x}_{\sigma(m)} \to \vec{x}_1 \oplus \ldots \oplus \vec{x}_m$ is the evident morphism in $S(X)^\op$ induced
by~$\sigma$. It is common to represent maps $f \co I \to A[x_1, \ldots, x_n; x]$ as corollas of the form
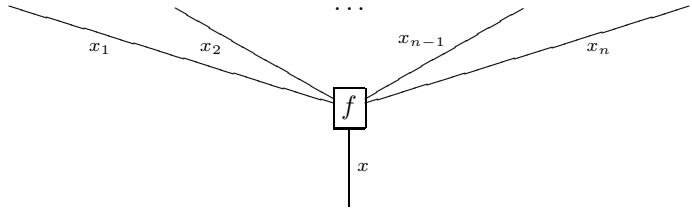
\[
\vcenter{\hbox{\xymatrix@C=2cm@R=1cm{
 \ar@{-}[drr]_(0.3){x_1} & \ar@{-}[dr]_(0.3){x_2}  & \ldots &  \ar@{-}[dl]_(0.5){x_{n-1}} &  \ar@{-}[dll]^(0.3){x_n}  \\
 & & *+[F] {f} \ar@{-}[d]^-{x}  & & \\
 & &  & &  }}}
\]
With this notation, the composition operation may be represented diagrammatically as a grafting operation. For example, the composite represented by the following grafting diagram
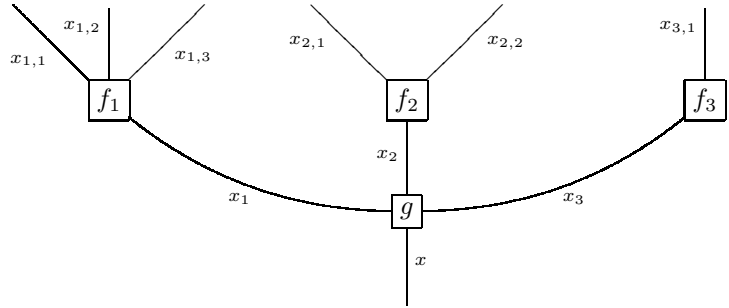
\[
\vcenter{\hbox{\xymatrix@C=1cm@R=1cm{
\ar@/1pc/@{-}[dr]_(0.3){x_{1,1}}  & \ar@{-}[d]_(0.3){x_{1,2}} &  \ar@{-}[dl]^-{x_{1,3}}  &    \ar@{-}[dr]_(0.3){x_{2, 1}} & &  \ar@{-}[dl]^(0.3){x_{2,2}} & &  \ar@{-}[d]_(0.3){ x_{3,1}  }  \\
            & *+[F] {f_1}   \ar@/_1pc/@{-}[drrr]_{x_1} &  & & *+[F] {f_2}  \ar@{-}[d]_{x_2} & &  & *+[F] {f_3}  \ar@/^1pc/@{-}[dlll]^{x_3}   \\
            &   &  & & *+[F] {g}\ar@{-}[d]^{x}   &  & & \\
            &   &  & &                                  & &  &   }}}
\]
is represented by
\[         
\vcenter
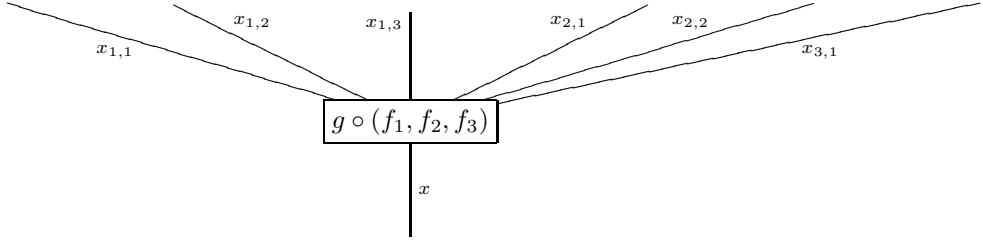
{\hbox{\xymatrix@C=2cm@R=0.5cm{
 \ar@{-}[ddrr]_(0.3){x_{1,1}} &   \ar@{-}[ddr]^(0.3){x_{1,2}} &   \ar@{-}[dd]_(0.2){x_{1,3} } &    \ar@{-}[ddl]_(0.3){x_{2,1}}  &  \ar@{-}[ddll]_(0.3){x_{2,2}}  &   \ar@{-}[ddlll]^(0.3){x_{3,1}}  \\     
 & & & & &    \\
            &  &  *+[F] {g \circ (f_1, f_2, f_3) } \ar@{-}[dd]^-{x} &  & & \\
             & & & & &    \\
            & &   &  & }}}
 \]
 where $g \circ (f_1, f_2, f_3) = \theta(g, f_1, f_2, f_3)$.
By the definition of the identity symmetric sequence in~\eqref{equ:symid}, the unit $\eta \co  \Id_X \to A$ then amounts to having a  morphism $1_x \co  I \to A[ (x); x]$ for each~$x \in X$. These give the identity operations of the operad. The associativity and unit axioms for a monad then correspond to the associativity and unit axioms for  operads.

An operad $(X,A)$ determines, for every symmetric $\catV$-rig $\rigR$, a monad $A \co \rigR^X \to \rigR^X$, whose underlying functor is the analytic functor associated to the symmetric sequence $A \co X
\to X$, which is given by the formula
\[
 A(T)(x) \defeq \int^{\vec{x} \in S(X)} A[\vec{x}, x] \otimes T^{\vec{x}} \, . 
 \]
We write $\Alg_\rigR(A)$ for the category of algebras and algebra morphisms for this monad, which is related to the category $\catV^X$ by a monadic adjunction 
\[
\xymatrix@C=1.5cm{\Alg_\rigR(A) \ar@<-1ex>[r]_-{U} \ar@{}[r]|-{\bot} & \rigR^X \ar@<-1ex>[l]_-{F} \, .} 
\]
\end{example}

It should be mentioned here 
that characterizations of  several kinds of operads as monads in appropriate bicategories are also given in~\cite{LeinsterT:higohc}.
Our setting is different to that considered in~\cite{LeinsterT:higohc}, since we work with bicategories of distributors rather than bicategories of spans (see~\cite{CurienPL:opecdl} for a discussion of the two settings). 

\medskip

The next definition recalls the standard notion of a left module for a monad, \ie what is often called a generalized algebra~\cite{LackS:a2cc}.

\begin{definition}  Let $A \co  \objX \to \objX$ be a monad in $\bcatE$. Let $\objK \in \bcatE$. 
\begin{enumerate}[(i)]
\item A \emph{left $A$-module with domain $\objK$} is a morphism~$M \co  \objK \rightarrow \objX$ equipped with a left $A$-action, \ie a 2-cell $\lambda \co   A\circ M \rightarrow M$ such that the following diagrams commute:
\[
\xymatrix@C=1.5cm{
A \circ A \circ M \ar[r]^-{A \circ \lambda} \ar[d]_-{\mu \circ M } & A \circ M \ar[d]^{\lambda} \\
A \circ M  \ar[r]_{\lambda} & M, } \hspace{1cm}
\xymatrix@C=1cm{
M \ar[r]^-{\eta\circ M} \ar@/_1pc/[dr]_-{1_M}  & A\circ M  \ar[d]^{\lambda}  \\
 & \,  M \, .}
 \]
\item  If~$M$ and~$M'$ are left $A$-modules with domain $\objK$, then a \myemph{
left $A$-module map} from~$M$ to~$M'$ is a 2-cell $f\co  M \rightarrow M'$ such that the following diagram commutes:
\[
\xymatrix{
A \circ M   \ar[r]^-{A \circ f }  \ar[d]_{\lambda} & A \circ M'  \ar[d]^{\lambda'}  \\
M \ar[r]_{f} & \, M' \, . }
\]
\end{enumerate}
We write~$\bcatE[\objK,\objX]^A$ for the category of left $A$-modules with domain~$\objK$ and 
left $A$-module maps. 
\end{definition}

\begin{example}[Left modules in a  monoidal category] For a monoidal category $\catC = (\catC, \otimes, I)$, a left module over a monoid $A$, viewed as a monad in $\Sigma(\catC)$, is the same thing as an object $M \in \catC$ equipped with a left $A$-action, \ie   a morphism $\lambda \co  A \otimes M \to M$ satisfying associativity and unit axioms. In particular,  
we obtain the familiar notion of left modules over a ring when~$\catC = \Ab$.
\end{example}

\begin{example}[Left modules for categories] For a small $\catV$-category $\catX$, viewed as a monad on $X \defeq \Obj(\catX)$ in~$\MatV$, left $\catX$-modules are families of presheaves on $\catX$, \ie controvariant $\catV$-functors from $\catX$ to~$\catV$. Indeed, a left $\catX$-module $M$ with domain~$K$ is a matrix $M \co  K \mat \Obj(\catX)$, \ie a functor $M \co  \Obj(\catX) \times K \to \catV$, equipped with a natural transformation with components
\[
\lambda_{x, k} \co  \bigsqcup_{x' \in \Obj(\catX)} \catX[x,x']  \otimes M[ x',k] \to M[x,k]  \, , 
\]  
satisfying associativity and unit axioms. It is immediate to see that this is the same thing as a family of~$\catV$-functors $M_k \co  \catX^\op \to \catV$, for $k \in K$.
\end{example}

\begin{example}[Left modules for operads] Let us consider an operad, given as a monad $(X,A)$ in~$\SymV$. A left $A$-module with domain~$K$ consists of an
symmetric sequence $M \co  K \sym X$, \ie a functor $M \co  S(K)^\op \times X      \to           \catV$ equipped with a left $A$-action
$\lambda \co   A  \circ M \to M$. By the definition of the composition operation in~$\SymV$, as given in~\eqref{equ:symcomp}, such a left action amounts to having maps in $\catV$ of the form
\[
M[ \vec{k}_1; x_1]  \otimes \ldots \otimes M[ \vec{k}_m; x_m] \otimes A[ x_1, \ldots, x_m;x] \to M[ \vec{k}_1 \oplus \ldots \oplus \vec{k}_m; x] \, , 
\]
which satisfy associativity and unit axioms and an equivariance condition. When $K = \emptyset$, we have $S(K) \iso 1$ and 
therefore  a left $A$-module with domain~$\emptyset$ is a family of objects~$M(x) \in \catV$, for~$x \in X$, equipped with maps  in $\catV$
of the form
\[
M(x_1) \otimes \ldots \otimes M(x_m)  \otimes A[x_1, \ldots, x_m; x] \to M(x)  \, ,
\]
satisfying the associativity and unit axioms and an equivariance condition.  Such left modules and their left module maps are exactly algebras and algebra morphisms 
for $A$ in $\catV$, in the sense of Example~\ref{thm:opdasmnd}, and so we have $\SymV[\emptyset, X]^A = \Alg_\catV(A)$.
Diagrammatically, if we represent a map $m \co I \to M(x)$ as 
\[
\xymatrix{
*+[F]{m} \ar@{-}[d]^-{x} \\
\quad }
\]
 then the left $A$-action can be seen, for example as acting as follows:
  \[
\vcenter{\hbox{\xymatrix@C=0.3cm@R=0.8cm{
            & *+[F]{m_1} \ar@/_1pc/@{-}[drrr]_(0.3){x_1}       &  & & *+[F]{m_2}   \ar@{-}[d]_{x_2}   &  & & *+[F]{m_3}    \ar@/^1pc/@{-}[dlll]^(0.3){x_3} \\
            &   &  & & *+[F] {f} \ar@{-}[d]^(0.6){x}   &  & & \\
            &   &  &  & \quad & &   }  }} \qquad \longmapsto \qquad
           \vcenter{\hbox{\xymatrix@C=0.3cm@R=1cm{
*++[F]{f \cdot (m_1, m_2, m_3)} \ar@{-}[d]^(0.6){x} \\ 
         \quad }}}
           \]
   
\end{example}

\begin{remark}
Let $\catK \in \bcatE$. The homomorphism~$\bcatE[\objK, -] \co 
\bcatE \to \Cat$ sends a monad $A \co  \objX \to \objX$ to a monad
$\bcatE[\objK,A]  \co  \bcatE[\objK,\objX] \to \bcatE[\objK,\objX]$.
Then, the category of left $A$-modules with domain~$\objK$ and left $A$-module maps
is exactly the category of algebras and algebra morphisms for the monad~$\bcatE[\objK,A]$. 
Therefore, we have an adjunction
\begin{equation*}
\xymatrix{
\bcatE[\objK,\objX]^A \ar@<-1.1ex>[r] \ar@{}[r]|-{\bot} & \bcatE[\objK,\objX] , \ar@<-1.1ex>[l] }
\end{equation*}
where the right adjoint is the forgetful functor and the left adjoint takes a morphism $M \co \objK\to \objX$
to the free left $A$-module on it, $A\circ M \co  \objK \to \objX$.
\end{remark}

Right modules and right module maps are defined in a dual way to left modules and left module maps. 
We state the explicit definition below.

\begin{definition} Let $A \co  \objX \to \objX$ be a monad in $\bcatE$. Let $\objK \in \bcatE$. 
\begin{enumerate}[(i)] 
\item A \emph{right $A$-module with codomain $\objK$}  is a morphism $M \co  \objX \rightarrow \objK$ equipped with a right $A$-action
$\rho \co  M\circ A \rightarrow M$ such that the following diagrams commute:
\[
\xymatrix{
M\circ A\circ A \ar[r]^(0.6){\rho\circ A}  \ar[d]_{M\circ \mu} & M\circ A \ar[d]^{\rho}  \\
M\circ A \ar[r]_{\rho} & \,  M \, ,} \qquad
\xymatrix{
M \ar[r]^(0.4){ M \circ \eta} \ar@/_1pc/[dr]_{1_M}  & M\circ A \ar[d]^{\rho} \\
 & \, M \, . } 
\]
\item If $M$ and $M'$ are two right $A$-modules with codomain $\objK$, then  a \myemph{right $A$-module map}
from $M$ to $M'$ is a 2-cell $f \co  M \rightarrow M'$ such that the following diagram commutes:
\[
\xymatrix{
M\circ A \ar[d]_{\rho} \ar[r]^{f\circ A}  & M'\circ A \ar[d]^{\rho'} \\
M \ar[r]_{f} & M'.}
\]
\end{enumerate}
We write $\bcatE[\objX, \objK]_A$ for the category of right $A$-module with codomain $\objK$ and 
right $A$-module maps. 
\end{definition}

\begin{example}[Right modules in a monoidal category] For a monoidal category $\catC = (\catC, \otimes, I)$,  
a right module over a monoid $A$, viewed as a monad in $\Sigma(\catC)$,  is the same thing as an object $M \in \catC$ equipped with a 
right $A$-action $\rho \co  M \otimes A \to M$. In particular, right modules in $\Ab$ are the same thing as right modules  over a ring in 
the usual sense.
\end{example}

\begin{example}[Right modules for categories] For a small $\catV$-category $\catX$, viewed as a monad $(X,A)$ in~$\MatV$, right $A$-modules 
are families of covariant $\catV$-functors from $\catX$ to $\catV$. Indeed, 
a right $A$-module with codomain~$K$ is a matrix $M \co  X \mat K$, \ie a functor $M \co  K \times X \to \catV$, equipped with a natural transformation with components
\[
\rho_{k,x} \co  \bigsqcup_{x' \in X} M[k,x'] \otimes A[x',x] \to M[k,x] \, , 
\]  
satisfying associativity and unit axioms. It is immediate to see that this is the same thing as a family of~$\catV$-functors $M_k \co  \catX \to \catV$, for $k \in K$.
\end{example}

\begin{example}[Right modules for operads] For a small $\catV$-operad, viewed as a monad $(X,A)$ in~$\SymV$, a right $A$-module with 
codomain~$K$ consists of a symmetric sequence~$M \co  X \to K$, \ie a functor $M \co  S(X)^\op \times K  \to   \catV$
equipped with a right $\catX$-action
$\rho \co   M \circ A  \to M$. By the definition of composition in $\SymV$, as given in~\eqref{equ:symcomp}, such an action amounts to having maps in $\catV$ of the form
\[
A[ \vec{x}_1; x_1]  \otimes \ldots \otimes A[\vec{x}_m; x_m]  \otimes M[ x_1, \ldots, x_m;k]  \to M[ \vec{x}_1 \oplus \ldots \oplus \vec{x}_m; k]  
\]
which satisfy associativity and unit axioms, as well as an equivariance condition. 
Diagrammatically, using notation analogous to the one adopted above, we have, for example, that

         \[
         \smallskip
\vcenter
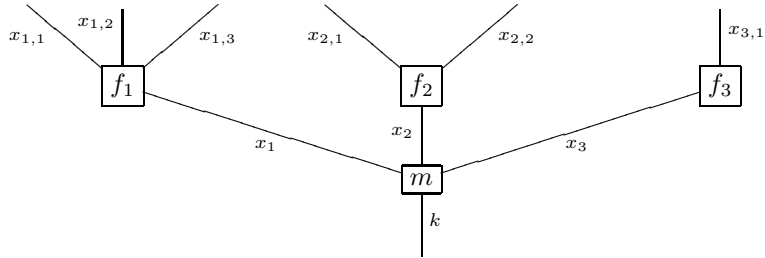
{\hbox{\xymatrix@C=1cm@R=0.8cm{
\ar@{-}[dr]_(0.3){x_{1,1}} &  \ar@{-}[d]_(0.3){x_{1,2}} & \ar@{-}[dl]^(0.3){x_{1,3}}  &    \ar@{-}[dr]_(0.3){x_{2, 1}} & &  \ar@{-}[dl]^(0.3){x_{2,2}} & &     \ar@{-}[d]^(0.4){x_{3,1}}  \\
            & *+[F]{f_1}    \ar@{-}[drrr]_{x_1} &  & & *+[F]{f_2}  \ar@{-}[d]_{x_2} & &  & *+[F]{f_3}  \ar@{-}[dlll]^{x_3}  \\
            &   &  & & *+[F]{m} \ar@{-}[d]^k   &  & & \\
            &   &  & & \quad &  & &   } }}
            \]
            is mapped to
            \[
             \smallskip
                \vcenter
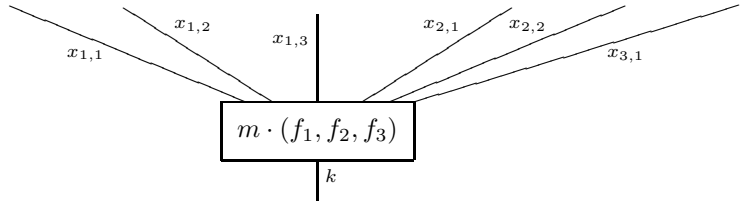
{\hbox{\xymatrix@C=1.3cm@R=0.5cm{
\ar@{-}[ddrr]_(0.3){x_{1,1}} &   \ar@{-}[ddr]^(0.3){x_{1,2}} &    \ar@{-}[dd]_(0.3){x_{1,3}} &    \ar@{-}[ddl]_(0.3){x_{2,1}}  &  \ar@{-}[ddll]_(0.3){x_{2,2}} &   \ar@{-}[ddlll]^(0.3){x_{3,1}}  \\     
 & & & & &    \\
            &  & *++[F]{m \cdot (f_1, f_2, f_3)}  \ar@{-}[d]^(0.6){k} &  & & \\
                         & & \quad &  & }}}
            \]
When $X = 1$,  these maps have the form
\[
A[n_1] \otimes \ldots \otimes A[n_m] \otimes M_k(m) \to M_k(n_1 + \ldots + n_m) \, ,
\]
where we write $M_k(n)$ for $M[n;k]$. These are $K$-indexed families of right $A$-modules for the operad, as usually defined in the literature  (see,
for example,~\cite{FresseB:modoof,KapranovM:modmto}). \end{example}

\begin{remark}
Let $\catK \in \bcatE$. The homomorphism~$\bcatE[\objK, -] \co 
\bcatE \to \Cat$ sends a monad $A \co  \objX \to \objX$ to a monad
$\bcatE[\objK, A] \co  \bcatE[\objX, \objK]  \to \bcatE[\objX, \objK]$.
Right $A$-modules with domain $\objK$ and right $A$-module maps are
the algebras and the algebra morphisms for the monad $\bcatE[\objK, A]$. Hence, we have an adjunction
\begin{equation*}
\xymatrix{
\bcatE[\objK,\objX]_A \ar@<-1.1ex>[r] \ar@{}[r]|-{\bot} & \bcatE[\objK,\objX] , \ar@<-1.1ex>[l] }
\end{equation*}
where the right adjoint is the forgetful functor and the left adjoint takes a morphism $M \co \objX \to \objK$
to the right $A$-module $M \circ A \co  \objK \to \objX$.
\end{remark}

\medskip

Next, we define the notions of a bimodule and bimodule map, which will play a fundamental role throughout the rest of the paper.
In particular, in Section~\ref{sec:regbbb} we will recall how  (under appropriate assumptions on $\bcatE$) monads, bimodules and 
bimodule maps form a bicategory.

\begin{definition}  Let $A \co  \objX \to \objX$ and $B \co  \objX \to \objX$ be monads in $\bcatE$.
\begin{enumerate}[(i)] 
\item  A \myemph{$(B,A)$-bimodule} is a morphism
$M \co  \objX \rightarrow \objY$ equipped with a left $B$-action 
$\lambda \co  B \circ M \rightarrow M$ and a right $A$-action
$\rho \co  M\circ A \rightarrow M$ which commute with each other,
in the sense that the following diagram commutes:
\begin{equation}
\label{equ:commutingactions}
\xycenter{
B\circ M \circ A \ar[d]_-{B \circ \rho} \ar[r]^-{\lambda\circ A} &  M \circ A \ar[d]^{\rho} \\
B \circ  M  \ar[r]_-{\lambda} & M \, .}
\end{equation}
\item If $M, M \co  \objX \to \objY'$ are $(B,A)$-bimodules, then a {\em bimodule map} from $M$ to $M'$
is a 2-cell $f \co  M \to M'$ that is a map of left $B$-modules and of 
right $A$-modules. 
\end{enumerate}
We write $\bcatE[\objX,\objY]^B_A$ for the
category of $(B,A)$-bimodules and bimodule maps.
\end{definition}

\begin{example}[Bimodules in a monoidal category] For a monoidal category $\catC = (\catC, \otimes, I)$, 
bimodules in a $\Sigma(\catC)$ are the same thing as objects of $\catC$ equipped with a 
right action and a left action by a monoid which distribute over each other. In particular, bimodules in~$\Sigma(\Ab)$ 
in the sense of the previous definition are the same thing as bimodules over a ring in the standard algebraic sense. 
\end{example}

\begin{example}[Bimodules for categories] As is well-known, bimodules in the bicategory $\MatV$  are exactly distributors, in the sense of Definition~\ref{thm:dis}.
 Indeed, for a small $\catV$-category $\catX$ with set of objects $X$
and a small $\catV$-category $\catY$ with set of objects $Y$, 
a  $(\catY, \catX$)-bimodule is a function $M \co  Y \times X \to \catV$ equipped with natural transformations with components
\[
\rho_{y,x} \co  \bigsqcup_{x' \in X} M[y,x'] \otimes \catX[x',x] \to M[y,x] \, , \quad
\lambda_{x,y} \co  \bigsqcup_{y' \in X}  \catY[y,y'] \otimes M[y',y]  \to M[y,x] \, , 
\]  
satisfying associativity and unit axioms. It is immediate to see that this is the same thing as a functor~$M \co  \catY^\op  \tensorvcat \catX \to \catV$, \ie a distributor
$M \co  \catX \mat \catY$.
\end{example}

\begin{example}[Operad bimodules] \label{thm:operadbimodule}
Bimodules in $\SymV$ are operad bimodules~\cite{MarklM:modo,RezkC:spaasc}, which we define explicitly below. Let us consider two operads $\catX = (X,A)$ and $\catY = (Y, B)$. Then,
an $(B, A)$-bimodule consists of a symmetric sequence $M \co  X \sym Y$, 
\ie a functor $M \co  S(X)^\op \times Y \to \catV$, equipped with a right $A$-action $\rho \co  M \circ A   \to M$ and a left $B$-action $\lambda \co  B \circ  M  
 \to M$ satisfying the compatibility condition in~\eqref{equ:commutingactions}. Explicitly, the right $A$-action amounts to having maps of the form
\[
A[\vec{x}_1; x_1] \otimes \ldots \otimes A[\vec{x}_m; x_m] \otimes M[x_1, \ldots, x_m;y] \to M[\vec{x}_1 \oplus \ldots \oplus \vec{x}_m; y]  \, ,
\]
while the left $B$-action amounts to having maps of the form
\[
M[ \vec{x}_1; y_1] \otimes \ldots \otimes M[ \vec{x}_m; y_m] \otimes B[ y_1, \ldots, y_m;y]  \to M[ \vec{x}_1 \oplus \ldots \oplus \vec{x}_m;y] \, ,
\]
all satisfying associativity, unit, compatibility and equivariance conditions. Bimodules for non-symmetric operads were defined
in~\cite[Definition~2.36]{LeinsterT:higohc}.
\end{example}

\begin{proposition} 
\label{bimoduleisbimonadic}
The forgetful functor $U \co  \bcatE[\objX,\objY]^B_A\to \bcatE[\objX,\objY]$
is monadic.
\end{proposition}

\begin{proof} 
Observe that the endofunctor
$\bcatE[A,B]\co  \bcatE[\objX,\objY] \to \bcatE[\objX,\objY]$
has the structure of a monad with multiplication $\mu= \bcatE[\mu_A, \mu_B]$
and unit $\eta=\bcatE[\eta_A, \eta_B]$.
A $(B,A)$-bimodule is the same thing as an $ \bcatE[A,B]$-algebra, which is a morphism $M\co \objX \to \objY$,
equipped with a 2-cell $\alpha \co  B\circ M\circ A \rightarrow M$ such that the following diagrams commute:
\[
\xymatrix{
B\circ B\circ M\circ A\circ A \ar[rr]^-{B\circ \alpha \circ A}  \ar[d]_{\mu_B\circ M\circ \mu_A} && B\circ M\circ A \ar[d]^{\alpha}  \\
B\circ M\circ A \ar[rr]_-{\alpha} && \, M \, ,} \qquad
\xymatrix{
M \ar[rr]^-{\eta_B\circ M\circ \eta_A} \ar@/_1pc/[drr]_{1_M}  && B\circ M\circ A \ar[d]^{\alpha} \\
& & M.}
\]
From  the 2-cell $\alpha$
we obtain two actions $\lambda=\alpha \cdot (B\circ M\circ \eta_A)$ 
and $\rho=\alpha\cdot (\eta_B\circ M\circ A)$ which commute with each other.
Conversely, from a commuting pair of actions $(\lambda,\rho)$ we obtain a 
2-cell~$\alpha$ which makes the required
diagrams commute by letting $\alpha \defeq \rho\cdot (\lambda \circ A)=\lambda\cdot (B\circ \rho)$, i.e.\
the common value of the composites in~\eqref{equ:commutingactions}.
\end{proof}

\begin{remark}
 The category $\bcatE[\objX,\objY]^B_A$ is related to the categories $\bcatE[\objX,\objY]^B$ and $\bcatE[\objX,\objY]_A$
by the following commutative squares of monadic forgetful functors (all written $U$) and left adjoints (all written~$F$):
\[
\xymatrix{
\bcatE[\objX,\objY]^B_A\ar[r]^U   \ar[d]_-{U}   & \bcatE[\objX,\objY]^B   \ar[d]^-{U} \\
\bcatE[\objX,\objY]_A  \ar[r]_-{U}    &  \bcatE[\objX,\objY]   \, ,}
\quad \quad 
\xymatrix{
\bcatE[\objX,\objY]^B_A &
\ar[l]_-{F} \bcatE[\objX,\objY]^B \\
\bcatE[\objX,\objY]_A   \ar[u]^-{F}  & \ar[l]^-{F} \bcatE[\objX,\objY]   \ar[u]_-{F}    \, , }
\]
\[
\xymatrix{
\bcatE[\objX,\objY]^B_A \ar[d]_-{U}  & \bcatE[\objX,\objY]^B  \ar[l]_-{F}  \ar[d]^-{U} \\
\bcatE[\objX,\objY]_A   & \ar[l]^-{F} \bcatE[\objX,\objY]   \, ,}
\quad \quad 
\xymatrix{
\bcatE[\objX,\objY]^B_A\ar[r]^-{U}    &\bcatE[\objX,\objY]^B   \\
\bcatE[\objX,\objY]_A \ar[r]_-{U}    \ar[u]^-{F}  & \bcatE[\objX,\objY]   \ar[u]^-{F}  \, .}
\]
\end{remark}

 \section{Tame bicategories and bicategories of bimodules}
\label{sec:regbbb}

We review the bimodule construction, which assembles monads, bimodules and bimodule maps in a bicategory $\bcatE$ satisfying appropriate assumptions into a new bicategory $\bimE$. For more information on the bimodule construction, 
see~\cite{BettiR:varte,CarboniA:axibb,GarnerR:enrcfc,KoslowskiJ:monib,StreetR:enrcc}.

\begin{definition} \label{def:tame} \hfill 
\begin{enumerate}[(i)] 
\item We say that $\bcatE$ is \emph{tame} if for every~$\objX,\objY \in \bcatE$ the 
category $\bcatE[\objX,\objY]$ has reflexive coequalizers and the horizontal composition functor of $\bcatE$
preserves coequalizers in each variable. 
\item If $\bcatE$ and $\bcatF$ are tame bicategories, we say that a homomorphism $\Phi\co \bcatE \to \bcatF$
is \emph{tame} if for every~$\objX,\objY\in \bcatE$, the functor $\Phi_{\objX,\objY} \co  \bcatE[\objX,\objY] \to \bcatF[\Phi \objX,\Phi\objY]$ 
preserves reflexive coequalizers.
\end{enumerate}
\end{definition}

As we will see below, the condition that a bicategory is tame is used in order to define the composition of bimodules.  If $\bcatE $ and $\bcatF$ are tame bicategories, we write $\REG[\bcatE, \bcatF]$ for
the full sub-bicategory of $\HOM[\bcatE, \bcatF]$
whose objects are tame homomorphisms from
$\bcatE$ to $\bcatF$.

\medskip

In order to make evident how different components of the structure of a bimodule play different roles in the definition of the
composition functors, we discuss ways of combining a bimodule with a (right or left) module.  From now until the end of the section, let $\bcatE$ be a fixed tame bicategory.  Let $A \co \objX \to \objX$, $B \co \objY \to \objY$ be monads in
 $\bcatE$. For a  left $A$-module $M \co \objK \to \objX$ and a $(B,A)$-bimodule $F \co \objX \to \objY$, we define a left $B$-module $F \circ_A M 
 \co \objK \to \objY$ as follows. Its underlying morphism is defined by following reflexive coequalizer diagram:
 \begin{equation}
 \label{equ:compositionleftmodulebimodule}
\xymatrix@C=1.2cm{
F  \circ  A  \circ M  \ar@<0.8ex>[r]^-{\rho\circ  M}  \ar@<-0.8ex>[r]_-{F \circ \lambda} 
&  F \circ   M  \ar[r]^q &  F \circ_{A} M \, .}
\end{equation}
The left $B$-action is determined by the universal property of coequalizers, as follows:
\[
\xymatrix{
B \circ F \circ A \circ M  \ar@<0.8ex>[rr]^{B \circ \rho\circ M}   \ar@<-0.8ex>[rr]_{B \circ F \circ \lambda}
\ar[d]_{\lambda \circ A \circ M} & & 
B \circ F  \circ M  \ar[rr]^{B \circ q} \ar[d]^{\lambda \circ M} &&  B \circ F \circ_A   M  \ar[d]^{\lambda}   \\
F  \circ  A  \circ M  \ar@<0.8ex>[rr]^-{\rho\circ  M}  \ar@<-0.8ex>[rr]_-{F \circ \lambda} 
& &  F \circ   M  \ar[rr]_q & &  F \circ_{A} M  \, ,}
\]
Here, the top row is a coequalizer diagram by the assumption that $\bcatE$ is tame, being obtained from the diagram in~\eqref{equ:compositionleftmodulebimodule}
by composition with $B$. The verification of the axioms for a left~$A$-action uses the fact that the actions of $A$ and $B$  commute with each other and it
is essentially straightforward. Furthermore, this definition can easily be shown to  extend to a functor
\[
(-) \circ_A (-) \co\bcatE[\objX, \objY]^B_A  \times  \bcatE[\objK, \objX]^A  \to \bcatE[\objK, \objY]^B \, .
\]
Dually, for a $(B,A)$-bimodule $F \co \objX \to \objY$ and a right $B$-module $M \co \objY \to \objK$, we define a right $A$-module $M \circ_B F 
 \co \objX \to \objK$ as follows. Its underlying morphism is defined by following reflexive coequalizer diagram:
 \begin{equation}
 \label{equ:compositionbimodulerightmodule}
\xymatrix@C=1.2cm{
M  \circ  B  \circ F  \ar@<0.8ex>[r]^-{M \circ \lambda}  \ar@<-0.8ex>[r]_-{\rho \circ F} 
&  M \circ   F  \ar[r]^q &  M \circ_{B} F \, . }
\end{equation}
The left $B$-action is determined by the universal property of coequalizers, as follows:
\begin{equation}
\label{equ:leftBactioninduced}
\xymatrix{
M \circ B \circ F \circ A  \ar@<0.8ex>[rr]^{M \circ \lambda \circ A}   \ar@<-0.8ex>[rr]_{\rho \circ F \circ A}
\ar[d]_{M \circ B \circ \rho} & & 
M  \circ F  \circ A  \ar[rr]^{q \circ A} \ar[d]^{M \circ \rho} &&  M \circ_B F \circ   A  \ar[d]^{\rho}   \\
M  \circ  B  \circ F  \ar@<0.8ex>[rr]^-{M \circ \lambda}  \ar@<-0.8ex>[rr]_-{\rho \circ F} 
& &  M \circ  F  \ar[rr]_q & &  M \circ_{B} F  \, .}
\end{equation}
In this way, we obtain a functor
\[
(-) \circ_B (-) \co \bcatE[\objX, \objY]^B_A \times \bcatE[\objY, \objK]_B \to \bcatE[\objX, \objK]_A \, .
\]

Let us now assume we have monads $A \co \objX \to \objX$, $B \co \objY \to \objY$ and $C \co \objZ \to \objZ$,
a $(B,A)$-bimodule $F \co \objX \to \objY$ and a $(C,B)$-bimodule $G \co \objY \to \objZ$. If we consider
$F$ as a left $B$-module and $G$ as a $(C,B)$-bimodule, the morphism given by the formula 
in~\eqref{equ:compositionleftmodulebimodule} coincides with the morphism given by the formula in~\eqref{equ:compositionbimodulerightmodule},
applied considering $F$ as a $(B,A)$-bimodule and $G$ as a right $B$-module. Hence, this morphism $G \circ_B F \co \catX \to \catZ$
is equipped with both a left $C$-action and a right $A$-action, which can be easily seen to 
 commute with each other, thus giving us a $(C,A)$-bimodule $G \circ_B F \co \objX \to \objZ$. Explicitly, the morphism~$G \circ_B F$  is  defined by the  reflexive coequalizer
\[
\xymatrix@C=1.2cm{
G \circ  B  \circ F  \ar@<0.8ex>[r]^-{\rho\circ  F}  \ar@<-0.8ex>[r]_-{G \circ \lambda} 
&  G \circ   F  \ar[r]^q &  G \circ_{B} F \, . }
\]
Its right $A$-action is determined by the diagram
\[
\xymatrix{
G \circ B \circ F \circ A  \ar@<0.8ex>[rr]^{\rho\circ F\circ A}   \ar@<-0.8ex>[rr]_{G\circ \lambda\circ A}
\ar[d]_{G\circ B\circ \rho} & & 
G \circ F  \circ A  \ar[rr]^{q\circ A} \ar[d]^{G\circ \rho} &&  (G \circ_B F )  \circ A  \ar[d]^{\rho}   \\
G \circ B \circ F  \ar@<0.8ex>[rr]^{\rho\circ F} \ar@<-0.8ex>[rr]_{G\circ \lambda} & & 
G \circ F \ar[rr]_q &&  G \circ_B F \, ,}
\]
while its left $C$-action  is given by the diagram
\[
\xymatrix{
C \circ G \circ B \circ F  \ar@<0.8ex>[rr]^{C\circ G\circ  \lambda}   \ar@<-0.8ex>[rr]_{C \circ \rho \circ F}
\ar[d]_{\lambda \circ B \circ F } & & 
C \circ G \circ F \ar[rr]^{C\circ q} \ar[d]^{\lambda\circ F} && C \circ (  G \circ_B F )   \ar[d]^{\lambda}   \\
G \circ B \circ F  \ar@<0.8ex>[rr]^{G\circ \lambda } \ar@<-0.8ex>[rr]_{\rho\circ F} & & G \circ 
F \ar[rr]_q &&  G \circ_B F \, . }
\]
In this way, we obtain a functor
\begin{equation}
\label{equ:bimcompfunctor} 
(-) \circ_B (-) \co \bcatE[\objY, \objZ]^C_B \times  \bcatE[\objX, \objY]^B_A \to \bcatE[\objX, \objZ]^C_A \, ,
\end{equation}
which we sometimes call \emph{relative composition}. This operation generalizes  the 
circle over construction defined by Rezk~\cite[Section~2.3.10]{RezkC:spaasc}, which is called the relative composition product in~\cite[Section~5.1.5]{FresseB:modoof}.

\medskip

 We can now define the bicategory $\bimE$. The objects of~$\bimE$ are 
 pairs of the form $(\objX, A)$, where $\objX \in \bcatE$ and $A \co  \objX \to \objX$ is a monad.
In order to simplify the notation, if we consider a pair~$(\objX,A)$ as an object of $\bimE$ we write it as $\objX/A$ and sometimes refer to it simply as a monad. For a pair of monads~$\objX/A$ and~$\objY/B$, we then define
\[
\bimE[\objX/A,\objY/B] \defeq  \bcatE[\objX,\objY]^B_A \, .
\]
Hence, a morphism in $\bimE$ from $\objX/A$ to $\objY/B$ is a $(B,A)$-bimodule $M \co  \objX/A \to \objY/B$
and a 2-cell $f \co  M\to N$ in $\bimE$ is a bimodule map. The composition operation of $\bimE$ is then 
given by the relative composition of bimodules in~\eqref{equ:bimcompfunctor}. 
For an object $\objX/A \in \bimE$, the identity bimodule $1_{\objX/A} \co  \objX/A \to \objX/A$
is given by the morphism $A \co  \objX \rightarrow \objX$, viewed as an $(A,A)$-bimodule by 
taking the monad multiplication $\mu \co  A \circ A \to A$ 
as both the left and the right $A$-action.
In order to complete the definition of the data of the bicategory $\bimE$, it remains to
exhibit the associativity and unit isomorphisms. For the associativity isomorphisms, 
let us define the joint composition of three bimodules
\[
\xymatrix{
(\objV,D)  \ar[r]^L   &(\objX,A)     \ar[r]^{M}   &(\objY,B)  \ar[r]^{N}  &(\objZ,C)   
}
\]
as the colimit  $N \circ_B M  \circ_A L$ of  a double (reflexive) graph
\[
\xymatrix@R=0.5cm@C=0.5cm{
N \circ B \circ M \circ A \circ L  \ar@<0.8ex>[dd]  \ar@<-0.8ex>[dd]  \ar@<0.8ex>[rr]  \ar@<-0.8ex>[rr]   &&  
N\circ B\circ M\circ L  \ar@<0.8ex>[dd]  \ar@<-0.8ex>[dd]   \\
 \\
N\circ M\circ A\circ L \ar@<0.8ex>[rr]  \ar@<-0.8ex>[rr]  && N\circ M\circ L \, .}
\]
If the colimit is calculated horizontally and then vertically, we obtain $ N \circ_B( M  \circ_A L)$. If the colimit is calculated vertically and then horizontally, instead, we obtain $(N \circ_B M)  \circ_A L$. 
Thus, we have the required isomorphism $a_{L,M,N} \co  N \circ_B( M  \circ_A L) \to  (N \circ_B M)  \circ_A L. $
This isomorphism is the unique 2-cell $a$
fitting in the commutative diagram  of canonical maps,
\[
\xymatrix{
N \circ (M \circ L) \ar[d]  \ar@{=}[r] & (N \circ M) \circ L \ar[d]\\
 N \circ_B (M \circ_A L)   \ar[r]_-a &(N \circ_B M) \circ_A L \, . }
\]
For the unit isomorphisms, observe that for $\objX/A, \objY/B\in \bimE$ and  $M \co  \objX/A \rightarrow \objY/B$,
 we have an isomorphism $\ell_M \co  B \circ_B M \to M $ which fits in the diagram
\begin{equation}
  \label{equ:leftunitisomorphism}
{\vcenter{\hbox{\xymatrix@C=1.5cm{
B \circ B \circ M  \ar@<1ex>[r]^-{M\circ \mu} \ar@<-1ex>[r]_-{\rho\circ A}  & 
B \circ M             \ar[r]^-{q} \ar@/_1pc/[dr]_-{\rho}     &  
B \circ_B M          \ar[d]^{\ell_M}   \\   
 & 
 & 
 M  }}}}
\end{equation}
since 
\[
\xymatrix@C=1.5cm{
B \circ B \circ M 
\ar@<0.8ex>[r]^-{\mu \circ M} 
\ar@<-0.8ex>[r]_-{ B\circ \lambda} 
& B \circ M \ar[r]^-{\lambda}   \ar@/_2pc/[l]_-{\eta\circ B\circ M}  &   M  \ar@/_2pc/[l]_-{\eta\circ M}    }
\]
is a split fork. Dually, we have also an isomorphism $r_M \co  M \circ_A A \to M$ making the following diagram commute
  \begin{equation}
  \label{equ:rightunitisomorphism}
  {\vcenter{\hbox{\xymatrix@C=1.5cm{
M \circ A \circ A 
\ar@<0.8ex>[r]^-{M\circ \mu} 
\ar@<-0.8ex>[r]_-{\rho\circ A} 
&  M \circ A \ar[r]^-{q}   \ar@/_1pc/[dr]_{\rho} &   M \circ_A A \ar[d]^{r_M} \\
 &    &   M  }}}}
\end{equation}
since 
\[
\xymatrix@C=1.5cm{
M \circ A \circ A 
\ar@<0.8ex>[r]^-{M\circ \mu} 
\ar@<-0.8ex>[r]_-{\rho\circ A} 
&  M \circ A \ar[r]^-{\rho}   \ar@/_2pc/[l]_-{M\circ A\circ \eta}  &   M  \ar@/_2pc/[l]_-{M\circ \eta}    }
\]
is a split fork. The verification of the coherence  axioms for a bicategory is a straightforward diagram-chasing argument.

\begin{remark} \label{rem:je}
For every tame bicategory $\bcatE$, there is a homomorphism 
\[
\JE \co  \bcatE \rightarrow \bimE
\]
which maps an object~$\objX \in \bcatE$ to the identity monad  $\objX/1_\objX$. The action of $\JE$ on morphisms and 2-cells is evident 
and hence we do not spell it out.  
\end{remark}

\begin{remark} If the monad $B \co \objY \to \objY$ is the identity $1_\objY \co  \objY \to \objY$, then we have a canonical isomorphism $G \circ_{1_\objY} F \iso
G \circ F$, which we will consider as an equality for simplicity.
\end{remark}

\begin{example}[Bimodules in a monoidal category] For a monoidal category $\catC$ with reflexive coequalizers in which the
tensor product preserves reflexive coequalizers, the bicategory $\Bim(\catC)$ has monoids in $\catC$ as 
objects, bimodules as morphisms and bimodule maps as 2-cells (see also~\cite{BarrM:chuc} for a
discussion of this example). In particular, $\Bim(\Ab)$ is the
bicategory of rings, ring bimodules and bimodule maps. Given a~$(B,A)$-bimodule~$M$ and a $(C,B)$-bimodule~$N$, where~$A$,~$B$ 
and~$C$ are rings, their tensor product~$N \otimes_B M$ fits in the following reflexive coequalizer in $\Ab$:
\[
\xymatrix@C=1.2cm{
N  \otimes  B  \otimes M  \ar@<0.8ex>[r]^-{\rho\otimes  M}  \ar@<-0.8ex>[r]_-{N \otimes \lambda} 
&  N \otimes   M  \ar[r]^q &  N \otimes_{B} M \, . }
\]
\end{example} 

\begin{example}[Distributors] The bicategory of bimodules in the bicategory  of matrices $\MatV$ is exactly the bicategory of distributors:
\[
\DistV = \Bim(\MatV) \, .
\]  
Indeed, we have seen that small $\catV$-categories, which are the objects of $\DistV$, are monads in $\MatV$, distributors
$F \co  \catX \mat \catY$ are the same thing as $(\catY, \catX)$-bimodules, and  $\catV$-natural transformations between distributors
are exactly a bimodule maps. Furthermore, composition and identities in~$\DistV$ arise as special cases of the general definition of
composition and identities in bicategories of bimodules, as direct calculations show. 
\end{example}

\begin{remark} \label{thm:dualitybimod}
For every tame bicategory $\bcatE$, there is an isomorphism
\[
\bimE^\op \iso \Bim(\bcatE^\op) \, .
\]
Recall that we write $F^\op \co \catY \to \catX$ for the 
morphism in $\bcatE^\op$ associated to a morphism $F \co \catX \to \catY$ in $\bcatE$.
The required isomorphism sends an object $\objX/A \in  \bimE^\op$ to the object 
$\objX/A^\op \in \Bim(\bcatE^\op)$. Given $\objX/A \, , \objY/B \in \bimE$,
if $M \co \objX \to \objY$ has a $(B,A)$-bimodule structure in $\bcatE$, then
$M^\op \co \objY \to \objX$ has an $(A^\op,B^\op)$-bimodule structure in $\bcatE^\op$ and so we have an isomorphism
\[
\Bim(\bcatE^\op) [\objY/B^\op, \objX/A^\op] \iso  \bimE[\objX/A, \objY/B] = (\bimE)^\op[\objY/B, \objX/A] \, .
\]
Moreover, if $N \co \objY \to \objZ$ has a $(C,B)$-bimodule structure, we
have $(N \circ_{B} M)^\op = M^\op \circ_{B^\op} N^\op$.
\end{remark}

\begin{remark} \label{thm:bimfunctorial} For every tame homomorphism $\Phi \co  \bcatE \to \bcatF$, there is a homomorphism 
\[
\Bim(\Phi) \co  \bimE \to \Bim(\bcatF)
\]
defined by letting $\Bim(\Phi)(\objX/A) \defeq \Phi(\objX)/{\Phi(A)}$
for $\objX/A\in \Bim(\bcatE)$,
$  \Bim(\Phi)(M) \defeq \Phi(M)$ for $M \co  \objX/A\to \objX/B$
 and  $\Bim(\Phi)(\alpha) \defeq \Phi(\alpha)$ for $\alpha \co  M \to N$.
The condition that $\Phi$ is tame ensures that $\Bim(\Phi) \co  \bimE \to \Bim(\bcatF)$ preserves
composition and identity morphisms up to coherent natural isomorphism. 
In this way, an inclusion $\bcatE \subseteq \bcatF$ determines an inclusion $ \bimE \subseteq \bimF$.
\end{remark}

\section{Monad morphisms and bimodules}
\label{sec:monadmorphisms}

The aim of this section is to relate the bicategory of monads, bimodules and bimodule maps introduced in Section~\ref{sec:regbbb} to the bicategory of monads, monad morphisms and maps of monad morphisms defined as in~\cite{StreetR:fortm}.  We begin by 
recalling  some definitions from~\cite{StreetR:fortm}, using a slightly different terminology.

\begin{definition}  Let $A \co \objX \to \objX \, , B \co \objY \to \objY$ be monads in $\bcatE$.
\begin{enumerate}[(i)]
\item A  \emph{lax monad morphism} $(F,\phi)\co (\objX,A) \to (\objY,B)$ consists of a morphism
$F \co  \objX \to \objY$ in~$\bcatE$ and a 2-cell $\phi \co  B\circ F \to F \circ A$ such that 
 the following diagrams commute:
\[
\xymatrix@C=1.2cm{
B\circ B\circ F\ar[d]_{\mu_B\circ F}\ar[r]^{B\circ \phi }&B \circ F \circ A \ar[r]^{\phi \circ A} & F \circ  A\circ A \ar[d]^{F \circ \mu_A} \\
B\circ F \ar[rr]_\phi&&  \, F \circ A \, , 
}\quad\quad
\xymatrix{
 F \ar[r]^-{\eta_B \circ F}  \ar@/_1pc/[dr]_-{F \circ \eta_A } & B \circ F \ar[d]^-{\phi}  \\
 &  F \circ A \, . }
\]
\item A \emph{map of lax monad morphisms} $f \co  (F, \phi) \rightarrow (F', \phi')$ 
 is a 2-cell $f \co   F \rightarrow  F'$ in $\bcatE$ such that
 the following diagram commutes:
 \[
 \xymatrix@C=1.2cm{
 B \circ F   \ar[r]^{\phi} \ar[d]_{B \circ f}  & F \circ A\ar[d]^{f \circ A} \\
B \circ F'   \ar[r]_{\phi'}  &  F' \circ A \, .} 
 \]
\end{enumerate}
\end{definition} 

Following~\cite{StreetR:fortm}, we write $\MndE$ for the bicategory whose objects are pairs~$(\objX, A)$,
where~$\objX \in \bcatE$ and~$A \co \objX \to \objX$ is a monad, morphisms are lax monad morphisms
and 2-cells are maps of lax monad morphisms. The composition operation of morphisms in $\MndE$ is defined in the following way: for monad morphisms $(F, \phi) \co  
(\objX, A) \to (\objY, B)$ and $(G, \psi) \co  (\objY, B) \to (\objZ, C)$, their composite is given by
the morphism $G \circ F 
\co  \objX \to \objZ$ equipped with the 2-cell
\[
\xymatrix{
C \circ G \circ F \ar[r]^{\psi \circ F} & G \circ B \circ F \ar[r]^{G \circ \phi} & G \circ F \circ A \, .}
\]
Note that a lax monad morphism $(F,\phi)\co (\objX,A)\to (\objY,B)$  is an equivalence in $\MndE$ if and only if $F
\co \objX \to \objY$ is an equivalence in $\bcatE$ and $\phi$ is invertible. Also note that a homomorphism $\Phi\co  \bcatE \to \bcatF$ induces a homomorphism $\Mnd(\Phi)\co \Mnd(\bcatE)\to \Mnd(\bcatF)$, defined in the evident way. As we show in Lemma~\ref{monadmorphtomod} below, every lax monad morphism gives rise to a bimodule and every map of lax monad morphisms gives rise to a bimodule map.

\begin{lemma} \label{monadmorphtomod} \hfill
\begin{enumerate}[(i)]
\item  If $(F,\phi)\co (\objX,A) \to (\objY,B)$ is a lax monad morphism, then the morphism $F \circ A \co \objX \to \objY$
has the structure of a $(B,A)$-bimodule.
\item If $f \co (F,\phi)\to (F',\phi')$ 
is a map of lax monad morphisms, then $f \circ A\co F \circ A\to F' \circ A$
is a bimodule map.
\end{enumerate}
\end{lemma}

\begin{proof} The  morphism $F \circ A  \co  \objX \to \objY$ has the structure 
of a free right $A$-module with the right action $\rho=F \circ \mu_A\co (F \circ A)\circ A\to F \circ A$.
The left action $\lambda\co  B\circ (F \circ A)\to F \circ A$ of  the monad~$B$
is defined to be the composite
\[
\xymatrix{
B\circ F \circ A\ar[r]^{\phi\circ A} & F \circ A\circ A\ar[rr]^{F \circ \mu_A}&& F \circ A  \, . }
\]
The proof the required properties is a straightforward diagram-chasing argument.
\end{proof}

Lemma~\ref{monadmorphtomod} allows us to define a homomorphism $R \co  \MndE \to \bimE$
as follows. 
Its action on objects is the identity. For a 
monad morphism $(F,\phi)\co  (\objX,A) \to (\objY,B)$, we define $R(F,\phi) \co  \objX/A  \to \objY/ B$
to be the $(B,A)$-bimodule with underlying morphism $F\circ A \co  \objX \to \objY$, as in the proof
of Lemma~\ref{monadmorphtomod},
Given a monad $2$-cell $f \co (F,\phi)\to (F',\phi')$, we define $R(f) \co F \circ A \to F' \circ A$ to be $f \circ A\co F \circ A\to F' \circ A$,
which is a bimodule map by part (ii) of Lemma~\ref{monadmorphtomod}. The remaining data necessary to define a homomorphism can be
derived easily. In particular, for lax monad morphisms  $(F,\phi)\co (\objX,A) \to (\objY,B)$ and $(G,\psi)\co  (\objY,B)\to (\objZ,C)$, we have an isomorphism
\[
R(G,\psi)\circ_B R(F,\phi) \iso
R(G \circ F, (G \circ \phi) \cdot (\psi \circ F))   \, , 
\]
since 
\[
R(G, \psi)\circ_B R(F,\phi)  = (G \circ B) \circ_B (F \circ A) \iso G \circ F \circ A = R(G \circ F) \, .
\]

Recall that for a bicategory $\bcatE$, we write $\bcatE^\op$ for the bicategory obtained from $\bcatE$
by formally reversing the direction of morphisms, but not that of 2-cells. Since $(\Bim(\bcatE^\op))^\op = 
\bimE$, Lemma~\ref{monadmorphtomod} admits a dual, which we state explicitly below since it will be useful in the
following. We begin by recalling from~\cite{StreetR:fortm} an explicit description of the morphisms and 2-cells 
in~$(\Mnd(\bcatE^\op))^\op$.

\begin{definition}
Let $A \co \objX \to \objX \, , B \co \objY \to \objY$ be monads in $\bcatE$.
\begin{enumerate}[(i)]
\item An  \emph{oplax monad morphism} $(F,\phi)\co (\objX,A) \to (\objY,B)$ consists of a morphism
$F \co  \objX \to \objY$ in~$\bcatE$ and a 2-cell $\psi \co  F \circ A \to B\circ F$ such that 
 the following diagrams commute:
\[
\xymatrix@C=1.2cm{
F \circ A \circ A \ar[d]_{F \circ \mu_A}\ar[r]^{\psi \circ A}& B \circ F \circ A  \ar[r]^{B \circ \psi} & B \circ  B \circ  F \ar[d]^{\mu_B \circ F} \\
F \circ A \ar[rr]_\psi&&  \, B \circ F \, , 
}\quad\quad
\xymatrix{
 F \ar[r]^-{F \circ \eta_A}  \ar@/_1pc/[dr]_-{\eta_B  \circ F } & F \circ A \ar[d]^-{\psi}  \\
 &  B \circ F \, . }
\]
\item A \myemph{map of oplax monad morphisms} $f \co  (F, \psi) \rightarrow (F', \psi')$ 
 is a 2-cell $f \co   F \rightarrow  F'$ in $\bcatE$ such that
 the following diagram commutes:
 \[
 \xymatrix@C=1.2cm{
 F \circ A \ar[r]^{\psi} \ar[d]_{f \circ A}  & B \circ F \ar[d]^{B \circ f} \\
 F' \circ A \ar[r]_{\psi'}  & B \circ F' \, .} 
 \]
\end{enumerate}
\end{definition}

The bicategory $\Mnd(\bcatE^\op)^\op$ has the same objects 
as $\MndE$, oplax monad morphisms as morphisms and maps of oplax
monad morphisms as 2-cells. We now state the dual of Lemma~\ref{monadmorphtomod}.

\begin{lemma} \label{monadmorphtomod2} \hfill
\begin{enumerate}[(i)]
\item  If $(F,\psi)\co (\objX,A) \to (\objY,B)$ is an oplax monad morphism
in a bicategory $\bcatE$, then the morphism $B \circ F  \co \objX \to \objY$
has the structure of a $(B,A)$-bimodule.
\item If $f \co (F,\psi)\to (F',\phi')$ 
is a map of oplax monad morphisms, then $B \circ f \co B \circ F \to B \circ F'$
is a bimodule map.
\end{enumerate}
\end{lemma}

\begin{proof} The  morphism $B \circ F \  \co  \objX \to \objY$ has the structure 
of a free left $B$-module with the left action $\mu  \circ F \co B \circ B \circ F \to B \circ B$.
The right $A$-action is defined to be the composite
\[
\xymatrix{
B \circ F \circ A\ar[r]^{B \circ \psi} & F \circ A\circ A\ar[rr]^{\mu_B \circ A}&& B  \circ F  \, . }
\]
As for Lemma~\ref{monadmorphtomod}, we omit the details of the verification.
\end{proof}

By Lemma~\ref{monadmorphtomod2}, it is possible to define a homomorphism
$L \co (\Mnd(\bcatE^\op))^\op \to \bimE$, in complete analogy with the way in
which we defined the homomorphism $R \co  \MndE \to \bimE$ using Lemma~\ref{monadmorphtomod}.
We omit the details, which are straightforward. The following remark will be useful in 
the proofs of Proposition~\ref{thm:bimcart} and Theorem~\ref{thm:ccbimod}.

\begin{remark} \label{thm:pseudonatmonadmorphism} 
Let $\Phi, \Psi\co  \bcatE \to \bcatF$ be homomorphisms  and let $F \co  \Phi \to \Psi$ be a 
pseudo-natural transformation. If $A \co \objX \to \objX$ is a monad in $\bcatE$,  we have monads 
$\Phi(A) \co \Phi(\objX) \to \Phi(\objX)$ and $\Psi(A) \co \Phi(\objX) \to \Psi(\objX)$ in $\bcatF$. Then,
the morphism $F_\objX \co \Phi(\objX) \to \Psi(\objX)$ and the pseudo-naturality 2-cell
\[
\xymatrix{
\Phi(\objX) \ar[r]^{F_\objX}  \ar[d]_{\Phi(A)} 
\ar@{}[dr]|{\quad \Downarrow \, f_A} 
& \Psi(\objX) \ar[d]^{\Psi(A)} \\
\Phi(\objX) \ar[r]_{F_\objX} & \Psi(\objX) }
\]
give us a lax monad morphism 
\[
(F_\objX, f_A) \co (\Phi(\objX), \Phi(A)) \to (\Psi(\objX), \Psi(A)) \, .
\]
Since the 2-cell $f_A$ is invertible, we also have an oplax monad morphism
\[
(F_\objX, f_A^{-1}) \co (\Phi(\objX), \Phi(A)) \to (\Psi(\objX), \Psi(A)) \, .
\]
\end{remark} 

\medskip

We now show how, under appropriate assumptions, for monads $A \co \objX \to \objX$ and $B \co \objY \to \objY$ in~$\bcatE$,
an adjunction  $(F, G, \eta, \varepsilon) \co \objX \rightarrow \objY$ induces an adjunction 
$(F', G', \eta', \varepsilon')  \co (\objX, A) \rightarrow (\objY, B)$ in $\bimE$. In order to do this, we will exploit
Lemma~\ref{monadmorphtomod}
and Lemma~\ref{monadmorphtomod2}.  Let us begin by observing that if $(F, G) \co \objX \rightarrow \objY$ is an adjunction in $\bcatE$, then a monad
$B \co \objY \to \objY$ induces a monad~$B' \co \objX \to \objX$, where~$B' \defeq G  \circ B 
\circ F$, with multiplication defined as the 
composite 
\[
\xymatrix@C=2cm{
G \circ B \circ F \circ G \circ B \circ F \ar[r]^-{G \circ B \circ \varepsilon \circ B \circ F} & 
G \circ B \circ B \circ F \ar[r]^-{G \circ \mu_B \circ F} & 
G \circ B \circ F \, ,}
\]
and unit $\eta' \co 1_\objX \to G \circ B \circ F$ defined as the composite 
\[
\xymatrix@C=1.5cm{
1_\objX \ar[r]^-{\eta} & G \circ F \ar[r]^-{G \circ \eta_B \circ F} & G \circ B \circ F\, .}
\]

Theorem~\ref{thm:transportadjunction} below allows us, under appropriate hypotheses, to construct
adjunctions in $\bimE$ from adjunctions in  $\bcatE$.

\begin{theorem} \label{thm:transportadjunction} Let $\bcatE$ be a tame bicategory.
Let $A \co \objX \to \objX \, , B \co \objY \to \objY$ be monads  in $\bcatE$ and let $(F, G, \eta, \varepsilon) \co \objX \rightarrow \objY$
be an adjunction in $\bcatE$. Then, a monad map $\xi \co A \to G \circ B \circ F$ determines an adjunction 
$(F', G', \eta', \varepsilon')  \co (\objX, A) \rightarrow (\objY, B)$ in $\bimE$.
\end{theorem}

\begin{proof} Given a monad map (in the sense of Definition~\ref{thm:monad}) $\xi \co A  \to G \circ B \circ F$, we define the 2-cell $\psi \co F \circ A \to B \circ F$ as the composite
\[
\xymatrix@C=1.2cm{F \circ A \ar[r]^-{F \circ \xi} &  F \circ G \circ B \circ F \circ \ar[r]^-{ \varepsilon \circ B \circ F} & B \circ F \, .}
\]
A  standard diagram-chasing argument shows that  $(F, \psi) \co (\objX,A) \to (\objY, B)$ is an oplax monad morphism. Similarly, we define the 2-cell 
$\phi \co A \circ G \to G \circ B$ as the composite
\[
\xymatrix@C=1.2cm{A \circ G \ar[r]^-{\xi \circ G} & G \circ B \circ F \circ G \ar[r]^-{G \circ B \circ \varepsilon} & G \circ B}
\]
and obtain a lax monad morphism $(G, \phi) \co (\objY, B) \to (\objX ,A)$. We define $F' \co (\objX, A) \rightarrow (\objY, B)$ to be the $(B,A)$-bimodule associated to the oplax monad morphism $(F, \psi) \co (\objX ,A) \to (\objY, B)$, as in Lemma~\ref{monadmorphtomod2}. Explicitly, the morphism $F' \defeq B \circ F$ is  
equipped with the left $B$-action $\lambda \co B \circ F' \to F'$ given by $\mu \circ F \co B \circ B \circ F \to B \circ F$ and the 
right $A$-action $\rho \co F' \circ A \to F'$  given by the composite 
\[
(\mu_B \circ F) \cdot (B \circ \psi) \co B \circ F \circ A \to B \circ F \, .
\]
Similarly, the right adjoint $G' \co (\objY, B) \to (\objX, A)$ is the $(A,B)$-bimodule associated to the lax monad 
morphism $(G, \phi) \co (\objY, B) \to (\objX ,A)$, as in Lemma~\ref{monadmorphtomod}. Explicitly,  $G' \defeq G \circ B$
is  equipped with the left $A$-action $\lambda \co A \circ G' \to G'$ given by the 
composite
\[
(G \circ \mu_B) \cdot (\phi \circ B) \co  A \circ G \circ B \to G \circ B
\] 
and the 
right $A$-action $\rho \co G' \circ A \to G'$ given by $G \circ \mu_B \co G \circ B \circ B \to G \circ B$.

In order to define the unit of the adjunction $\eta' \co 1_{(\objX, A)} \to G' \circ_B F'$, observe that by the
definition of the relative composition we have
\[
G' \circ_B F' \iso G \circ B \circ F \, .
\] 
Hence, we define $\eta'$ 
to be the monoid map $\xi \co A \to G \circ B \circ F$. The counit $\varepsilon' \co F' \circ_A G' \to 1_{(\objY, B)}$
is obtained via the universal property of coqualizers, via the diagram
\[
\xymatrix{
F' \circ A \circ G' \ar@<1ex>[r] \ar@<-1ex>[r]   & F' \circ G' \ar[r]  
\ar@/_1pc/[dr]_-{\sigma} & F' \circ_A G'  \ar[d]^{\varepsilon'} \\
  & & B \, , }
\]
where $\sigma$ is the composite
\[
\xymatrix{
B \circ F \circ G \circ B \ar[r]^-{B \varepsilon B} & B \circ B \ar[r]^-{\mu} & B \, .}
\]
Unfolding the relevant definitions, the triangular laws amount to the commutativity of the diagrams
\[
\xymatrix@C=1.5cm@R=1.5cm{
B \circ F \circ_A A \ar[r]^-{B \circ F \circ_A \xi} \ar@/_1pc/[dr]_{\ell_{B \circ F}} & B \circ F \circ_A G \circ B \circ F \ar[d]^{\varepsilon' \circ F} \\
 & B \circ F \, , } \qquad
 \xymatrix@C=1.5cm@R=1.5cm{
 A \circ_A  G \circ B \ar[r]^-{\xi \circ_A G \circ B} \ar@/_1pc/[dr]_{r_{G \circ B}} & G \circ B \circ F \circ_A  G \circ B \ar[d]^{G \circ \varepsilon'}  \\
 & G \circ B \, ,}
 \]
 where $\ell_{B \circ F}$ and $r_{G \circ B}$ are the unit isomorphisms of~$\bimE$, defined as in~\eqref{equ:leftunitisomorphism}
 and~\eqref{equ:rightunitisomorphism}. In both cases, the required commutativity follows by the universal property defining the relative composition operation and
in particular the definition of the unit isomorphisms.
\end{proof} 

Let us remark that the adjunction $(F', G', \eta', \varepsilon')$ in $\bimE$ constructed in
the proof of Theorem~\ref{thm:transportadjunction} is not obtained by
applying the homomorphism $\JE \co  \bcatE \rightarrow \bimE$ to the adjunction $(F, G, \eta, \varepsilon)$ in $\bcatE$. 
Indeed, applying $\JE$
to the latter would give us an adjunction in $\bimE$ between the identity monads $(X,1_X)$ and $(Y, 1_Y)$.

\section{Tameness of bicategories of symmetric sequences}
\label{sec:regbss}

The aim of this section is to prove that the bicategory $\SymV$ of symmetric sequences, defined in Section~\ref{sec:symsaf}, is tame. 
This generalizes the corresponding fact for the category of single-sorted symmetric sequences, equipped with the substitution monoidal
structure, proved in~\cite{RezkC:spaasc}. Showing that $\SymV$ is tame allows us to organize
 operads, operad bimodules and operad bimodule maps into a bicategory, called the bicategory of operad bimodules and denoted 
 by~$\OpdV$, using the bimodule construction of Section~\ref{sec:regbbb}. More precisely, we define
\begin{equation}
\label{equ:opd}
\OpdV \defeq \Bim(\SymV) \, . 
\end{equation}
In particular, for operads $(X,A)$ and $(Y,B)$, we have 
\[
\OpdV[(X,A), (Y,B)] = \SymV[X,Y]^B_A \, .
\] 

In order to prove that $\SymV$ is tame, we will show that the bicategory $\SDistV$ defined of Section~\ref{sec:sdist} is tame. 
Since, for small $\catV$-categories $\catX$ and $\catY$, we have
\[
\SDistV[\catX, \catY] = \CATV[S(\catY)^\op \otimes \catX, \catV]  \, ,
\]
the existence of reflexive coequalizers in the hom-categories of $\SDistV$ is clear. Thus, it remains to show that the
composition functors of $\SDistV$ preserve reflexive coequalizers in both variables.  We will actually show something stronger, namely that the composition functors in $\SDistV$ preserve 
sifted colimits in the first variable and are cocontinuous in the second variable, which implies that composition functors
preserve reflexive coequalizers, since reflexive coequalizers are sifted colimits.  For the convenience of the reader, we begin by recalling the notion of a sifted category and recall some basic facts about it. For further information on this notion, 
see~\cite[Chapter~3]{AdamekJ:algt}.

\begin{definition} We say that a small category $\catK$ is~\emph{sifted} 
if the colimit functor
\[
\colim_\catK\co  \Set^\catK\to \Set
\]
preserves finite products. A category $\catK$ is said to be \emph{cosifted}
if the opposite category $\catK^{\op}$ is sifted.
\end{definition}

A category $\catK$ is sifted if and only if  it is non-empty and the diagonal functor $d_\catK \co \catK \to \catK\times \catK$ is cofinal.
Dually, $\catK$ is cosifted if and only if  it is non-empty and the diagonal functor $d_\catK \co \catK\to \catK\times \catK$ is coinitial.
Let us say that a presheaf $X\co \catK^\op\to \Set$ is \emph{connected} if $\colim_\catK X=1$.
Then a category $\catK$ is cosifted if and only if  it is non-empty and
the cartesian product $\catK(-,j)\times \catK(-,k)$ of representable presheaves is connected.
A sifted (or cosifted) category is connected. In the statement of the next lemma, we write
$\Delta$ for the usual category of finite ordinals and monotone maps, and  $\Delta_{|_1}$
for  its full subcategory spanned by the objects $[0]$ and $[1]$.

\begin{lemma} The categories $\Delta$ and  $\Delta_{|_1}$ are cosifted.
\end{lemma}

\begin{proof} The colimit of a simplicial set $X\co \Delta^{\op}\to \Set$
is the set $\pi_0(X)$ of its connected components.
It is well known that the canonical map $\pi_0(X\times Y)\to \pi_0(X)\times \pi_0(Y)$
is bijective for any pair of simplicial sets $X$ and $Y$. Moreover,
$\pi_0(\Delta[0])=1$. Similarly, a presheaf $X$ on $\Delta_{|_1}$
is a reflexive graph and its colimit 
is the set $\pi_0(X)$ of its connected components.
It is easy to verify that the canonical map $\pi_0(X\times Y)\to \pi_0(X)\times \pi_0(Y)$
is bijective for any pair of reflexive graphs~$X$ and~$Y$.
Moreover, $\pi_0(\Delta[0])=1$. 
\end{proof}

\begin{remark} Since a reflexive graph in $\bcatE$ is exactly a contravariant functor $X\co \Delta_{|_1}^{\op}\to \bcatE$, reflexive coequalizers are sifted colimits.
\end{remark}
 
 We now establish some auxiliary facts which will allow us to establish that $\SDistV$ is tame. 
If  $\catK$ is a small category and $\rigR$ is a symmetric $\catV$-rig, then category $\rigR^\catK$
of $\catK$-indexed diagrams in $\rigR$ has a symmetric monoidal (closed) 
structure with the pointwise tensor product:
\[
(A\otimes B)(k) \defeq A(k)\otimes B(k) \, .
\]
The unit object for the pointwise tensor product
is the  constant diagram $cI\co \catK\to \rigR$
with value the unit object $I\in \rigR$.
If $A,B\in \rigR^\catK$
then the canonical map
\begin{equation}
\label{equ:colimmonoidal}
\colim_{(i,j)\in \catK\times \catK} A(i)\otimes B(j) \to \colim_{i\in \catK} A(i) \otimes  \colim_{j\in \catK}  B(j)
\end{equation}
is an isomorphism, since the tensor product functor of $\rigR$
is cocontinuous in each variable. If $\catK$ is sifted, then
the canonical map 
\[
\colim_{i\in \catK} A(i)\otimes B(i)\to \colim_{(i,j)\in \catK\times \catK}A(i)\otimes B(j)
\]
is an isomorphism, since the diagonal $d_\catK \co \catK\to \catK\times \catK$ is cofinal.

\begin{proposition}\label{siftedtensor1}
Let $\catK$ be a small category and $\rigR$ a symmetric $\catV$-rig.
If $\catK$ is sifted,
then $\colim_\catK\co  \rigR^\catK\to \rigR$
is a symmetric monoidal functor.
\end{proposition}

\begin{proof} We saw above that the map in~\eqref{equ:colimmonoidal}
is an isomorphism for any pair of diagrams 
$A,B \co \catK\to  \rigR$.
Moreover, the canonical map 
$\colim_{\catK} cI \to I$ is an isomorphism,
since $\catK$ is connected.
\end{proof}

We define the  $n$-fold tensor product functor $T^n\co \rigR^n \to \rigR$ by letting
\[
T^n(X_1,\ldots, X_n) \defeq X_1\otimes \ldots \otimes  X_n \, .
\]

\begin{corollary}\label{nfoldtensorproduct}
The $n$-fold tensor product functor $T^n\co \rigR^n \to \rigR$ preserves sifted colimits for every $n\geq 0$.
\end{corollary}

\begin{proof}
If $\catK$ is a sifted category,
then the colimit functor
$\colim_{\catK}\co \rigR^\catK\to \rigR$
is monoidal by Proposition~\ref{siftedtensor1}. It thus 
preserves $n$-fold tensor products, from which the claim follows.
 \end{proof}

Let $\mathbb{W}$ be a small $\catV$-category and $[\mathbb{W},\rigR]$
be the $\catV$-category of  $\catV$-functors from $\mathbb{W}$ to $\rigR$.
If $w\in \mathbb{W}$, then the evaluation functor $ev_w\co [\mathbb{W},\rigR]\to [\mathbb{W},\rigR]$
defined by letting $ev_w(F)=F(w)$ is cocontinuous.
If $\vec{w}=(w_1,\ldots, w_n) \in \mathbb{W}^n$, let us put 
\[
ev_{\vec{w}}(F)=F(w_1)\otimes \ldots \otimes F(w_n) \, .
\]

\begin{lemma}\label{tensorpower} 
The functor $ev_{\vec{w}}\co [\mathbb{W},\rigR] \to \rigR$
preserves sifted colimits
for every $\vec{w}\in \mathbb{W}^n$.
\end{lemma}

\begin{proof} The functor $ev_{\vec{w}}$ is the composite
of the functor
$\rho_{\vec{w}}\co [\mathbb{W},\rigR] \to \rigR^n$
defined by letting
$\rho_{\vec{w}}(F) \defeq (F(w_1), \ldots ,F(w_n))$
followed by the 
$n$-fold tensor product functor $T^n\co \rigR^n \to \rigR$.
The first functor is cocontinuous, while the second preserves
sifted colimits by Corollary~\ref{nfoldtensorproduct}.
\end{proof}

Recall from Section~\ref{sec:sdist} that to an $S$-distributor $F\co  \catX \sym \catY$, \ie
a distributor $F \co \catX \mat S(\catY)$, 
we associate a distributor $F^e  \co S(\objX) \mat S(\catY)$ defined as in~\eqref{equ:fprime2}.

\begin{lemma}\label{coend}  
For every pair of small $\catV$-categories  $\catX, \catY$ the functor 
\[
(-)^e \co  \DistV[\catX, S(\catY)] \to \DistV[S(\catX), S(\catY)] 
\]
 preserves sifted colimits.
\end{lemma}

\begin{proof} The claim follows if we show that the functor
\[
(-)^e \co \CATV[\catX, PS(\catY)] \to  \CAT[S(\catX), PS(\catY)]
\]
preserves sifted colimits. But this is a consequence of Lemma~\ref{tensorpower} since for a $\catV$-functor $F \co \catX \to PS(\catY)$, we have $F^e(\vec{x}) = ev_{\vec{x}}(F)$.
\end{proof}

\begin{theorem}  \label{thm:sdisttame} The bicategory $ \SDistV$ is tame. In particular, 
the horizontal composition functors of~$\SDistV$ preserve sifted colimits in the first variable and are cocontinuous in the second 
variable.
\end{theorem}

\begin{proof} The hom-categories of $\SDistV$
clearly have reflexive coequalizers. Recall from~\eqref{equ:Sdistcomp} that for
$F\in   \SDistV[\objX,\objY]$ and
$G\in   \SDistV[\objY,\objZ]$ we have
\begin{align*} 
(G\circ F)[ \vec{z}; x]  & =\int^{\vec{y} \in S(\catY)} G^e[ \vec{z}; \vec{y}] \otimes F[ \vec{y};x]  \\
 & = (G^e \circ F)[ \vec{z}; x]  \,  . \phantom{\int^{\vec{x}}} 
\end{align*} 
But the composition functors in $\DistV$ are  cocontinuous
in each variable. 
Hence, the functor $F \mapsto G\circ F$ is cocontinuous, while the  functor $G\mapsto G^e \circ F$ preserves sifted
colimits by Lemma~\ref{coend}.
\end{proof}

 \begin{corollary} \label{thm:smattame} The bicategories $\CatSymV$ and $\SymV$ are tame. 
 \end{corollary}

\begin{proof} Since a bicategory is tame if and only if its opposite is so, the claim that the bicategory $\CatSymV$ is tame follows from 
Theorem~\ref{thm:sdisttame}, since $\CatSymV$ is the opposite of $\SDistV$, which is tame by Theorem~\ref{thm:sdisttame}. The bicategory $\SymV$ is tame since it is a full sub-bicategory 
of the tame bicategory~$\CatSymV$. 
\end{proof}

Since we defined the bicategory of categorical symmetric sequences as the opposite of the bicategory of
$\SDistV$, by letting $\CatSymV \defeq \SDistV^\op$
and the bicategory of operad bimodules by letting
$\OpdV \defeq \Bim(\SMatV^\op)$, the inclusion $\SMatV^\op \subseteq \SDistV^\op$ induces an inclusion 
\begin{equation}
\label{equ:opdbim}
\OpdV \subseteq \Bim(\CatSymV) \, .
\end{equation}
As we will see in Theorem~\ref{thm:keybiequiv}, this inclusion is actually an equivalence.

\section{Analytic functors}
\label{sec:anaf}

We define the analytic functors associated to an operad bimodule. Recall that for an operad~$(X,A)$, we write $\Alg_\rigR(A)$ for its category of algebras and algebra morphisms in 
a symmetric $\catV$-ring $\rigR$. Let $(X,A)$ and $(Y, B)$ be operads.
Given an operad bimodule $F \co (X,A) \to (Y,B)$, for a symmetric $\catV$-rig $\rigR$ we define the \emph{analytic functor}
\[
\Alg_\rigR(F) \co \Alg_\rigR(A) \to \Alg_\rigR(B)
\]
associated to $F$ as follows. For an $A$-algebra~$M$, we let
\begin{equation}
\label{equ:anafunalg}
\Alg_\rigR(F)(M) \defeq F \circ_A M \, ,
\end{equation}
where $F \circ_A M$ is given  by the following reflexive coequalizer diagram
\[
\xymatrix@C=1.2cm{
F \circ A \circ M \ar@<1ex>[r]^-{\rho \circ M} \ar@<-1ex>[r]_-{F \circ \lambda} & F \circ M \ar[r] & F \circ_A M \, .}
\]
This object has a $B$-algebra structure given as in~\eqref{equ:leftBactioninduced}. This functor fits in the following commutative diagram 
\[
\xymatrix@C=1.5cm{
\Alg_\rigR(A) \ar[r]^{\Alg_\rigR(F)} & \Alg_\rigR(B) \ar[d] \\
\rigR^X \ar[u] \ar[r]_{F}  & \rigR^Y \, , }
\]
where the vertical arrows are the evident free algebra and forgetful functors, and $F \co \rigR^X \to \rigR^Y$ is the analytic functor associated to 
the symmetric sequence $F \co X \sym Y$, defined in~\eqref{equ:anasym}. 

\medskip

We now show how the \emph{restriction functor} and \emph{extension functor} (see, \eg, \cite{FresseB:modoof} for these functors
in the case of single-sorted operads) associated to an operad morphism are analytic functors. This will follow by an application of Theorem~\ref{thm:transportadjunction}.
Let us begin by recalling some definitions.  Let $\catX = (X, A)$,  $\catY = (Y, B)$
be operads, viewed as monads in the bicategory~$\SymV$. Thus,  $X \, , Y$ are sets and $A \co X \sym X$, $B \co Y \sym Y$ are symmetric sequences
equipped with multiplication and unit. Let us also fix an operad morphism~$(u,\xi) \co \objX \to \catY$, which consists of a function $u \co X \to Y$ and a monoid 
morphism~$\xi \co A \to B'$, where $B' \co X \sym X$ is the monad whose underlying symmetric sequence is defined by letting
\begin{equation}
\label{equ:bprime}
B'[ x_1, \ldots, x_n; x]  \defeq B[ u x_1, \ldots, u x_n; u x] \, .
\end{equation}
It will be useful to define the symmetric sequence $u^\circ \co X \sym Y$ and $u_\circ \co Y \to X$ by letting
\begin{equation}
\label{equ:ucirc}
u^\circ[ x_1, \ldots, x_n; y]  \defeq B[ u x_1, \ldots, u x_n; y] \, , \quad 
u_\circ[ y_1, \ldots, y_n; x]  \defeq B[ y_1, \ldots, y_n; u x]  \, .
\end{equation}
We wish to show that $u_\circ$ and $u^\circ$ form an adjunction in $\SymV$. We make some preliminary observations.
As a special case of the corresponding facts for $\catV$-functors and distributors, recalled in Section~\ref{sec:dist}, 
the function $u \co X \to Y$ determines an adjunction $(u_\bullet, u^\bullet) \co X \mat Y$ in $\MatV$.
The homomorphism $\delta \co \MatV \to \SMatV$ takes this adjunction to an adjunction $(\delta(u_\bullet), \delta(u^\bullet)) 
\co X \sym Y$ in $\SMatV$, which gives us an adjunction $(\delta(u^\bullet), \delta(u_\bullet)) \co X \sym Y$ in~$\SymV$ by duality. 
Explicitly, the symmetric sequence $\delta(u^\bullet) \co X \sym Y$ and $\delta(u_\bullet) \co Y \sym X$ are defined by
letting
\[
\delta(u^\bullet)[ \vec{x}; y]  \defeq
\left\{
\begin{array}{ll}
I & \text{if } \vec{x} = (x) \text{ and } ux = y \, , \\
0 & \text{otherwise,}
\end{array}
\right. \quad
\delta(u_\bullet)[\vec{y}; x] \defeq
\left\{
\begin{array}{ll}
I & \text{if } \vec{y} = (ux) \, , \\
0 & \text{otherwise.}
\end{array}
\right.
\]
Here, $I$ and $0$ denote the unit and the initial object of $\catV$, respectively.  
We  now state and prove a lemma which relates these symmetric sequences
with those defined in~\eqref{equ:ucirc}.

\begin{lemma} \label{thm:circbullet} There are isomorphisms 
\begin{enumerate}[(i)]
\item $u^\circ \iso B \circ \delta(u^\bullet)$,
\item $u_\circ \iso  \delta(u_\bullet) \circ B$.
\end{enumerate}
\end{lemma} 

\begin{proof} For the proof, we write $u \vec{x}$ for $(ux_1, \ldots, ux_n)$, where $\vec{x} = (x_1, \ldots, x_n)$. 
For (i), observe that by the definition of composition in $\SymV$, we have
\begin{multline*}
(B \circ \delta(u^\bullet))[ \vec{x}; y]  = \bigsqcup_{m \in \mathbb{N}} 
\int^{(y_1, \ldots, y_m) \in S^n(Y)}  \int^{\vec{x}_1 \in S(X)} \cdots \int^{\vec{x}_m \in S(X)} \delta(u^\bullet)[ \vec{x}_1; y_1]  \otimes \ldots \\[1ex]
 \ldots \otimes \delta(u^\bullet)[ \vec{x}_m; y_m] \otimes S(X)[\vec{x}, \vec{x}_1 \oplus \ldots \oplus \vec{x}_m] \otimes B[y_1, \ldots y_m; y] \, .
\end{multline*}
By the definition of $\delta(u^\bullet)$, the right-hand side is isomorphic to
\[
 \bigsqcup_{m \in \mathbb{N}} 
 \int^{x_1 \in X} \cdots \int^{x_m \in X}  S(X)[\vec{x}, (x_1, \ldots, x_m)] \otimes B[ ux_1, \ldots ux_m; y] \, .
 \]
This, in turn, is isomorphic to $B[ u \vec{x}; y]$, as required. The proof of (ii) is similar.
\end{proof}

The next lemma gives an alternative description of the monad $B' \co X \sym X$ defined in~\eqref{equ:bprime}.

\begin{lemma} \label{thm:bprimebullet}
There is an isomorphism $B' \iso \delta(u_\bullet)  \circ B \circ \delta(u^\bullet)$.
\end{lemma}

\begin{proof} By part (i) of Lemma~\ref{thm:circbullet}, it is sufficient to exhibit an isomorphism $B' \iso  \delta(u_\bullet)  \circ u^\circ$. We have
\begin{multline*}
(\delta(u_\bullet) \circ u^\circ)[ \vec{x}; x]  = \bigsqcup_{n \in \mathbb{N}} 
\int^{(y_1, \ldots, y_n) \in S^n(Y)}  \int^{\vec{x}_1 \in S(X)} \cdots \int^{\vec{x}_n \in S(X)}  u^\circ[ \vec{x}_1; y_1] \otimes \ldots \\[1ex]
 \ldots \otimes u^\circ[ \vec{x}_n; y_n]  \otimes S(X)[\vec{x}, \vec{x}_1 \oplus \ldots \oplus \vec{x}_n] \otimes \delta(u_\bullet)[ y_1, \ldots, y_n; x] \, .
\end{multline*}
By the definition of $\delta(u_\bullet)$, the left-hand side is isomorphic to
\[
\int^{\vec{x}' \in S(X)}  u^\circ[ \vec{x}'; ux]  \otimes S(X)[\vec{x}, \vec{x}'] \, ,
\]
which is isomorphic to $u^\circ[ \vec{x}; ux]$. But, by definition of $u^\circ$ and $B'$, we have $u^\circ[ \vec{x}; ux] = B'[ \vec{x}; x]$.
\end{proof}

\begin{lemma} \label{thm:adjointcirc} 
The symmetric sequence $u^\circ \co X \sym Y$ has the structure of a $(B,A)$-bimodule, the 
symmetric sequence $u_\circ \co Y \sym X$ has the structure of an $(A,B)$-bimodule and the resulting 
operad bimodules form an adjunction $(u^\circ, u_\circ) \co (X, A) \to (Y, B)$ in the bicategory $\OpdV$.
\end{lemma}

\begin{proof} First of all, observe that we have an adjunction $(\delta(u^\bullet), \delta(u_\bullet)) \co X \sym Y$ in the 
bicategory~$\SymV$. Secondly, the monoid morphism $\xi \co A \to B'$ determines, by Lemma~\ref{thm:bprimebullet}, 
a monoid morphism $\xi' \co A \to \delta(u_\bullet) \circ B \circ \delta(u^\bullet)$. By Proposition~\ref{thm:transportadjunction}, it 
follows that we have an adjunction 
\[
(B \circ \delta(u^\bullet), \delta(u_\bullet) \circ B) \co (X,A) \sym (Y,B) 
\]
in $\OpdV$. But, by Lemma~\ref{thm:circbullet}, the symmetric sequences $B \circ \delta(u^\bullet)$ and $\delta(u_\bullet) \circ B$ are
isomorphic to the symmetric sequences~$u^\circ$ and~$u_\circ$, which therefore inherit a bimodule structure so as to give us the required adjunction.
\end{proof} 

We can apply Lemma~\ref{thm:adjointcirc} to give a general version of the restriction and extension functors between categories of algebras for operads. We continue
to consider a fixed operad morphism $(u, \xi) \co (X,A) \to (Y, B)$. For a set $K$, we define the restriction functor $u^* \co  [K,X]^A \to [K,Y]^B$ as follows.
 For a left $B$-module $N \co K \sym Y$, we define the left $A$-module $u^*(N) \co K \sym X$  by letting 
\[
u^*(N)[ k_1, \ldots, k_n; x] \defeq N[ k_1, \ldots, k_n; u(x)]   \, . 
\]

\begin{theorem} For every operad morphism $(u,\xi) \co \objX \to \objY$, the functor 
\[
u^* \co  [K,X]^A \to [K,Y]^B
\] 
is
 an analytic functor and it has a left adjoint 
 \[
 u_! \co [K,Y]^B \to [K,X]^A \, ,
 \] 
 which is also an analytic functor.
\end{theorem}

\begin{proof} We show that $u^*$ is the  analytic functor associated to the operad bimodule $u_\circ \co (Y,B) \to (X,A)$. By the formula
in~\eqref{equ:anafunalg}, this
amounts to showing that we have a natural isomorphism 
\begin{equation}
\label{equ:upperstaranalytic}
u^*(N) \iso u_\circ \circ_B N \, , 
\end{equation}
for every left $B$-module $N \co K \sym Y$.
By part (ii) of Lemma~\ref{thm:circbullet} and the unit isomorphism of~$\OpdV$, we have
$u_\circ \circ_B N  \iso \delta(u_\bullet) \circ B \circ_B N \iso \delta(u_\bullet) \circ N$.
Hence, it suffices to exhibit an isomorphism $u^*(N) \iso  \delta(u_\bullet) \circ N$, which can be done using calculations
similar to those in the proofs of Lemma~\ref{thm:circbullet} and Lemma~\ref{thm:bprimebullet}. By the isomorphism 
in~\eqref{equ:upperstaranalytic} and Lemma~\ref{thm:adjointcirc}, it follows that we can define the left adjoint $u_! \co [K,Y]^B \to [K,X]^A$
as the analytic functor associated to the operad bimodule $u^\circ : (X,A) \to (Y,B)$. Explicitly, for a left $A$-module $M \co K \sym X$, we have
\[
 u_{!}(M) \defeq u^\circ \circ_A M \, .
\]
The required adjointness $u_{!} \dashv u^*$ now follows immediately from the adjointness $u^\circ \dashv u_\circ$ proved in Lemma~\ref{thm:adjointcirc}.
\end{proof}

\chapter{Cartesian closure of operad bimodules}
\label{cha:carcob}

The goal of this chapter is to prove that the bicategory of operad bimodules $\OpdV$ is cartesian closed, 
which is our second main result. The proof of this fact uses two auxiliary results. The first  is that, for a 
tame bicategory~$\bcatE$, if $\bcatE$ is cartesian closed, then so is $\bimE$. We establish this in
Section~\ref{sec:carcbb}. The second auxiliary result is that the inclusion $\OpdV \subseteq \Bim(\CatSymV)$
determined by the inclusion $\SymV \subseteq \CatSymV$ is an equivalence. We establish this in 
Section~\ref{sec:bicbem}, as a consequence of the development of some aspects of monad theory within 
tame bicategories and bicategories of bimodules, which is given in Section~\ref{sec:montrb} and
Section~\ref{sec:montbb}. In particular, we show how, for a tame
bicategory $\bcatE$, the bicategory $\bimE$ can be seen as the Eilenberg-Moore completion of $\bcatE$
as a tame bicategory, which is a special case of a general result obtained independently in~\cite{GarnerR:enrcfc}.
This is proved in Appendix~\ref{app:proof}.

\section{Cartesian closed bicategories of bimodules}
\label{sec:carcbb}

We show that if a tame bicategory $\bcatE$ is cartesian closed, then so is the bicategory $\bimE$.
We begin by considering the cartesian structure.

\begin{proposition} \label{thm:bimcart}
 Let $\bcatE$ be a tame bicategory. If $\bcatE$ is cartesian, then so
is~$\bimE$.
\end{proposition}

\begin{proof}  Let us first verify that a terminal object $\top$ in $ \bcatE$ 
remains a terminal object in~$\bimE$.  For this, we have to show that
the category~$\bcatE[\objX/A,\top]$
is equivalent to the terminal category for every object $\objX/A\in \bimE$.  
By definition,  $\bcatE[\objX/A,\top]=\bcatE[\objX,\top]_A$ is the category
of algebras of the monad $\bcatE[A,\top]$ acting on the category $\bcatE[\objX,\top]$.
But the monad $\bcatE[A,\top]$ is isomorphic to the identity monad,
since every morphism in $\bcatE[\objX,\top]$ is invertible. 
It follows that the category $\bcatE[\objX,\top]_A$ is equivalent
to the category $\bcatE[\objX,\top]$. Hence the category $\bcatE[\objX/A,\top]$
is equivalent to the terminal category.

Let us now show that the category $\Bim(\bcatE)$ 
admits binary cartesian products.
The cartesian product homomorphism $(-)\times (-)\co \bcatE\times \bcatE \to \bcatE$
takes a monad $(B_1,B_2) \co (\objY_1,\objY_2) \to (\objY_1,\objY_2)$ in~$\bcatE\times \bcatE$
to a monad $B_1\times B_2 \co \objY_1\times \objY_2 \to \objY_1\times \objY_2$  in $\bcatE$.
We will prove that 
\[
\objY_1/B_1\times \objY_2/B_2=(\objY_1\times \objY_2)/(B_1\times B_2) \, , 
\]
\ie that the product of $\objY_1/B_1$ and $\objY_2/B_2$ in 
 $\Bim(\bcatE)$ is given by $(\objY_1\times \objY_2)/(B_1\times B_2)$.
The projections $\pi_1\co \objY_1\times \objY_2\to Y_1$
and $\pi_2\co \objY_1\times \objY_2 \to \objY_2$
are components
of pseudo-natural transformations.
Hence the left hand  square in the following diagrams commute up to a canonical
isomorphism  $\sigma_1\co B_1\circ \pi_1\iso \pi_1\circ (B_1\times B_2)$ 
and  the right hand square  up to a canonical
isomorphism  $\sigma_2\co B_2\circ \pi_2 \iso \pi_2\circ (B_1\times B_2)$: 
\begin{equation}
 \label{equ:b1b2}
 {\vcenter{\hbox{\xymatrix@C=2cm{
\obj Y_1 \ar[d]_{B_1} & \objY_1\times \objY_2 \ar[r]^-{ \pi_2} \ar[l]_-{ \pi_1} \ar[d]_{B_1\times B_2} & 
\objY_2\ar[d]^{B_2} \\
 \obj Y_1  & \objY_1\times \objY_2 \ar[r]_-{\pi_2} \ar[l]^-{\pi_1} & \objY_2 \, . }}}}
 \end{equation}
 Moreover,  $(\pi_1,\sigma_1) \co  (\objY_1\times \objY_2, B_1\times B_2)\to ( \objY_1, B_1)$
and  $(\pi_2,\sigma_2) \co  (\objY_1\times \objY_2, B_1\times B_2)\to (\objY_2, B_2)$ are lax monad morphisms by Remark~\ref{thm:pseudonatmonadmorphism}.
It follows by Lemma~\ref{monadmorphtomod} 
that the morphism 
\[
\tilde\pi_1\defeq\pi_1\circ (B_1\times B_2)\co  \objY \to \objY_1
\]
has the structure of a 
$(B_1, B_1\times B_2)$-bimodule
and that the morphism 
\[
\tilde \pi_2\defeq \pi_2\circ (B_1\times B_2)\co  \objY \to \objY_2
\]
has the structure of a $(B_2, B_1\times B_2)$-bimodule.
Let us put $\obj Y \defeq \objY_1\times \objY_2$ and $B \defeq B_1\times B_2$ and show that the
 object $ \objY/B\in \Bim(\bcatE)$ equipped with the 
morphisms $\tilde \pi_1\co  \objY/B\to \objY_1/B_1$
and $\tilde \pi_2\co  \objY/B\to \objY_2/B_2$ is the cartesian product of the objects 
$\objY_1/B_1$ and $\objY_2/B_2$.
For this we have to show that
the functor 
 \begin{equation}
  \label{equ:xoveryoverb1yoverb2}
  {\vcenter{\hbox{\xymatrix{
(\tilde \pi_1,\tilde \pi_2)\circ_B(-)\co 
\bcatE[\objX,\objY]^{B_1\times B_2}_A\ar[r]& \bcatE[\objX,\objY_1]^{B_1}_A \times \bcatE[\objX,\objY_2]^{B_2}_A }}}}
\end{equation}
defined by letting $(\tilde \pi_1,\tilde \pi_2)\circ_B M=(\tilde \pi_1\circ_B M,\tilde \pi_2\circ_B M)$
is an equivalence of categories for every object $\objX/A\in \Bim(\bcatE)$. The equivalence of categories
\[
(\pi_1,\pi_2)\circ(-)=(\bcatE[\objX, \pi_1],\bcatE[\objX, \pi_2])\co  \bcatE[\objX,\objY]\to \bcatE[\objX,\objY_1]\times \bcatE[\objX,\objY_2]
\]
is the component  associated to  the triple
$(\objX, \objY_1,\objY_2)\in \bcatE^\op\times  \bcatE \times \bcatE$ of a pseudo-natural transformation.
By Remark~\ref{thm:pseudonatmonadmorphism}, it induces a lax monad morphism
 \[
\big( \bcatE[\objX, \objY_1 \times \objY_2], \bcatE[A,B_1\times B_2]\big)  \to 
\big( \bcatE[\objX, \objY_1], \bcatE[A,B_1] \big) \times \big( \bcatE[\objX,  \objY_2], \bcatE[A,B_2] \big) 
 \]
in $\Cat$. This is an equivalence in the 2-category $\Mnd(\Cat)$ 
since the functor $ (\bcatE[\objX, \pi_1],\bcatE[\objX, \pi_2])$ is an equivalence. Hence the induced functor 
 \begin{equation}
  \label{equ:yb1yb2}
  {\vcenter{\hbox{\xymatrix{
\bcatE[\objX,\objY]^{B_1\times B_2}_A\ar[r]& \bcatE[\objX,\objY_1]^{B_1}_A \times \bcatE[\objX,\objY_2]^{B_2}_A \, , 
}}}}
\end{equation}
which  takes a $M\in \bcatE[\objX,\objY]^{B}_A$ to $(\pi_1\circ M,\pi_2\circ M)$,  is an equivalence of categories. But we have
\[
\tilde \pi_1\circ_{B} M \iso \pi_1\circ B \circ_{B}\circ M
 \iso \pi_1\circ M \quad{\rm and} \quad \tilde \pi_2\circ_{B} M \iso \pi_2\circ B\circ_{B}\circ M
\iso \pi_2\circ M \, .
\]
Thus, $(\tilde \pi_1\circ_B M,\tilde \pi_2\circ_B M) \iso (\pi_1\circ M,\pi_2\circ M)$.
This shows that the functor in (\ref{equ:xoveryoverb1yoverb2}) is an equivalence of categories. 
\end{proof}

\begin{theorem} \label{thm:ccbimod}
Let $\bcatE$ be a tame bicategory. If $\bcatE$ is cartesian closed, then so is
 $\bimE$.
\end{theorem}

\begin{proof} 
The internal hom homomorphism $(-)^{(-)}\co \bcatE^\op \times \bcatE \to \bcatE$
takes a monad $(B,C)$ on the object $(\objY,\objZ)\in \bcatE^\op \times \bcatE$
to a monad $C^B$ on the object $\obj Z^\objY \in \bcatE$.
We will prove that 
\begin{equation}\label{expbimE}
(\objZ/C )^{\objY/B}=\objZ^\objY/C^B \, ,
\end{equation}
\ie that the exponential of $(\objZ/C )$ by ${\objY/B}$ 
in the bicategory $\Bim(\bcatE)$ is given by $\objZ^\objY/C^B$.
If $\ev\co \objY^\obj Z\times \obj Y\to \obj Z$ is the evaluation
then the adjunction
$$\theta\co \bcatE[\obj X,\objZ^\objY ] \to \bcatE[\objX\times \obj Y, \objZ] $$
is defined by letting
$\theta(M) \defeq \ev\circ(M\times \objY)$ for every $\obj X\in \bcatE$ and $M\co \objX\to \objZ^\objY$.
If $(A,B,C)$ is a monad on the object $(\objX,\objY,\objZ)\in \bcatE^\op\times \bcatE^\op \times \bcatE$,
then by Remark~\ref{thm:pseudonatmonadmorphism}, $\theta$ induces a lax monad morphism
\[
\big(\bcatE[\objX, \objZ^\objY], \bcatE[A,C^B] \big) \to  \big(\bcatE[\objX \times \objY, \objZ], \bcatE[A\times B, C] \big) \, , 
\]
since it is a component of a pseudo-natural transformation.
This lax monad morphism induces an equivalence between the categories of algebras
over these monads
\[
\tilde \theta\co  \bcatE[\obj X, \objZ^\objY]_A^{C^B}\to \bcatE[\objX\times \obj Y, \objZ]^C_{A\times B} 
\]
since $\theta$ is an equivalence. 
By definition, we have a commutative square
\[
\xymatrix@C=2cm{
\bcatE[\objX, \objZ^\objY]_A^{C^B} \ar[r]^{\tilde \theta} \ar[d] & \bcatE[\objX \times \objY, \objZ]^C_{A\times B} 
\ar[d] \\
\bcatE[\objX, \objZ^\objY] \ar[r]_{\theta} & \bcatE[\objX \times \objY, \objZ] \, ,}
\]
in which the vertical arrows are forgetful functors.
Note that 
\[
\bcatE[\objX, \objZ^\objY]_A^{C^B}=\Bim(\bcatE)[\objX/A,\  \objZ^\objY/C^B]
\]
and 
\[
\bcatE[\objX \times \objY, \objZ]^C_{A\times B} =\Bim(\bcatE)[\objX/A \times\! \objY/B, \ \objZ/C] \, .
\]
Let us show that the equivalence
\[
\tilde \theta\co \Bim(\bcatE)[\objX/A,\  \objZ^\objY/C^B]\to\Bim(\bcatE)[\objX/A \times\! \objY/B, \ \objZ/C]
\]
is natural in  $\objX/A\in \Bim(\bcatE)$.
But $\tilde \theta$ is natural if and only if it is of the form 
\[
\tilde\theta (M)= \widetilde{\ev} \circ_{C^B\times B} (M\times  \objY/B)
\]
for some morphism  $ \widetilde{\ev} \co  \objZ^\objY/C^B\times \objY/B \to \objZ/C$
in the bicategory $\Bim(\bcatE)$. If we apply this formula to the case $M= 1_{\objZ^\objY} =C^B$,
we obtain that
\[
\widetilde{\ev} =\tilde\theta(C^B\times B)=\ev \circ (C^B\times B) \, .
\]
Conversely, let us define $ \widetilde{\ev}\co \objZ^\objY\times\! \objY\to \objZ$
by letting 
\[
\widetilde{\ev} \defeq \ev \circ (C^B\times B) \, .
\]
Let us show that the morphism $\widetilde{\ev}$ so defined
 has the structure of a $(C, C^B\times B)$-bimodule.
Note that $C^B\times B=(C^\objY\times \objY) \circ  (\objZ^B\times B)$
and that the monad $C^\objY\times \objY$ commutes with the monad~$\objZ^B\times B$, since we have
\[
C^\objY\circ \objZ^B=C^B= \objZ^B \circ C^\objY
\]
by functoriality.
The morphism $\ev\co  \objZ^\objY \times \obj Y\to \objZ$
is a component of a pseudo-natural transformation.
By Remark~\ref{thm:pseudonatmonadmorphism} it defines a lax monad morphism
$(\ev,\alpha)\co   (\objZ^\objY \times \objY, C^\objY\times \objY)\to  (\objZ, C)$
and it follows by Lemma~\ref{monadmorphtomod} that the morphism $\ev \circ (C^\objY\times \objY)$
has the structure of a $(C, C^\objY\times \objY)$-bimodule.
Hence the morphism
\[
\widetilde{\ev} \defeq \ev \circ (C^B\times B)=   \ev \circ(C^\objY\times \objY) \circ  (\objZ^B\times B)
\]
has the structure  of a $(C, C^B\times B)$-bimodule, 
since the monad $\objZ^B\times B$ commutes with the monad~$C^\objY\times \objY$.
For every $M\in\bcatE[\objX, \objZ^\objY]_A^{C^B}$
we have
\[
\tilde \theta(M) =\ev\circ  (M\times \objY)=\ev\circ (C^B\times B)\circ_{(C^B\times B)} 
 (M\times \objY)=\widetilde{\ev}\circ_{(C^B\times B)} (M\times \objY) \, .
 \]
This shows that the equivalence  $\tilde \theta$ is natural and hence that 
$(\objZ/C)^{\objY/B}=\objZ^\objY/C^B$.
\end{proof}

Recall from Theorem~\ref{thm:smondistiscc} that the tame bicategory $\CatSymV$ is cartesian closed. Hence, by
Theorem~\ref{thm:ccbimod}, the bicategory~$\Bim(\CatSymV)$ is also cartesian closed. In order to prove that 
$\OpdV$ is cartesian closed, recall that  $\OpdV = \Bim(\SymV)$ and that have an inclusion $\SymV \subseteq \CatSymV$. Thus, we have an induced inclusion 
\begin{equation*}
\OpdV \subseteq \Bim(\CatSymV)
\end{equation*}
Thus, in order to prove that $\OpdV$ is cartesian
closed, it is sufficient to show that this inclusion is an equivalence. 
The next two, final, sections of this paper lead to a proof of this fact.

\section{Monad theory in tame bicategories}
\label{sec:montrb}

The aim of this section is to develop some aspects of the formal theory of monads in the setting of a tame bicategory. 
We begin by reviewing some notions and results  from~\cite{StreetR:fortm}, adapting them from the setting of 2-category to that of a bicategory, and then focus on the particular aspects that arise in tame bicategories.
 Let $\bcatE$ be a fixed bicategory. 
Recall from Section~\ref{sec:monmb} that, for a monad $A \co  \objX \to \objX$ in~$\bcatE$ and~$\objK \in \bcatE$,
we  write~$\bcatE[\objK,\objX]^A$ for the category of left $A$-modules with domain $\objK$ and  left $A$-module maps. The category $\bcatE[\objK,\objX]^A $ depends 
pseudo-functorially on $\objK$ and so we obtain a prestack $\bcatE[-,\objX]^A \co  \bcatE^{\op} \rightarrow \Cat$.

\begin{definition}  An \emph{Eilenberg-Moore object} for a monad  $A \co \objX \to \objX$ in $\bcatE$ is a 
representing object for the prestack $\bcatE[-,\objX]^A \co  \bcatE^{\op} \rightarrow \Cat$.
\end{definition}

Concretely, an Eilenberg-Moore object for a monad $A \co  \objX \to \objX$ consists of an object $\objX^A \in \bcatE$
and a morphism~$U \co  \objX^A \rightarrow \objX$ equipped with a left $A$-action $\lambda \co  A\circ U \to U$, 
which is universal in the following sense: the functor 
\begin{equation}
\label{equ:emchar}
 \bcatE[\objK ,\objX^A]\to  \bcatE[\objK ,\objX]^A
\end{equation}
which takes a morphism $N\co  \objK\to   \objX^A$
to the morphism $U \circ N \co  \objK \to \objX$ equipped with the left 
action~$\lambda \circ N\co  A\circ U \circ N \to U \circ N$ is an equivalence 
of categories for every object $\objK\in \bcatE$.
In particular,
for any  left $A$-module $M\co \objK\to \objX$ 
 there exists a morphism $N\co \objK\to \objX^A$ together with 
 an invertible 2-cell $\alpha\co M \to U  \circ N$ such that the following square
 commutes,
\[
\xymatrix{
 A \circ M  \ar[d]_{\lambda}  \ar[rr]^{A\circ \alpha} \ar[d] && A\circ U \circ N \ar[d]^{\lambda \circ N}  \\
M \ar[rr]_\alpha  &&  U \circ N.
 }
 \]
 In the following, we will adopt a slight abuse of language and refer to either the object $\objX^A$ or the morphism $U \co  \objX^A \to \objX$
 as the Eilenberg-Moore object for a monad $A \co  \objX \to \objX$. Note that the universal property characterizing an Eilenberg-Moore object  considered here 
is weaker than introduced in~\cite{StreetR:fortm}, even in a 2-category, since we require the functor in~\eqref{equ:emchar} to be an equivalence rather than
an isomorphism.
In particular, an Eilenberg-Moore object for a monad, as defined here, is unique up to equivalence rather than up to isomorphism as in~\cite{StreetR:fortm}.

 \begin{proposition} \label{thm:EMconservative} 
 If $U \co  \objX^A\to  \objX$ 
 is an  Eilenberg-Moore object 
 for a monad~$A \co  \objX \to \objX$,
 then the functor
$\bcatE[\objK, U] \co \bcatE[\objK,\objX^A]\to \bcatE[\objK,\objX] $
is monadic for every  $\objK\in \bcatE$.
  \end{proposition}

\begin{proof} 
We have a commutative diagram of functors
\[
\xymatrix{
\bcatE[\objK,\objX^A]\ar[rr]^{\simeq} \ar@/_1pc/[rrd]_{\bcatE[\objK, U]}  && \bcatE[\objK,\objX]^A \, , \ar[d] \\
&&  \bcatE[\objK,\objX] \, , 
}
\]
where the vertical arrow is the evident forgetful functor, which is monadic by construction.
Hence also the functor $\bcatE[\objK, U]$, since it is the composite of a monadic functor
and an equivalence of categories.
\end{proof}

The next proposition is a version of~\cite[Theorem~2]{StreetR:fortm} in the context of bicategories. We omit the proof.

 \begin{proposition} \label{thm:dualityexact} 
 Every Eilenberg-Moore object $U \co  \objX^A\to  \objX$
 for a monad $A \co  \objX \to \objX$ has a left adjoint
 $F \co  \objX \to \objX^A$
 and the monad map $\pi\co A\to U \circ F$ given by the composite
 \[
 \xymatrix@C=1cm{
 A  \ar[r]^-{A \circ \eta} & A \circ U \circ F \ar[r]^-{\lambda \circ F} &  U \circ F}
 \] 
 is invertible. \qed
 \end{proposition}

Let  $A \co \objX \to \objX$ be a monad in~$\bcatE$. 
Recall from Section~\ref{sec:monmb}  that we  write $\bcatE[\objK,\objX]_A$ for the category of right $A$-modules with domain $\objK$ and 
right~$A$-module maps between them. The category $\bcatE[\objK,\objX]_A $ depends 
pseudo-functorially 
on the object $\objK$ and so we obtain a prestack $\bcatE[ \objX, -]_A \co  \bcatE^{\op} \rightarrow \Cat$.

\begin{definition}  A \myemph{Kleisli object} for a monad $A \co \objX \to \objX$ in a bicategory $\bcatE$
is a representing object for the prestack $\bcatE[\objX,-]_A \co  \bcatE^{\op} \rightarrow \Cat$.
\end{definition}

Concretely, a Kleisli object for a monad $A \co  \objX \to \objX$ consists of an object $\objX_A \in \bcatE$ and
a morphism $F \co  \objX \rightarrow \objX_A$ equipped with a right $A$-action $\rho \co  F\circ A \to F$
which is universal in the following sense: the functor 
\[
\bcatE[F, \objK]  \co \bcatE[\objX_A ,\objK]\to  \bcatE[\objX ,\objK]_A
\]
which takes a morphism $N\co  \objX_A\to   \objK$
to the morphism $N\circ F_A \co  \objX \to \objK$ equipped with the right action
$N\circ \rho \co  N \circ F  \circ A  \to N\circ F$ is an equivalence 
of categories for every  $\objK\in \bcatE$.
In particular, for any right $A$-module $M\co \objX\to \objK$ there exists a morphism $N\co \objX_A \to \objK$ together with 
 an invertible 2-cell $\alpha\co M \to N \circ F$ such that the following square
 commutes:
\[
\xymatrix{
 M \circ A  \ar[d]_{\rho}  \ar[rr]^{\alpha \circ A} \ar[d] &&N \circ F  \circ A \ar[d]^{N \circ \rho}  \\
M \ar[rr]_\alpha  && \,  N\circ F \, .
 }
 \]
Observe that a right $A$-module $F \co  \objX \to \objX_A$ is a Kleisli object for
a monad $A \co \objX \to \objX$ in $\bcatE$ if and only if the left $A^\op$-module $F^\op \co \objX_A \to \objX$ in $\bcatE^\op$
is an Eilenberg-Moore object for the monad $A^\op \co \objX \to \objX$ in $\bcatE^\op$.

 \begin{proposition} \label{thm:KLconservative} 
 If $F \co  \objX\to  \objX_A$ 
 is a Kleisli object 
 for a monad  $A \co \objX \to \objX$, 
 then the functor
$\bcatE[F, \objK] \co \bcatE[\objX_A ,\objK]\to  \bcatE[\objX ,\objK]$
is monadic for every object $\objK\in \bcatE$.
  \end{proposition}

\begin{proof} This follows by duality from Proposition~\ref{thm:EMconservative}.
\end{proof}

 \begin{proposition} \label{thm:dualityexactKL} 
 Every Kleisli object $F \co  \objX_A \to  \objX$
 for a monad  $A \co \objX \to \objX$ has a right adjoint
 $U \co \objX\to \objX_A$
 and the monad map $\pi\co A\to U  \circ F$ given by the composite
\[
 \xymatrix@C=1cm{
 A  \ar[r]^-{\eta \circ A} &  U \circ F \circ A \ar[r]^-{U  \circ \rho} &  U \circ F}
 \] 
 is invertible.
 \end{proposition}

\begin{proof} This follows by duality from Proposition~\ref{thm:dualityexact}.
\end{proof}

The next definition introduces the notions of an opmonadic  and bimonadic adjunction. Opmonadic adjunctions should not be confused with the
comonadic adjunctions, which involve comonads rather than monads.

\begin{definition} \label{thm:monadicopmonadic}
We say that an adjunction $(F, U, \eta,\varepsilon)\co \objX \rightarrow \objY$  in $\bcatE$ is
\begin{enumerate}[(i)] 
\item  \emph{monadic} if the morphism 
$U\co \objY \to \objX$, equipped with the left action by the monad $U\circ F$, is an Eilenberg-Moore object for the monad $U\circ F \co \objX \to \objX$,
\item \emph{opmonadic} if the morphism 
$F\co \objX\to \objY$, equipped with the right action of the monad $U \circ F$, is a Kleisli object for the monad $U\circ F \co \objX \to \objX$, 
\item \emph{bimonadic} if it is both monadic and opmonadic.  
\end{enumerate}
\end{definition}

In the 2-category $\Cat$, an adjunction is monadic in the sense of Definition~\ref{thm:monadicopmonadic} if and only if it is monadic in
the usual sense~\cite[Section 3.3]{BarrM:toptt}. An adjunction $(F, U) \co \objX \rightarrow \objY$ in a bicategory
$\bcatE$ is monadic if and only if the adjunction  $(\bcatE[ \objK, F],  \bcatE[\objK, G])  \co  \bcatE[\objK,\objX] \rightarrow \bcatE[\objK, \objY]$ is monadic 
in $\Cat$ for every $\objK \in \bcatE$.
By Proposition~\ref{thm:dualityexact}, the adjunction $(F, U) \co \objX \rightarrow \objX^A$   associated to
an Eilenberg-Moore object is monadic. Dually, by Proposition~\ref{thm:dualityexactKL}, the adjunction $(F, U) \co \objX \rightarrow \objX_A$
associated to a Kleisli object  is opmonadic. Observe that an adjunction  $(F, U, \eta,\varepsilon)$
is opmonadic   in $\bcatE$ if and only if the opposite adjunction $(U^\op, F^\op, \eta,\varepsilon)$  is monadic in $\bcatE^\op$.  Consequently, the notion of a bimonadic
adjunction is self-dual, in the sense that an adjunction  $(F, U, \eta,\varepsilon)$ in $\bcatE$
is bimonadic if  and only if the opposite adjunction 
$(U^\op, F^\op, \eta,\varepsilon)$  in  $\bcatE^\op$
is bimonadic.

\begin{definition}
We will say that an adjunction $(F, U, \eta,\varepsilon)\co  \objX \rightarrow \objY$
in  $\bcatE$  is \emph{effective}  if the fork
 \begin{equation*}
\xymatrix{
F \circ U\circ F \circ U 
\ar@<0.8ex>[rr]^-{ \varepsilon\circ  F \circ U} 
\ar@<-0.8ex>[rr]_-{F\circ U\circ \varepsilon } 
&&  F \circ U \ar[r]^-{\varepsilon}  & 1_\objY 
}\end{equation*} 
 is a coequalizer diagram.  
 \end{definition}

The notion of an effective adjunction is self-dual, in the sense that an adjunction $(F, U, \eta,\varepsilon)$
  is effective in~$\bcatE$ if and only if the opposite adjunction~$(U^\op, F^\op, \eta,\varepsilon)$ is effective  in $\bcatE^\op$. Effective
 adjunctions have been studied extensively in category theory (see~\cite{KellyGM:adjwca} and
 references therein for further information).

\begin{proposition} \label{monadiciseffective}  \label{opmonadiciseffective} \hfill 
\begin{enumerate}[(i)] 
\item A monadic adjunction is effective.
\item An opmonadic adjunction is effective.
\end{enumerate}
\end{proposition}

\begin{proof}  For part (i), let us show that a monadic adjunction
$(F, U, \eta,\varepsilon)\co \objX\rightarrow \objY$ is effective.
The adjunction
\[
\xymatrix@C=1.5cm{
\bcatE[\objK,\objX] \ar@<1ex>[r]^{\bcatE[\objK, F]} \ar@{}[r]|{\bot} & \bcatE[\objK,\objY]  \ar@<1ex>[l]^{\bcatE[\objK, U]}}
\]
is  monadic for every $\objK\in \bcatE$, since  $U \co \objY \to \objX$ is 
an Eilenberg-Moore object for the monad $U\circ F \co \objX \to \objX$.
This is true in particular in the case where $\objK=\objY$.
Let us now show that the fork 
 \begin{equation}\label{fork12}
\xymatrix@C=1.2cm{
F \circ U\circ F \circ U 
\ar@<0.8ex>[r]^-{ \varepsilon\circ  F \circ U} 
\ar@<-0.8ex>[r]_-{F\circ U\circ \varepsilon } 
&  F \circ U \ar[r]^-{\varepsilon}  & 1_\objY 
}\end{equation} 
is a coequalizer diagram. But the image of the fork in~\eqref{fork12}
by the functor $U\circ (-)$ is 
a split coequalizer:
 \begin{equation*}
\xymatrix@C=1.8cm{
U\circ  F \circ U\circ F \circ U 
\ar@<0.8ex>[r]^-{ U\circ  \varepsilon\circ  F \circ U} 
\ar@<-0.8ex>[r]_-{U\circ  F\circ U\circ \varepsilon } 
&  U\circ  F \circ U  \ar@/_3pc/[l]_-{\eta\circ U \circ F\circ U}  \ar[r]^-{U\circ \varepsilon} 
& U   \ar@/_2pc/[l]_-{\eta\circ U}.
}\end{equation*} 
By Beck's monadicity theorem, which we can apply because the adjunction~$\bcatE[\objK, F] \dashv \bcatE[\objK, U]$ is monadic, 
 the fork in~\eqref{fork12} is a coequalizer diagram. Part~(ii) follows by duality.
\end{proof}

Theorem~\ref{effectiveismonadic}
below shows that in a tame bicategory $\bcatE$ also the converse of each of the implications of
Proposition~\ref{monadiciseffective} also holds and therefore monadic, opmonadic and effective adjunctions
coincide.  Its proof makes use of the Crude Monadicity Theorem, according to which
an adjunction $(F, U) \co  \objX \rightarrow \objY$ in the 2-category $\Cat$, where $\objY$ has 
reflexive coequalizers, is monadic if the functor $U$  preserves reflexive coequalizers
and is conservative~\cite[Section~3.5]{BarrM:toptt}.

\begin{theorem} \label{effectiveismonadic} For an adjunction $(F, U, \eta,\varepsilon)\co \objX \rightarrow \objY$ 
in a tame bicategory $\bcatE$, the conditions of being effective, monadic, opmonadic and
bimonadic are equivalent. 
\end{theorem}

\begin{proof} We already saw in Proposition~\ref{monadiciseffective} 
that every monadic adjunction is effective.
Conversely,  let us show that if $\bcatE$ is  tame,
then every  effective adjunction 
$(F, U, \eta,\varepsilon)\co \objX \rightarrow \objY$ in $\bcatE$ 
 is monadic.
For this, we show 
that the adjunction
\[
\xymatrix@C=1.5cm{
\bcatE[\objK,\objX] \ar@<1ex>[r]^{\bcatE[\objK, F]} \ar@{}[r]|{\bot} & \bcatE[\objK,\objY]  \ar@<1ex>[l]^{\bcatE[\objK, U]}}
\]
is  monadic for every $\objK\in \bcatE$.
We apply the Crude Monadicity Theorem.
The category~$\bcatE[\objK,\objY]$ has reflexive coequalizers 
and the functor~$\bcatE[\objK, U]\co \bcatE[\objK, \objY] \to \bcatE[\objK, \objX]$ preserves them
by the hypothesis on $\bcatE$.
Hence it remains to show that the functor~$\bcatE[\objK, U]$ is conservative.
Let $f\co M\to N$ be a 2-cell in~$\bcatE[\objK,\objY ]$ and suppose
that the 2-cell $U\circ f\co U\circ M\to U\circ N$ is invertible.
Let us show that $f$ is invertible.
Consider the following commutative diagram,
\[
\xymatrix{
F \circ U\circ F \circ U \circ M \ar[d]_{F \circ U\circ F \circ U \circ f}
\ar@<0.8ex>[rr]^-{ \varepsilon\circ  F \circ U\circ M} 
\ar@<-0.8ex>[rr]_-{F\circ U\circ \varepsilon\circ M } 
&&  F \circ U\circ M \ar[d]^{ F \circ U \circ f}  \ar[rr]^-{\varepsilon\circ M }  && M  \ar[d]^{ f}\\
F \circ U\circ F \circ U \circ N
\ar@<0.8ex>[rr]^-{ \varepsilon\circ  F \circ U\circ N} 
\ar@<-0.8ex>[rr]_-{F\circ U\circ \varepsilon\circ N } 
&&  F \circ U\circ N \ar[rr]_-{\varepsilon\circ N}  && \, N \, . }
\]
The top fork of the diagram is a coequalizer diagram, since the
 the adjunction $F\dashv U$ is effective and
the functor $\bcatE[M, \objY] \co \bcatE[\objY, \objY] \to \bcatE[\objK, \objY]$ 
preserves reflexive coequalizers.
Similarly,  the bottom fork of the diagram is a coequalizer diagram.
But the left and middle vertical cells of the diagram
are invertible, since the 2-cell  $U\circ f$ is invertible.
It follows that $f$ is invertible.
We have proved that the adjunction~$F\dashv U$ in $\bcatE$ is monadic.
It follows by duality that the adjunction~$U^\op \dashv F^\op$ in $\bcatE^\op$
is monadic if and only if it is effective. But the adjunction $U^\op \dashv F^\op$
is monadic in $\bcatE^\op$ if and only if the adjunction $F \dashv U$ is opmonadic in $\bcatE$.
Moreover, the adjunction~$U^\op \dashv F^\op$
is effective in~$\bcatE^\op$ if and only if the adjunction $F \dashv U$ is effective.
This proves that the adjunction $F\dashv U$ is opmonadic
if and only if it is effective.
\end{proof}

\begin{corollary} \label{thm:EM=eff=K} Let $\bcatE$ be a tame bicategory and $(F, U) \co \objX \rightarrow  \objY$
 be an adjunction in $\bcatE$. The following conditions are equivalent:
\begin{enumerate}[(i)]
\item the adjunction $(F, U) \co \objX \rightarrow  \objY$  is effective, 
\item the morphism $U \co  \objY \to \objX$ is an Eilenberg-Moore object for the monad $U\circ F \co \objX \to \objX$,
\item the morphism $F \co  \objX \to \objY$ is a Kleisli object for the monad $U \circ F \co \objX \to \objX$. \qed
\end{enumerate}
\end{corollary}

\begin{corollary} \label{thm:EM=K} Let $ \bcatE$ a tame bicategory.
If $U \co \objX^A\to \objX$ is an Eilenberg-Moore object for a monad $A \co \objX \to \objX$,
then its left adjoint   $F \co \objX \to \objX^A$ is a Kleisli object for $A$.
Conversely, if $F \co \objX \to \objX_A$ is a Kleisli  object for $A$,
then its right adjoint $U \co \objX_A\to \objX$ is  an Eilenberg-Moore object for $A$.
\qed
\end{corollary}

Note that Corollary~\ref{thm:EM=K} implies that Kleisli objects and Eilenberg-Moore objects coincide in a tame bicategory.
The next proposition involves the notion of a tame homomorphism between tame bicategories introduced in
Definition~\ref{def:tame}.

\begin{proposition} \label{properpreserves}  Let $\bcatE$ and $\bcatF$ be tame bicategories.  A tame homomorphism 
$\Phi\co \bcatE \to \bcatF$
preserves monadic (and hence also opmonadic, bimonadic and effective) adjunctions. It thus preserves Eilenberg-Moore 
(and hence also Kleisli) objects.
\end{proposition}

\begin{proof} Let us show that $\Phi$ preserves effective adjunctions.
If $(F, U, \eta,\varepsilon)\co \objX \to \objY$ is an effective adjunction   in  $\bcatE$,  let us show that
the adjunction $( \Phi F, \Phi U,  \Phi \eta,\Phi\varepsilon)\co \Phi \objX \to \Phi \objX$ is effective.
The fork 
 \begin{equation*}
\xymatrix{
F \circ U\circ F \circ U 
\ar@<0.8ex>[rr]^-{ \varepsilon\circ  F \circ U} 
\ar@<-0.8ex>[rr]_-{F\circ U\circ \varepsilon } 
&&  F \circ U \ar[r]^-{\varepsilon}  & 1_\objY 
}\end{equation*} 
 is a reflexive coequalizer diagram in  $\bcatE[\objY,\objY]$, 
 since the adjunction $(F, U, \eta,\varepsilon)$ is effective.
 Hence also  the fork 
 \begin{equation*}
\xymatrix{
\Phi F \circ \Phi U\circ \Phi F \circ \Phi U 
\ar@<0.8ex>[rr]^-{ \Phi \varepsilon\circ  \Phi F \circ \Phi U} 
\ar@<-0.8ex>[rr]_-{\Phi F\circ \Phi U\circ \Phi \varepsilon } 
&&  \Phi F \circ \Phi U \ar[r]^-{\Phi \varepsilon}  & 1_{ \Phi \objY} 
}\end{equation*} 
is a reflexive coequalizer, since 
the functor $\Phi_{\objX, \objY} \co  \bcatE[\objY,\objY] \to \bcatF[\Phi \objY,\Phi\objY]$ 
preserves reflexive coequalizers for every $\objY\in \bcatE$.
This proves that $\Phi$ preserves effective adjunctions.
It then follows from Theorem~\ref{effectiveismonadic} that
$\Phi$ preserves monadic, opmonadic and bimonadic adjunctions.
It thus preserves Eilenberg-Moore and Kleisli objects
by Corollary~\ref{thm:EM=eff=K}.
\end{proof}

\section{Monad theory in bicategories of bimodules} 
\label{sec:montbb}

Recall from Remark~\ref{rem:je} that  for every tame bicategory $\bcatE$ there is a homomorphism 
\[
\JE \co  \bcatE \rightarrow \bimE
\]
which maps an object~$\objX \in \bcatE$ to the identity monad  $\objX/1_\objX$
We introduce the appropriate notion of morphism between tame bicategories and make an
observation about the functorial character of the bimodule construction.

\begin{proposition} \label{thm:EMinBimpartone}
Let $\bcatE$ be a tame bicategory. The bicategory $ \bimE$ is tame
and the homomorphism 
$\JE \co  \bcatE \rightarrow \bimE$ is full and faithful and proper.
\end{proposition}

 \begin{proof} Let us show that 
$\JE \co  \bcatE \rightarrow \bimE$ 
is  fully  faithful.
A morphism $M\co \objX\to \objY$ in $\bcatE[\objX,\objY]$  has a unique structure 
of~$(1_\objY,1_\objX)$-bimodule, and every 2-cell $f\co M\to N$ in $\bcatE[\objX,\objY]$
is a map of~$(1_\objY,1_\objX)$-bimodules.
This shows that the functor 
\[
\myJ_{\objX, \objY} \co  \bcatE[\objX,\objY]\to \bimE[\objX/1_\objX,\objY/1_\objY]
\]
is  an isomorphism of categories for every pair of objects $\objX,\objY\in \bcatE$.

Let us now show that the bicategory $ \bimE$ is tame. First of all, 
recall that the  category $ \bcatE[ \objX, \objY]$ has  reflexive coequalizers by the assumption that $\bcatE$ is tame.
Moreover,  the monad $B\circ (-)\circ A\co  \bcatE[ \objX, \objY]\to  \bcatE[ \objX, \objY]$
preserves reflexive coequalizers since it is defined by composition.
By Proposition~\ref{bimoduleisbimonadic} 
 the category~$\bcatE[ \objX, \objY]^B_A$ has reflexive coequalizers
and the  forgetful functor $\bcatE[ \objX, \objY]^B_A\to \bcatE[ \objX, \objY]$
preserves and reflects reflexive coequalizers. 
Let us now show that the horizontal composition functors of $\bimE$ preserve coequalizers on the left.
For this we have to 
show that the functor $N\circ_B(-) \co  [ \objX/A, \objY/B]\to [ \objX/A, \objZ/C]$
preserves reflexive coequalizers for every  morphism
$N\co  \objY/B\to  \objZ/C$ in $\bimE$. 
The following square commutes
\[
\xymatrix{
\bcatE[ \objX, \objY]^B_A\ar[rr]^{N\circ_B(-)}   \ar[d]_-{U_1}  & & \bcatE[ \objX, \objZ]^C_A   \ar[d]^-{U_2}    \\
\bcatE[ \objX, \objY]^B  \ar[rr]_{N\circ_B(-)}   &&  \bcatE[ \objX, \objZ]  }
\]
by definition of the functor $N\circ_B(-)$, where $U_1$
and $U_2$ are the forgetful functors.
Moreover, the functors $U_1$
and $U_2$ preserve and reflect  reflexive coequalizers.
Hence, it suffices to show that the composite
\[
U_2(N\circ_B(-))\co \bcatE[ \objX, \objY]^B_A\to\bcatE[ \objX, \objY]^C_A\to \bcatE[ \objX, \objZ]
\]
preserves reflexive coequalizers.
But for this it suffices to show that the functor
\[
N\circ_B(-)\co \bcatE[ \objX, \objY]^B \to \bcatE[ \objX, \objZ]
\]
preserves reflexive coequalizers, since the square commutes
and the functor $U_1$ preserves reflexive coequalizers.
For every $M\in \bcatE[ \objX, \objY]^B$
we have a coequalizer diagram
\begin{equation}
\label{equ:relcomptensor}
\xymatrix@C=1.2cm{
N  \circ  B \circ M  \ar@<0.8ex>[r]^-{\rho  \circ M}  \ar@<-0.8ex>[r]_-{N \circ \lambda} 
&  N \circ   M \ar[r]^{q} & N\circ_B M  }
\end{equation}
in the category $\bcatE[ \objX, \objZ]$,
where $\rho_N$ is the right action of $B$ on $N$ and $\lambda_M$
is the left action of $B$ on $M$.
The functors 
\[
N\circ B\circ (-)  \co [ \objX, \objY]\to [ \objX, \objZ] \, , \quad 
 N\circ (-)  \co [ \objX, \objY]\to [ \objX, \objZ]
\]
preserve reflexive coequalizers, since the bicategory $\bcatE$ is tame.
Hence also their composite 
with the  forgetful functor $[\objX, \objY]^B\to [ \objX, \objY]$.
This shows that the functor $N\circ_B (-)$ is a colimit of functors
preserving reflexive coequalizers.
It follows that the functor $N\circ_B (-)$ 
preserves reflexive coequalizers, since colimits commute with colimits.
We have proved that  the horizontal composition functors of $\bimE$ preserve coequalizers on the left.
It follows by the duality of Remark~\ref{thm:dualitybimod} that the horizontal composition functors of $\bimE$ preserve coequalizers 
also on the right, and hence $\bcatE$ is  tame.

It remains to show that the  homomorphism 
$\JE \co  \bcatE \rightarrow \bimE$ is proper.
But $\JE$ preserves local reflexive
coequalizers, since the functor
$\myJ_{\objX, \objY} \co  \bcatE[ \objX, \objY]\to \bimE[ \objX/1_\objX, \objY/1_\objY]$
 is an equivalence of categories
for every  pair of objects $ \objX, \objY\in \bcatE$.
\end{proof}

Proposition~\ref{thm:EMinBimpartone} implies that $\JE \co  \bcatE \to \bimE$ can be regarded 
as an inclusion $\bcatE \subseteq \bimE$. Because of this, in the following we will 
 identify an object $\objX\in  \bcatE$ 
with the object $\objX/1_\objX$ of $\bimE$ and the category $ \bcatE[\objX,\objY]$  with the category $\bimE[\objX/1_\objX,\objY/1_\objX]$
 for every pair $\objX,\objY\in \bcatE$.  Our next goal is to establish that $\bimE$ is 
Eilenberg-Moore complete. We begin with  two observations about the relationship between monads 
in~$\bcatE$ and in~$\bimE$.

\begin{proposition} \label{thm:EMinBimparttwo}
Let $\bcatE$ a tame bicategory.
\begin{enumerate}[(i)]
\item An adjunction in $\bcatE$
is monadic in $\bcatE$ if and only if it is monadic  in  $\bimE $.
\item A morphism $U\co \objY\to \objX$  in $\bcatE$
is an Eilenberg-Moore object for a monad $A \co \objX \to \objX$ 
if and only if it is an Eilenberg-Moore object for  $A \co \objX/1_\objX \to \objX/1_\objX$
in  $\bimE$. 
\end{enumerate}
\end{proposition}

\begin{proof} 
Let us show that an adjunction $\bcatE$
is monadic if and only if it is monadic  in  $\bimE $.
But an adjunction in $\bcatE$
 is  effective if and only if it is effective in $\bimE $,
since the functor 
\[
\mathrm{J}_{\objX, \objY} \co  \bcatE[\objY, \objY]\to  \bimE[\objY, \objY]
\]
is an equivalence of categories (it is actually an isomorphism of categories).
The result then follows from Theorem~\ref{effectiveismonadic},
since  $\bcatE$ and $\bimE$ are tame.
Corollary~\ref{thm:EM=eff=K} implies that a morphism $U\co \objY\to \objX$ in $\bcatE$
is an Eilenberg-Moore object for a monad  $A \co \objX \to \objX$ 
if and only if it is an Eilenberg-Moore object for~$A \co \objX/1_\objX \to \objX/1_\objX$
in $\bimE$. 
\end{proof}

Let  $A \co \objX \to \objX$ be a monad in a tame bicategory $\bcatE$.
Then the morphism $A\co  \objX\to  \objX$
has the structure of a left
$A$-module $F\co  \objX \to  \objX/A$ and of a right $A$-module $U \co   \objX/A\to  \objX$.
We wish to show that we have an adjunction $F \dashv U$ in $\bimE$.
Since  $U \circ_{A} F = A\circ_A A \iso A$ and $F \circ  U  = A \circ A$, we 
define the unit $\eta \co  1_ \objX \to U \circ_{A} F$  to be $\eta_A\co 1_ \objX\to A$ and the counit $\varepsilon \co  
F \circ U \to 1_{\objX/A}$  to be $\mu_A\co A\circ A\to A$. 
To prove that we have an adjunction, we need to show that the triangular identities 
hold. This amounts to proving that
\[
(A\circ_A \mu_A)\cdot (\eta_A \circ  A)=1_A \, , \quad  
(\mu_A \circ_A A)\cdot (A\circ \eta_A)=1_A  \, .
\]
But we have $A\circ_A \mu_A =\mu_A $,
since the 2-cell $\mu_A \co  A\circ A\to A$ is a map of left $A$-modules. 
Thus,
\[
(A\circ_A \mu_A )\cdot (\eta_A \circ  A)=\mu_A \cdot (\eta_A \circ A)=1_A \, , 
\]
since $\eta_A$ is a unit for the multiplication $\mu_A$. Dually, we have
$\mu_A \circ_A A=\mu_A$, since the 2-cell $\mu_A\co A\circ A\to A$ is a map of right $A$-modules.
Thus,
\[
(\mu_A \circ_A A)\cdot (A\circ \eta )=\mu_A \cdot ( A \circ \eta_A)=1_A \, , 
\]
since $\eta_A$ is a unit for the multiplication $\mu_A$.

\begin{lemma} \label{thm:adjunctionfrommonad5}
Let $\bcatE$ be a tame bicategory.
Let  $A \co  \objX \to \objX$ be a monad in $\bimE$.
Then the adjunction $(F, U) \co  \objX  \rightarrow \objX/A$ in $\bimE$ described above is monadic
and the monad $U \circ_A F\co  \objX \to \objX$ is  isomorphic to $A \co  \objX \to \objX$. Hence,
$U \co  \objX/A \to \objX$ is an Eilenberg-Moore object for~$A$. 
\end{lemma}

\begin{proof}
Let us begin by verifying that the monad $U \circ_A F$ is  isomorphic to $A$.
Obviously, $U \circ_A F = A\circ_A A= A$.
The unit  of the monad $U \circ_A F$ is defined to be the unit $\eta$ of the 
adjunction $F \dashv U$.  But we have  $\eta=\eta_A$ by definition.
The multiplication of the monad $U \circ_A F$ is defined to be the 2-cell
$U \circ_A \varepsilon \circ_A F$.
But we have 
\[
U\circ_A \varepsilon \circ_A F=A\circ_A \mu_A \circ_A A=\mu_A \, , 
\]
since $\mu_A\co A\circ A\to A$ is a map of $(A,A)$-bimodules.
Let us now show that the adjunction  $(F, U, \eta,\varepsilon)$
is monadic. By Proposition~\ref{thm:EMinBimpartone} the bicategory $\bimE$ is tame
and therefore, by Theorem~\ref{effectiveismonadic},  it suffices to show that the adjunction
is effective. For this we have to show that the fork
 \begin{equation*}
\xymatrix{
F\circ U\circ_A F\circ U 
\ar@<0.8ex>[rr]^-{\varepsilon \circ_A  F\circ U } 
\ar@<-0.8ex>[rr]_-{ F\circ U \circ\!_A\ \varepsilon} 
&& F \circ U \ar[r]^-{\varepsilon}  &   1_{\objX/A} }
 \end{equation*}
is a coequalizer diagram in the category of $(A,A)$-bimodules.
 But this fork is isomorphic to 
 \begin{equation}
 \label{equ:fork537}
\xymatrix{
A\circ A\circ A
\ar@<0.8ex>[rr]^-{\mu_A \circ A } 
\ar@<-0.8ex>[rr]_-{A\circ \mu_A } 
&& A\circ A \ar[r]^-{\mu_A}  &   A }
 \end{equation}
since $F=U=A$, $A\circ_A A \iso A$ and $\varepsilon=\mu_A$.
Let us show that the fork in~\eqref{equ:fork537}  is a coequalizer diagram.
But its image by
the forgetful functor $U\co \bcatE[\objX,\objX]^A_A\to \bcatE[\objX,\objX]$
splits in the category~$\bcatE[\objX,\objX]$, as we have the following diagram
 \begin{equation*}
\xymatrix{
A\circ A\circ A
\ar@<0.8ex>[rr]^-{\mu_A \circ A } 
\ar@<-0.8ex>[rr]_-{A\circ \mu_A } 
&& A\circ A \ar[r]^-{\mu_A}   \ar@/_2pc/[ll]_-{\eta_A \circ A\circ A}   &   \, A \, . \ar@/_2pc/[l]_-{\eta_A \circ A} }
 \end{equation*}
This shows that the fork in~\eqref{equ:fork537} is a coequalizer diagram, since the functor $U$ 
is monadic, as observed in Proposition~\ref{bimoduleisbimonadic}.
 \end{proof}

\begin{remark} Let $\mathrm{Und}_\bcatE \co \Mnd(\bcatE) \to \bcatE $ be the homomorphism mapping a monad $(\objX, A)$ 
to its underlying object~$\objX$. For a monad $(\objX, A)$ in $\bcatE$, the bimodule $U \co  \objX/A \to \objX$ 
of Lemma~\ref{thm:adjunctionfrommonad5} can be viewed as  a morphism 
\[
U^A \co   R(\obj X,A) \to (\JE \circ \mathrm{Und}_\bcatE) (\objX,A) \, , 
\]
in $\bimE$, where $R_\bcatE \co \MndE \to \bimE$ is the homomorphism defined via Lemma~\ref{monadmorphtomod}.
The family of morphisms $U^A \co \objX/A\to  \objX$, for $(X, A) \in \MndE$, can then be seen as the components of a pseudo-natural transformation 
fitting in the diagram
\[
\xymatrix{
\Mnd(\bcatE) \ar[r]^-{\mathrm{Und}_\bcatE} \ar@/_1pc/[dr]_-{R_\bcatE}    & \bcatE \ar[d]^{\JE} \\
 & \bimE \, .}
 \]
For a monad morphism $(F,\phi)\co  (\objX,A) \to (\objY,B)$, the required
pseudo-naturality 2-cell, which should fit in the diagram
\[
\xymatrix@C=2cm{
\objX/A  \ar[r]^{U^A} \ar[d]_{R(F)}  \ar@{}[dr]|{\Downarrow \, u_F} & \objX/1_{\objX}  \ar[d]^{F}\\
\objY/B \ar[r]_{U^B}    & \objY/1_{\objY} , }
\]
is given by the following chain of isomorphisms and equalities:
\[
F \circ U^A = F\circ A \iso B\circ_B F\circ A=B\circ_B R(F) = U^B \circ_B R(F) \, . \qedhere
\]
\end{remark}

\medskip

Lemma~\ref{thm:adjunctionfrommonad5} shows that every monad in $\bcatE$ admits an Eilenberg-Moore 
object in $\bimE$. Below, we show that in fact every monad in $\bimE$ has an Eilenberg-Moore object in
$\bimE$. In order to do this, we need some preliminary observations. 
Let~$(A, \mu_A, \eta_A)$ be a monad on~$\objX \in \bcatE$. Then, for a monad 
 $(B, \mu_B, \eta_B)$ on $\objX$ and a map of monads $\pi\co A\to B$, 
the morphism $B \co  \objX \to \objX$ has the structure of an $(A,A)$-bimodule, which we 
denote by~$B_A \co  \objX/ A \to \objX/A$. 
The left action~$\lambda\co A\circ B\to B$ and the right action
 $\rho\co B \circ A\to B$ are defined by the following diagrams
  \begin{equation*}
\xymatrix{
  A \circ B \ar[r]^{\pi\circ B} \ar@/_1pc/[dr]_-{\lambda}  & B  \circ B \ar[d]_-{\mu_B} &
 \ar[l]_{B  \circ \pi } B\circ A  \ar@/^1pc/[dl]^-{\rho}  \\ 
 &B \, . & }
\end{equation*}
Furthermore, the bimodule $B_A\co  \objX/A \to \objX/A$ has the structure of a monad  in $\bimE$
with the multiplication $\mu^{B}\co B\circ_A B\to B$
 defined by the following commutative diagram, determined by the universal property of $B \circ_A B$, as follows:
\begin{equation*}
\xymatrix{
B \circ A \circ B 
\ar@<0.8ex>[r]^-{ \rho\circ B} 
\ar@<-0.8ex>[r]_-{B \circ  \lambda } 
&  B \circ B \ar[r]^{q} \ar@/_1pc/[dr]_-{\mu_B}  &B  \circ_A B \ar[d]^-{\mu^{B}} \\ 
 & & B \, .}
\end{equation*} 
The unit of the monad $B_A \co  \objX/A \to \objX/A$ is the 2-cell $\pi\co A\to B$.
This defines a functor $(-)_A\co A\backslash \Mon_\bcatE(\objX)\to  \Mon_{\bimE}(\objX/A)$.

\begin{lemma} \label{monadovermonda}
For every monad $(A, \mu_A, \eta_A)$ in a tame category $\bcatE$,
the functor 
\[
(-)_A\co A\backslash \Mon_\bcatE(\objX)\to  \Mon_{\bimE}(\objX/A)
\]
is essentially surjective.
\end{lemma}

\begin{proof}  Let $(E, \mu, \eta)$ be a monad  on $\objX/A$ in $\bimE$. Thus,
we have a morphism $E \co  \objX \to \objX$ equipped with the structure of an $(A,A)$-bimodule
together with bimodule maps $\mu \co  E\circ_A E\to E$ and $\eta \co A\to E$  satisfying the 
monad axioms. We define a monad $(B, \mu_B, \eta_B)$ on $\objX$ as follows. The morphism
$B \co  \objX \to \objX$ is given by $E \co  \objX \to \objX$ itself. 
The multiplication $\mu_B \co  B \circ B \to B$
is obtained by composing the canonical map
$q \co  E\circ E\to E\circ_A E$ with $\mu \co  E\circ_A E\to E$,
and the unit $\eta_B \co  1_{\objX}\to B$ is defined to be the composite 
of the unit $\eta_A \co  1_{\objX}\to A$ of the monad $A$
with the unit $\eta \co  A \to E$ of the monad $E$. It is easy to verify that the monad axioms are satisfied.
We can also define a monad map $\pi \co A\to B$ by letting $\pi \defeq \eta$. 
It is now immediate that~$B_A \iso E$, as required.
\end{proof}

Let $\objX \in \bcatE$. Let $A = (A, \mu_A, \eta_A)$, $B =(B, \mu_B, \eta_B)$ be monads on $\objX$.
If $\pi\co A\to B$ is a map of monads, then the morphism $B\co  \objX \to  \objX$ has the structure of 
both an $(A,B)$-bimodule  and a $(B,A)$-bimodule. These bimodules  will be denoted 
by~$U \co   \objX/B\to  \objX/A$ and $F\co   \objX/A\to   \objX/B$, respectively. We wish to show that these morphisms are adjoint.
In order to do so, let us define the unit of the adjunction $\eta \co  \Id_{\objX/A} \to U \circ_B F$ as the composite of 
$\pi \co  A \to B$ with the isomorphism $B \iso B \circ_B B$. We then define the counit  of the adjunction $\varepsilon \co  F \circ_A U
\to \Id_{\objX/B}$  as the multiplication $\mu^B \co  B \circ_A B \to B$ of the monad $B_A \co  \objX/A \to \objX/A$ defined above.
It remains to verify the triangular laws, which in this case amounts to verifying the commutativity of the following diagrams:
\[
\xymatrix@C=1.5cm{
B \ar[r]^-{B \circ_A \pi} \ar@/_1pc/[dr]_-{1_B} & B \circ_A B \circ_B  B \ar[d]^-{\mu^B \circ_B B} \\
 & B \, , } \qquad
 \xymatrix@C=1.5cm{
 B \circ_B \circ_A B \ar[d]_-{B \circ_B \mu^B} & B \ar[l]_-{\pi \circ_A B}  \ar@/^1pc/[dl]^-{1_B} \\
  B \, . & }
\]
For the diagram on the left-hand side, observe that $\mu^B \circ_B B=\mu^B$, since  
$\mu^B \co  B\circ_A B\to B$ is a map of right $B$-modules. 
Thus,
\[
(\mu^B \circ_B B) \cdot (B \circ_A \pi)=\mu^B \cdot (B\circ_A \pi)=1_B \, ,
\]
since $\pi$ is the unit of the monad $B_A \co  \objX/A \to \objX/A$. 
For the diagram on the right-hand side, dually, we have $B\circ_B \mu^B =\mu^B$,
since $\mu^B\co B\circ_A B\to B$ is a map of left $B$-modules. 
Thus,
\[
(U\circ_B  \mu^B)\cdot (\pi \circ_A U)=\mu^B \cdot (\pi \circ_A B)=1_B=1_U \, , 
\]
since $\pi$ is the unit of the monad $B_A$. We have therefore proved that $(F, U, \eta,\varepsilon) \co \objX/A \to \objY/B$
is an adjunction.

\begin{proposition} \label{EMovermonad}
The adjunction $(F, U, \eta, \varepsilon) \co \objX/A \to \objX/B$ defined above is monadic and the
monad $U \circ_B F \co  \objX/A \to \objX/A$ is isomorphic to the monad 
$B_A \co  \objX/A \to \objX/A$. Hence, the bimodule $U \co  \objX/B \to \objX/A$
is an Eilenberg-Moore object for the monad $B_A \co  \objX/A \to \objX/A$.
\end{proposition}

\begin{proof} 
Let us show that the adjunction is monadic. 
By Proposition~\ref{thm:EMinBimpartone},  the bicategory $\bimE$ is tame
and therefore, by Theorem~\ref{effectiveismonadic}, it suffices to show that the adjunction is effective.
For this we have to show that the fork
 \begin{equation}\label{fork539}
\xymatrix{
F\circ_A U\circ_B F\circ_A U 
\ar@<1.2ex>[rrr]^-{\mu^B \circ_B  F\circ_A U } 
\ar@<-1.2ex>[rrr]_-{ F\circ_A U \circ_B\mu^B } 
&&& F \circ_A U \ar[rr]^-{\mu^B }  &&   B }
 \end{equation}
is a coequalizer diagram in the category of $(B,B)$-bimodules.
But the fork in~\eqref{fork539}  is isomorphic to the fork
 \begin{equation}\label{fork540}
\xymatrix{
B\circ_A B \circ_A B 
\ar@<1.2ex>[rr]^-{\mu^B \circ_A B } 
\ar@<-1.2ex>[rr]_-{B \circ_A \mu^B } 
&&B\circ_A B \ar[rr]^-{\mu^B}  &&   B \, , }
 \end{equation}
since $F=U=B$ and $B\circ_B B=B$.
But the  image of the fork in~\eqref{fork540} under 
the forgetful functor $U\co \bcatE[\objX,\objX]^B_B\to \bcatE[\objX,\objX]$
splits in the category $\bcatE[\objX,\objX]$,
  \begin{equation*}
 \xymatrix{
B\circ_A B \circ_A B 
\ar@<1.0ex>[rr]^-{\mu^{B/A} \circ_A B } 
\ar@<-1.0ex>[rr]_-{B \circ_A \mu^{B/A} } 
&&B\circ_A B \ar[rr]^-{\mu^{B/A}}    \ar@/_3pc/[ll]_-{\pi \circ_A B \circ_A B}  &&    \ar@/_3pc/[ll]_-{\pi \circ_A B} B \, .
} 
\end{equation*}
Since the functor $U$ 
is monadic, as observed in Proposition~\ref{bimoduleisbimonadic},
this shows that the fork in~\eqref{fork540} is a coequalizer diagram.
We have proved that the adjunction $(F, U, \pi,\mu^B)$
is monadic.

Finally, let us show that the monad $U\circ_B F \co  \objX/A \to \objX/A$ is isomorphic to $B_A \co  \objX/A \to \objX/A$. 
First of all, we have
\[
U \circ_B F = B \circ_B B =  B \, .
\]
Secondly, the multiplication of $U \circ_B F$ is given by $U \circ_B \varepsilon \circ_B F \co  U \circ_B F \circ_A  U \circ_B F \to U \circ_B F$,
which is $B \circ_B \mu^B \circ_B B \co  B \circ_B B \circ_A  B \circ_B B \to  B$. But $B \circ_B \mu^B \circ_B = \mu^B$, as required.
Finally, the units of $U \circ_B F$ and $B_A$ coincide, since both are given by $\pi \co  A \to B$.
\end{proof}

\begin{definition} 
We say that a bicategory $\bcatE$ is \emph{Eilenberg-Moore complete} (\resp \emph{Kleisli complete}) 
if every monad in $ \bcatE$ admits an Eilenberg-Moore object (\resp Kleisli object).
\end{definition}

For example, the 2-category $\Cat$ is both Eilenberg-Moore complete and Kleisli complete. 
A bicategory $\bcatE$ is Eilenberg-Moore complete if and only if its opposite $\bcatE^\op$ is Kleisli complete. Since Eilenberg-Moore and Kleisli objects coincide
in a tame bicategory by Corollary~\ref{thm:EM=K}, a tame bicategory is Eilenberg-Moore complete if and only if it is Kleisli 
complete.

\begin{theorem} \label{Bimbiseffective}
For any tame category $\bcatE$, the bicategory $\bimE$ is Eilenberg-Moore complete.
\end{theorem}

\begin{proof} Let us show that every monad $E = (E, \mu, \eta)$ 
over an object $ \objX/A\in \bimE$
admits an Eilenberg-Moore object. By Lemma~\ref{monadovermonda},
we have $E \iso B_A$ for a monad $(B, \mu_B, \eta_B)$ on $\objX$ and a map of monads
$\pi\co A\to B$ in  $\Mon(\objX)$.
It then follows from Proposition~\ref{EMovermonad} that
the bimodule $U \co   \objX/B\to   \objX/A$ defined above
is an Eilenberg-Moore object for the monad $B_A \co \objX \to \objX$.
\end{proof}

\begin{proposition} \label{prop:equivbimod}
A tame bicategory $\bcatE$ is Eilenberg-Moore complete if and only if the homomorphism
$\JE \co  \bcatE \rightarrow \bimE$
is an equivalence.
\end{proposition}

\begin{proof} If $\JE$
is an equivalence, then  $\bcatE$ is Eilenberg-Moore complete because $\bimE$ is Eilenberg-Moore complete, as proved in Theorem~\ref{Bimbiseffective}.
Conversely, let us assume that $\bcatE$ is Eilenberg-Moore complete and show 
that~$\JE$ is an equivalence. Recall that~$\JE$ is full and faithful by Proposition~\ref{thm:EMinBimpartone}.
Hence it suffices to show that~$\JE$ is essentially surjective.
For this we have to show that every object  $\objX/A\in \bimE $
is equivalent to an object of~$\bcatE$.
The monad~$(\objX,A)$ admits an Eilenberg-Moore object
$U \co \objX^A\to \objX$ in $\bcatE$, since $\bcatE$ is Eilenberg-Moore complete by hypothesis.
This Eilenberg-Moore object is also an Eilenberg-Moore object in 
$\bimE$ by Proposition~\ref{thm:EMinBimparttwo}.
Hence we have an equivalence $\objX^A\simeq \objX/A$
in $\bimE$, since any two Eilenberg-Moore objects for a monad 
are equivalent.
\end{proof}

\section{Bicategories of bimodules as Eilenberg-Moore completions}
\label{sec:bicbem}

Recall from Section~\ref{sec:regbbb}  that, for tame bicategories~$\bcatE $ and~$\bcatF$, we write $\REG[\bcatE, \bcatF]$ for
the full sub-bicategory of~$\HOM[\bcatE, \bcatF]$ 
whose objects are tame homomorphisms. Clearly, the composite of two tame homomorphisms
is proper.  

\begin{definition} Given  a tame bicategory~$\bcatE$ and a tame and Eilenberg-Moore complete bicategory~$\bcatE'$, we say that a tame homomorphism~$J\co \bcatE \to \bcatE'$
exhibits~$\bcatE'$ as  the \emph{Eilenberg-Moore completion of~$\bcatE$ as a tame bicategory}
if the  homomorphism
\[
(-)\circ J\co  \REG[\bcatE', \bcatF] \to \REG[\bcatE, \bcatF]
\]
is a biequivalence for any tame and Eilenberg-Moore complete bicategory $ \bcatF$.
\end{definition}

 It follows from this definition that such an Eilenberg-Moore completion of a tame bicategory, if it exists, is unique up to biequivalence. If $\bcatE$ a tame bicategory, then
the bicategory~$\bimE$ is tame by Theorem~\ref{thm:EMinBimpartone} and Eilenberg-Moore complete by Theorem~\ref{Bimbiseffective}
and the homomorphism~$\JE\co \bcatE\to \bimE$ 
is tame by Proposition~\ref{thm:EMinBimpartone}.
The next theorem is a special case of a result obtained independently in~\cite{GarnerR:enrcfc}.

 \begin{theorem} \label{thm:completiontheorem} For a tame bicategory $\bcatE$, the homomorphism
$\JE \co  \bcatE\to \Bim(\bcatE)$ exhibits~$\bimE$ as the Ei\-len\-berg-Moore completion of $\bcatE$ as a tame bicategory.
\end{theorem}

\begin{proof} See Appendix~\ref{app:proof}.
\end{proof}

\begin{remark} As shown in~\cite{LackS:fortm}, for a 2-category $\bcatE$, not necessarily tame, it is possible to define its Eilenberg-Moore completion
$\emE$, which comes equipped with a 2-functor $\mathrm{I}_\bcatE \co  \bcatE \to \emE$ satisfying a suitable universal property. The definitions of $\emE$ 
and $\mathrm{I}_\bcatE \co  \bcatE \to \emE$ make sense also when $\bcatE$ is a bicategory, in which case $\emE$ is also a bicategory and $\mathrm{I}_\bcatE$ is a homomorphism.
We can then relate $\emE$ and $\bimE$ via a homomorphims $\Gamma \co  \emE \to \bimE$, defined below, which makes the following diagram commute:
\[
\xymatrix{
\bcatE \ar[r]^-{\mathrm{I}_\bcatE} \ar@/_1pc/[dr]_-{\JE} & \emE \ar[d]^-{\Gamma}  \\
 & \bimE \, .}
 \]
 For a bicategory $\bcatE$, the objects and the morphisms of $\emE$ are the same as those of the bicategory $\MndE$ recalled
 in Section~\ref{sec:monadmorphisms}. Given morphisms $(M, \phi), (G, \psi) \co 
 (\objX, A) \to (\objY, B)$, a 2-cell $f \co  (M,\phi) \to (M', \phi')$ in $\emE$, instead, is a 2-cell $f \co  M \to M' \circ A$ making the following diagram
 commute:
 \[
 \xymatrix{
 B \circ M \ar[r]^{\phi} \ar[d]_{B \circ f} & M \circ A \ar[r]^{f \circ A} & M' \circ A \circ A \ar[d]^{M' \circ \mu} \\
 B \circ M' \circ A \ar[r]_{\phi' \circ A} & M' \circ A  \circ A \ar[r]_{M' \circ \mu} & M' \circ A \, .}
 \]
 The homomorphism $\Gamma \co  \emE \to \bimE$ is defined exactly as the homomorphism $R \co  \MndE \to \bimE$ of 
 Section~\ref{sec:monadmorphisms} on objects and morphisms. For 
 a 2-cell $f \co  (M, \phi) \to (M', \phi')$ in $\emE$, we define $\Gamma(f) \co  M \circ A \to M' \circ A$ as the composite
 \[
 \xymatrix@C=2cm{
 M \circ A \ar[r]^-{f \circ A} &  M' \circ A  \circ A \ar[r]^-{M' \circ \mu_A} & M' \circ A}
 \]
The commutativity of the required diagrams follows easily.
\end{remark} 

\medskip

\begin{proposition} \label{thm:bimequiv}
Let $\bcatE \subseteq \bcatF$ be an inclusion of tame bicategories. 
If every object of
$\bcatF$ is an Eilenberg-Moore object (or, equivalently, a Kleisli object) for a monad in $\bcatE$, then the 
induced inclusion~$\bimE \subseteq \bimF$ is an equivalence.
\end{proposition}

\begin{proof} We have the following diagram:
\[
\xymatrix{
\bcatE \ar[r] \ar[d]_{\JE} & \bcatF \ar[d]^{\JF} \\
\bimE \ar[r] & \, \bimF \, .}
\]
We show that  $\bimE \subseteq \bimF$ is essentially surjective.
 Let $(\objY, B) \in \bimF$. By the hypothesis, 
  $\objY \in \bcatF$ is an Eilenberg-Moore object for a monad~$A \co \objX \to \objX$ in $\bcatE$. 
By  Proposition~\ref{thm:EMinBimpartone}, the homomorphism
$\JF \co  \bcatF \to \bimF$ is tame and so, by Proposition~\ref{properpreserves},
 it preserves Eilenberg-Moore objects. Hence, $(\objY, 1_\objY) \in \bimF$ is an Eilenberg-Moore object 
 for~$A \co (\objX, 1_\objX) \to (\objX, 1_\objX)$ in~$\bimF$. But also $(\objX,A) \in \bimE$ is an Eilenberg-Moore object for the same monad
by Lemma~\ref{thm:adjunctionfrommonad5} and so it is also an Eilenberg-Moore object for it  in~$\bimF$.
Therefore, there is an equivalence~$\objY \simeq (\objX, A)$ in $\bimF$ and so there is also an equivalence~$(\objY, B) \simeq  ( (\objX, A),  E)$ in~$\bimF$ for some monad $E \co (\objX, A) \to 
(\objX,A)$ in $\bimE$. By Lemma~\ref{monadovermonda}, $E$ must have
the form~$B'_A \co  (\objX, A) \to (\objX,A)$ for some monad $B' \co \objX \to \objX$ and monad map~$\pi \co  A \to B'$. 
By Proposition~\ref{EMovermonad}, an Eilenberg-Moore object for $E$ is  given by~$(\objX,B')$, which is therefore equivalent to~$(\objY, B)$.
Since~$(\objX, B') \in \bimE$, we have
the required essential surjectivity.
\end{proof}

Proposition~\ref{thm:bimequiv} implies the known fact that the inclusion~$\bimE \subseteq \Bim(\bimE)$ is an equivalence~\cite{CarboniA:axibb}. In 
particular, we have that $\DistV \subseteq \Bim(\DistV)$ is an equivalence.  We now apply Proposition~\ref{thm:bimequiv} to the  inclusion $\OpdV \subseteq \Bim(\CatSymV)$ of~\eqref{equ:opdbim}.

\begin{theorem} \label{thm:keybiequiv}
The inclusion $\OpdV \subseteq \Bim(\CatSymV)$ is an equivalence.
\end{theorem}

\begin{proof} By duality, it is sufficient to prove that the inclusion $\Bim(\SMatV) \subseteq \Bim(\SDistV)$ is an equivalence.
In order to do so, we apply Proposition~\ref{thm:bimequiv} and show that every object of $\SDistV$ is a Kleisli object for
a monad in $\SMatV$. Let $\catX$ be  a small $\catV$-category. Since monads in
$\SMatV$ are operads, we can regard $\catX$ also as a monad in $\SMatV$. So, we show
that $\catX$, viewed as an object of $\SDistV$, is a Kleisli object for $\catX$, viewed as a monad in~$\SMatV$. In order
to do so, we begin by defining a right $\catX$-module with domain $\Obj(\catX)$, 
\[
F  \co  \Obj(\catX) \sym \catX \, , 
\] 
in $\SDistV$. The $\catV$-functor $F \co  S(\catX)^\op \tensorvcat \Obj(\catX) \to \catV$ is defined by letting
\[
F[ \vec{x}; x']   \defeq 
\left\{
\begin{array}{ll} 
\catX[ x,x']   & \text{if } \vec{x} = (x) \text{ for some }  x \in \catX \, , \\
 0  & \text{otherwise.} 
\end{array}
\right.
\]
The right $\catX$-action is then defined by the composition operation of $\catX$ in the evident way. In order to show that
$F \co  \Obj(\catX) \sym \catX$ is the required Kleisli object, we need to show that, for every small~$\catV$-category~$\catK$, the functor
\[
\SDistV[\catX, \catK]  \to   \SDistV[\Obj(\catX), \catK]_\catX \, , 
\]
defined by composition with $F$, is an equivalence of categories, where
$ \SDistV[\Obj(\catX), \catK]_\catX$ denotes the category of right $\catX$-modules with 
codomain $\catK$. In order to see this, observe that to give a right $\catX$-action
on an $S$-distributor $M \co  \Obj(\catX) \sym \catK$, \ie a $\catV$-functor
$M \co  S(\catK)^\op \tensorvcat \Obj(\catX) \to \catV$, is the same thing as extending
$M$ to a $\catV$-functor  $M' \co  S(\catK)^\op \tensorvcat \catX \to \catV$.
\end{proof}

We can now prove the main result of the paper.

\begin{theorem} \label{thm:opdcc}
The bicategory $\OpdV$ is cartesian closed.
\end{theorem}

\begin{proof} Recall that the bicategory $\CatSymV$ is cartesian closed by Theorem~\ref{thm:smondistiscc} and so the associated bicategory
of bimodules~$\Bim(\CatSymV)$ 
is cartesian closed by Theorem~\ref{thm:ccbimod}.  The result follows since, as stated in Theorem~\ref{thm:keybiequiv}, $\OpdV$ is equivalent to~$\Bim(\CatSymV)$. 
\end{proof}

We conclude by illustrating the cartesian closed structure of $\OpdV$. For simplicily, let us now denote an operad by the name
of its underlying symmetric sequence and omit the mention of its underlying set of sorts. Thus, we write simply $A$ rather than 
$(X,A)$. If we denote products, exponentials and the terminal object in $\OpdV$ by
\[
A \sqcap B \, ,  \quad B^A \, , \quad \top
\]
respectively, the cartesian closed structure gives us equivalences
\begin{eqnarray*} 
\OpdV[ A, B_1 \sqcap B_2]  & \simeq  & \OpdV[  A,B_1] \times \OpdV[ A, B_2] \, , \\
\OpdV[ A \sqcap B, C ]  & \simeq &  \OpdV[ A, C^B ] \, , \\
\OpdV[ A, \top]  & \simeq & \mathrm{1} \, . 
\end{eqnarray*}
Furthermore, since the operad $\top$ has an empty set of sorts, we have that 
\[
\OpdV[\top, A] \simeq \Alg_\catV(A) \, .
\]
Therefore, as a special case of the natural isomorphisms characterizing products and exponentials, we have
equivalences
\begin{eqnarray*}
\Alg_\catV(A \sqcap B) & \simeq & \Alg_\catV(A) \sqcap \Alg_\catV(B) \, ,  \\
\Alg_\catV(B^A) & \simeq & \OpdV(A,B) \, .
\end{eqnarray*}
This shows that the algebras for $A_1 \sqcap A_2$ are pairs consisting of an $A$-algebra and a $B$-algebra, while
the algebras for $B^A$ are the $(B,A)$-bimodules. These equivalences can actually be generalized by replacing 
$\catV$ with an arbitrary $\catV$-rig $\rigR = (\rigR, \diamond, e)$, thus obtaining the equivalences mentioned in the introduction, since 
the canonical $\catV$-rig homomorphism from $\catV$ to $\rigR$ (defined by mapping $X \in \catV$ 
to~$X \otimes e \in \rigR$) preserves all the relevant structure.

\appendix

\chapter{A compendium of bicategorical definitions}
\label{sec:combd}

\begin{definition*} 
A \emph{bicategory} $\bcatE$ consists of the data in (i)-(vi) subject to the axioms in (vii)-(viii), below. 
\begin{enumerate}[(i)]
\item  A class $\Ob(\bcatE)$ of \emph{objects}. We will 
write simply $\objX \in \bcatE$ instead of $\objX \in \Ob(\bcatE)$.
\item For  every $\objX, \objY \in \bcatE$, a category $\bcatE[\objX,\objY]$ . An object of $\bcatE[\objX,\objY]$
is called a \emph{$1$-cell} or a \emph{morphism} and written $A \co  \objX \rightarrow \objY$, while a morphism $f\co  A\rightarrow B$ of $\bcatE[\objX,\objY]$ is called 
a~\emph{$2$-cell} . 
\item For every $\objX, \objY, \objZ \in \bcatE$, a functor 
\[
(-) \circ (-)  : \bcatE(\objY,\objZ) \times \bcatE(\objX,\objY) \to \bcatE(\objX,\objZ) 
\]
which associates a  morphism $B \circ A \co 
 \objX \rightarrow \objZ$ to a pair  of morphisms $A \co  \objX \rightarrow \objY$, $B \co  \objY \rightarrow \objZ$  and 
a 2-cell~$f \circ g : B \circ A \rightarrow
B' \circ A'$
 to a pair of 2-cells~$f \co  A \rightarrow A'\co  \objX \rightarrow \objY $
 and $g \co  B \rightarrow B'\co   \objY \rightarrow \objZ$.
\item For every $\objX \in \bcatE$, a morphism $1_\objX \co  \objX \rightarrow \objX$.
\item A natural isomorphism with components
\[
\alpha_{A,B,C} \co  (C \circ B) \circ A \rightarrow C \circ (B \circ A) 
\]
for $A \co  \objX \rightarrow \objY$, $B \co  \objY \rightarrow \objZ$ and $C \co  \objZ \rightarrow U$.
\item Two natural isomorphisms with components
\[
\lambda_A \co  1_\objY \circ A \rightarrow A  \, , \quad
\rho_A \co  A \circ 1_\objX \rightarrow A
\]
for  $A \co  \objX \rightarrow \objY$. 
\item For all $A \co  \objX\rightarrow \objY$, $B \co  \objY \rightarrow \objZ$, $C \co  \objZ \rightarrow U$ and $D \co U \rightarrow V$,
the following diagram commutes:
\[
\begin{xy} 
(0,20)*+{((D \circ C) \circ B) \circ A} ="1"; 
(-40,0)*+{(D \circ (C \circ B) ) \circ A}="2"; 
(-25,-20)*+{D \circ ( ( C \circ B) \circ A) }="3"; 
(25,-20)*+{D \circ (C  \circ (B \circ A))\, .}="4"; 
(40,0)*+{ (D \circ C) \circ (B \circ A)}="5";
{\ar_{\alpha \circ A} "1";"2"};
{\ar^{\alpha} "1";"5"};
{\ar^{\alpha} "5";"4"};
{\ar_{\alpha} "2";"3"};
{\ar_{D \circ \alpha} "3";"4"};
\end{xy}
\] 
\item For all $A \co  \objX \rightarrow \objY$ and $B \co  \objY \rightarrow \objZ$, the following diagram commutes:
\[
\begin{xy} 
(-25,10)*+{(B \circ 1_\objY) \circ A}="1"; 
(25,10)*+{B \circ (1_\objY \circ A) }="2"; 
(0,-10)*+{B \circ A \, .}="3"
{\ar^{\alpha}"1";"2"}
{\ar_{\rho \circ A} "1";"3"}
{\ar^{B \circ \lambda} "2";"3"}
\end{xy}
\]
\end{enumerate}
\end{definition*}

\begin{definition*} 
A \emph{homomorphism}  $\Phi : \bcatE \rightarrow \bcatF$ 
of bicategories consists of the data in (i)-(iv) subject to axioms (v)-(vi) below.
\begin{enumerate}[(i)]
\item A function $\Phi : \Ob(\bcatE) \rightarrow \Ob(\bcatF)$.
\item For every $\objX, \objY \in \bcatE$,  a functor 
\[
\Phi_{\objX,\objY} : \bcatE(\objX,\objY) \rightarrow \bcatF(
\Phi \objX, \Phi \objY) \, .
\] 
Below, we write $\Phi$ instead of $\Phi_{\objX,\objY}$.
\item A natural isomorphism with components
\[
\phi_{B,A} \co  \Phi(B \circ A) \rightarrow \Phi(B) \circ \Phi(A) \, , 
\]
for $A \co  \objX \rightarrow \objY$ and $B \co  \objY \rightarrow \objZ$.
\item For each $\objX \in \bcatE$, an isomorphism $\iota_\objX \co  \Phi(1_\objX) \rightarrow 1_{\Phi \objX}$.
\item For every $A \co  \objX \rightarrow \objY$, $B \co  \objY \rightarrow \objZ$
and $C \co  \objZ \rightarrow \objU$, the following diagram commutes:
\[
\begin{xy} 
(0,20)*+{\Phi( ( C \circ B) \circ A)} ="1"; 
(-40,0)*+{\Phi( C \circ B) \circ \Phi(A)}="2"; 
(-40,-20)*+{(\Phi(C) \circ \Phi(B)  ) \circ \Phi(A)}="3"; 
(40,-20)*+{ \Phi(C) \circ \Phi(B \circ A)  }="4"; 
(40,0)*+{\Phi( C \circ  ( B \circ A) )}="5";
(0,-40)*+{  \Phi(C) \circ ( \Phi(B) \circ \Phi(A) ) \, . }="6";
{\ar_{\phi} "1";"2"};
{\ar^{\Phi(\alpha)} "1";"5"};
{\ar^{ \phi} "5";"4"};
{\ar_{\phi \circ \Phi(A)} "2";"3"};
{\ar_{\alpha} "3";"6"};
{\ar^{\Phi(C) \circ \phi} "4";"6"};
\end{xy}
\] 
\item For every $A \co  \objX \rightarrow \objY$, the following diagrams commute:
\[
\xymatrix{
\Phi(1_\objY \circ A) \ar[r]^-{\phi} \ar@/_2pc/[ddr]_-{\Phi(\lambda)} & \Phi(1_\objY) \circ \Phi(A) \ar[d]^{\iota \circ \Phi(A)} \\
 & 1_{\Phi(\objY)} \circ \Phi(A) \ar[d]^{\lambda} \\
 & \Phi(A)  \, . }  \hspace{1.2cm}
\xymatrix{
\Phi(A) \circ \Phi(1_\objX) \ar[d]_{\Phi(A) \circ \iota} & \Phi(A \circ 1_\objX) \ar[l]_-{\phi} \ar@/^2pc/[ddl]^-{\Phi(\rho)} \\
\Phi(A) \circ 1_{\Phi(\objX)} \ar[d]_{\rho} & \\
\Phi(A) \, . }
 \]
 \end{enumerate}
 \end{definition*}

\begin{definition*} 
A {\em pseudo-natural transformation} 
$P \co  \Phi \rightarrow \Psi$ between two homomorphism $\bcatE \rightarrow \bcatF$
consists of the data in (i)-(ii) subject to the axioms (iii)-(iv) below.
\begin{enumerate}[(i)]
\item For each  $\objX \in \bcatE$, a morphism $P(\objX) \co  \Phi\objX \rightarrow \Psi\objX$ 
\item For each morphism  $A \co  \objX \rightarrow \objY$, an invertible 2-cell
\[
\xymatrix@C=1.5cm@R=1.5cm{
\Phi \objX \ar[r]^{P({\objX })} \ar[d]_{\Phi(A)}  \ar@{}[dr]|{\Downarrow \, p_A} &  \Psi\objX \ar[d]^{\Psi(A)}   \\
\Phi \objY \ar[r]_{P({\objY })}  & \Psi Y \, .}
\]
\item  For  every $A \co  \objX \rightarrow \objY$ and $B \co  \objY \rightarrow \objZ$, we have
\[
{\vcenter{\hbox{\xymatrix@C=1.5cm@R=1.5cm{
\Phi \objX  \ar[r]^{P(\objX)}  \ar[d]_{\Phi A}  \ar@{}[dr]|{\Downarrow \, p_A }& \Psi \objX \ar[d]^{\Psi A} \\
\Phi \objY \ar[r]^{P(\objY)}  \ar[d]_{\Phi B}  \ar@{}[dr]|{\Downarrow \, p_B} & \Psi \objY \ar[d]^{\Psi B}  \\
\Phi \objZ \ar[r]_{P(\obj Z)} & \Psi \objZ }}}}  \quad = \quad
{\vcenter{\hbox{\xymatrix@C=1.5cm@R=1.5cm{
\Phi \objX  \ar[r]^{P(\objX)} \ar[dd]_{\Phi (B \circ A)}  \ar@{}[ddr]|{\Downarrow \, p_{B \circ A}} & \Psi \objX \ar[dd]^{\Psi(B \circ A)} \\
& \\
\Phi \objZ \ar[r]_{P(\obj Z)} & \, \Psi \objZ \, .}}}}
\]
 \item For every $\objX \in \bcatE$, 
\[
{\vcenter{\hbox{\xymatrix@C=1.5cm@R=1.5cm{
\Phi \objX \ar[r]^{P({\objX })} \ar[d]_{\Phi(1_X)} \ar@{}[dr]|{\Downarrow \, p_{1_\objX} } &  \Psi\objX \ar[d]^{\Psi(1_{\objX})}     \\
\Phi \objX \ar[r]_{P({\objX})} & \Psi\objX}}}}
\quad = \quad 
{\vcenter{\hbox{\xymatrix@C=1.5cm@R=1.5cm{
\Phi \objX \ar[r]^{P({\objX })} \ar[d]_{1_{\Phi \obj X}}  &   \Psi\objX \ar[d]^{1_{\Psi \objX}}  \\
\Phi \objX \ar[r]_{P(\objX ) }   & \Psi\objX  \, .}}}}
\]
\end{enumerate}
\end{definition*}

\begin{definition*} 
Let $P, Q : \Phi \rightarrow \Psi$ be pseudo-natural transformations. A {\em modification} 
$\sigma : P \rightarrow Q$
consists of a family of 2-cells $\sigma_\objX : P(\objX) \rightarrow Q(\objX)$ such that
\[
\begin{xy}
(-12,8)*+{\Phi \objX}="4";
(-12,-12)*+{\Phi \objY}="3";
(12,8)*+{\Psi \objX}="2";
(12,-12)*+{\Psi \objY}="1"; 
(0,-4)="5";
(0,-9)="6";
(0,10)="7";
(0,6)="8";
{\ar@/^1pc/^{P(\objX)} "4";"2"};
{\ar@/_1pc/_{Q(\objX)} "4";"2"};
{\ar@/_1pc/_{Q(\objY)} "3";"1"};
{\ar_{\Phi(A)} "4";"3"};
{\ar^{\Psi(A)} "2";"1"};
{\ar@{=>}^{\ \scriptstyle{q_\objY}} "5";"6"};
{\ar@{=>}^{\; \scriptstyle{\sigma_\objX}} "7";"8"};
\end{xy}
\qquad = \qquad  
\begin{xy}
(-12,8)*+{\Phi \objX}="4";
(-12,-12)*+{\Psi \objX }="3";
(12,8)*+{\Psi \objX}="2";
(12,-12)*+{\Psi \objY \, .}="1"; 
(0,6)="5";
(0,0)="6";
(0,-10)="7";
(0,-14)="8";
{\ar@/^1pc/^{P(\objX)} "4";"2"};
{\ar@/^1pc/^{P(\objY)} "3";"1"};
{\ar@/_1pc/_{Q(\objY)} "3";"1"};
{\ar_{\Phi A} "4";"3"};
{\ar^{\Psi A} "2";"1"};
{\ar@{=>}^{\ \scriptstyle{p_A}} "5";"6"};
{\ar@{=>}^{\; \scriptstyle{\sigma_\objY}} "7";"8"};
\end{xy}
\]

\end{definition*}

\chapter{A technical proof}
\label{app:proof}

\section{Preliminaries} 

In order to prove Theorem~\ref{thm:completiontheorem} we need to recall some facts about the the
formal theory of monads. Let us consider a fixed bicategory $\bcatE$, which
for the moment we do not need to suppose being tame. The inclusion  homomorphism 
\[
\mathrm{Inc}_\bcatE \co  \bcatE\to \Mnd(\bcatE)
\]
takes an object~$\objX $ to the pair~$(\objX, 1_\objX)$,
where~$1_\objX$ is the identity monad on~$\objX$.

\medskip

As proved in~\cite{StreetR:fortm} in the context of 2-categories,
the bicategory $\bcatE$ admits Eilenberg-Moore objects if and only if 
the inclusion homomorphism~$\mathrm{Inc} \co  \bcatE\to \Mnd(\bcatE)$
has a right adjoint 
\[
\EM\co \Mnd(\bcatE)\to \bcatE \, .
\]
It will be useful to give an explicit description of it.
The homomorphism $\EM \co  \Mnd(\bcatE)\to \bcatE$
 takes  a monad $(\objX,A)$ to
the Eilenberg-Moore object $\objX^{A}$
and the counit of the adjunction is the monad morphism 
$(U^A,\lambda^A)\co(\objX^{A}, 1_{\objX^{A}})\to (\objX,A)$,
where  $\lambda^A \co A \circ U^A \to U^A$ is the left action of $A$ on 
 $U^A \co\objX^{A} \to \objX$ that is part of the structure of an Eilenberg-Moore object.
The image by $\EM$ of a lax  monad morphism
$(M,\phi)\co(\objX,A) \to (\objY,B)$ is constructed as follows:
the 2-cell
\[
\xymatrix@C=2cm{
B\circ M\circ U^A \ar[r]^{\phi \circ U^A} & 
M\circ A\circ U^A \ar[r]^{M\circ \lambda^A} &
M\circ  U^A}
\]
is  a left action of the monad $B$ on $M\circ U^A$. 
There is then a morphism $M' \co\objX^{A}\to 
\objY^{B}$ together with an isomorphism of left $B$-modules 
\[
\xymatrix@C=1.5cm{
\objX^A \ar[r]^{U^A} \ar[d]_{M'} 
\ar@{}[dr]|{\quad \Downarrow \, \sigma_M} & \objX \ar[d]^{M}  \\
\objY^{B} \ar[r]_{U^B}  &  \, \objY \, .}
\]
By definition, $\EM(M,\phi) \defeq M'$.
Finally, let us describe the image by $\EM$ of a
monad $2$-cell  $\alpha\co(M_1, \phi_1) \rightarrow (M_2, \phi_2)$.
By construction, $\EM( \alpha) $
is the unique 2-cell $\alpha' \co M'_1  \to M'_2$ such that the following 
cylinder commutes:
\[
\begin{xy}
(-12,8)*+{\objX^A}="4";
(-12,-12)*+{\objX}="3";
(12,8)*+{\objY^B}="2";
(12,-12)*+{\objY}="1"; 
(0,-4)="5";
(0,-9)="6";
(0,10)="7";
(0,6)="8";
(0,-10)="9";
(0,-14)="10";
{\ar@/^1pc/^{M'_1} "4";"2"};
{\ar@/_1pc/_{M'_2} "4";"2"};
{\ar@/^1pc/^{M_1} "3";"1"};
{\ar@/_1pc/_{M_2} "3";"1"};
{\ar_{U^A} "4";"3"};
{\ar^{U^B} "2";"1"};
{\ar@{=>}^{\; \scriptstyle{\alpha'}} "7";"8"};
{\ar@{=>}^{\; \scriptstyle{\alpha}} "9";"10"};
\end{xy}
\]

\section{The proof} 

Let us now recall the statement of Theorem~\ref{thm:completiontheorem} and give its proof.

\begin{theorem} 
\label{thm:appcompletiontheorem} If $ \bcatE$ is a tame bicategory,
then the  homomorphism 
\[
\JE \co \bcatE\to \bimE
\]
exhibits $\bimE$ as an Eilenberg-Moore completion of $\bcatE$ as a tame bicategory.
\end{theorem}

\begin{proof} We will prove that the homomorphism
\begin{equation}
\label{equ:compwithj}
(-)\circ \JE \co \mathrm{HOM}^p \big( \bimE, \bcatF) \big) \to 
\mathrm{HOM}^p \big(\bcatE , \bcatF \big)
\end{equation}
is a surjective equivalence for any Eilenberg-Moore complete tame bicategory $ \bcatF$.

\medskip

Let us first show that the homomorphism is surjective on objects.
The  homomorphism $\JF \co \bcatF \to  \bimF$
is an equivalence by Proposition~\ref{prop:equivbimod}, since the bicategory $\bcatF $ is Eilenberg-Moore complete.
Hence there is a homomorphism $\Theta\co \Bim(\bcatF)\to \bcatF$
together with a pseudo-natural transformation 
\[
E\co \Id_\bcatF\to \Theta \circ \JF \, , 
\]
whose components are equivalences in $\bcatF$.
Let us show that the pair $(\Theta,E)$ can be chosen so that the pseudo-natural transformation $E$
is the identity. The homomorphism $\Theta$ is constructed 
by first choosing
an Eilenberg-Moore object $U^A\co\objX^A\to \objX$ in $\bcatF$
for each object $\objX/A\in \Bim(\bcatF)$ and letting $\Theta(\objX/A)=\objX^A$, 
and then by choosing 
for each morphism
 $M\co \objX/A\to \objY/B$
in $\Bim(\bcatF)$, a morphism $\Theta(M)\co\objX^A \to \objY^B$
in $\bcatF$
together with an isomorphism of left $B$-modules $\sigma_M \co  R^B\circ \Theta(M) \to M\circ_A R^A$
\[
\xymatrix{
\objX^{A}\ar[d]_{U^A}\ar[r]^{\Theta(M)} \ar@{}[dr]|{\Downarrow \, \sigma_M} & \objY^{B}  \ar[d]^{U^B} \\
\objX  \ar[r]_{M} & \, \objY \, .}
\]
The value of $\Theta$ on 2-cells is determined by these choices afterward.
More precisely, if $\alpha\co M\to N$
is a 2-cell, then $\Theta(\alpha)\co \Theta(M)\to \Theta(N)$ is the unique 2-cell
such that the following equality of pasting diagrams holds:
\[
\begin{xy}
(-12,8)*+{\objX^A}="4";
(-12,-12)*+{\objX}="3";
(12,8)*+{\objY^B}="2";
(12,-12)*+{\objY}="1"; 
(0,-4)="5";
(0,-9)="6";
(0,10)="7";
(0,6)="8";
{\ar@/^1pc/^{\Theta(M)} "4";"2"};
{\ar@/_1pc/_{\Theta(N)} "4";"2"};
{\ar@/_1pc/_{N} "3";"1"};
{\ar_{R^A} "4";"3"};
{\ar^{R^B} "2";"1"};
{\ar@{=>}^{\ \scriptstyle{\sigma_N}} "5";"6"};
{\ar@{=>}^{\; \scriptstyle{\Theta(\alpha)}} "7";"8"};
\end{xy}
\qquad = \qquad  
\begin{xy}
(-12,8)*+{\objX^A}="4";
(-12,-12)*+{\objX }="3";
(12,8)*+{\objY^B}="2";
(12,-12)*+{\objY}="1"; 
(0,6)="5";
(0,0)="6";
(0,-10)="7";
(0,-14)="8";
{\ar@/^1pc/^{\Theta(M)} "4";"2"};
{\ar@/^1pc/^{M} "3";"1"};
{\ar@/_1pc/_{N} "3";"1"};
{\ar_{R^A} "4";"3"};
{\ar^{R^B} "2";"1"};
{\ar@{=>}^{\ \scriptstyle{\sigma_M}} "5";"6"};
{\ar@{=>}^{\; \scriptstyle{\alpha}} "7";"8"};
\end{xy}
\]
If $A=1_\objX$, we can choose $\objX^A=\objX$ and $U^A=1_\objX$.
And if  $M\co \objX/1_\objX\to \objY/1_\objY$ 
we can choose $\Theta(M)=M$ and $\sigma_M$ to be the identity 2-cell.
It follows from these choices that we have $\Theta(\alpha)=\alpha$ for every
 $\alpha\co M \to N$.
Thus, $ \Theta \circ \JF= \Id_\bcatF$.

\medskip

Let us now show that every tame homomorphism $\Phi\co\bcatE \to  \bcatF$
can be extended as a tame homomorphism  $\Phi'\co \Bim(\bcatE)  \to  \bcatF$.
As observed in~Remark~\ref{thm:bimfunctorial}, the homomorphism $\Phi\co\bcatE \to  \bcatF$
has a natural extension $\Bim(\Phi)\co \Bim(\bcatE) \to  \Bim(\bcatF)$.
 Then, the following square commutes by construction
\[
\xymatrix{
\bcatE  \ar[rr]^\Phi \ar[d]_{\JE}&& \bcatF\ar[d]^{\JF } \\
\Bim(\bcatE)  \ar[rr]_{\Bim(\Phi)} && \Bim(\bcatF) \, .
}
\]
It is easy to verify that the homomorphism $\Bim(\Phi)$ is proper.
Let us put $\Phi'=\Theta \circ \Bim(\Phi)$.
Then we have
\[
\Phi' \circ \JE =\Theta \circ \Bim(\Phi)\circ \JE=
\Theta \circ \JE \circ \Phi= \Id_\bcatF\circ \Phi=\Phi \, .
\]
We have proved that the homomorphism in~\eqref{equ:compwithj} 
is surjective on objects. 
For any tame homomorphism $\Phi\co\Bim(\bcatE) \to  \bcatF$, let us define
$\Phi_{\mid \bcatE} \defeq \Phi \circ \JE $. 
For a pair of tame  homomorphisms  $\Phi,\Psi\co\Bim(\bcatE) \to  \bcatF$, 
we can then define the restriction functor
\[
\Res(\Phi,\Psi)\co[\Phi,\Psi]\to [\Phi_{\mid   \bcatE}, \Psi_{\mid  \bcatE}]
\]
in the evident way.  It remains to show that the functor $\Res(\Phi,\Psi)$ is an equivalence of categories
surjective on objects.
We will prove that the functor $\Res(\Phi,\Psi)$ is surjective on objects
in Lemma~\ref{restriction1-cell}, and that it is full and faithful in Lemma~\ref{restriction2-cell}.
\end{proof}

We need to prove a few intermediate results.
We first recall the bicategory of 1-cells $ \bcatE^{(1)}$  of a bicategory $ \bcatE$.
The bicategory $ \bcatE^{(1)}$ is equipped with a universal pseudo-natural transformation $U\co s_0\to t_0$
between two homomorphisms 
\[
\xymatrix{ 
\bcatE^{(1)}\ar@<1.0ex>[rr]^-{s_0} \ar@<-1.0ex>[rr]_-{t_0}  &&  \bcatE.
}
\]
The universality of $U$ means that for any bicategory $ \bcatF$
and any pseudo-natural transformation $M\co \Phi\to \Psi$ between a  pair of homomorphisms 
$\Phi,\Psi\co   \bcatF \to  \bcatE$, there exists a \emph{unique} homomorphism
$\Theta\co  \bcatF \to \bcatE^{(1)}$ such that $s_0\circ \Theta=\Phi$, $t_0\circ \Theta=\Psi$ and $U\Theta=M$.
By construction, an  object of $ \bcatE^{(1)}$ is a 1-cell $M\co \objX_0\to \objX_1$
in the bicategory $ \bcatE$.  A 1-cell $M\to N$ of $ \bcatE^{(1)}$ 
is a pseudo-commutative square
\[
\xymatrix{
\objX_0\ar[d]_{M}\ar[r]^{E_0} \ar@{}[dr]|{\stackrel{\alpha}{\rightarrow}} & \objY_0 \ar[d]^{N} \\
\objX_1 \ar[r]_{E_1}& \objY_1
}
\]
defined by a triple $(E_0,E_1,\alpha)$,  where $\alpha \co  E_1\circ M \to  N\circ E_0$ is
an isomorphism. 
Composition of 1-cells is defined by pasting squares.
If $(F_0,F_1,\beta)\co  M\to N$ is another pseudo-commutative square
then a  2-cell $(E_0,E_1,\alpha)\to (F_0, F_1,\beta)$ in $ \bcatE^{(1)}$
is a cylinder
\[
\xymatrix{
\objX_0\ar[dd]_{M} \ar@/^1pc/[rr]^-{E_0} 
\ar@{}[rr]|{\Downarrow \, \gamma_0}  \ar@/_1pc/[rr]_-{F_0} && \objY_0 \ar[dd]^{N} \\
&&\\
\objX_1 \ar@/^1pc/[rr]^-{E_1}
\ar@{}[rr]|{\Downarrow \, \gamma_1}  \ar@/_1pc/[rr]_-{F_1}  && \objY_1.
}
\]
defined by a pair of 2-cells $\gamma_0\co E_0\to F_0$,  $\gamma_1\co E_1\to F_1$
such that the following square commutes,
\[
\xymatrix@C=1.5cm{
E_1\circ M  \ar[r]^{\gamma_1 \circ M} \ar[d]_{\alpha}  &\ar[d]^{\beta}  F_1\circ M \\
N \circ E_0 \ar[r]_{N\circ \gamma_0 }&  N\circ F_0  }
\]
Vertical composition of 2-cells in $\bcatE^{(1)}$ is defined  component-wise.
Observe that the source and target
homomorphisms
\[
\xymatrix{ 
\bcatE^{(1)}\ar@<0.8ex>[rr]^-{s_0 } \ar@<-0.8ex>[rr]_-{t_0}  &&  \bcatE
}
\]
are preserving composition strictly. The pseudo-natural stansformation $U\co s_0\to t_0$
takes the object $M\co \objX_0\to \objX_1$ of $ \bcatE^{(1)}$ to the 1-cell 
$M\co s_0(M)\to t_0(M)$ of $ \bcatE$, and it takes the 1-cell $(E_0,E_1,\alpha)\co M\to N$
to the isomorphism $\alpha\co E_1\circ M \to  N\circ E_0$.

\begin{lemma}\label{tameit1cell} 
If $\bcatE$ is tame, then so is its bicategory of $1$-cells $\bcatE^{(1)}$.
\end{lemma}

\begin{proof} If $M\co \objX_0\to \objX_1$ and $N\co \objY_0\to \objY_1$,
then the bicategory $ \bcatE^{(1)}[M,N]$ is defined by a pseudo-pullback square
\[
\xymatrix{
 \bcatE^{(1)}[M,N]\ar[d]\ar[rr]&&  \bcatE[\objX_0, \objY_0] \ar[d]^{N\circ (-)} \\
 \bcatE[ \objX_1, \objY_1] \ar[rr]_{(-)\circ M} &&  \bcatE[\objX_0,\objY_1] }
\]
It follows that the bicategory $ \bcatE^{(1)}[M,N]$ admits reflexive coequalizers, since the functors $N\circ (-)$
and $(-)\circ M$ are preserving reflexive coequalizers. 
Moreover, the functor
\[
(s,t)\co  \bcatE^{(1)}[M,N]\to  \bcatE[\objX_0, \objY_0] \times
 \bcatE[\objX_1, \objY_1] 
 \]
preserves and reflects reflexive coequalizers.
It follows that the composition functor
\[
(-) \circ (-)  \co  \bcatE^{(1)}[N,P]\times  \bcatE^{(1)}[M,N]\to  \bcatE^{(1)}[M,P]
\]
preserves reflexive coequalizers in each variable.
Hence the bicategory $ \bcatE^{(1)}$ is tame.
\end{proof}

A monad on an object  $M\co \objX_0\to \objX_1$ of the
 bicategory $ \bcatE^{(1)}$ is a 6-tuple 
 \[
 A=(A_0,\mu_0,\eta_0, A_1, \mu_1,\eta_1,\alpha) \, , 
 \]
where $A_0=(A_0,\mu_0,\eta_0)$ is a monad on $\objX_0$, $A_1=(A_1,\mu_1,\eta_1)$ is a monad on $\objX_1$
and $\alpha \co  A_1\circ M \to   M\circ A_0$ is an isomorphism
satisfying the coherence conditions expressed by the  diagrams:
\[
\xymatrix@C=1.2cm{
A_1\circ A_1\circ M\ar[d]_-{\mu_1\circ M}\ar[r]^-{A_1\circ \alpha}&A_1\circ M\circ A_0 \ar[r]^-{\alpha\circ A_0} & M\circ  A_0\circ A_0\ar[d]^-{M\circ \mu_0} \\
A_1\circ M \ar[rr]_\alpha&& M\circ A_0
}\quad\quad 
\xymatrix@C=1.2cm{
M \ar[r]^-{\eta_1\circ M} \ar@/_1pc/[dr]_-{M\circ \eta_0} & A_1\circ M \ar[d]^-{\alpha}\\
  & M\circ A_0}.
\]
It follows that we have a morphism $(M,\alpha) \co (\objX_0, A_0)\to (\objX_1, A_1)$
in the bicategory~$\Mnd(\bcatE)$, defined in Section~\ref{sec:monadmorphisms}.

\begin{lemma} If $\bcatE$ is Eilenberg-Moore complete, then so is its 
bicategory of $1$-cells $\bcatE^{(1)}$.
\end{lemma}

\begin{proof} Let us show that every monad
in $ \bcatE^{(1)}$ admits an Eilenberg-Moore object.
We will use the homomorphism $\EM\co  \Mnd(\bcatE)\to \bcatE$.
If $A=(A_0,\mu_0,\eta_0, A_1, \mu_1,\eta_1,\alpha)$
is a monad on an object  $M\co \objX_0\to \objX_1$ of the
 bicategory $ \bcatE^{(1)}$,
 then  $(M,\alpha) \co (\objX_0, A_0)\to (\objX_1, A_1)$ is a morphism
 in~$\Mnd(\bcatE)$.
 Let $U_0\co \objX_0^{A_0}\to \objX_0$
 be an Eilenberg-Moore object for $A_0$
and $U_1\co \objX_1^{A_1}\to \objX_1$
 be an Eilenberg-Moore object for $A_1$.
There is then a morphism  $M^A \defeq \EM(M,\alpha)\co \objX_0^{A_0}\to 
 \objX_1^{A_1}$ in the bicategory $ \bcatE$
 together with an isomorphism of left $A_1$-modules 
 $\sigma\co  U_1\circ M^{A} \to M\circ U_0$,
 \[
 \xymatrix{
\objX_0^{A_0}\ar[d]_{U_0}\ar[r]^{M^A} \ar@{}[dr]|{\Downarrow \, \sigma} & \objX_1^{A_1} \ar[d]^{U_1} \\
\objX_0 \ar[r]_{M} & \objX_1.
}
\]
It is easy to verify that the morphism $(U_0, U_1,\sigma)\co  M^{A} \to M$ 
in $ \bcatE^{(1)}$ is an Eilenberg-Moore object
for the monad $A$.
 \end{proof}

\begin{lemma}\label{restriction1-cell} 
The restriction functor
$\Res(\Phi,\Psi)\co [\Phi,\Psi]\to [\Phi\mid  \bcatE, \Psi \mid\bcatE]$
is surjective on objects
for any pair of tame  homomorphisms 
$\Phi,\Psi\co \Bim(\bcatE) \to  \bcatF$.
\end{lemma}

\begin{proof} Let us show that any pseudo-natural transformation $M\co \Phi_{\mid  \bcatE} \to \Psi_{\mid\bcatE}$
admits an extension $M'\co \Phi \to \Psi$.
The pseudo-natural transformation $M$ is defining a homomorphism
$M\co  \bcatE \to  \bcatF^{(1)}$.
The homomorphism $M$ is proper, since the homomorphisms $s_0(M)=\Phi_{\mid  \bcatE}$
and $t_0(M)=\Psi_{\mid  \bcatE}$ are proper.
It follows by the first part of the proof of Theorem~\ref{thm:appcompletiontheorem}  
that the homomorphism $M$ can be extended
as a tame homomorphism $M'\co \Bim(\bcatE) \to  \bcatF^{(1)}$.
It remains to show that $M'$ can be chosen so that $s_0(M') =\Phi$
and $t_0(M')=\Psi$.
\[
\xymatrix{
\bcatE  \ar[rr]^{M}\ar[d]_{  \JE }&& \bcatF^{(1)} \ar[d]^{(s_0,t_0)} \\
\Bim(\bcatE) \ar[rru]^{M'}  \ar[rr]_{(\Phi,\Psi)}&& \bcatF  \times \bcatF  }
\]
The homomorphism $M\co  \bcatE \to  \bcatF^{(1)}$ sends 
$\objX \in \bcatE$ to a morphism $M(\objX)\co \Phi (\objX)\to \Psi (\objX)$ in $\bcatF$
and a monad~$A$ on~$\objX$ in $\bcatE$, to a monad~$M(A)$ on~$(M(\objX)$ 
in~$\bcatF^{(1)}$. Thus,
\[
M(A)=(\Phi(A), \Phi(\mu), \Phi( \eta), \Psi(A), \Psi(\mu), \Psi(\eta), \alpha) \, , 
\]
where  $\Phi(A)=(\Phi(A), \Phi(\mu), \Phi( \eta))$ is a monad on $\Phi(\objX)$,
 $\Psi(A)=( \Psi(A), \Psi(\mu), \Psi(\eta))$  is a monad on $\Psi(\objX)$
and $\alpha$ is an invertible 2-cell 
\[
\alpha\co  \Psi(A)\circ M(\objX)\to  M(\objX) \circ \Phi(A)
\]
defining a monad morphism $(M(\objX),\alpha)\co (\Phi (\objX),\Phi (A))\to
(\Psi (\objX),\Psi (A)) $,
\[
\xymatrix{
\Phi(\objX)\ar[d]_{M(\objX)}\ar[r]^{\Phi(A)}   \ar@{}[dr]|{\Downarrow \, \alpha} & \Phi(\objX) \ar[d]^{M(\objX)} \\
\Psi(\objX) \ar[r]_{\Psi(A)}& \Psi(\objX) \, . } 
\]
The morphism $A\co \objX/A \to \objX/1_\objX$ in $\Bim(\bcatE)$
is an Eilenberg-Moore object for the monad $A \co \objX \to \objX$
by Lemma~\ref{thm:adjunctionfrommonad5}. Hence, the morphism
$\Phi(A)\co \Phi(\objX/A) \to  \Phi( \objX)$
is an Eilenberg-Moore object for the monad $\Phi(A)$ on the object $\Phi (\objX)$,
since the homomorphism~$\Phi$ is proper.
Similarly,  the morphism~$\Psi(A)\co \Psi(\objX/A) \to  \Psi( \objX)$
is an Eilenberg-Moore object for the monad $\Psi(A)$,
since the homomorphism~$\Psi$ is proper. 
The homomorphism~$M' \co \bimE \to \bcatF^{(1)}$ can be constructed by choosing for each $\objX/A \in  \Bim(\bcatE)$
 a morphism $M'(\objX/A)\co \Phi(\objX/A)\to \Psi(\objX/A)$
together with an isomorphism of left $\Psi(A)$-modules
$\sigma_A\co \Psi(A)\circ M(\objX/A)\to M(\objX)\circ \Phi(A)$,
\[
\xymatrix@C=1.5cm{
\Phi(\objX/A)\ar[r]^{ M'(\objX/A)}\ar[d]_{\Phi(A)}  \ar@{}[dr]|{\Downarrow \, \sigma_A} & \Psi(\objX/A) \ar[d]^{\Psi(A)} \\
\Phi(\objX) \ar[r]_{M(\objX)} & \Psi(\objX) }
\]
We then have $s_0(M')=\Phi$ and   $t_0(M')=\Psi$.
If $A=1_\objX$, we can choose $M(\objX/A)=M(\objX)$
and $\sigma_A=\id$. In which case we have $M'_{\mid \bcatE}=M$, as required.
\end{proof}

We now recall the definition of the bicategory of 2-cells $ \bcatE^{(2)}$ of the bicategory $ \bcatE$.
The bicategory $ \bcatE^{(2)}$ is equipped with a universal modification $\alpha\co s_1\to t_1$
between a pair of pseudo-natural transformations
\[
\xymatrix{ 
 \bcatE^{(2)} \ar@<1.0ex>[rr]^-{s_1 } \ar@<-1.0ex>[rr]_-{t_1}  && \bcatE^{(1)} }
\]
such that
 \[
s_0 \, s_1=s_0 \, t_1 \, , \quad t_0 \, s_1  = t_0 \, t_1 \, .
\]
By construction, an object of $\bcatE^{(2)}$ is a 
2-cell $\alpha\co M_0\to M_1$ of the
bicategory $ \bcatE$. We refer to $\alpha$ as a \emph{disk}
and represent it as follows:
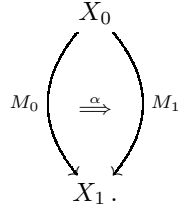
\[
\xymatrix{
\objX_0 \ar@/_1.5pc/[dd]_-{M_0} \ar@{}[dd]|{\stackrel{\alpha}{\Longrightarrow}}
\ar@/^1.5pc/[dd]^-{M_1}  & \\
&\\
\objX_1 \, .& }
\]
If  $\beta\co N_0 \to N_1\co \objY_0\to \objY_1$ is another disk,
then a  1-cell $(E_0, E_1,\gamma_0, \gamma_1) \co \alpha \to \beta$ in  $\bcatE^{(2)}$
is a 4-tuple where $\gamma_0\co E_1\circ M_0\to N_0\circ E_0$
and $\gamma_1\co E_1\circ M_1\to N_1\circ E_0$ are invertible 2-cells
such that the following square commutes,
\[
\xymatrix{
E_1\circ M_0\ar[d]_{E_1\circ \alpha}  \ar[r]^{\gamma_0}& N_0\circ E_0\ar[d]^{\beta \circ E_0} \\
E_1 \circ M_1\ar[r]_{ \gamma_1 }&N_1\circ E_0 \, . }
\]
We represent such a 2-cell by  a cylinder
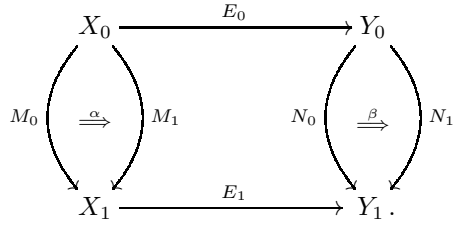
\[
\xymatrix{
\objX_0 \ar@/_1.5pc/[dd]_-{M_0} \ar@{}[dd]|{\stackrel{\alpha}{\Longrightarrow}} \ar[rrr]^{E_0}
\ar@/^1.5pc/[dd]^-{M_1}
&&&\objY_0 \ar@/_1.5pc/[dd]_-{N_0} \ar@{}[dd]|{\stackrel{\beta}{\Longrightarrow}} \ar@/^1.5pc/[dd]^-{N_1} \\
&&&\\
\objX_1\ar[rrr]^{E_1}&&& \; \objY_1 \, . } 
\]
Composition of 1-cells in $\bcatE^{(2)}$  is defined  by pasting cylinders.
A 2-cell 
\[
(\sigma_0, \sigma_1) \co (E_0, E_1,\gamma_0, \gamma_1) \to (F_0, F_1,\delta_0, \delta_1)
\] 
in $\bcatE^{(2)}$ is a pair of 2-cells $\varepsilon_0\co E_0\to F_0$ and  $\varepsilon_1\co E_1\to F_1$ such that
 the following two squares commute:
\[
\xymatrix@C=1.2cm{
E_1\circ M_0\ar[d]_{\varepsilon_1\circ M_0}  \ar[r]^{\gamma_0}& N_0\circ E_0\ar[d]^{N_0\circ \varepsilon_0} \\
F_1 \circ M_0\ar[r]_{\delta_0 }& N_0\circ F_0 \, , 
}\quad \quad  \quad 
\xymatrix@C=1.2cm{
E_1\circ M_1\ar[d]_{\varepsilon_1\circ M_1}  \ar[r]^{\gamma_1}& N_1\circ E_0\ar[d]^{N_1\circ \varepsilon_0} \\
F_1 \circ M_1\ar[r]_{\delta_1 }& N_1\circ F_0 \, .
}
\]
We represent such a 2-cell as a ``deformation'' of cylinders
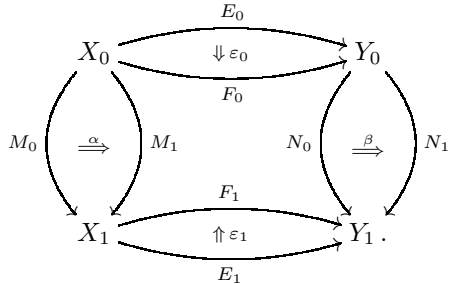
\[
\xymatrix{
\objX_0 \ar@/_1.5pc/[dd]_-{M_0} \ar@{}[dd]|{\stackrel{\alpha}{\Longrightarrow}} \ar@/^.8pc/[rrr]^-{E_0} 
\ar@{}[rrr]|{ \Downarrow \, \varepsilon_0}  
\ar@/_.8pc/[rrr]_-{F_0}
\ar@/^1.5pc/[dd]^-{M_1}
&&&\objY_0 \ar@/_1.5pc/[dd]_-{N_0}  \ar@{}[dd]|{\stackrel{\beta}{\Longrightarrow}} \ar@/^1.5pc/[dd]^-{N_1} \\
&&&\\
\objX_1 \ar@/^.8pc/[rrr]^-{F_1}
\ar@/_.8pc/[rrr]_-{E_1}  \ar@{}[rrr]|{\Uparrow \, \varepsilon_1}  &&&\objY_1 \, .
}
\]
Composition of 2-cells in $\bcatE^{(1)}$  is defined  component-wise.
The source and target
homomorphisms
\[
\xymatrix{ 
\bcatE^{(2)}\ar@<0.8ex>[rr]^-{s_1 } \ar@<-0.8ex>[rr]_-{t_1}  &&  \bcatE^{(1)}
}
\]
are preserving composition strictly.

\begin{lemma}\label{tameit2cell}  
If $\bcatE$ is tame, then so is its bicategory of $2$-cells $\bcatE^{(2)}$
 \end{lemma}

\begin{proof} Similar to the proof of Lemma~\ref{tameit1cell}.
\end{proof}

A monad on an object  $\gamma\co M\to N\co \objX_0\to \objX_1$ of the
 bicategory $ \bcatE^{(2)}$ is an 8-tuple 
 \[
 A=(A_0,\mu_0,\eta_0, A_1, \mu_1,\eta_1,\alpha, \beta) \, ,
 \]
where $(A_0,\mu_0,\eta_0)$ is a monad on $\objX_0$,
$(A_1,\mu_1,\eta_1)$ is a monad on $\objX_1$,
 $\alpha\co A_1\circ M \to M\circ A_0$  and $\beta\co A_1\circ N \to N\circ A_0$
are invertible 2-cells such that $(M, \alpha)$ and $(N, \beta)$ are monad morphisms and the following diagram commutes,
\[
\xymatrix{
 A_1\circ M  \ar[d]_{A_1\circ \gamma }\ar[r]^{\alpha}&M\circ A_0 \ar[d]^{\gamma\circ A_0} \\
A_1\circ N \ar[r]_{\beta}& N\circ A_0.
}
\]
It follows that  we have a monad 2-cell $ \gamma \co (M,\alpha)\to (N,\beta)$ in $\mathrm{Mnd}(\bcatE)$,
\[
\xymatrix{
\objX_0 \ar@/_1.5pc/[dd]_-{M} \ar@{}[dd]|{\stackrel{\gamma}{\Longrightarrow}} \ar[rrr]^{A_0}
\ar@/^1.5pc/[dd]^-{N}
&&&\objX_0 \ar@/_1.5pc/[dd]_-{M} \ar@{}[dd]|{\stackrel{\gamma}{\Longrightarrow}} \ar@/^1.5pc/[dd]^-{N} \\
&&&\\
\objX_1\ar[rrr]^{A_1}&&&\objX_1 \, .}
\]

\begin{lemma} If $\bcatE$ is Eilenberg-Moore complete, then so is 
its bicategory of $2$-cells $\bcatE^{(2)}$.
\end{lemma}

\begin{proof}  Let us show that every monad
in $ \bcatE^{(2)}$ admits an Eilenberg-Moore object. 
We will use the homomorphism $\EM\co  \Mnd(\bcatE)\to \bcatE$.
Let $A=(A_0,\mu_0,\eta_0, A_1, \mu_1,\eta_1,\alpha, \beta)$
be a monad on an object  $\gamma\co M\to N$ of $\bcatE^{(2)}$,
where $M \, N \co \objX_0 \to \objX_1$. 
Then,  $\alpha\co A_1\circ M \to M\circ A_0$ 
determines a monad morphism $(M,\alpha)\co A_0\to A_1$
 in~$\mathrm{Mnd}(\bcatE)$,  
 $\beta\co A_1\circ N \to N\circ A_0$
 determines a monad morphism $(N,\beta)\co A_0\to A_1$,
 and  $\gamma\co  M \rightarrow  N$  determines a
monad 2-cell $( M, \alpha) \rightarrow~(N, \beta)$.

Let $U_0\co \objX_0^{A_0}\to \objX_0$
 be an Eilenberg-Moore object for $A_0$
and $U_1\co \objX_1^{A_1}\to \objX_1$
 be an Eilenberg-Moore object for $A_1$.
There then is a morphism  $M^A \co \objX_0^{A_0}\to 
 \objX_1^{A_1}$ in the bicategory $ \bcatE$, defined by letting $M^A \defeq \EM(M,\alpha)$,
 together with an isomorphism of left $A_1$-modules 
\[
\xymatrix{
\objX_0^{A_0}\ar[d]_{U_0}\ar[rr]^{M^A} \ar@{}[drr]|{\Downarrow \, \sigma_M} &&  \objX_1^{A_1} \ar[d]^{U_1} \\
\objX_0 \ar[rr]_{M}&& \objX_1 \, . }
\]
Similarly,  there is a morphism  $N^A \co \objX_0^{A_0}\to 
 \objX_1^{A_1}$ in~$\bcatE$
 together with an isomorphism of left $A_1$-modules 
 \[
 \xymatrix{
\objX_0^{A_0}\ar[d]_{R_0}\ar[rr]^{N^A} \ar@{}[drr]|{\Downarrow \, \sigma_N} &&  \objX_1^{A_1} \ar[d]^{R_1} \\
\objX_0 \ar[rr]_{N}&& \objX_1 \, . }
\] 
If we define $\gamma^A \co  M^A\to N^A$ by letting $\gamma ^A \defeq \EM(\gamma)$, then the following 
square commutes 
\[
\xymatrix{
U_1 \circ M^A \ar[d]_{R_1 \circ \gamma^A}  \ar[r]^{\sigma_M }& M \circ R_0 \ar[d]^{\gamma \circ R_0} \\
U_1 \circ N\ar[r]^{ \sigma_N}& N \circ R_0 \, . }
\]
This means that the following cylinder in $\bcatE$ commutes:
\[
\xymatrix{
\objX_0^{A_0} \ar@/_1.5pc/[dd]_-{M^A} \ar@{}[dd]|{\stackrel{\gamma^A}{\Longrightarrow}} \ar[rrr]^{U_0}
\ar@/^1.5pc/[dd]^-{N^A}
&&&\objX_0 \ar@/_1.5pc/[dd]_-{M} \ar@{}[dd]|{\stackrel{\gamma}{\Longrightarrow}} \ar@/^1.5pc/[dd]^-{N} \\
&&&\\
\objX_1^{A_1} \ar[rrr]_{U_1}&&&\objX_1 \, .}
\]
It is easy to verify that the morphism $(U_0, U_1,\sigma_M, \sigma_N)\co \gamma^A\to \gamma$ 
in $ \bcatE^{(2)}$ is an Eilenberg-Moore object for the monad $A$.
\end{proof}

\begin{lemma}\label{restriction2-cell} 
The restriction functor
$\Res(\Phi,\Psi)\co [\Phi,\Psi]\to [\Phi_{\mid  \bcatE}, \Psi_{\mid\bcatE}]$
is full and faithful
for any pair of tame  homomorphisms 
$\Phi,\Psi\co \Bim(\bcatE) \to  \bcatF$.
\end{lemma}

\begin{proof} If $M,N\co \Phi \to \Psi$ are pseudo-natural transformations,
let us show that  every modification $\alpha\co M_{\mid \bcatE} \to N_{\mid \bcatE}$
admits a unique extension $\alpha'\co M \to N$.
For every $\objX/A \in \bcatE$, we have an open cylinder,
\[
\xymatrix{
\Phi(\objX/A)  \ar@/_1.5pc/[dd]_-{M(\objX/A)}^{} \ar[rrr]^{\Phi(A)}
\ar@/^1.5pc/[dd]^-{N(\objX/A)}
&&&\Phi(\objX). \ar@/_1.5pc/[dd]_-{M(\objX)}
\ar@{}[dd]|{\stackrel{\alpha(\objX)}{\Longrightarrow}}
\ar@/^1.5pc/[dd]^-{N(\objX)} \\
&&&\\
\Psi(\objX/A) \ar[rrr]^{\Psi(A)}&&&\Psi(\objX).
}
\]
with back and front faces given by isomorphisms
\[
\sigma_M\co \Psi(A)\circ M(\objX/A)\simeq M(\objX)\circ \Phi(A) \, ,  \quad 
\sigma_N\co \Psi(A)\circ N(\objX/A)\simeq N(\objX)\circ \Phi(A) \, .
\]
The morphism $\Phi(A)\co \Phi(\objX/A) \to  \Phi( \objX)$
is an Eilenberg-Moore object for~$\Phi(A)$
and the morphism $\Psi(A)\co \Psi(\objX/A) \to  \Psi( \objX)$
is an Eilenberg-Moore object for~$\Psi(A)$.
The 2-cell $\alpha(\objX)\co M(\objX) \to N(\objX)$
is a monad 2-cell. It follows that there is a unique 2-cell
$\alpha'(\objX/A)\co M(\objX/A) \to N(\objX/A)$
such that the following square commutes
\[
\xymatrix{
 \Psi(A) \circ M(\objX/A)  \ar[d]_{\Psi(A)\circ \alpha'(\objX/A)}\ar[r]^{\sigma_M}&M(\objX) \circ  \Phi(A) \ar[d]^{\alpha(\objX)\circ \Phi(A)} \\
\Psi(A)\circ N(\objX/A) \ar[r]_{\sigma_N}& N(\objX)\circ \Phi(A) \, .}
\]
We obtain a full cylinder
\[
\xymatrix{
\Phi(\objX/A)  \ar@/_2.pc/[dd]_-{M(\objX/A)}
\ar@{}[dd]|{\stackrel{\alpha'(\objX/A)}{\Longrightarrow}}
 \ar[rrr]^{\Phi(A)}
\ar@/^2.pc/[dd]^-{N(\objX/A)}
&&&\Phi(\objX) \ar@/_2.pc/[dd]_-{M(\objX)}
\ar@{}[dd]|{\stackrel{\alpha(\objX)}{\Longrightarrow}}
\ar@/^2.pc/[dd]^-{N(\objX)} \\
&&&\\
\Psi(\objX/A) \ar[rrr]_{\Psi(A)}&&&\Psi(\objX).
}
\]
This defines the extension $\alpha'\co M \to N$
of the modification $\alpha\co M_{\mid \bcatE} \to N_{\mid \bcatE}$.
The uniqueness of $\alpha'$ is clear from the construction.
\end{proof} 

\backmatter

\end{document}